\documentclass[10pt,reqno,a4paper,oneside]{article}

\usepackage{breakurl}             
\usepackage{underscore}           

\usepackage[english]{babel}

\usepackage[a4paper]{geometry}

\usepackage{amssymb}
\usepackage{amsthm}
\usepackage{amsmath}
\usepackage{textcmds}
\usepackage{listings}
\usepackage{comment}
\usepackage{tikz-cd} 
\usetikzlibrary{decorations.markings}
\usetikzlibrary{calc}
\usepackage[ruled,vlined]{algorithm2e}
\usepackage{xcolor}
\usepackage{geometry}
\usepackage[percent]{overpic}
\usepackage{wrapfig}
\usepackage{array}
\usepackage{graphicx}
\usepackage{tabularx}

\definecolor{customblue}{HTML}{166dde}
\definecolor{customgrey}{HTML}{d3d3d3}
\definecolor{customred}{HTML}{e32636}

\newcommand{\boldblue}[1]{\textcolor{customblue}{\textbf{#1}}}
\newcommand{\boldred}[1]{\textcolor{customred}{\textbf{#1}}}

\usepackage[colorlinks=true, allcolors = customblue]{hyperref}


\usepackage{titling}
\usepackage[skip=10pt]{parskip}
\usepackage{subfig}
\usepackage{caption}
\usepackage{thmtools}
\usepackage{thm-restate}
\usepackage{tikz}
\usepackage{setspace}
\usepackage[framemethod=tikz]{mdframed}
\usepackage[export]{adjustbox}
\usepackage{makecell}

\newcounter{computationalnote}
\newmdenv[
  linewidth=0.5pt,
  linecolor=customred,
  backgroundcolor=white,
  skipabove=10pt,
  skipbelow=10pt,
  innerleftmargin=8pt,
  innerrightmargin=8pt,
  innertopmargin=6pt,
  innerbottommargin=6pt,
  frametitlefont=\bfseries\color{customred},
  frametitle={Computational Note~\thecomputationalnote},
  settings={\refstepcounter{computationalnote}} 
]{computationalnote}
\newcommand{\compnote}[1]{%
  {\hypersetup{linkcolor=customred}\hyperref[#1]{Computational Note~\ref*{#1}}}%
}

\newcounter{relatedwork}
\newmdenv[
  linewidth=0.5pt,
  linecolor=customblue,
  backgroundcolor=white,
  skipabove=10pt,
  skipbelow=10pt,
  innerleftmargin=8pt,
  innerrightmargin=8pt,
  innertopmargin=6pt,
  innerbottommargin=6pt,
  frametitlefont=\bfseries\color{customblue},
  frametitle={Related work},
  settings={\refstepcounter{relatedwork}} 
]{relatedwork}

\newmdenv[
  linewidth=0.5pt,
  linecolor=customblue,        
  backgroundcolor=white,      
  skipabove=10pt,
  skipbelow=10pt,
  innerleftmargin=8pt,
  innerrightmargin=8pt,
  innertopmargin=6pt,
  innerbottommargin=6pt,
  frametitlefont=\bfseries\color{customblue},
  frametitle={Main contributions} 
]{contributions}

\newmdenv[
  linewidth=0.5pt,
  linecolor=customred,        
  backgroundcolor=white,      
  skipabove=10pt,
  skipbelow=10pt,
  innerleftmargin=8pt,
  innerrightmargin=8pt,
  innertopmargin=6pt,
  innerbottommargin=6pt,
  frametitlefont=\bfseries\color{customred},
  frametitle={Code availability} 
]{code}

\usepackage{titlesec}

\makeatletter
\let\@internalcite\cite
\def\cite{\def\citeauthoryear##1##2{##1, ##2}\@internalcite}
\def\shortcite{\def\citeauthoryear##1{##2}\@internalcite}
\def\@biblabel#1{\def\citeauthoryear##1##2{##1, ##2}[#1]\hfill}
\makeatother

\newcommand{\R}{\mathbb{R}}
\newcommand{\C}{\mathbb{C}}
\newcommand{\Z}{\mathbb{Z}}

\newcommand{\N}{\mathbb{N}}

\newcommand{\E}{\mathbb{E}}

\newcommand{\M}{\mathcal{M}}
\newcommand{\A}{\mathcal{A}}
\DeclareMathOperator{\sign}{sign}
\DeclareMathOperator{\diag}{diag}
\newcommand\inp[2]{\langle #1, #2 \rangle}

\newcommand{\inv}{^{-1}}

\newcommand{\ord}{\mathcal{O}}

\newcommand{\thupper}{^{\text{th}}}

\newcommand{\p}[1]{\phi_{#1}}

\newtheorem{theorem}{Theorem}[section]
\newtheorem{corollary}[theorem]{Corollary}
\newtheorem{lemma}[theorem]{Lemma}
\newtheorem{prop}[theorem]{Proposition}

\theoremstyle{definition}
\newtheorem{definition}[theorem]{Definition}

\theoremstyle{definition}

\theoremstyle{definition}
\newtheorem{example}[theorem]{Example}
 
\theoremstyle{remark}

\title{Computing Diffusion Geometry}
\author{
Iolo Jones 
\\ University of Oxford
\and
David Lanners \\ Durham University}
\date{February 2026}

\begin{document}

\maketitle

\begin{abstract}
\noindent
Calculus and geometry are ubiquitous in the theoretical modelling of scientific phenomena, but have historically been very challenging to apply directly to real data as statistics.
Diffusion geometry is a new theory that reformulates classical calculus and geometry in terms of a diffusion process, allowing these theories to generalise beyond manifolds and be computed from data.
This work introduces a new computational framework for diffusion geometry that substantially broadens its practical scope and improves its precision, robustness to noise, and computational complexity.
We present a range of new computational methods, including all the standard objects from vector calculus and Riemannian geometry, and apply them to solve spatial PDEs and vector field flows, find geodesic (intrinsic) distances, curvature, and several new topological tools like de Rham cohomology, circular coordinates, and Morse theory.
These methods are data-driven, scalable, and can exploit highly optimised numerical tools for linear algebra.
\end{abstract}

\begin{figure}[h!]
  \centering
    \begin{overpic}[width=\linewidth,grid=false,
      clip=true, trim=0 400 0 0]{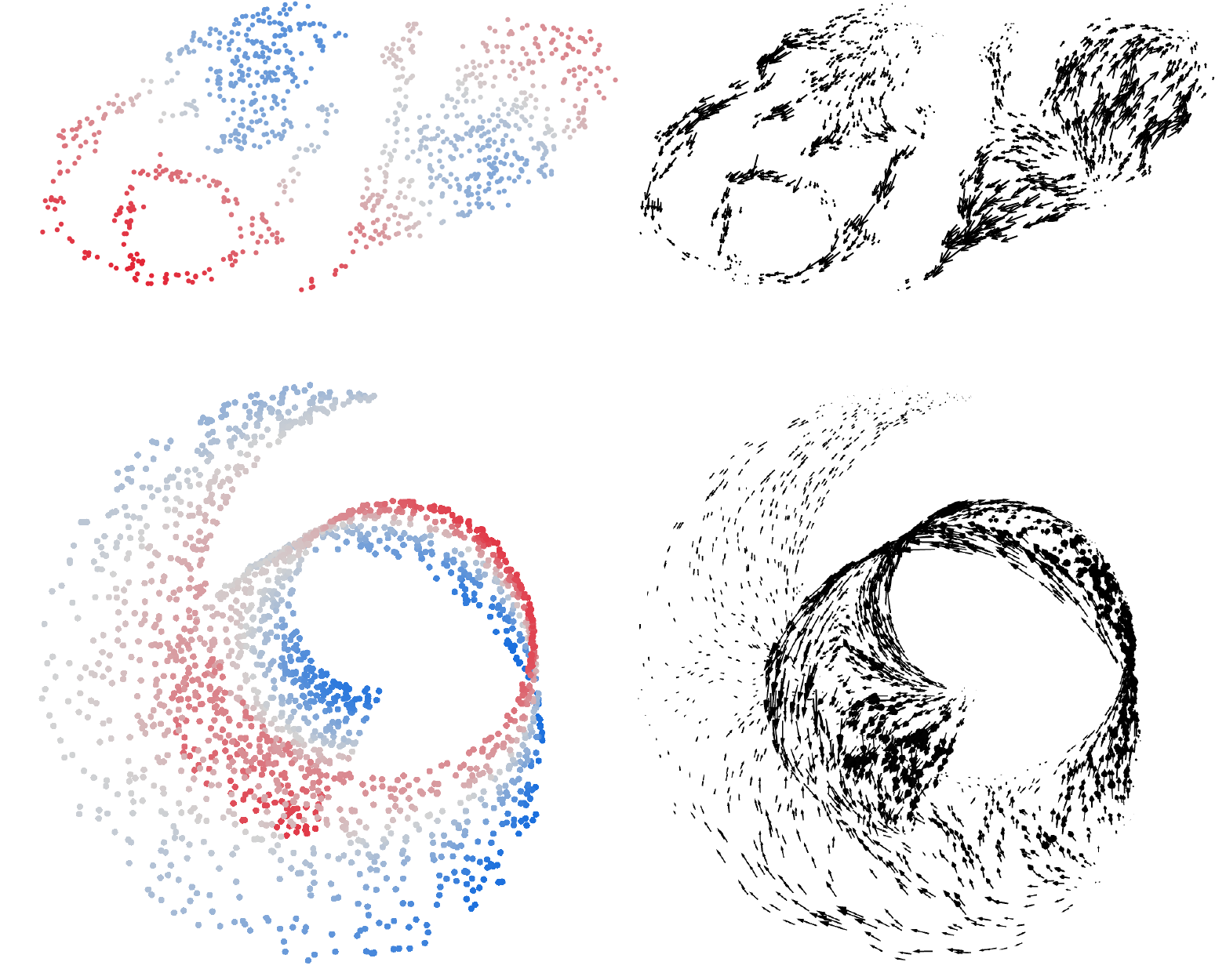}
\put(25.0,0){\makebox(0,0)[c]{$f$}}
\put(75.0,0){\makebox(0,0)[c]{$\nabla f$}}
  \end{overpic}
  \caption{A function $f$ represents a signal on the point cloud (blue denotes negative values, and red positive).
  We use diffusion geometry to compute its \boldblue{gradient vector field $\nabla f$} with respect to the data geometry.
  }
\end{figure}

\vspace{-1.2em}

\begin{figure}[h!]
  \centering
  \begin{overpic}[width=\linewidth,grid=false]{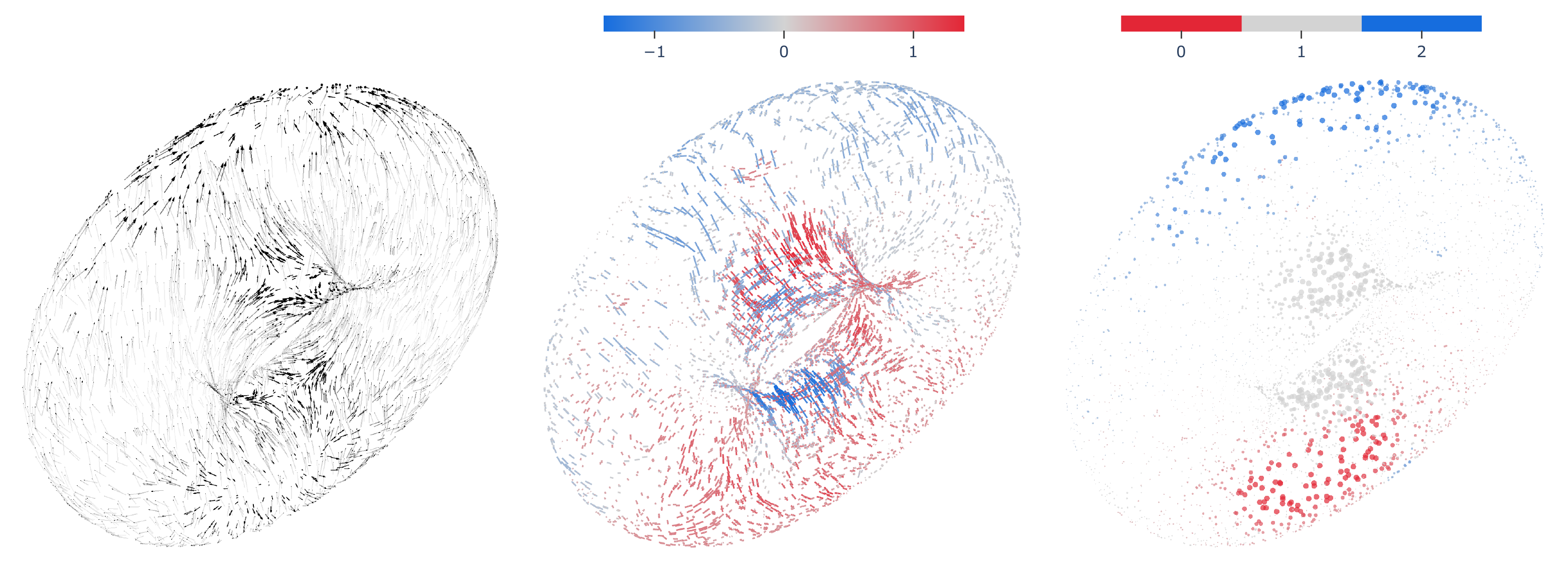}
\put(16.7,-1.5){\makebox(0,0)[c]{$\nabla h$ near critical points}}
\put(50.0,-1.5){\makebox(0,0)[c]{eigenvectors of $H(h)$}}
\put(83.3,-1.5){\makebox(0,0)[c]{indices of critical points}}
  \end{overpic}
  \vspace{0.em}
  \caption{
    The height function $h$ measures elevation on the torus.
    It has a maximum at the top, a minimum at the bottom, and two saddle points in between.
    The gradient $\nabla h$ reveals the flow into and out of these critical points.
    We compute the \boldblue{Hessian $H(h)$}, which measures the expansion and contraction of $h$ in a $2 \times 2$ matrix at each point.
    We plot its eigenvectors, which indicate the directions of expansion, coloured by their eigenvalues, which measure the rate of expansion (so negative implies contraction).
    These tell us the \boldblue{indices of the critical points}: there is expansion in both directions at a minimum (degree 0), contraction in both directions at a maximum (degree 2), and both expansion and contraction at a saddle (degree 1).
  }
\end{figure}

\newpage
\begingroup
  \begin{spacing}{0.9}
    \tableofcontents
  \end{spacing}
\endgroup
\newpage

\section{Introduction}

\subsection{The problem of data-driven calculus}

\boldblue{Calculus} is the mathematical language for describing change.
It is a foundational method in many areas of science, where objects like functions and vector fields describe all sorts of phenomena, and calculus methods like derivatives and differential equations are used to study them.
For example, a wave can be modelled by a function which measures its height at each point, and its movement is described by the wave equation.

Calculus can measure change through time, but also change through space.
When this space is curved, the derivatives will become deformed and affect the behaviour of the calculus.
For example, an electrical wave travelling across the surface of the brain will speed up and slow down as it moves in and out of bumps along the surface.
In this way, the calculus of curved spaces implicitly encodes something about their shape, and the ways in which this geometry can be inferred from calculus are studied in \boldblue{Riemannian geometry}.

Despite the great significance of calculus as a theoretical tool for modelling dynamics and measuring geometry, it has historically been very difficult to apply directly to data.
What happens when we only have a finite sample of points from a space, and want to estimate its calculus?
We might have an MRI scan of a brain and want to study how its shape will affect the propagation of electrical waves across the surface.

There are two major obstacles.
First, calculus involves taking derivatives, which are the infinitesimal rates of change measured by \q{zooming in} infinitely closely to a point.
Unlike in continuous space, finite data are all isolated, and the gaps between them prevent any infinitesimal measurements.
Second, the classical theory of Riemannian geometry only applies when the space is a \boldblue{manifold}.
Manifolds are highly specialised spaces that satisfy very specific constraints.
Most real data do not meet these constraints (as we discuss below), and so, even if we could compute it, the classical Riemannian geometry would rarely apply in practice.


\boldblue{Diffusion geometry} \cite{jones2024diffusion,jones2024manifold} is a new theory that simultaneously overcomes both of these obstacles by reformulating calculus and geometry in terms of the \boldblue{heat diffusion} on the underlying space.
Heat diffusion (as described by the classical heat equation) may appear to be an extremely specific dynamical process and unrelated to the other calculus objects, but, remarkably, it contains enough information to fully reconstruct all the calculus and geometry of the space.

This has two important consequences.
First, the heat diffusion can be expressed using the \boldblue{heat kernel}, which measures how heat spreads from one point to another over time.
We can evaluate the heat kernel on finite data to get a discrete approximation of the diffusion process, which can then be used to compute calculus objects using their reformulation in diffusion geometry.
The heat kernel is easily computable and highly robust to noise, and all our methods will inherit these qualities.
Second, the heat diffusion can be defined on very general spaces, not just manifolds.
This means that diffusion geometry and the methods we derive from it will apply to all real-world data geometries, without requiring special assumptions.

\subsection{Why the \qq{manifold hypothesis} has the wrong name}

A space must satisfy several strict criteria to be a manifold.
It must have a dimension $d$, meaning that, at every point, the space locally \q{looks like} the standard $d$-dimensional space $\R^d$.
As such, it cannot have locally variable dimension (like a line, which is 1d, joined to a sphere, which is 2d) or any tree-like \q{branching points} where different parts of the space split off.
While these criteria are met in special cases, most real-world data have branches and variable dimensions, so are not manifolds (see Figure \ref{fig:intro_fig_gradients}).
On the other hand, a fundamental observation in data science is that real-world data tend to have \boldblue{low-dimensional} structure.
For example, the space of natural images is intrinsically low-dimensional because, given some initial image, there are only certain \q{directions} in which you can change that image and obtain another natural image (you could safely change the colour of the background, but changing pixel values randomly just looks like noise).

Confusingly, the reasonable and broadly applicable assumption that the data are low-dimensional has been mislabelled the \qq{manifold hypothesis}, and the low-dimensional data are called \qq{data manifolds} \cite{cayton2005algorithms}.
In fact, the two properties are unrelated.
For example, three lines that meet at a single point (like a letter \q{Y}) form a 1-dimensional space that is not a manifold, and a 999,999-dimensional sphere in 1,000,000-dimensional space is a manifold, but is not low-dimensional.

The goal of diffusion geometry is to redefine Riemannian geometry, which only exists on actual manifolds, so that the same methods can be applied to the things that people mean when they say \qq{data manifold}.

\subsection{Paper outline}

In \cite{jones2024diffusion}, we computed diffusion geometry objects by expanding on the \textit{Spectral Exterior Calculus} framework \cite{berry2020spectral}, in which Riemannian geometry objects are computed from the eigenvalues and eigenfunctions of the Laplacian.
However, there are several practical drawbacks to this approach that we discuss in Section \ref{sec: background}.
This paper introduces a new computational framework based on a simple and general recipe directly involving the \textit{carré du champ} operator \cite{bakry2014analysis} from the theory of diffusion geometry.
This leads to a much simplified mapping from theory to computation, as well as improved performance and scalability.
It provides a general, robust approach to data-driven Riemannian geometry, which can be exploited for statistics and machine learning.

In Section \ref{sec: overview}, we briefly illustrate the main ideas of diffusion geometry and demonstrate our computational approach by computing gradient vector fields on point clouds.
The central objects of calculus and geometry are functions, vector fields, differential forms, and tensors: in Section \ref{sec: functions_vector_fields_forms_tensors}, we describe a framework in which they can be represented and manipulated.
In Section \ref{sec: frame_theory_weak_formulations}, we provide a theoretical justification for the approach, showing that it is well-posed and numerically stable.
The interesting information in calculus and geometry is captured by differential operators that map between these spaces, such as the exterior derivative and Lie bracket, and these are described in Section \ref{sec: differential_operators}.

We can use this framework to compute the calculus, geometry, and topology of point clouds with diffusion geometry.
To demonstrate the potential of the Riemannian geometry toolkit in statistics and machine learning, we illustrate a range of possible applications and methods.
In Section \ref{sec: differential_equations}, we describe methods for solving partial differential equations, including the heat and wave equations, and vector field flows.
In Section \ref{sec: geometric data analysis}, we present new geometric methods like finding the geodesic (intrinsic) distances and curvature.
In Section \ref{sec: TDA}, we introduce methods for topological data analysis based on differential topology, including de Rham cohomology, circular coordinates and Morse theory.
These applications all deliver promising results, and we leave their thorough development and evaluation for future work.

In Section \ref{sec: computational_complexity}, we analyse the computational complexity and scaling laws of our framework on a topological data analysis task, where it outperforms comparable methods like persistent homology by several orders of magnitude.
We discuss the historical background and related work in Section \ref{sec: background}.

\begin{contributions}
\vspace{-0.3em}
The theory of diffusion geometry reformulates Riemannian geometry in terms of a Markov process, leading to a simple and more general theory.
In this paper, we show that the same principle applies to computation,
where diffusion geometry directly translates the theories of vector calculus, Riemannian geometry, geometric analysis, and differential topology into simple, general statistics.
\begin{enumerate}
  \item We introduce a new computational framework for representing vector fields, differential forms, and tensors in diffusion geometry by using embedding or immersion coordinates as generators, leading to improved statistical and computational performance.
  These can be computed from any Markov process, and we offer a simple data-driven solution based on the heat kernel.
  \item We represent all the major differential operators from calculus and geometry by solving weak formulations, and study their stability.
  \item We use these tools to develop a meshless system for differential equations and geometric analysis on arbitrary data geometries.
  \item We solve the geodesic distance equation to recover the intrinsic metric of point cloud data.
  \item We compute the sectional curvature via the Levi-Civita connection, which is the first time a classical curvature tensor has been estimated from non-manifold data.
  \item We introduce novel methods for topological data analysis by applying techniques from differential topology, such as de Rham cohomology via harmonic forms, circular coordinates, and Morse theory.
  These can be computed in several orders of magnitude less time and space than Vietoris-Rips persistent homology, and exhibit significantly better robustness to noise and outliers.
\end{enumerate}
\end{contributions}

\vspace{-1em}
\begin{code}
\vspace{-0.3em}
We release a Python software package at \url{github.com/Iolo-Jones/DiffusionGeometry}.
\end{code}

\newpage
\section{Overview and main concepts}
\label{sec: overview}

To illustrate the main ideas in diffusion geometry, we show how to compute gradient vector fields on a point cloud.
This will demonstrate a simple and very general method that we apply throughout this paper.

\subsection{Vector calculus on $\R^n$ and manifolds}

A fundamental object in vector calculus is a \boldblue{vector field}, which is an assignment of a \q{direction} to each point in space.
We can represent this direction at a point $p \in \R^d$ by the vector $(f_1(p),...,f_d(p)) \in \R^d$, which is expressed by the notation
\begin{equation}
\label{eq: intro vector field eq}
X = \sum_{i=1}^d f_i \nabla x_i.
\end{equation}
The term $\nabla x_i$ refers to the vector field $e_i = (0,...,1,...,0)$, which represents a constant flow with speed 1 in the direction of the $i\thupper$ axis.
In the left column of Figure \ref{fig: vector calculus R2 and manifold}, we 
demonstrate this construction in $\R^2$.

In geometry, we traditionally work with \boldblue{manifolds}, which we can think of as subsets $\M \subseteq \R^d$ that meet certain criteria.
A manifold must have a dimension $d' \leq d$, which means that, at every point $p \in \M$, the space $\M$ locally \q{looks like} $\R^{d'}$.
This is encoded in the \boldblue{tangent space} $T_p \M$, which is a $d'$-dimensional vector space that contains all the infinitesimal directions through which a curve in $\M$ can pass through $p$.
We can define vector fields on $\M$ by projecting them from $\R^d$ onto $T_p \M$ at each point $p$, a process which is called the \boldblue{pullback}.
We can use the same notation as Equation (\ref{eq: intro vector field eq}), where $f_i$ now means a function $\M \to \R$, and $\nabla x_i$ means the projection of $e_i$ to $T_p\M$.
In the right column of Figure \ref{fig: vector calculus R2 and manifold}, we demonstrate the construction of vector fields on a manifold $\M \subset \R^2$ using the pullbacks of functions and vector fields from $\R^2$.

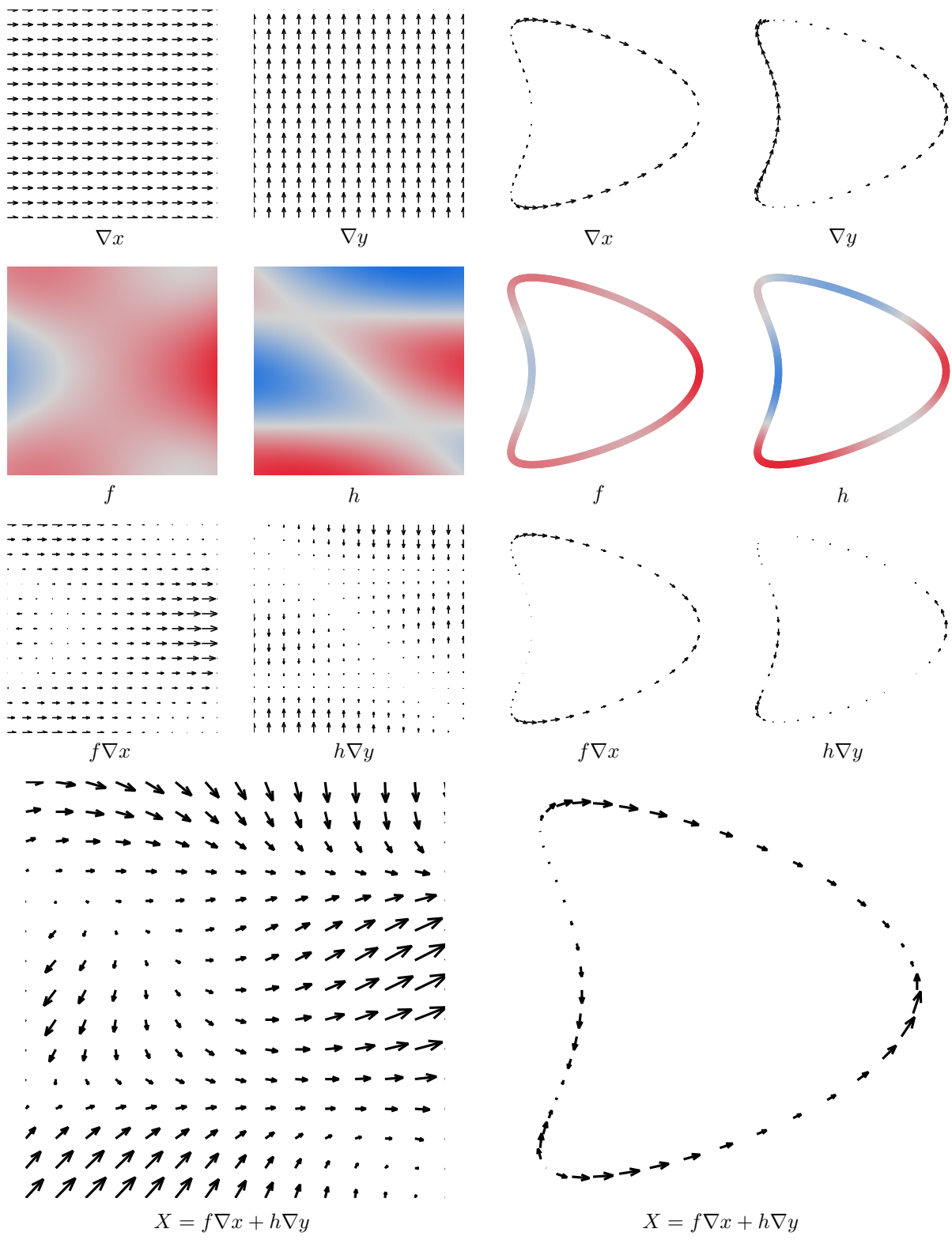
\begin{figure}[]
  \centering
  \begin{overpic}[width=\linewidth,grid=false]{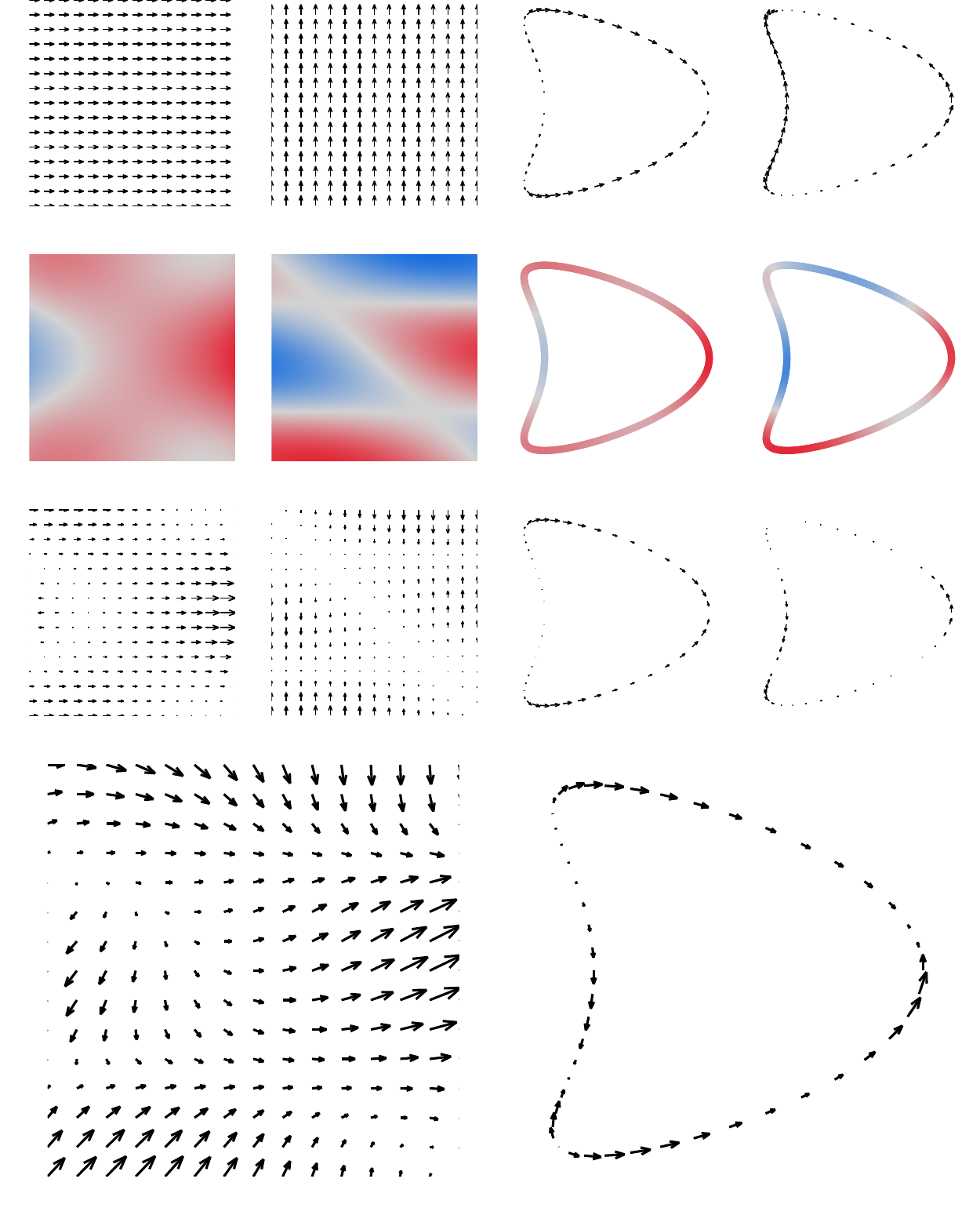}
\put(10.6,81.6){\makebox(0,0)[c]{$\nabla x$}}
\put(30.2,81.6){\makebox(0,0)[c]{$\nabla y$}}
\put(49.8,81.6){\makebox(0,0)[c]{$\nabla x$}}
\put(69.4,81.6){\makebox(0,0)[c]{$\nabla y$}}
\put(10.6,60.8){\makebox(0,0)[c]{$f$}}
\put(30.2,60.8){\makebox(0,0)[c]{$h$}}
\put(49.8,60.8){\makebox(0,0)[c]{$f$}}
\put(69.4,60.8){\makebox(0,0)[c]{$h$}}
\put(10.6,40.0){\makebox(0,0)[c]{$f \nabla x$}}
\put(30.2,40.0){\makebox(0,0)[c]{$h \nabla y$}}
\put(49.8,40.0){\makebox(0,0)[c]{$f \nabla x$}}
\put(69.4,40.0){\makebox(0,0)[c]{$h \nabla y$}}
\put(20.4,2){\makebox(0,0)[c]{$X = f \nabla x + h \nabla y$}}
\put(59.6,2){\makebox(0,0)[c]{$X = f \nabla x + h \nabla y$}}
  \end{overpic}
  \vspace{1em}
  \caption{\textbf{Vector fields on $\R^2$ and a manifold $\M$.}
  In both cases, we construct a vector field $X$ by multiplying $\nabla x$ and $\nabla y$ by coefficient functions $f$ and $h$.
  The functions and vector fields on the right are the pullbacks of those on the left from $\R^2$ to the manifold $\M$ (i.e. we restrict functions to $\M$ and project vectors to the tangent spaces $T_p\M$).}
\label{fig: vector calculus R2 and manifold}
\end{figure}

\subsection{Diffusion geometry on $\R^n$ and manifolds}

The two aims of diffusion geometry are to generalise the above constructions to more general spaces than manifolds, and to do so in a way that can be computed from data.
These goals share a common solution, which is to reformulate the classical theory in terms of a \boldblue{diffusion process}.

In our definition of vector fields on manifolds, we used the pointwise projection of an ambient vector field onto the tangent space $T_p \M$, and this projection is defined using the pointwise inner product.
More generally, the pointwise inner product $X \cdot Y$ of two vector fields (which defines a function) is a fundamental construction in calculus.
This idea is abstracted by the theory of {Riemannian geometry}, which is formulated entirely in terms of the pointwise inner product, called the \boldblue{Riemannian metric} and denoted $g(X,Y)$.
Diffusion geometry generalises this further by noting that, if $X = \sum_{i=1}^d f_i \nabla x_i$ and $Y = \sum_{j=1}^d h_j \nabla x_j$, then
\[
g(X,Y) = \sum_{i,j=1}^d f_i h_j g(\nabla x_i, \nabla x_j),
\]
where the terms $g(\nabla x_i, \nabla x_j)$ are called the \boldblue{carré du champ}~\cite{bakry2014analysis} of $x_i$ and $x_j$, denoted $\Gamma(x_i,x_j)$.
The carré du champ can be expressed in terms of a diffusion process on the manifold $\M$, but, crucially, can also be defined using more general diffusions on more general spaces.
By generalising the diffusion, we can generalise the carré du champ, and so generalise the geometry.

The specific diffusion process that captures the geometry of Euclidean space $\R^d$ and manifolds $\M$ is the \boldblue{heat flow}.
On $\R^d$, it is defined by the \boldblue{heat kernel}
\[
p_t(x, y)
= \frac{1}{(4\pi t)^{d/2}} \exp\!\left(-\frac{\|x - y\|^2}{4t}\right)
\]
for $t > 0$ and $x,y \in \R^d$, which we can interpret as the probability of a Brownian motion transitioning from $x$ to $y$ after $t$ time.
For our purposes, the important geometric information is summarised by the following proposition, which relates the heat kernel to the carré du champ.

\begin{restatable}{reprop}{covarianceeuclideanformula}
\label{prop: covarianceeuclideanformula}
If $f, h : \R^d \to \R$ are differentiable functions then
\begin{equation}
\label{eq: covarianceeuclideanformula}
\begin{split}
\Gamma(f,h)(p)
& = \lim_{t\to 0} \frac{1}{2t} \int p_t(p,y) (f(y) - f(p)) (h(y) - h(p)) \, dy
\\
& = \lim_{t\to 0} \frac{1}{2t} \, \textnormal{Cov}\!\left[f(X), h(X)\,:\, X \sim \mathcal{N}(p,2t\textbf{I})\right].
\end{split}
\end{equation}
\end{restatable}

We state this proposition for $\R^d$ (see Appendix Section \ref{sec:markov-appendix} for a proof), but it holds on manifolds and for general Markov processes \cite{bakry2014analysis}.
By writing the carré du champ as the \boldblue{infinitesimal covariance} of the diffusion process, we obtain a straightforward formula to compute it from data.

\subsection{Computing the carré du champ}

The covariance formula (\ref{eq: covarianceeuclideanformula}) relates the carré du champ $\Gamma$ to the heat kernel.
We can compute a data-driven carré du champ by first computing the kernel $p_t$ on the data and then applying the same formula.
Given $n$ data points in $\R^d$, we evaluate $p_t(x_i,x_j)$ for all $i,j = 1,...,n$, to obtain an $n \times n$ matrix.
We row-normalise to obtain a matrix $\textbf{P}$ whose rows sum to 1, so $\textbf{P}$ defines a discrete diffusion process or \boldblue{Markov chain}.
A function $f$ on the data can be represented by a vector $\textbf{f} \in \R^n$, i.e. $\textbf{f}_i = f(x_i)$, and we can apply (\ref{eq: covarianceeuclideanformula}) to compute the discrete carré du champ $\boldsymbol\Gamma(f,h) \in \R^d$ as
\begin{equation}
\label{eq: cdc covariance computation intro}
\boldsymbol\Gamma_i(f,h) = \frac{1}{2t} \sum_{j=1}^n \textbf{P}_{ij} (\textbf{f}_j - \textbf{f}_i) (\textbf{h}_j - \textbf{h}_i).
\end{equation}

\subsection{Gradient vector fields}

As a simple application of our carré du champ formula, we can compute the \boldblue{gradient} $\nabla f$ of a function $f$, which is a vector field measuring the speed and direction in which $f$ is increasing the fastest.
The gradient is inherently spatial, because it measures the variation of $f$ with respect to the underlying geometry.

We can visualise any vector field $X$ by taking the derivative of each coordinate function $x_i$ in the direction of $X$, and plotting the \q{arrow} given by these derivatives $(X(x_1),..., X(x_d)) \in \R^d$ at each point.
This derivative is given by the Riemannian metric (i.e. pointwise inner product), because $X(h) = g(X, \nabla h)$ for any function $h$.
When the vector field $X$ is $\nabla f$, this is a straightforward computation with the carré du champ
\[
\begin{pmatrix}
\nabla f(x_1) \\
... \\
\nabla f(x_d) \\
\end{pmatrix}
=
\begin{pmatrix}
g(\nabla f, \nabla x_1) \\
... \\
g(\nabla f, \nabla x_d) \\
\end{pmatrix}
=
\begin{pmatrix}
\Gamma(f,x_1) \\
... \\
\Gamma(f,x_d) \\
\end{pmatrix}
\]
which we can compute using (\ref{eq: cdc covariance computation intro}).
So, at the $i\thupper$ point in the data, we get a vector
\[
\big(\boldsymbol\Gamma_i(f,x_1), ... , \boldsymbol\Gamma_i(f,x_d) \big) \in \R^d.
\]
In Figure \ref{fig:intro_fig_gradients}, we verify that this carré du champ gradient computation aligns with the ground truth when computed from a sparse sample of data from a manifold (second and third columns).
However, this method does not assume that the data lie on a manifold.
We can directly apply it to \boldblue{non-manifold} data that has noise, variable density, variable dimension, and singularities, and get meaningful results (fourth column).

\vspace{1em}
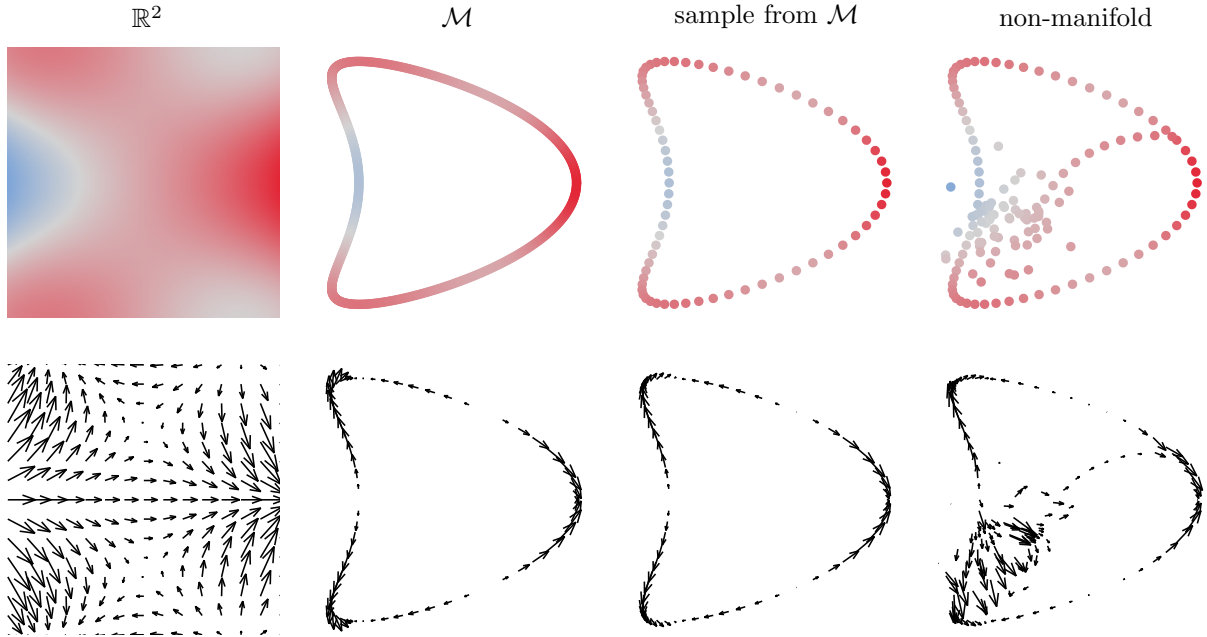
\begin{figure}[h!]
    \centering
      \begin{overpic}[width=\linewidth,grid=false]{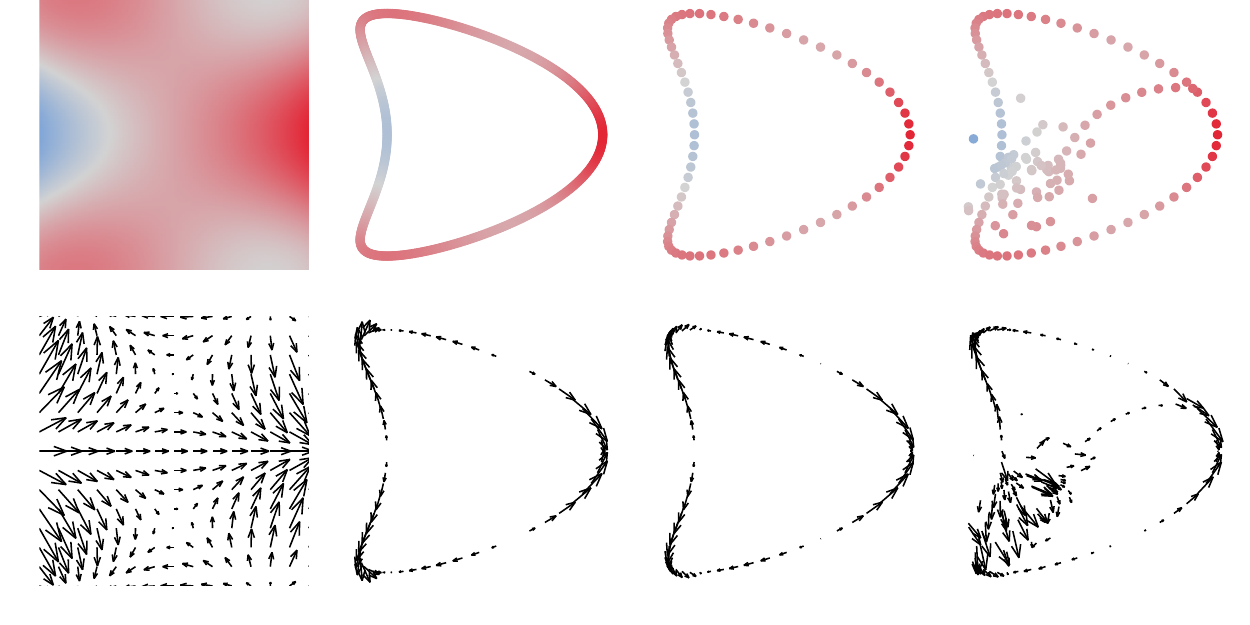}
\put(14.2,52.5){\makebox(0,0)[c]{$\R^2$}}
\put(38.8,52.5){\makebox(0,0)[c]{$\M$}}
\put(63.3,52.5){\makebox(0,0)[c]{sample from $\M$}}
\put(87.8,52.5){\makebox(0,0)[c]{non-manifold}}
  \end{overpic}
    \caption{\textbf{Gradients of functions on manifolds and non-manifolds.}
    The top row shows a function $f$ on different spaces, and the bottom row shows its gradient $\nabla f$.
    The first column is in $\R^2$ and the second column is a 1-dimensional \boldblue{manifold} $\M$, where we can compute $\nabla f$ from $f$ exactly.
    The third column uses diffusion geometry to estimate $\nabla f$ on $\M$ from a sample of data, using the carré du champ formula (\ref{eq: cdc covariance computation intro}).
    The fourth column shows \boldblue{non-manifold} data, where we can still use diffusion geometry to compute a sensible notion of gradient in this more general setting.
    The data is not from a manifold, because it is both 1d and 2d in different regions (2d on the patch on the left, and 1d elsewhere), and has a branching point on the right where three 1-dimensional paths meet.
    }
    \label{fig:intro_fig_gradients}
\end{figure}

\subsection{General recipe for computing diffusion geometry}

The example of computing gradients demonstrates a simple and general procedure for computing calculus objects with diffusion geometry:
\begin{enumerate}
    \item \textbf{Theoretical step.}
    \begin{enumerate}
        \item Find a definition for the object in terms of functions and the \boldblue{carré du champ} $\Gamma$.
    \end{enumerate}
    \item \textbf{Practical steps.}
    \begin{enumerate}
    \item Construct a \boldblue{Markov chain} on the data $(P_{ij})_{1 \leq i,j \leq n}$.
    \item Compute the carré du champ of functions $f,h$ as the \boldblue{covariance}
    \[
    \Gamma(f,h)(p_i) = \text{Cov}(f(X), h(X) : X \sim \mathbf{P}_i).
    \]
    \item Plug these computations into the formula for the object.
    \end{enumerate}
\end{enumerate}
We will now apply this method to introduce a wide array of computational methods for calculus, geometry, and topology.

\newpage

\section{Functions, vector fields, forms, and tensors}
\label{sec: functions_vector_fields_forms_tensors}

The fundamental objects of calculus and geometry are \boldblue{tensors}, such as \boldblue{functions}, \boldblue{vector fields}, and \boldblue{differential forms}.
These can all be computed using the carré du champ $\Gamma$ constructed from a Markov chain on the data, as demonstrated above.




\subsection{Carré du champ and measure from a Markov chain}

The \q{geometry} in diffusion geometry is derived from a \boldblue{diffusion process} on the space $M$, which provides us with a \boldblue{measure} $\mu$ and a \boldblue{carré du champ} operator $\Gamma$.
When we have a finite point cloud $\textbf{M} = \{p_1,...,p_n\} \subset \R^d$, we can approximate the diffusion process on $M$ by constructing a \boldblue{Markov chain} on $\textbf{M}$.
However, the computational methods that follow would work for any choice of measure and carré du champ on $\textbf{M}$, which could be constructed in different ways depending on the application.


\subsubsection{Constructing a Markov chain from a kernel}

An easy way to construct a Markov chain from data is to use a \boldblue{kernel} $k : \R^d \times \R^d \to \R_+$, which measures the \q{affinity} between a pair of points.
A standard choice is the gaussian kernel\footnote{also called the \q{heat kernel} or \q{radial basis function}}
\begin{equation}
\label{eq: fixed bandwidth heat kernel}
    k(x,y) = \exp\Big( -\frac{\parallel x - y \parallel^2}{t^2} \Big)
\end{equation}
for some \boldblue{bandwidth} parameter $t > 0$.
The bandwidth $t$ sets the length scale of the kernel and needs to be tuned so that the Markov chain will encode the local geometry of the data.
When $t$ is too small, the kernel can only \q{see} the point on which it is centred, and the Markov chain becomes the identity matrix, failing to encode any of the relationships between points.
When $t$ is too large, the kernel is too smoothed out, and the Markov transition probabilities become uniform, missing the subtle details.
In practice, it is standard to use a \boldblue{variable bandwidth}\footnote{also called \q{adaptive} or \q{self-tuning}}
kernel
\begin{equation}
\label{eq: variable bandwidth heat kernel}
k(x,y) = \exp\Big( -\frac{\parallel x - y \parallel^2}{\rho(x)\rho(y)} \Big)
\end{equation}
for some \boldblue{bandwidth function} $\rho: \R^d \to \R_+$.
These lead to better performance on variable-density and heterogeneous data, because the bandwidth can be increased at low-density points (where the gaps between data are larger) and decreased at high-density points (to capture the finer details).

Although finding a good bandwidth function $\rho$ is essential for the performance of the kernel, there are simple heuristics that attain very good results, such as setting $\rho(p_i)$ to be the distance to the $k\thupper$ nearest neighbour of $p_i$ (usually $k=8$).
In our implementation, we closely follow \cite{berry2016variable} in constructing a local density estimate $q(x)$, and then set $\rho(x) = q(x)^{b}$ for some negative power $b$.

We can obtain a Markov chain from $k$ by computing the \boldblue{kernel matrix} and then row-normalising:

\begin{enumerate}
    \item Let $\mathbf{K}_{ij} = k(p_i,p_j)$ be the $n \times n$ kernel matrix.
    \item Let $\mathbf{D}_i = \sum_{j=1}^n \mathbf{K}_{ij}$ be the row sums (so $\mathbf{D} \in \R^n$).
    \item Define a Markov chain $\mathbf{P}_{ij} = \mathbf{K}_{ij}/\mathbf{D}_i$.
\end{enumerate}

\begin{computationalnote}
Computing the whole kernel matrix $\mathbf{K}$ takes $\ord(n^2 d)$ time and $\ord(n^2)$ space, which is infeasible for large datasets.
To avoid this, we can compute only the $k$ nearest neighbours of each point (for some small $k$, usually $k=32$), and set $\mathbf{K}_{ij} = 0$ for all other pairs $(i,j)$.
This leads to a sparse kernel matrix with only $\ord(kn)$ nonzero entries, and costs only $\ord(n \log n + d)$ time and $\ord(kn)$ space to compute using a KD-tree.
\end{computationalnote}


\subsubsection{Measure}

The vector $\mathbf{D}$ can be interpreted as an estimate of the density of the underlying measure $\mu$ from which the data points $\mathbf{M}$ were sampled.
For example, if the points in $\mathbf{M}$ are drawn from a distribution with density $q$, then for large $n$,
\[
\frac{\mathbf{D}_i}{n} = \frac{1}{n} \sum_{j=1}^n k(p_i, p_j) \approx \int_M k(p_i, y) q(y) dy
\]
is a \boldblue{kernel density estimate} of the function $q$. 
We can normalise $\mathbf{D}$ into a probability measure (i.e. its entries sum to 1) by setting
\[
\boldsymbol{\mu} = \frac{\mathbf{D}}{\sum_{k=1}^n \mathbf{D}_k} \in \R^n,
\]
which will play the role of the measure $\mu$ in our computations.

\subsubsection{Carré du champ}
\label{sub: cdc}

We can compute a data-driven carré du champ using the covariance formula (\ref{eq: covarianceeuclideanformula}), which acts on functions defined on the data points.
These are represented by vectors $\textbf{f} \in \R^n$, i.e. $\textbf{f}_i = f(x_i)$.
If the kernel has a variable bandwidth $\rho$ as in (\ref{eq: variable bandwidth heat kernel}), then we compute
\begin{equation}
\label{eq: cdc covariance variable bandwidth}
\begin{split}
\boldsymbol\Gamma_i(f,h)
&= \frac{1}{2\rho(x_i)^2} \, \textnormal{Cov}\!\left[f(X), h(X)\,:\, X \sim \textbf{P}_i\right] \\
&= \frac{1}{2\rho(x_i)^2} \sum_{j=1}^n \textbf{P}_{ij} \Big(\textbf{f}_j - \sum_{k=1}^n \textbf{P}_{jk}\textbf{f}_k \Big) \Big(\textbf{h}_j - \sum_{k=1}^n \textbf{P}_{jk}\textbf{h}_k \Big),
\end{split}
\end{equation}
which reduces to (\ref{eq: cdc covariance computation intro}) when $\rho = t$ is constant (\ref{eq: fixed bandwidth heat kernel}).
In \cite{bamberger2025carr}, we found that this \q{covariance} formulation of the carré du champ led to scalable and robust results on high-dimensional data.

\newpage
\begin{computationalnote}
The summand in the covariance formula \eqref{eq: cdc covariance variable bandwidth} could be replaced by the simpler $\textbf{P}_{ij} (\textbf{f}_j - \textbf{f}_i) (\textbf{h}_j - \textbf{h}_i)$ that we used in \eqref{eq: cdc covariance computation intro}.
This would be equivalent to computing the carré du champ by the formula $\Gamma_i(f,h) = \frac{1}{2} (f \mathbf{L}(h) + h \mathbf{L}(f) - \mathbf{L}(fh))$ where $\mathbf{L} = \text{diag}(\rho(x_i)^{-2})(\mathbf{P} - \mathbf{I})$ is the generator of the Markov chain.
This approach is equivalent to computing the carré du champ with a graph Laplacian, which has previously been applied in \cite{lin2010ricci, lin2011ricci, berry2020spectral, jones2024diffusion}.
However, we find that using the mean-centred covariance gives slightly better robustness to noise and more precise results on manifold data.
\end{computationalnote}


\subsection{Functions}

In the theory of diffusion geometry \cite{jones2024diffusion}, the most fundamental objects are the measure space $M$ and the \boldblue{functions} $f: M \to \R$.
We call the space of functions $\A$, and $\A$ is an \boldblue{algebra}: it is a vector space that also has a notion of {multiplication} (because, if $f,h \in \A$ are functions, then so is $fh \in \A$).

Given a finite dataset $\textbf{M}$, we could define a discrete function space $\textbf{A}$ as the space of all functions on $\textbf{M}$.
We would then identify $\textbf{A} = \R^n$, where a function $f : \textbf{M} \to \R$ corresponds to the vector $(f(p_1),...,f(p_n)) \in \R^n$, and function multiplication is just componentwise multiplication of vectors.
However, while $\R^n$ contains all the functions that can possibly be represented on the data, it has dimension $n$, and so the time and memory complexity of downstream computations will gain a compounding factor of at least $\ord(n)$.
To improve this, we would like to restrict to a lower-dimensional subspace.

\subsubsection{Compressed function space}
\label{sub: function space}

There are many choices of subspace, and the properties of the functions they contain represent an \textit{inductive bias}.
A natural and general choice is to pick a subspace of the smoothest possible functions, partly because smoothness is an often desirable property, and partly because the smoothness of the function controls the numerical error of the carré du champ (see Theorem 3.1 in \cite{jones2024manifold}).
By restricting to the smoothest possible subspace, we are therefore considering only the functions on which our methods are the most precise.

The notion of \q{smoothness} we use will be data-driven: we can use the Markov chain $\mathbf{P}$ to define an energy functional
\begin{equation}
E(\textbf{f}) 
:= \frac{\inp{\textbf{f}}{\mathbf{P}\textbf{f}}_{L^2(\R^n, \boldsymbol{\mu})} }{\|\textbf{f}\|_{L^2(\R^n, \boldsymbol{\mu})}^2}
= \frac{\sum_{i,j=1}^n \textbf{f}_i \textbf{f}_j \mathbf{P}_{ij}\boldsymbol{\mu}_i}{\sum_{i=1}^n \textbf{f}_i^2\boldsymbol{\mu}_i},
\end{equation}
where $0 \leq E(f) \leq 1$\footnote{the upper bound is because $\mathbf{P}$ is a Markov chain so its largest eigenvalue is 1} measures the smoothness of $f$ with respect to $P$.
Since $P$ is a diffusion operator, $1 - E$ is a discrete analogue of the \boldblue{Dirichlet energy}.
If $n_0$ is our desired subspace dimension, the $n_0$ orthonormal functions in $L^2(\R^n, \boldsymbol{\mu})$ that maximise $E$ are precisely the first $n_0$ eigenvectors of the $n \times n$ matrix $P$.
The first eigenfunction $\phi_1$ will equal 1 everywhere, and the other $\phi_i$ will get progressively more oscillatory as $i$ increases.
See the left column of Figures \ref{fig:function-vf-basis-2d} and \ref{fig:function-vf-basis-3d}.

Since $P$ was constructed from a symmetric kernel, it is self-adjoint with respect to the inner product on $L^2(\R^n, \boldsymbol{\mu})$\footnote{see appendix for details}.
This makes $E$ a symmetric quadratic form which can be efficiently diagonalised, and the eigenvectors $\p{1},...,\p{n}$ of $P$ form an orthonormal basis of $L^2(\R^n, \boldsymbol{\mu})$.
We define the \boldblue{compressed function space}
\[
\textbf{A} = \text{Span}\{\p{1},...,\p{n_0}\} \subseteq L^2(\R^n, \boldsymbol{\mu}),
\]
and store the $n_0$ basis functions in an $n \times n_0$ matrix $\textbf{U}$.
Since the eigenfunctions $\p{i}$ minimise a Dirichlet energy, we can think of $\textbf{A}$ as a space of \boldblue{bandlimited functions} on the data.
The fact that the functions are orthonormal with respect to $\boldsymbol{\mu}$ means that $\textbf{U}^T \diag(\boldsymbol{\mu}) \textbf{U} = \textbf{I}_{n_0}$. 

\begin{computationalnote}
Computing the first $n_0$ eigenvectors of $P$ is a symmetric eigenvalue problem, which can be solved efficiently using iterative methods.
Since we construct a sparse $P$, this costs just $\ord(n n_0)$ time and $\ord(n n_0)$ space.
\end{computationalnote}

\begin{relatedwork}
If the bandwidths in the kernel are constant, then the eigenfunctions $\p{i}$ converge to the eigenfunctions of the \textit{Laplace-Beltrami operator} on a manifold as $n \to \infty$ \cite{belkin2003laplacian,COIFMAN20065,von2008consistency,hein2005graphs,belkin2006convergence,garcia2020error}.
A slightly different normalisation is required to get this convergence when using variable bandwidths \cite{berry2016variable}, which ensures that the resulting eigenfunctions do not depend on the sampling density of the data.
We do not use that normalisation here and directly diagonalise the Markov chain, based on the heuristic that incorporating the density into the basis choice is a more data-driven approach.
\end{relatedwork}

\subsubsection{Projection into the subspace}

We can use the matrix $\textbf{U}$ and measure $\boldsymbol{\mu}$ to project functions $\textbf{f} \in \R^n$ into the compressed function space.
If $\textbf{f} \in \R^n$ is a function, then
\[
\textbf{f}^* := \textbf{U}^T \diag(\boldsymbol{\mu}) \textbf{f} \in \R^{n_0}
\]
contains the coefficients of the $n_0$ eigenfunctions that best approximate $\textbf{f}$ with respect to the squared $L^2(\R^n, \boldsymbol{\mu})$ norm.
If $n_0 = n$, then this is just a change of basis, but we will usually take $n_0 \ll n$ so the projection is into a proper subspace, resulting in an approximation
\[
\textbf{f} \approx \sum_i \textbf{f}^*_i \p{i} = \textbf{U} \textbf{f}^* \in \R^n.
\]
We can think of $\textbf{f}^*$ as the \boldblue{bandlimited Fourier transform} of $\textbf{f}$.

\begin{computationalnote}
\label{note: function space is not an algebra}
One cost of using a compressed function space is that function multiplication is no longer closed in $\textbf{A}$.
If $\textbf{f}, \textbf{h} \in \textbf{A}$, we can compute their product $\textbf{f} \odot \textbf{h} \in \R^n$ by componentwise multiplication and then projecting back into $\textbf{A}$ as $(\textbf{f} \odot \textbf{h})^*$, i.e. 
\begin{equation}
\label{eq: function product projection}
(\boldsymbol{\phi}_i \boldsymbol{\phi}_{i'})_k
= \sum_{p=1}^n \textbf{U}_{pk} (\textbf{U}_{pi} \textbf{U}_{pi'}) \boldsymbol{\mu}_p
\end{equation}
for $k = 1,...,n_0$, but this generally loses information when $n_0 < n$.
In other words, $\textbf{A}$ is a subspace of $L^2(\R^n, \boldsymbol{\mu})$, but not a subalgebra.
\end{computationalnote}

\begin{computationalnote}
\label{note: multiplication tensor}
In this work, we will always evaluate function products using \eqref{eq: function product projection}, which costs $\ord(n n_0 m)$ time to compute the product of $m$ pairs of functions.
An alternative approach, used in Spectral Exterior Calculus (SEC) \cite{berry2020spectral}, is to precompute the $n_0 \times n_0 \times n_0$ tensor of function products $\textbf{c}_{i i' k}$ \eqref{eq: function product projection} in $\ord(n n_0^3)$ time, which allows $m$ function products to be computed in $\ord(n_0^3 m)$ time.
This gives a faster multiplication when $n$ and $m$ are both greater than $\ord(n_0^2)$.
This trade-off is justified in the SEC, where it is used extensively in backend computation to approximate expressions in the eigenfunction basis (so $m$ is very high), but this leads to a compounding projection error.
In this work, we never perform backend computation with these expansions, and so this error does not arise.
\end{computationalnote}

\subsection{Vector fields}

Formally, vector fields in diffusion geometry are defined as linear combinations
\[
X = \sum f_i \nabla h_i
\]
for some functions $f_i$ and $h_i$.
Instead of working with arbitrary functions $h_i$, we will restrict to the \boldblue{coordinate functions} $x_1,...,x_d : M \to \R$ defined by the ambient embedding of the data in $\R^d$, as in Section \ref{sec: overview}.
In general, we can use any coordinates $x_1,...,x_d : M \to \R$ that define an \boldblue{immersion}\footnote{It is important to point out that this assumption is not presently justified in the theory of diffusion geometry.
The correct definitions of morphisms between Markov triples, immersions, embeddings, and isometries are just some of the many outstanding questions in the field.
However, it will certainly be the case that, when these things have been defined, the above assumption will hold (it may even itself serve as a good definition of an immersion into $\R^d$).
} $(x_1,...,x_d) : M \to \R^d$.
We will suppose that an arbitrary vector field $X$ can be written as a sum
\[
X = \sum_{j=1}^d f_j \nabla x_j,
\]
where $x_j$ is the $j^{th}$ coordinate function.
In order to compute vector fields, we will discretise the coefficient functions using $n_1 \leq n_0$ basis functions from $\textbf{A}$, to get
\[
\textbf{X} = \sum_{i=1}^{n_1} \sum_{j=1}^d \textbf{X}_{ij} \p{i} \nabla x_j,
\]
so vector fields can be represented by the $n_1 \times d$ coefficient matrices $(\textbf{X}_{ij})_{i\leq n_1, j\leq d}$.
We flatten these into vectors of length $n_1d$, and so will think of the space of vector fields $\mathfrak{X}(M)$ as being represented by $\mathfrak{X}(\textbf{M}) = \R^{n_1d}$.
This choice of spanning set is illustrated for 2d data in Figure \ref{fig:function-vf-basis-2d} and for 3d data in Figure \ref{fig:function-vf-basis-3d}.

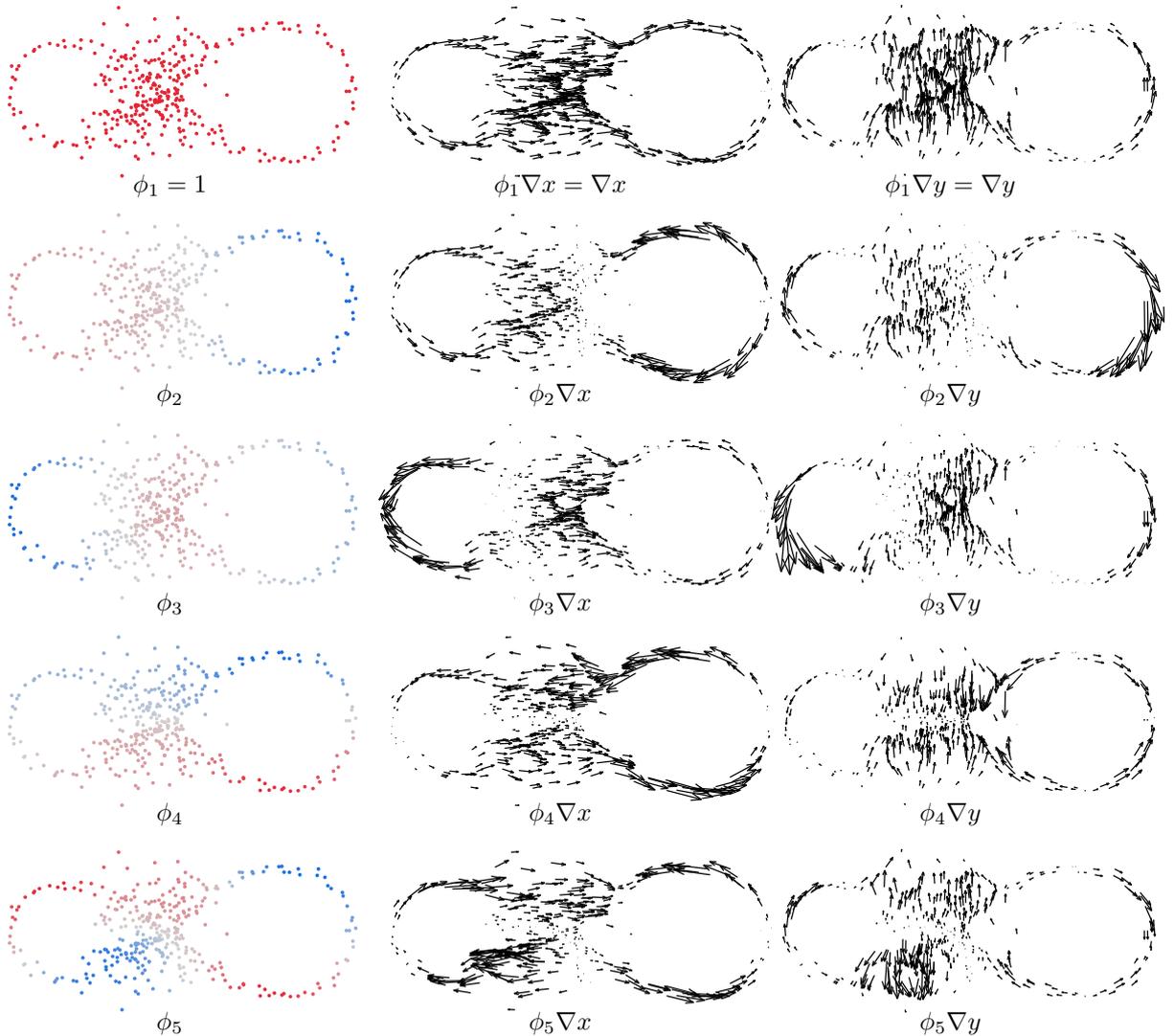
\begin{figure}[h!]
  \centering
  \begin{overpic}[width=\linewidth,grid=false]{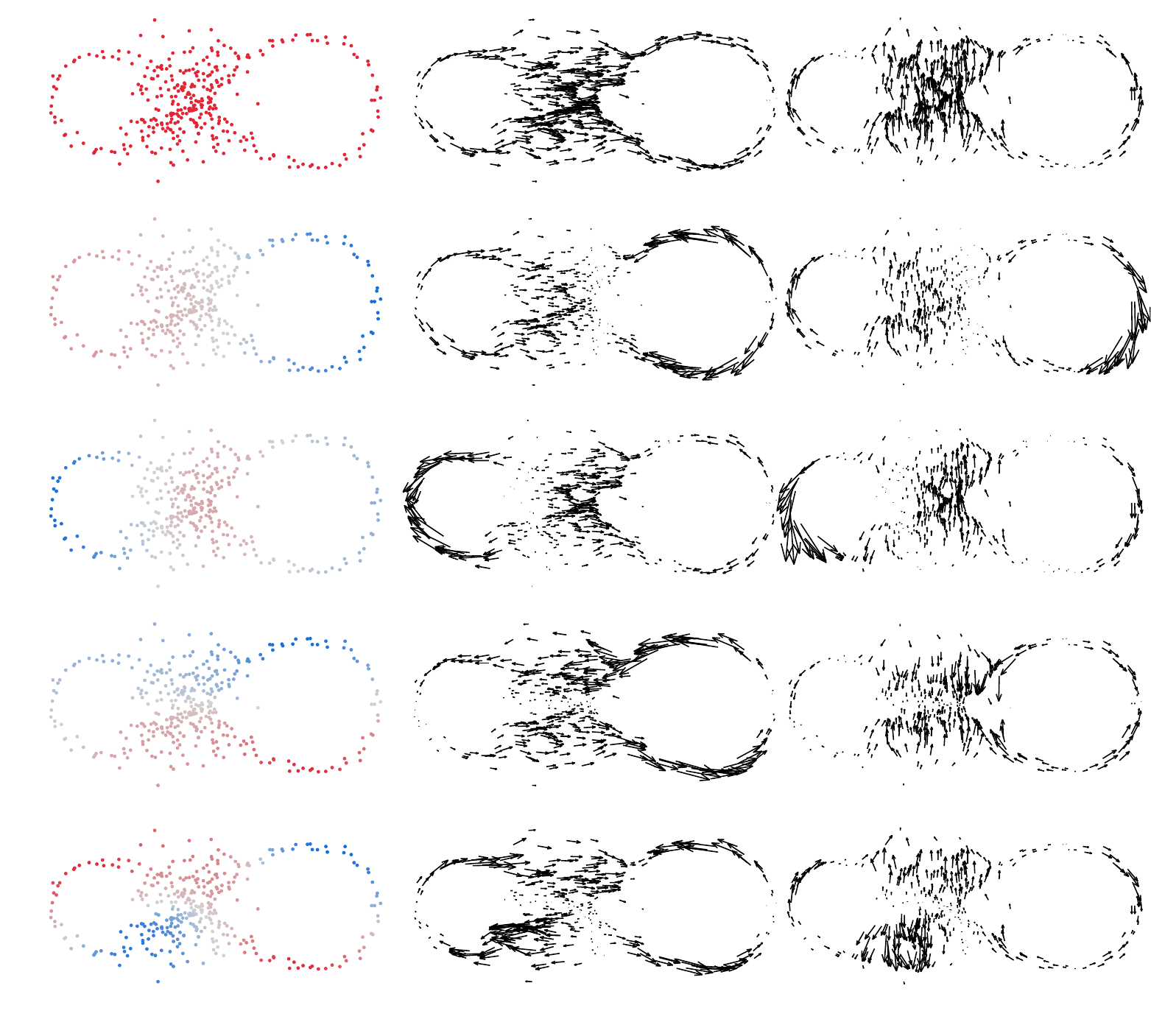}
\put(17.7,73.4){\makebox(0,0)[c]{$\phi_1 = 1$}}
\put(50.0,73.4){\makebox(0,0)[c]{$\phi_1 \nabla x = \nabla x$}}
\put(82.3,73.4){\makebox(0,0)[c]{$\phi_1 \nabla y = \nabla y$}}
\put(17.7,56.0){\makebox(0,0)[c]{$\phi_2$}}
\put(50.0,56.0){\makebox(0,0)[c]{$\phi_2 \nabla x$}}
\put(82.3,56.0){\makebox(0,0)[c]{$\phi_2 \nabla y$}}
\put(17.7,38.5){\makebox(0,0)[c]{$\phi_3$}}
\put(50.0,38.5){\makebox(0,0)[c]{$\phi_3 \nabla x$}}
\put(82.3,38.5){\makebox(0,0)[c]{$\phi_3 \nabla y$}}
\put(17.7,21.0){\makebox(0,0)[c]{$\phi_4$}}
\put(50.0,21.0){\makebox(0,0)[c]{$\phi_4 \nabla x$}}
\put(82.3,21.0){\makebox(0,0)[c]{$\phi_4 \nabla y$}}
\put(17.7,3.6){\makebox(0,0)[c]{$\phi_5$}}
\put(50.0,3.6){\makebox(0,0)[c]{$\phi_5 \nabla x$}}
\put(82.3,3.6){\makebox(0,0)[c]{$\phi_5 \nabla y$}}
  \end{overpic}
  \caption{\textbf{Spanning sets for functions and vector fields in 2d.}
  We can use eigenfunctions $\phi_i$ of the Markov chain as the \textit{smoothest possible} function basis (left column).
  The two ambient coordinates $x$ and $y$ are an embedding of the data, so we can use their gradients $\nabla x$ and $\nabla y$ to construct a spanning set for the space of vector fields by multiplying them by basis functions (middle and right columns).}
  \label{fig:function-vf-basis-2d}
\end{figure}

\newpage
\begin{computationalnote}
We generally set $n_1 = n_0$, so that the coefficient functions can be just as complex as the functions in $\textbf{A}$, but setting $n_1 < n_0$ would save computation and memory.
By using functions from $\textbf{A}$ as coefficients, the space of vector fields inherits the inductive bias encoded by the choice of function space basis.
\end{computationalnote}

\begin{computationalnote}
\label{note: immersion dimension and coords}
The role of the coordinate functions $x_1,...,x_d$ is just to generate a spanning set for the space of vector fields as a module over the function algebra $\A$, which requires that they give an immersion of the data in $\R^d$.
However, any other set of functions that give an immersion could be used instead, and in practice we will often \q{smooth out} the original coordinates by multiplying them by the Markov chain, i.e. using $\mathbf{P} x_1,..., \mathbf{P} x_d$, or projecting them into the compressed function space $\textbf{A}$.
If the data are very high-dimensional, a smaller set of functions obtained from a dimensionality reduction method could also be used to reduce computation and memory.
The different choices of immersion will lead to different spanning sets, which, like the choice of function space, encode an inductive bias on the vector fields.
However, they do not lead to essentially different geometries (because that depends only on the carré du champ), and so they are not equivalent to dimensionality reduction applied beforehand as data preprocessing.
\end{computationalnote}

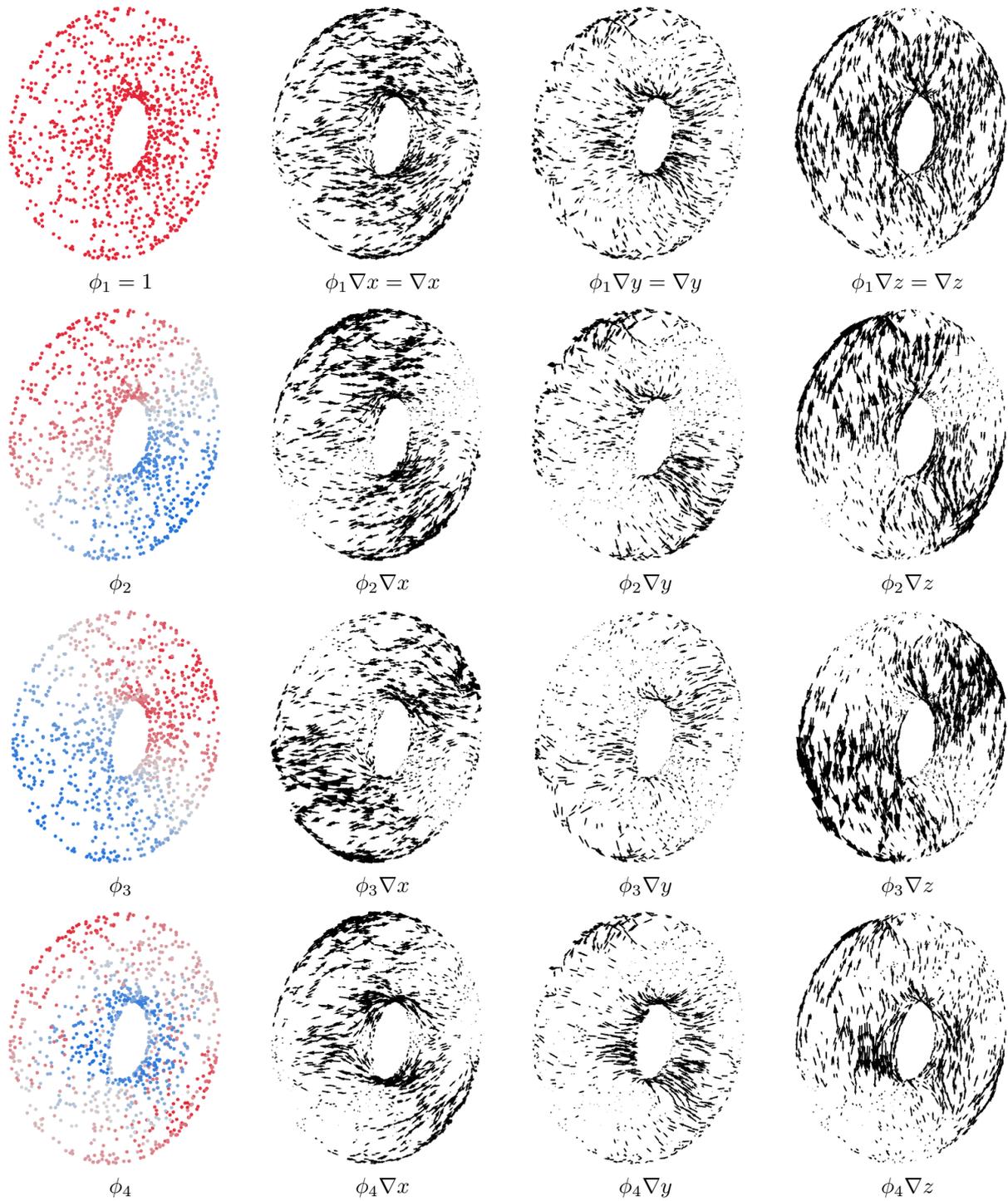
\begin{figure}[]
  \centering
  \begin{overpic}[width=\linewidth,grid=false]{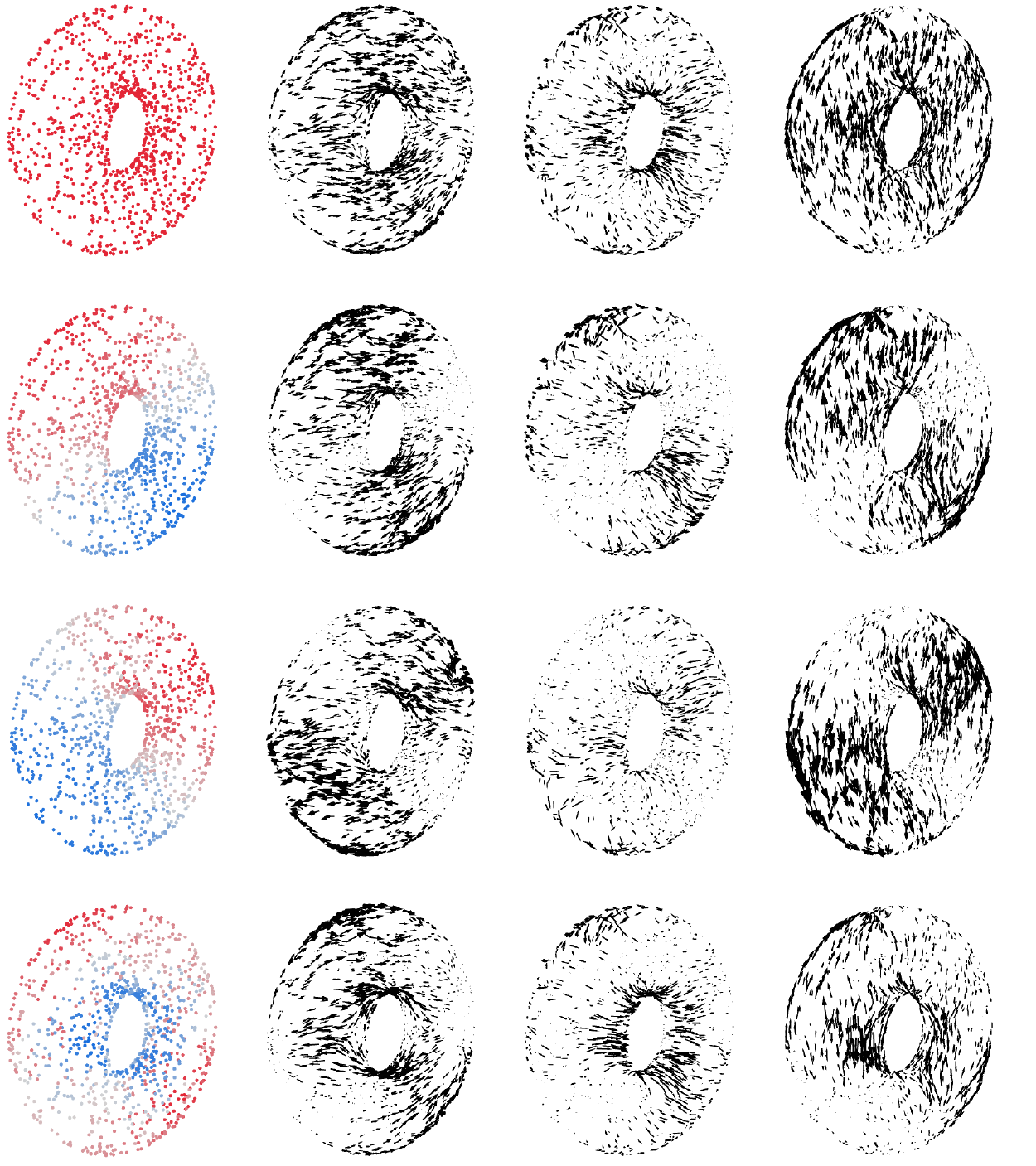}
\put(10.1,76.1){\makebox(0,0)[c]{$\phi_1 = 1$}}
\put(32.3,76.1){\makebox(0,0)[c]{$\phi_1 \nabla x = \nabla x$}}
\put(54.6,76.1){\makebox(0,0)[c]{$\phi_1 \nabla y = \nabla y$}}
\put(76.8,76.1){\makebox(0,0)[c]{$\phi_1 \nabla z = \nabla z$}}
\put(10.1,50.4){\makebox(0,0)[c]{$\phi_2$}}
\put(32.3,50.4){\makebox(0,0)[c]{$\phi_2 \nabla x$}}
\put(54.6,50.4){\makebox(0,0)[c]{$\phi_2 \nabla y$}}
\put(76.8,50.4){\makebox(0,0)[c]{$\phi_2 \nabla z$}}
\put(10.1,24.6){\makebox(0,0)[c]{$\phi_3$}}
\put(32.3,24.6){\makebox(0,0)[c]{$\phi_3 \nabla x$}}
\put(54.6,24.6){\makebox(0,0)[c]{$\phi_3 \nabla y$}}
\put(76.8,24.6){\makebox(0,0)[c]{$\phi_3 \nabla z$}}
\put(10.1,-1.1){\makebox(0,0)[c]{$\phi_4$}}
\put(32.3,-1.1){\makebox(0,0)[c]{$\phi_4 \nabla x$}}
\put(54.6,-1.1){\makebox(0,0)[c]{$\phi_4 \nabla y$}}
\put(76.8,-1.1){\makebox(0,0)[c]{$\phi_4 \nabla z$}}
  \end{overpic}
  \vspace{2em}
  \caption{\textbf{Spanning sets for functions and vector fields in 3d.}
  For 3d data, there are three ambient coordinates $x$, $y$, and $z$, and their gradients $\nabla x$, $\nabla y$, and $\nabla z$ can be used to construct a spanning set for the space of vector fields by multiplying them by basis functions (right three columns).}
  \label{fig:function-vf-basis-3d}
\end{figure}

\subsubsection{Riemannian metric}

The Riemannian metric defines a \boldblue{pointwise inner product} $g(X,Y)$ between any pair of vector fields $X$ and $Y$, and is related to the carré du champ by the formula $g(f \nabla h, f' \nabla h') = ff' \Gamma(h,h')$.
We can compute the metric of $\p{i}\nabla x_j$ and $\p{i'}\nabla x_{j'}$ at a point $p$ in the $n \times n_1 \times d \times n_1 \times d$ 5-tensor
\begin{equation}
\label{eq: g1}
\begin{split}
\textbf{g}_{pijij'} 
&= g(\p{i}\nabla x_j, \p{i'}\nabla x_{j'})(p) \\
&= \p{i}(p)\p{i'}(p)\Gamma(x_j,x_{j'})(p) \\
&= \textbf{U}_{pi}\textbf{U}_{pi'} \boldsymbol\Gamma_p(\textbf{x}_{j},\textbf{x}_{j'}). \\
\end{split}
\end{equation}
If we flatten $\textbf{g}_{piji'j'}$ into an $n \times n_1d \times n_1d$ 3-tensor $\textbf{g}_{pIJ}$ then, at a fixed point $p$, $\textbf{g}_p$ is a symmetric $n_1d \times n_1d$ matrix that represents the metric at $p$.
So, if $\textbf{X},\textbf{Y} \in \R^{n_1d}$ are vector fields, their metric is given pointwise by $(\textbf{X}^T \textbf{g}_p \textbf{Y})_{p=1,...,n} \in \R^n$.
The pointwise norm of a vector field $X$ is given by $\|X\| = \sqrt{g(X,X)}$, which we can compute as $\|\textbf{X}\| = \big((\textbf{X}^T \textbf{g}_p \textbf{X})^{1/2} \big)_{p=1,...,n} \in \R^n$.

\begin{computationalnote}
The metric $\textbf{g}_p$ at a point $p$ should be a positive semi-definite matrix.
Our implementation of the carré du champ \eqref{eq: cdc covariance variable bandwidth} is guaranteed to be positive semi-definite up to numerical precision, and so the metric computed using \eqref{eq: g1} will also be positive semi-definite.
\end{computationalnote}

\begin{figure}[h!]
  \centering
  \begin{overpic}[width=\linewidth,grid=false]{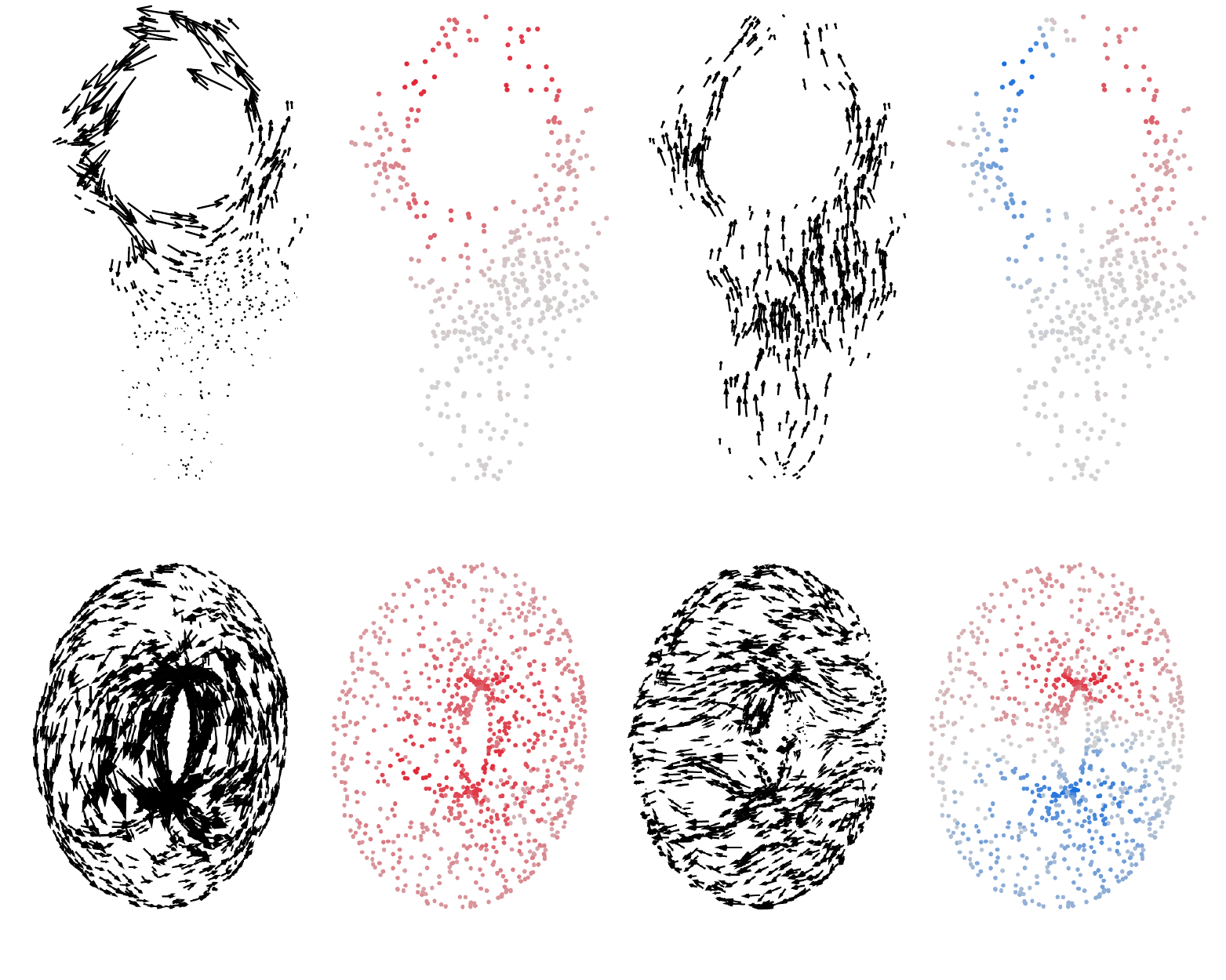}
\put(13.2,38.5){\makebox(0,0)[c]{$X$}}
\put(37.8,38.5){\makebox(0,0)[c]{$\|X\| = \sqrt{g(X,X)}$}}
\put(62.3,38.5){\makebox(0,0)[c]{$Y$}}
\put(86.8,38.5){\makebox(0,0)[c]{$g(X,Y)$}}
\put(13.2,2.6){\makebox(0,0)[c]{$Z$}}
\put(37.8,2.6){\makebox(0,0)[c]{$\|Z\| = \sqrt{g(Z,Z)}$}}
\put(62.3,2.6){\makebox(0,0)[c]{$W$}}
\put(86.8,2.6){\makebox(0,0)[c]{$g(Z,W)$}}
  \end{overpic}
  \caption{\textbf{Riemannian metric and pointwise norm of vector fields.}
  The carré du champ defines a Riemannian metric (pointwise inner product) on vector fields.
  We can use this to obtain the pointwise norm as $\|X\| = \sqrt{g(X,X)}$.}
  \label{fig:vf-metric}
\end{figure}

\subsubsection{Inner product}

The \boldblue{global inner product} of vector fields is given by integrating the metric, i.e. $\inp{X}{Y} = \int g(X,Y) d\mu$.
We can compute the $n_1d \times n_1d$ Gram matrix of vector fields as
\begin{equation}
\label{eq: G1}
\textbf{G}_{IJ} = \sum_{p=1}^n \textbf{g}_{pIJ} \boldsymbol{\mu}_p,
\end{equation}
so, if $\textbf{X},\textbf{Y} \in \R^{n_1d}$ are vector fields, their inner product is given by $\textbf{X}^T \textbf{G} \textbf{Y}\in \R$.

\begin{computationalnote} \label{comp:gram}
Although (\ref{eq: G1}) is a valid formula for the Gram matrix, it is unnecessarily expensive to materialise the entire $n \times n_1d \times n_1d$ 3-tensor $\textbf{g}_{pIJ}$ in memory if it has not already been computed for another purpose.
Instead, we can compute $\textbf{G}$ directly from (\ref{eq: g1}) in a single optimised tensor contraction.
This is a common theme that applies to all the inner product computations and many others later: most formulas for global objects are defined as the integral of pointwise objects.
By writing out the whole expression as a single tensor contraction, we can use the most optimal order of operations to save time and memory.
\end{computationalnote}

\subsubsection{Visualising vector fields}
\label{sub: visualising vector fields}

As discussed in Section \ref{sec: overview}, we can visualise a vector field $X \in \mathfrak{X}(M)$ as a collection of \q{arrows} on the data points by computing the Riemannian metric with the ambient coordinate vector fields
\[
(g(X, \nabla x_1),...,g(X, \nabla x_d)) \in \R^d.
\]
If $\textbf{X} = (\textbf{X}_{ij})_{i\leq n_1, j\leq d} \in \R^{n_1d}$ we can compute this vector at each point $p$ using the carré du champ as
\begin{equation}
\label{eq: VF vis}
\begin{split}
g(X, \nabla x_c)(p)
& = \sum_{i=1}^{n_1} \sum_{j=1}^d \textbf{X}_{ij} g(\p{i}\nabla x_j, \nabla x_c)(p) \\
& = \sum_{i=1}^{n_1} \sum_{j=1}^d \textbf{X}_{ij} \p{i}(p) \Gamma(x_j, x_c)(p) \\
& = \sum_{i=1}^{n_1} \sum_{j=1}^d \textbf{X}_{ij} \textbf{U}_{pi} \boldsymbol{\Gamma}_p(x_j, x_c)
\end{split}
\end{equation}
for each $c = 1,...,d$.

\subsection{Differential forms}

Differential forms are a fundamental object in geometry and topology that provide a general theory of \boldblue{volumes} and \boldblue{integrands}. 
We can also define them in diffusion geometry and compute with this framework.

\subsubsection{1-forms}

In degree 1, forms measure the \boldblue{lengths of curves}.
In diffusion geometry, the space of 1-forms is denoted $\Omega^1(M)$ and contains the formal sums
\[
\alpha = \sum_i f_i d h_i,
\]
which are \q{dual} to vector fields.
However, 1-forms and vector fields are essentially equivalent, due to the \boldblue{musical isomorphism} $\sharp : \Omega^1(M) \to \mathfrak{X}(M)$ given by $f dh \mapsto f \nabla h$, which is an isometric isomorphism.
As such, we use exactly the same representation for 1-forms as for vector fields, and write a 1-form $\alpha$ as a linear combination
\[
\boldsymbol\alpha = \sum_{i=1}^{n_1} \sum_{j=1}^d \boldsymbol\alpha_{ij} \p{i} d x_j,
\]
where $\p{i}$ is the $i^{th}$ basis function of $\textbf{A}$ and $x_j$ is the $j^{th}$ coordinate function.
We flatten the $n_1 \times d$ coefficient matrix $(\boldsymbol\alpha_{ij})_{i\leq n_1, j\leq d}$ into a vector of length $n_1d$, and so $\Omega^1(M)$ is represented by $\R^{n_1d}$.
As this choice of spanning set is just the dual to the spanning set for vector fields, it is also illustrated for 2d data in Figure \ref{fig:function-vf-basis-2d} and for 3d data in Figure \ref{fig:function-vf-basis-3d}.
Likewise, the metric on 1-forms is illustrated in Figure \ref{fig:vf-metric}.

The fact that the isomorphism $\sharp : \Omega^1(M) \to \mathfrak{X}(M)$ is \boldblue{isometric} means that 
1-forms and their dual vector fields share the same metric and inner product, i.e. $g(\alpha, \beta) = g(\alpha^\sharp, \beta^\sharp)$ and $\inp{\alpha}{\beta} = \inp{\alpha^\sharp}{\beta^\sharp}$.
We therefore equip the discrete space of 1-forms $\R^{n_1d}$ with the same metric tensor $\textbf{g}$ from (\ref{eq: g1}) and Gram matrix $\textbf{G}$ from (\ref{eq: G1}).
This makes the duality trivial, because the musical isomorphism $\sharp: \R^{n_1d} \to \R^{n_1d}$ is the identity map.
The inverse to $\sharp$ is denoted $\flat : \mathfrak{X}(M) \to \Omega^1(M)$, and is also the identity map on $\R^{n_1d}$.

\subsubsection{Visualising 1-forms}

A 1-form $\alpha$ can \boldblue{act on a vector field} $X$ to produce a function $\alpha(X)$, given by $\alpha(X) = g(\alpha, X^\flat)$.
We can use this action to visualise 1-forms as collections of \q{arrows} on the data points by computing
\[
(\alpha(\nabla x_1),...,\alpha(\nabla x_d))
= (g(\alpha, dx_1),...,g(\alpha, dx_d))
 \in \R^d
\]
at each point.
This is exactly the same as the visualisation of the dual vector field $\alpha^\sharp$, as described in (\ref{eq: VF vis}).

\subsubsection{$k$-forms}

Forms of higher degree measure volumes in higher dimensions.
For example, 2-forms measure the \boldblue{area of surfaces}, and 3-forms measure the \boldblue{volume of spaces}.
We write $k$-forms in the space $\Omega^k(M)$ as the sums
\[
\alpha = \sum_i f_i dh_i^1 \wedge \dots \wedge dh_i^k.
\]
Using the ambient coordinate functions, we get the expression
\[
\alpha 
= \sum_{1 \leq j_1<...<j_k \leq d} f_{j_1,...,j_k} dx_{j_1} \wedge \dots \wedge dx_{j_k}
= \sum_{J = 1}^{\binom{d}{k}} f_J dx_{(J)},
\]
where $J = (j_1,...,j_k)$ is a multi-index and $dx_{(J)}$ denotes the form $dx_{j_1} \wedge \dots \wedge dx_{j_k}$
We will compute $k$-forms by discretising the coefficient functions using $n_1$ basis functions from $\textbf{A}$, to get
\[
\boldsymbol\alpha = \sum_{i=1}^{n_1} \sum_{J = 1}^{\binom{d}{k}} \boldsymbol\alpha_{iJ} \p{i} dx_{(J)},
\]
so $k$-forms are represented by the $n_1 \times \binom{d}{k}$ coefficient matrices $(\boldsymbol\alpha_{iJ})$.
We flatten these into vectors of length $n_1 \binom{d}{k}$, and so identify the space of $k$-forms $\Omega^k(M)$ with $\R^{n_1\binom{d}{k}}$.
Notice that this representation of $k$-forms includes what was described above for $k=1$ as a special case.
Another special case is $k=0$, where the 0-forms are just functions $f \in \A = \Omega^0(M)$, which we can interpret as measuring the volume of points, i.e. \boldblue{signed densities}.
However, if $n_1 < n_0$, our discrete representations of $\textbf{A} = \R^{n_0}$ and $\Omega^0(\textbf{M}) = \R^{n_1}$ will differ, and so, to avoid confusion, we will never directly work with $\Omega^0(\textbf{M})$.
This choice of spanning set is illustrated for 2-forms on 2d data in Figure \ref{fig:2form-basis-2d} and on 3d data in Figure \ref{fig:2form-basis-3d}.
The spanning set for 3-forms on 3d data is illustrated in Figure \ref{fig:3form-basis-3d}.

\begin{figure}[h!]
  \centering
  \begin{overpic}[width=\linewidth,grid=false]{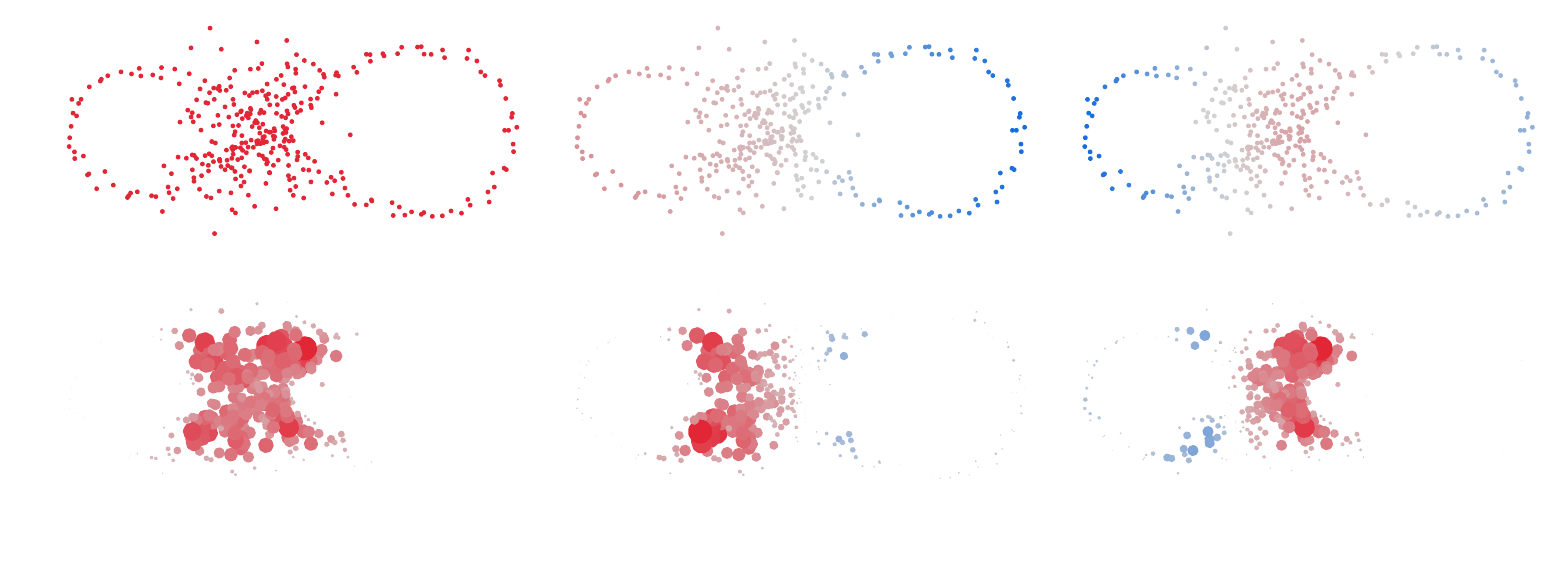}
\put(17.3,20.1){\makebox(0,0)[c]{$\phi_1 = 1$}}
\put(50.0,20.1){\makebox(0,0)[c]{$\phi_2$}}
\put(82.7,20.1){\makebox(0,0)[c]{$\phi_3$}}
\put(17.3,2.9){\makebox(0,0)[c]{$\phi_1 dx \wedge dy = dx \wedge dy$}}
\put(50.0,2.9){\makebox(0,0)[c]{$\phi_2 dx \wedge dy$}}
\put(82.7,2.9){\makebox(0,0)[c]{$\phi_3 dx \wedge dy$}}
  \end{overpic}
  \caption{\textbf{A spanning set for 2-forms in 2d.}
  In 2d, every 2-form is a multiple of $dx \wedge dy$.
  Since 2-forms measure \textit{area}, they are only non-zero where the space is at least 2-dimensional.
  In this example, the local dimension of the space varies from 1 to 2, and this is measured by the 2-forms.
  }
  \label{fig:2form-basis-2d}
\end{figure}

\begin{figure}[h!]
  \centering
  \begin{overpic}[width=\linewidth,grid=false]{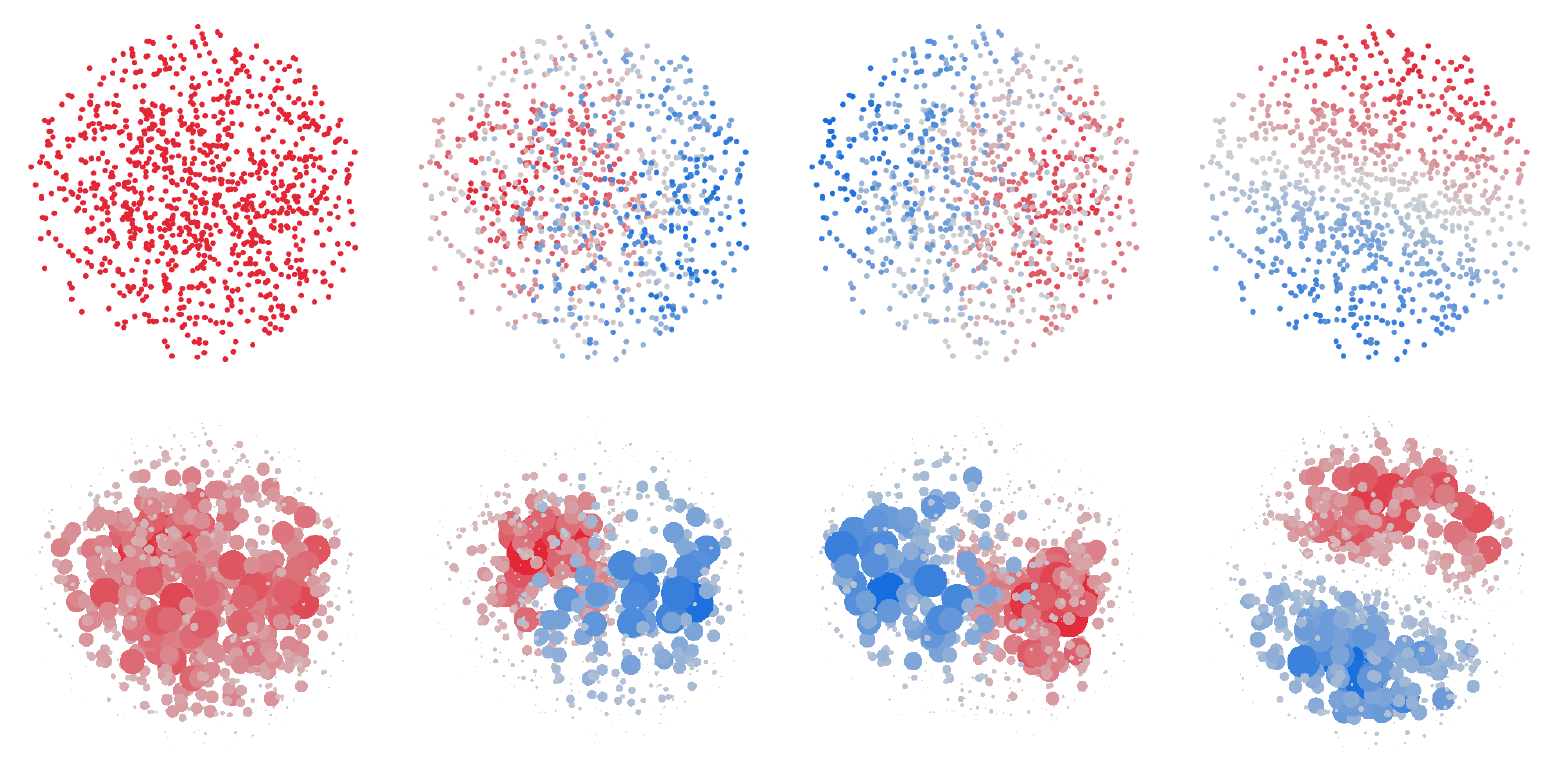}
\put(13.0,24.4){\makebox(0,0)[c]{$\phi_1 = 1$}}
\put(38.0,24.4){\makebox(0,0)[c]{$\phi_2$}}
\put(63.0,24.4){\makebox(0,0)[c]{$\phi_3$}}
\put(88.0,24.4){\makebox(0,0)[c]{$\phi_4$}}
\put(13.0,0.6){\makebox(0,0)[c]{$\phi_1 dx \wedge dy \wedge dz = dx \wedge dy \wedge dz$}}
\put(38.0,0.6){\makebox(0,0)[c]{$\phi_2 dx \wedge dy \wedge dz$}}
\put(63.0,0.6){\makebox(0,0)[c]{$\phi_3 dx \wedge dy \wedge dz$}}
\put(88.0,0.6){\makebox(0,0)[c]{$\phi_4 dx \wedge dy \wedge dz$}}
  \end{overpic}
  \caption{\textbf{A spanning set for 3-forms in 3d.}
  The data forms a dense ball in $\R^3$. 
  In 3d, every 3-form is a multiple of $dx \wedge dy \wedge dz$.
  Since 3-forms measure \textit{volume}, they are only non-zero where the space is at least 3-dimensional.
  In this example, they have more mass in the centre of the ball and vanish at the boundary sphere where the space loses volume.}
  \label{fig:3form-basis-3d}
\end{figure}

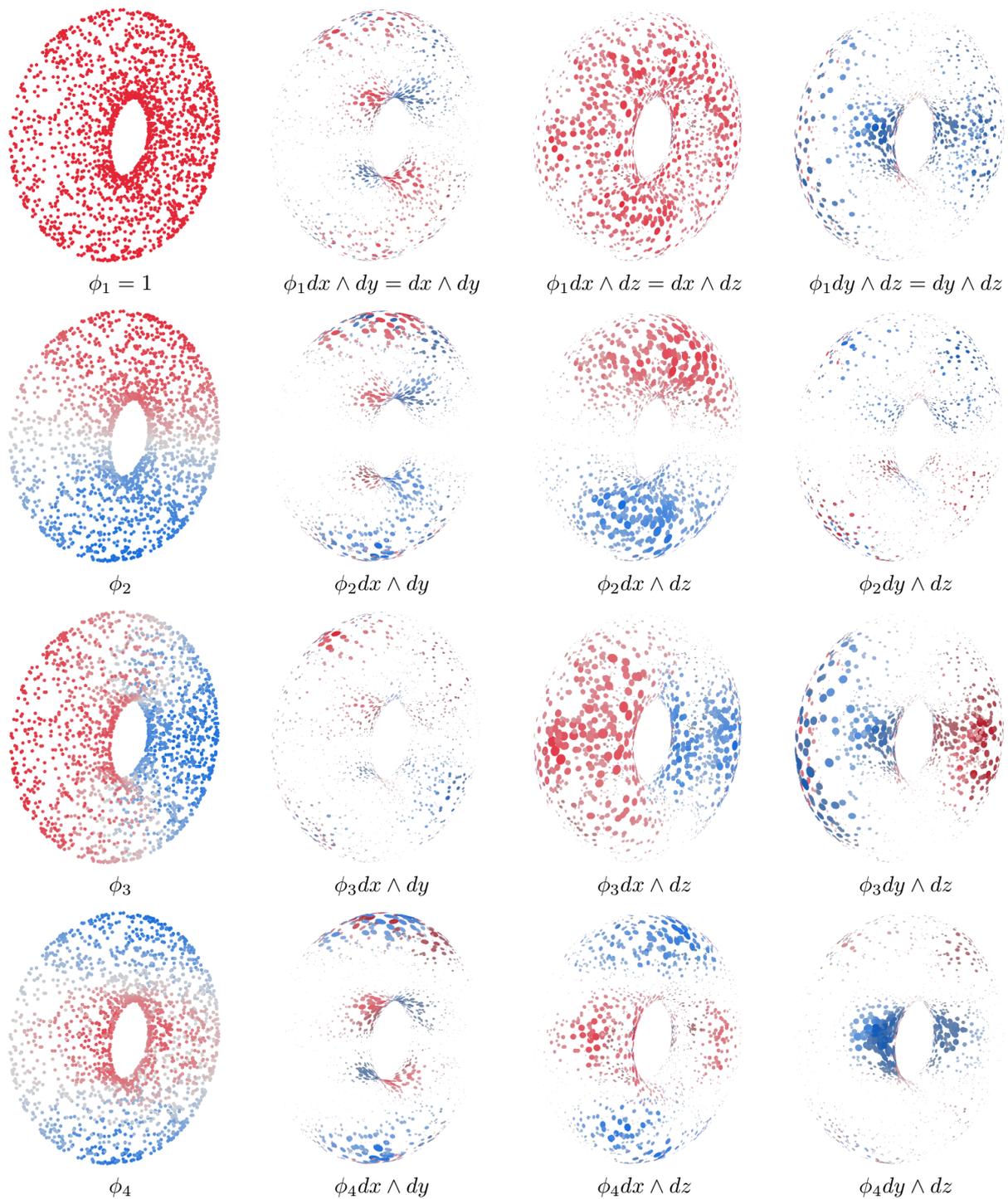
\begin{figure}[]
  \centering
  \begin{overpic}[width=\linewidth,grid=false]{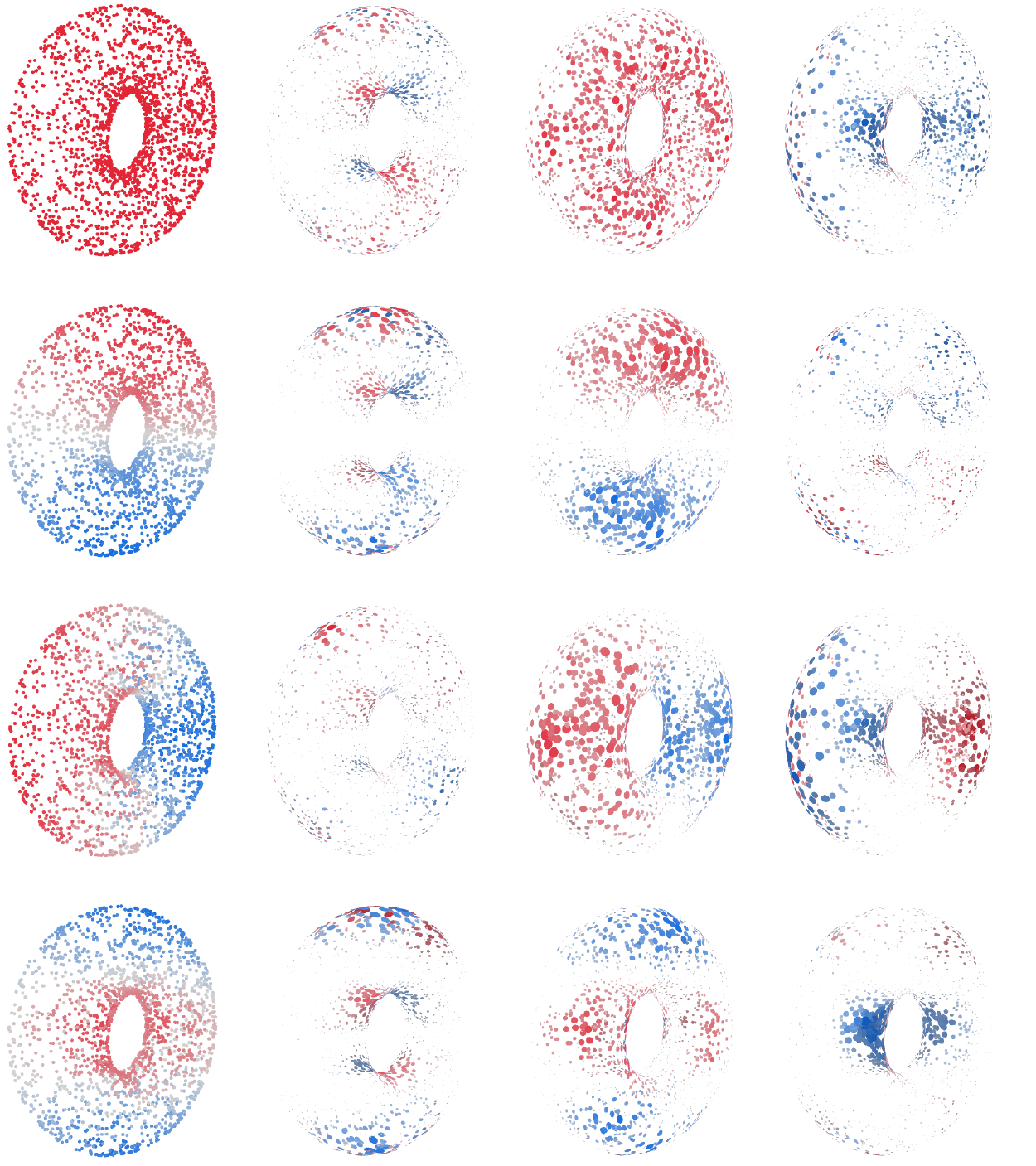}
\put(10.1,76.1){\makebox(0,0)[c]{$\phi_1 = 1$}}
\put(32.3,76.1){\makebox(0,0)[c]{$\phi_1 dx \wedge dy = dx \wedge dy$}}
\put(54.6,76.1){\makebox(0,0)[c]{$\phi_1 dx \wedge dz = dx \wedge dz$}}
\put(76.8,76.1){\makebox(0,0)[c]{$\phi_1 dy \wedge dz = dy \wedge dz$}}
\put(10.1,50.4){\makebox(0,0)[c]{$\phi_2$}}
\put(32.3,50.4){\makebox(0,0)[c]{$\phi_2 dx \wedge dy$}}
\put(54.6,50.4){\makebox(0,0)[c]{$\phi_2 dx \wedge dz$}}
\put(76.8,50.4){\makebox(0,0)[c]{$\phi_2 dy \wedge dz$}}
\put(10.1,24.6){\makebox(0,0)[c]{$\phi_3$}}
\put(32.3,24.6){\makebox(0,0)[c]{$\phi_3 dx \wedge dy$}}
\put(54.6,24.6){\makebox(0,0)[c]{$\phi_3 dx \wedge dz$}}
\put(76.8,24.6){\makebox(0,0)[c]{$\phi_3 dy \wedge dz$}}
\put(10.1,-1.1){\makebox(0,0)[c]{$\phi_4$}}
\put(32.3,-1.1){\makebox(0,0)[c]{$\phi_4 dx \wedge dy$}}
\put(54.6,-1.1){\makebox(0,0)[c]{$\phi_4 dx \wedge dz$}}
\put(76.8,-1.1){\makebox(0,0)[c]{$\phi_4 dy \wedge dz$}}
  \end{overpic}
  \vspace{3em}
  \caption{\textbf{A spanning set for 2-forms in 3d.}
  We can use eigenfunctions $\phi_i$ of the Markov chain as a function basis (left column).
  The three ambient coordinates $x$, $y$, and $z$ are an embedding of the data, so we can use their gradients $\nabla x$, $\nabla y$, and $\nabla_z$ to construct a spanning set for the space of vector fields by multiplying them by basis functions (middle and right columns).}
  \label{fig:2form-basis-3d}
\end{figure}

\subsubsection{Riemannian metric}

The metric of $k$-forms uses the determinant to compute a \boldblue{signed volume} in the formula
\[
g(f dh_1 \wedge \dots \wedge dh_k, f' dh_1' \wedge \dots \wedge dh_k') = ff' \det(\Gamma(h_i, h_j')).
\]
In order to compute the metric of $dx_{(J)}$ and $dx_{(J')}$ for multi-indices $J, J'$, we note that the $\binom{d}{k} \times \binom{d}{k}$ matrix of determinants $\det \big[ (\Gamma(x_{j_s},x_{j'_t})(p))_{s \in J,t \in J'}\big]$ is precisely the $k\thupper$ \boldblue{compound matrix} of the $d \times d$ matrix $(\Gamma(x_j,x_{j'})(p))_{j,j' \leq d}$.
For each $p$, we can then define the $\binom{d}{k} \times \binom{d}{k}$ matrix
\[
(\boldsymbol\Gamma_p(\textbf{x}_{(J)},\textbf{x}_{(J')}))_{J,J' \leq \binom{d}{k}} =
(g(dx_{(J)},dx_{(J')})(p))_{J,J' \leq \binom{d}{k}}
\]
to be the $k\thupper$ compound matrix of $(\boldsymbol\Gamma_p(\textbf{x}_j,\textbf{x}_{j'}))_{j,j' \leq d}$, so $(\boldsymbol\Gamma_p(\textbf{x}_{(J)},\textbf{x}_{(J')}))_{J,J' \leq \binom{d}{k}}$ is a 3-tensor of shape $(n, \binom{d}{k}, \binom{d}{k})$.
Compound matrices are a fundamental object in \boldblue{exterior algebra} (where differential forms are classically defined), and this construction of $\boldsymbol\Gamma_p(\textbf{x}_{(J)},\textbf{x}_{(J')})$ as a compound matrix naturally includes the cases $k = 1$, where $\binom{d}{1} = d$, as well as the degenerate case $k = 0$, where $\binom{d}{0} = 1$ and $\boldsymbol\Gamma(\cdot, \cdot)(p) = (1)$, the $1 \times 1$ identity matrix representing the \q{metric of functions}.

We can now compute the metric of $\p{i}dx_{(J)}$ and $\p{i'}dx_{(J')}$ at a point $p$ in the 5-tensor
\begin{equation}
\label{eq: gk}\textbf{g}^{(k)}_{pi J i' J'} 
= \boldsymbol{\phi}_i(p)\boldsymbol{\phi}_{i'}(p) \boldsymbol\Gamma_p(\textbf{x}_{(J)},\textbf{x}_{(J')}),
\end{equation}
which has dimension $n \times n_1 \times \binom{d}{k} \times n_1 \times \binom{d}{k}$.
We flatten $\textbf{g}^{(k)}$ into an $n \times n_1\binom{d}{k} \times n_1\binom{d}{k}$ 3-tensor $\textbf{g}^{(k)}_{pII'}$ so, at a fixed point $p$, $\textbf{g}^{(k)}_p$ is a symmetric $n_1\binom{d}{k} \times n_1\binom{d}{k}$ matrix which represents the metric at $p$.
If $\boldsymbol\alpha,\boldsymbol\beta \in \R^{n_1\binom{d}{k}}$ are $k$-forms, their metric is given by
$$
(\boldsymbol\alpha^T \textbf{g}^{(k)}_p \boldsymbol\beta)_{p=1,...,n} \in \R^n.
$$
The metric on 2-forms and 3-forms is illustrated in Figure \ref{fig:forms-metric}.

\begin{figure}[h!]
  \centering
  \begin{overpic}[width=\linewidth,grid=false]{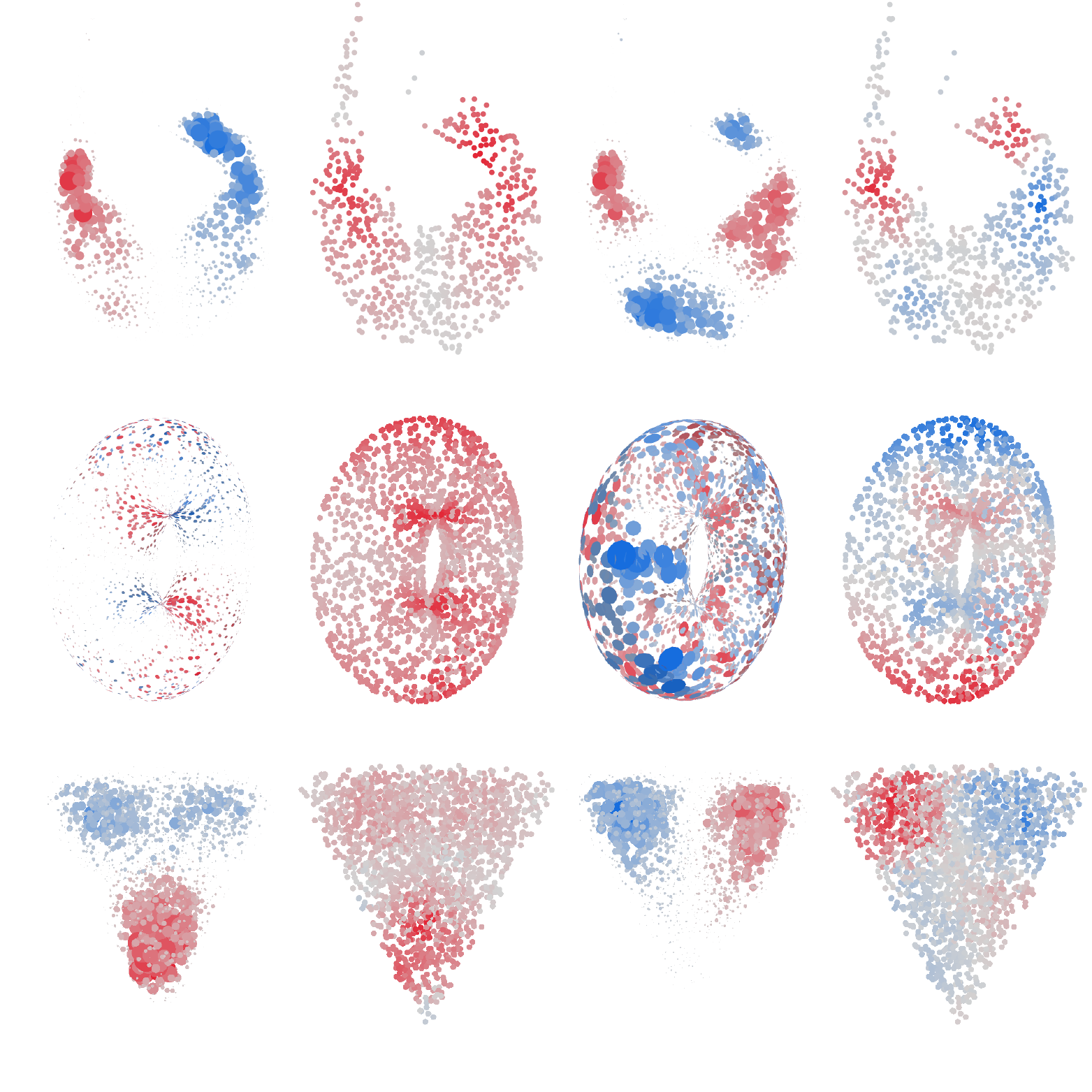}
\put(13.2,64.8){\makebox(0,0)[c]{$\alpha$}}
\put(37.8,64.8){\makebox(0,0)[c]{$\|\alpha\| = \sqrt{g(\alpha,\alpha)}$}}
\put(62.3,64.8){\makebox(0,0)[c]{$\beta$}}
\put(86.8,64.8){\makebox(0,0)[c]{$g(\alpha,\beta)$}}
\put(13.2,33.5){\makebox(0,0)[c]{$\alpha$}}
\put(37.8,33.5){\makebox(0,0)[c]{$\|\alpha\| = \sqrt{g(\alpha,\alpha)}$}}
\put(62.3,33.5){\makebox(0,0)[c]{$\beta$}}
\put(86.8,33.5){\makebox(0,0)[c]{$g(\alpha,\beta)$}}
\put(13.2,2.2){\makebox(0,0)[c]{$\alpha$}}
\put(37.8,2.2){\makebox(0,0)[c]{$\|\alpha\| = \sqrt{g(\alpha,\alpha)}$}}
\put(62.3,2.2){\makebox(0,0)[c]{$\beta$}}
\put(86.8,2.2){\makebox(0,0)[c]{$g(\alpha,\beta)$}}
  \end{overpic}
  \caption{\textbf{Riemannian metric and pointwise norm of differential forms.}
  The carré du champ defines a Riemannian metric (pointwise inner product) and pointwise norm as $\|\alpha\| = \sqrt{g(\alpha,\alpha)}$ on forms.
  We apply this to 2-forms in 2d (top row) and 3d (middle row) and to 3-forms in a dense 3d cone (bottom row).
  }
  \label{fig:forms-metric}
\end{figure}
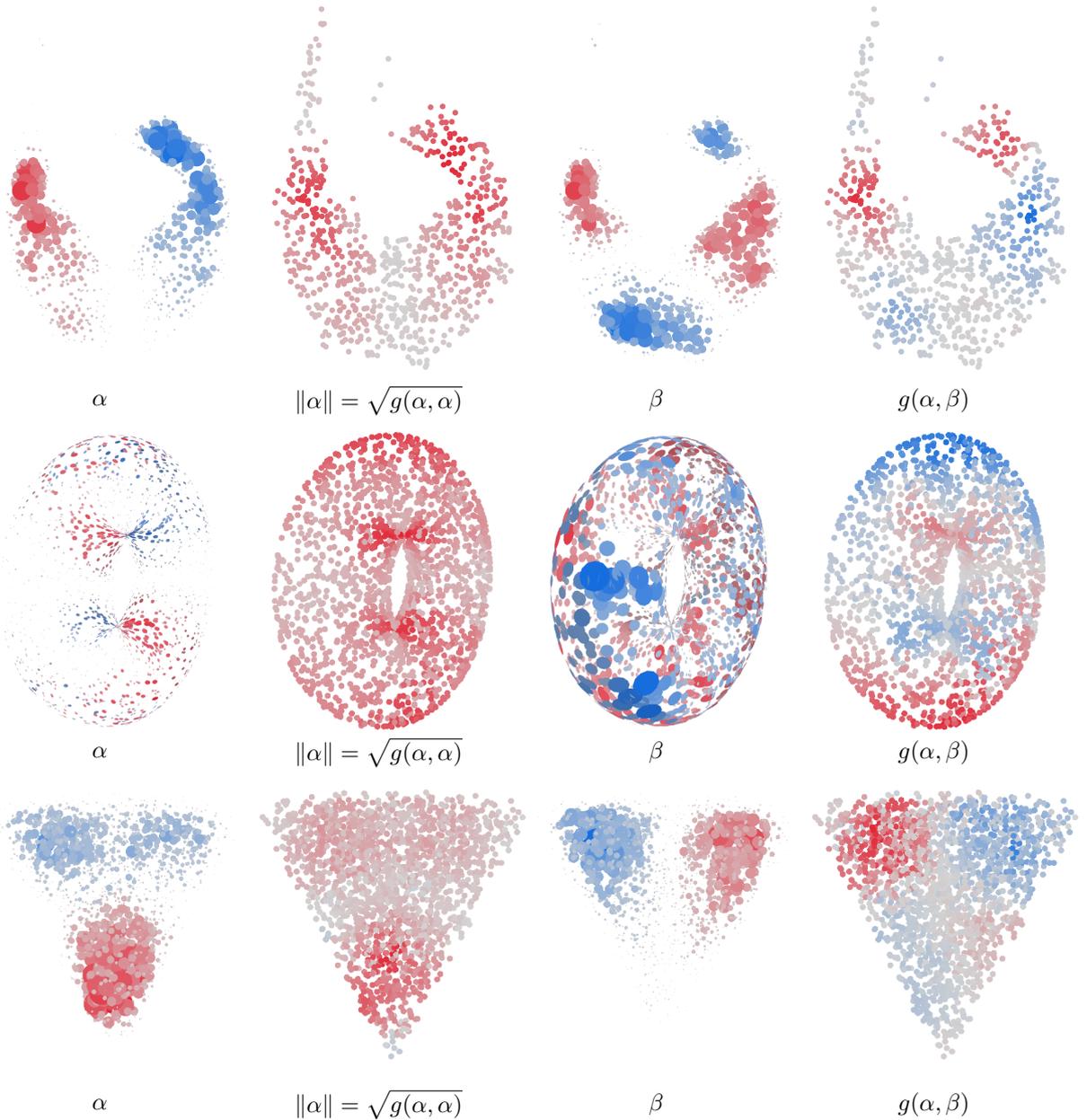

\subsubsection{Inner product}

The inner product of $k$-forms is given by the formula $\inp{\alpha}{\beta} = \int g(\alpha,\beta) d\mu$, so we can compute the $n_1\binom{d}{k} \times n_1\binom{d}{k}$ Gram matrix $\mathbf{G}^{(k)}$ of $k$-forms as
\begin{equation}
\label{eq: Gk}
\mathbf{G}^{(k)}_{IJ} = \sum_{p=1}^n \mathbf{g}^{(k)}_{pIJ} \boldsymbol{\mu}_p.
\end{equation}
If $\boldsymbol{\alpha}, \boldsymbol{\beta} \in \R^{n_1\binom{d}{k}}$ are $k$-forms, their inner product is given by $\boldsymbol{\alpha}^T \textbf{G}^{(k)} \boldsymbol{\beta}\in \R$.

\subsubsection{Wedge product}

Differential forms have a notion of multiplication that extends the usual multiplication of functions (i.e. 0-forms).
If $\alpha \in \Omega^k(M)$ and $\beta \in \Omega^l(M)$, we can form their \boldblue{wedge product} $\alpha \wedge \beta \in \Omega^{k+l}(M)$ by
\[
(f dh_1 \wedge \dots \wedge dh_k) \wedge (f' dh_1' \wedge \dots \wedge dh_l') = ff' dh_1 \wedge \dots \wedge dh_k \wedge dh_1' \wedge \dots \wedge dh_l'.
\]
In our discretisation, we can compute the wedge product of
\[
\boldsymbol\alpha = \sum_{i=1}^{n_1} \sum_{J = 1}^{\binom{d}{k}} \boldsymbol\alpha_{iJ} \p{i} dx_{(J)}
\qquad
\boldsymbol\beta = \sum_{i'=1}^{n_1} \sum_{J' = 1}^{\binom{d}{l}} \boldsymbol\beta_{i'J'} \p{i'} dx_{(J')}
\]
as
\[
\boldsymbol\alpha \wedge \boldsymbol\beta
= \sum_{i,i'=1}^{n_1} \sum_{J = 1}^{\binom{d}{k}}\sum_{J' = 1}^{\binom{d}{l}} \boldsymbol\alpha_{iJ} \boldsymbol\beta_{i'J'} \p{i} \p{i'} dx_{(J)} \wedge dx_{(J')}.
\]
Note that the multi-indices $J = (j_1,...,j_k)$ and $J' = (j'_1,...,j'_l)$ are in lexicographic (increasing) order, and we would like to obtain length-$(k+l)$ multi-indices in lexicographic order as well.
However, while it is true that $dx_{(J)} \wedge dx_{(J')} = dx_{(j_1,...,j_k,j'_1,...,j'_l)}$, the concatenated multi-index $(j_1,...,j_k,j'_1,...,j'_l)$ is not necessarily in lexicographic order.
We therefore need to sort this list and multiply the result by the sign of the permutation, so let
\[
\sign(J,J')
=
\begin{cases}
    0 & J \text{ and } J' \text{ are not disjoint} \\
    \sign(\pi) & \pi \text{ is a permutation of } (j_1,...,j_k,j'_1,...,j'_l) \text{ into lexicographic order}
\end{cases}
\]
to give the wedge product formula
\begin{equation}
\label{eq: wedge product formula}
\boldsymbol\alpha \wedge \boldsymbol\beta
= \sum_{i,i'=1}^{n_1} \sum_{J = 1}^{\binom{d}{k}}\sum_{J' = 1}^{\binom{d}{l}} \sign(J,J') \boldsymbol\alpha_{iJ} \boldsymbol\beta_{i'J'} (\boldsymbol{\phi}_i \boldsymbol{\phi}_{i'}) d\textbf{x}_{(\pi (j_1,...,j_k,j'_1,...,j'_l))}.
\end{equation}
As in \compnote{note: function space is not an algebra}, the product $(\boldsymbol{\phi}_i \boldsymbol{\phi}_{i'})$ is approximated in terms of the first $n_1$ functions from $\textbf{A}$ by
\begin{equation}
\label{eq: coefficient function projection wedge}
(\boldsymbol{\phi}_i \boldsymbol{\phi}_{i'})_k
= \sum_{p=1}^n \textbf{U}_{pk} (\textbf{U}_{pi} \textbf{U}_{pi'}) \boldsymbol{\mu}_p
\end{equation}
for $k = 1,...,n_1$.
The wedge product is defined between forms of all degrees, and so can encode a range of geometric and topological information.
See Figure \ref{fig:products} for several examples, as well as the function products (i.e. $k=0$) in Figures \ref{fig:function-vf-basis-2d}, \ref{fig:function-vf-basis-3d}, \ref{fig:2form-basis-2d}, \ref{fig:3form-basis-3d}, \ref{fig:2form-basis-3d}, and \ref{fig:2-tensor-basis}.


\begin{computationalnote}
\label{note: wedge product truncation}
Recall from \compnote{note: function space is not an algebra} that our compressed function space $\textbf{A}$ is not closed under multiplication, and so computing the product in $\textbf{A}$ will incur some loss of detail unless $n_0 = n$.
This limitation extends to the wedge product, because the projection \eqref{eq: coefficient function projection wedge} incurs the same loss unless $n_1 = n$. In addition, we make use of the fact that only disjoint multi-indices contribute nonzero terms to the wedge product: whenever $J \cap J' \neq \emptyset$, the sign factor in \eqref{eq: wedge product formula} vanishes.
Consequently, the number of admissible index pairs is reduced from $\binom{d}{k}\binom{d}{\ell}$ to $\binom{d}{k}\binom{d-k}{l}$, taking the overall computational complexity to
$\mathcal{O}\left(n n_1\binom{d}{k}\binom{d-k}{l} \right)$.
This substantially decreases the effective computational cost of forming wedge products. For instance, if $d=10$ and $k=l=3$, roughly $70\%$ of all possible pairs are immediately zero.
\end{computationalnote}

\begin{figure}[]
  \centering
\begin{overpic}[width=\linewidth,grid=false, yshift=2cm]{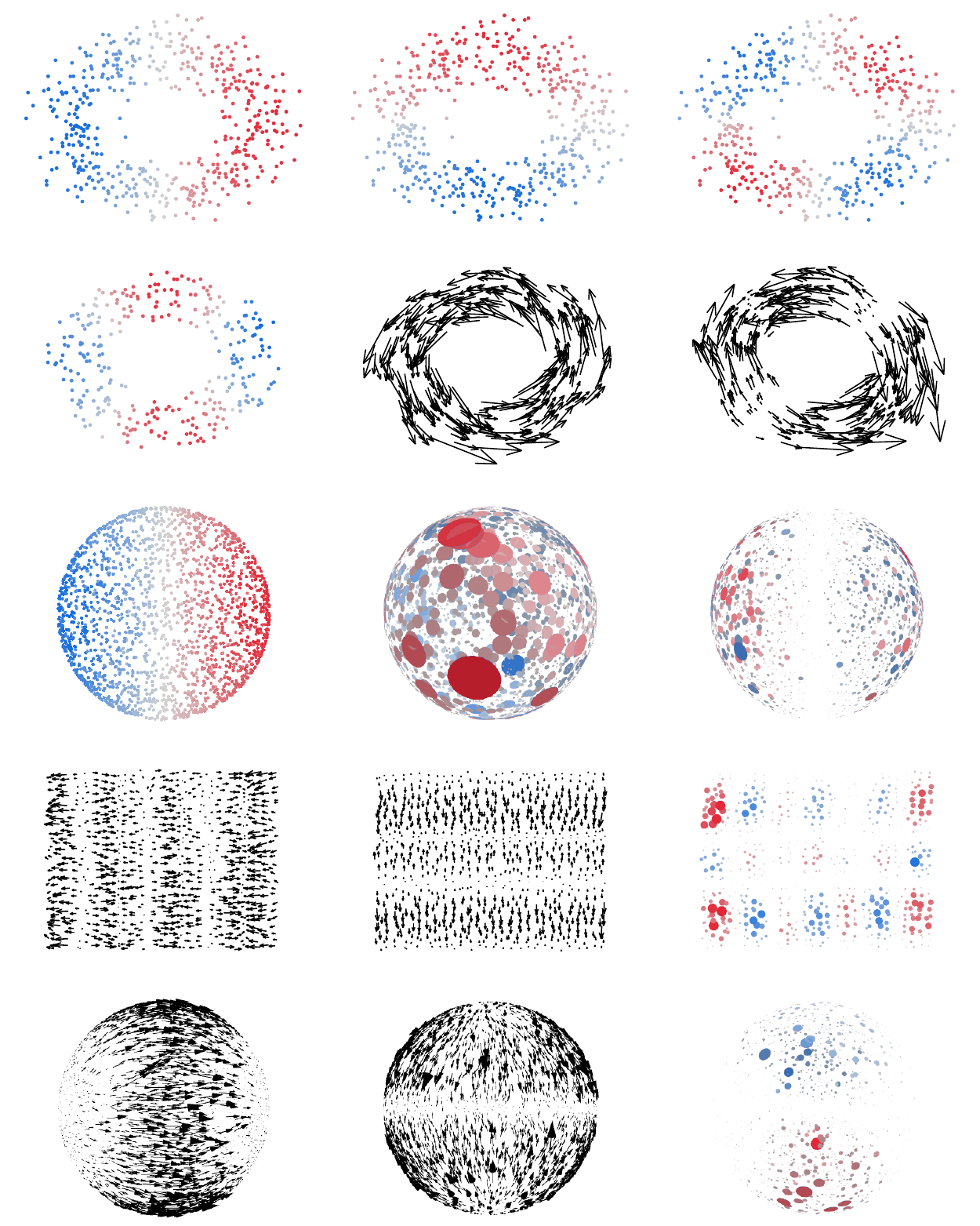}
\put(14.1,80.0){\makebox(0,0)[c]{$\phi_1$}}
\put(40.0,80.0){\makebox(0,0)[c]{$\phi_2$}}
\put(65.9,80.0){\makebox(0,0)[c]{$\phi_1\cdot\phi_2$}}
\put(14.1,61.0){\makebox(0,0)[c]{$\phi_3$}}
\put(40.0,61.0){\makebox(0,0)[c]{$\alpha$}}
\put(65.9,61.0){\makebox(0,0)[c]{$\phi_3\cdot \alpha$}}
\put(14.1,39.5){\makebox(0,0)[c]{$\phi_1$}}
\put(40.0,39.5){\makebox(0,0)[c]{$\beta$}}
\put(65.9,39.5){\makebox(0,0)[c]{$\phi_1\cdot \beta$}}
\put(14.1,20.5){\makebox(0,0)[c]{$f_1 dx$}}
\put(40.0,20.5){\makebox(0,0)[c]{$f_2 dy$}}
\put(65.9,20.5){\makebox(0,0)[c]{$f_1 f_2\, dx\wedge dy$}}
\put(14.1,-1.0){\makebox(0,0)[c]{$h_1 dx$}}
\put(40.0,-1.0){\makebox(0,0)[c]{$h_2 dy$}}
\put(65.9,-1.0){\makebox(0,0)[c]{$h_1 h_2\, dx\wedge dy$}}
  \end{overpic}
  \vspace{2em}
 \caption{
\textbf{Wedge product of forms.}
Each row illustrates a product of increasing degree:
(1) eigenfunctions $\phi_1$, $\phi_2$, and their product $\phi_1\phi_2$ on an annulus;
(2) eigenfunction $\phi_3$ and a rotational 1-form $\alpha$ centered around the origin;
(3) eigenfunction $\phi_1$ on the sphere $S^2$ and the Riemannian volume form $\beta = \mathrm{vol}_{S^2}$, giving $\phi_1\beta$;
(4) 1-forms $f_1dx$, $f_2dy$, and their wedge $f_1f_2\,dx\wedge dy$;
(5) analogous 1-forms $h_1dx$, $h_2dy$ on $S^2$, with wedge product $h_1h_2\,dx\wedge dy$.
}
  \label{fig:products}
\end{figure}
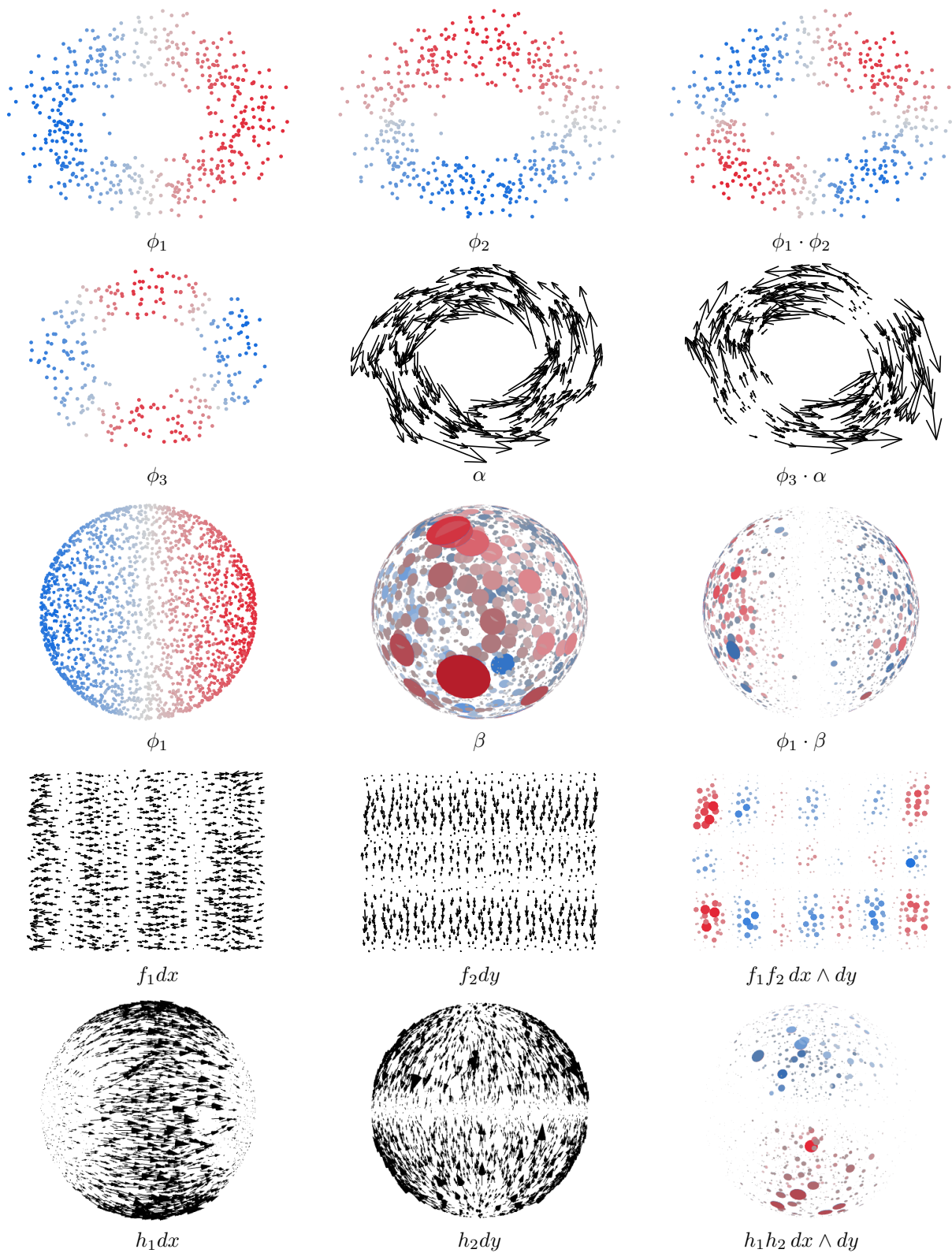

\subsubsection{Visualising $k$-forms}

\label{sec: visualising k forms}

We visualised 1-forms by considering their action on the ambient coordinate vector fields, to get pointwise vectors $(g(\alpha, dx_i))_{i=1,...,d}$.
These measure 1-dimensional volumes (lengths) and so are represented by a \boldblue{1-dimensional space} (a direction) as well as a \boldblue{signed magnitude} at each point (and directions with a signed magnitude are just vectors).

General $k$-forms measure $k$-dimensional volumes, so are represented by a \boldblue{$k$-dimensional subspace} (a $k$-plane) and a \boldblue{signed magnitude} at each point.
For example, a 2-form corresponds to a 2d plane with a signed magnitude.
We can visualise this as a disk aligned with that space and whose size varies with the magnitude, and opposite colours on either side to represent orientation.
A 3-form corresponds to a 3d volume with a signed magnitude, which we can visualise as a coloured sphere at each point, whose size varies with the magnitude and whose colour represents the sign.
To compute these objects, we generalise the approach for 1-forms to $k$-forms by computing, at each point, a $d \times \dots \times d$ $k$-tensor whose entries are given by
\[
\big(g(\alpha, dx_{j_1} \wedge \dots \wedge dx_{j_k})\big)_{j_1,...,j_k=1,...,d} \in \R^{d \times \dots \times d}
\]
which is skew-symmetric in its $k$ indices.
For example, a 2-form $\alpha$ will yield a skew-symmetric $d \times d$ matrix, and a 3-form $\beta$ will yield a skew-symmetric $d \times d \times d$ 3-tensor.

The method for computing the visualisation depends on the degree $k$ of the form and the ambient dimension $d$.
In degree 1, this directly gives a collection of vectors in $\R^d$ which we can plot as arrows.
A 2-form in 2d gives a skew symmetric $2 \times 2$ matrix
\[\begin{pmatrix}
0 & a \\
-a & 0
\end{pmatrix},\]
where $a \in \R$ is the signed area magnitude, which we can visualise as a coloured disk of size $|a|$ and signed colour $a$ (see Figure \ref{fig:2form-basis-2d}).
A 2-form in 3d gives a skew-symmetric $3 \times 3$ matrix which can be orthogonally diagonalised into the form
\[\begin{pmatrix}
0 & a & 0 \\
-a & 0 & 0 \\
0 & 0 & 0
\end{pmatrix},\]
where $a \in \R$ is the signed area magnitude in the plane spanned by the first two eigenvectors.
We can visualise it as a coloured disk in that plane of size $|a|$ with opposite colours on each side, depending on the sign of $a$ (see Figure \ref{fig:2form-basis-3d}).
A skew-symmetric 3-form in 3d takes the form
\[
\begin{bmatrix}
\begin{bmatrix}
0 & 0 & 0\\
0 & 0 & a\\
0 & -a & 0
\end{bmatrix}
\;
\begin{bmatrix}
0 & 0 & -a\\
0 & 0 & 0\\
a & 0 & 0
\end{bmatrix}
\;
\begin{bmatrix}
0 & a & 0\\
-a & 0 & 0\\
0 & 0 & 0
\end{bmatrix}
\end{bmatrix},
\]
where $a \in \R$ is the signed volume magnitude, which we can visualise as a coloured sphere of size $|a|$ and signed colour $a$ (see Figure \ref{fig:3form-basis-3d}).

\subsection{Other tensors}

Functions, vector fields, and differential forms are all examples of \boldblue{tensor fields}, meaning they represent a tensor (in the linear-algebraic sense) attached to each point in the space.
Vector fields and 1-forms have degree 1, and so correspond to rank-1 tensors (i.e. vectors) at each point.
2-forms correspond to rank-2 tensors (matrices) that are antisymmetric.
It is precisely these constructions that we use for visualisation in Section \ref{sec: visualising k forms}.
In general, tensor algebra is a powerful framework for encoding calculus and geometric information, and we can form tensor fields of any degree.

We can define any tensor in diffusion geometry and compute it with this framework, but we focus here on 2-tensors, which can be interpreted as a square matrix at each point.
General $(0,2)$-tensors in the space $\Omega^1(M)^{\otimes2}$ are the sums
\[
\alpha = \sum_i f_i dh_i^1 \otimes dh_i^2.
\]
The notation $(0,2)$ means that they correspond to a pair of 1-forms: in general, a $(p,q)$-tensor corresponds to $p$ vector fields and $q$ 1-forms.
We will compute 2-tensors in terms of the coordinate functions as
\[
\alpha = \sum_{j_1,j_2 = 1}^d f_J dx_{j_1} \otimes dx_{j_2},
\] 
where $J = (j_1,j_2)$, and discretise the coefficient functions
\[
\boldsymbol\alpha = \sum_{i=1}^{n_1} \sum_{j_1,j_2 = 1}^d \boldsymbol\alpha_{ij_1j_2} \p{i} dx_{j_1} \otimes dx_{j_2},
\]
so 2-tensors are represented by the $n_1 \times d \times d$ coefficient 3-tensors $(\boldsymbol\alpha_{ij_1j_2})_{i\leq n_1; j_1,j_2\leq d}$.
We flatten these into vectors of length $n_1d^2$, and so identify the space of 2-tensors $\Omega^k(M)^{
\otimes2}$ with $\R^{n_1d^2}$.

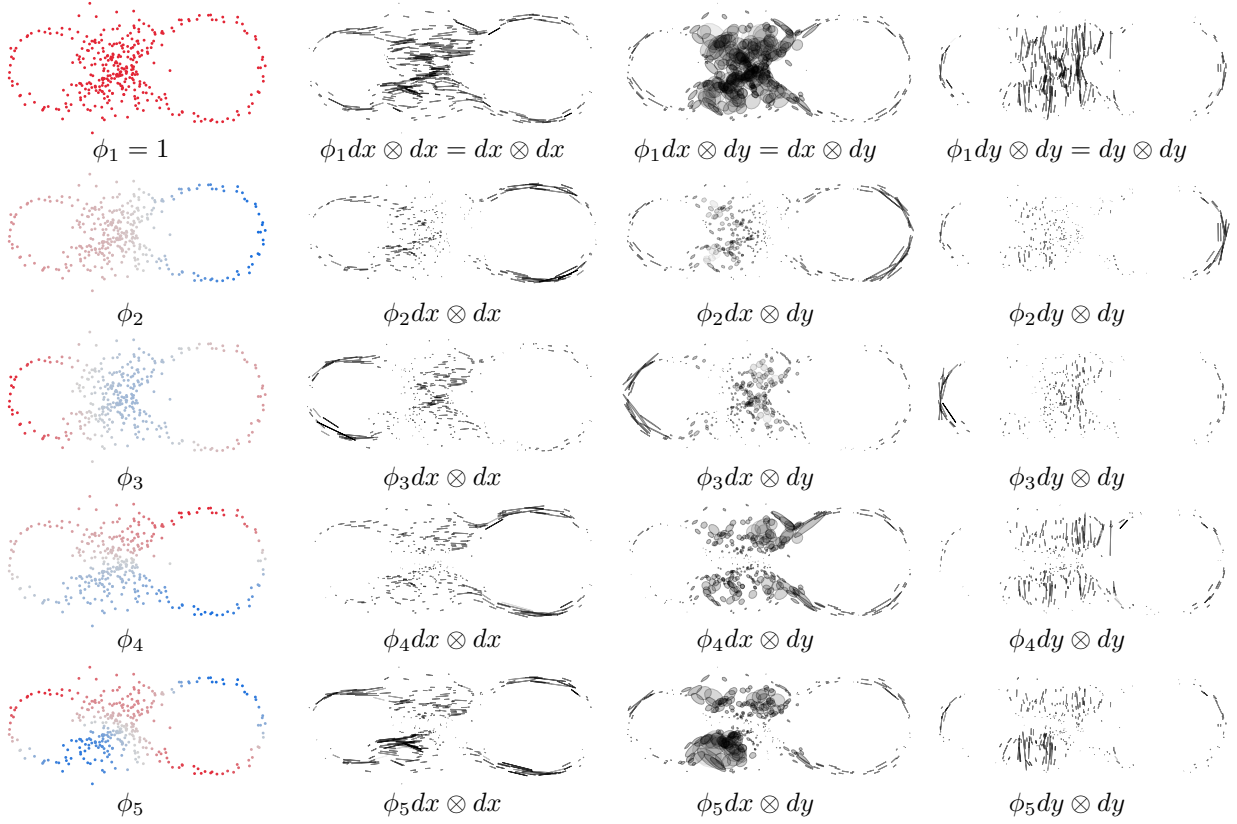
\begin{figure}[h]
  \centering
  \begin{overpic}[width=\linewidth,grid=false]{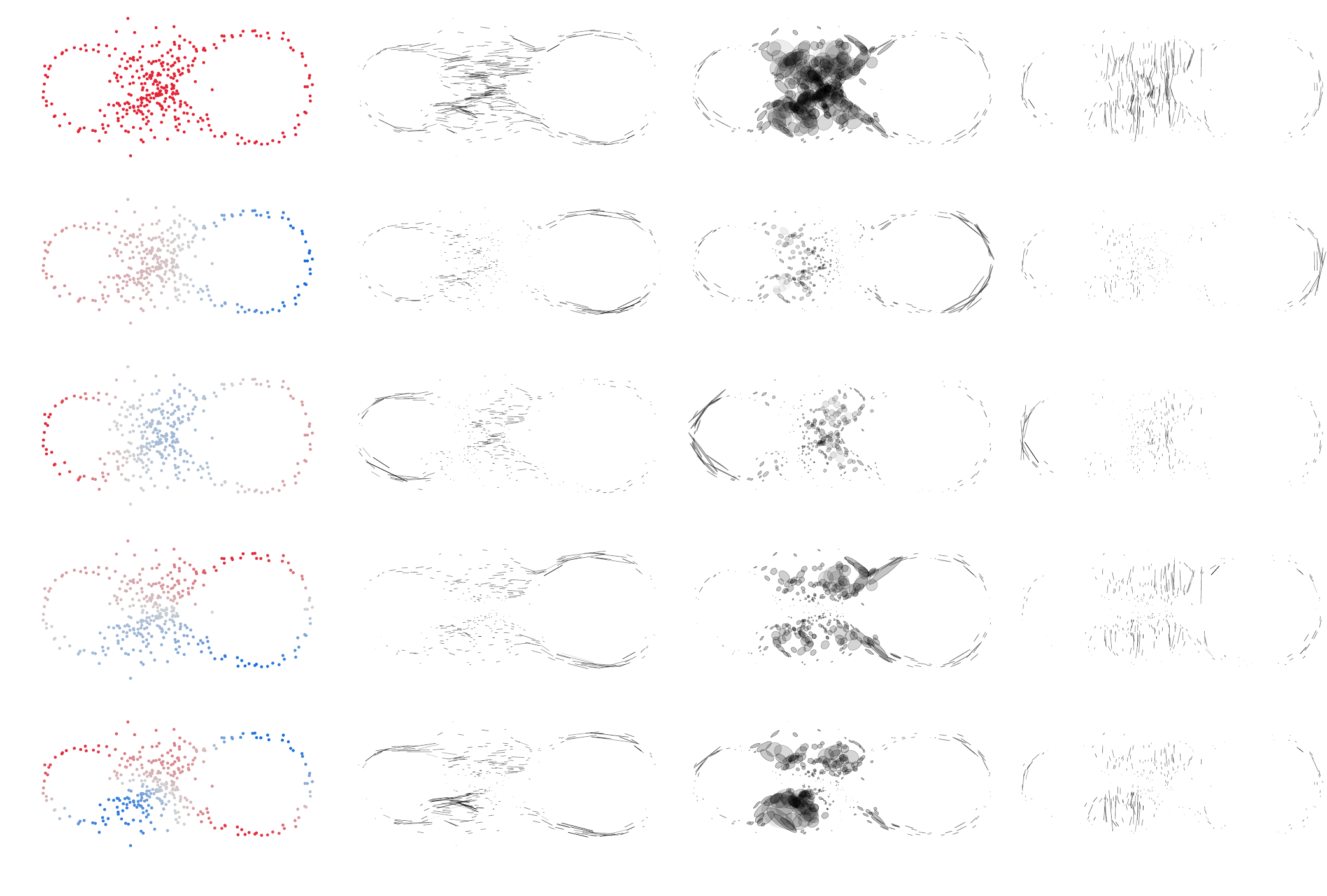}
\put(12.9,54.4){\makebox(0,0)[c]{$\phi_1 = 1$}}
\put(37.6,54.4){\makebox(0,0)[c]{$\phi_1 dx \otimes dx = dx \otimes dx$}}
\put(62.4,54.4){\makebox(0,0)[c]{$\phi_1 dx \otimes dy = dx \otimes dy$}}
\put(87.1,54.4){\makebox(0,0)[c]{$\phi_1 dy \otimes dy = dy \otimes dy$}}
\put(12.9,41.3){\makebox(0,0)[c]{$\phi_2$}}
\put(37.6,41.3){\makebox(0,0)[c]{$\phi_2 dx \otimes dx$}}
\put(62.4,41.3){\makebox(0,0)[c]{$\phi_2 dx \otimes dy$}}
\put(87.1,41.3){\makebox(0,0)[c]{$\phi_2 dy \otimes dy$}}
\put(12.9,28.2){\makebox(0,0)[c]{$\phi_3$}}
\put(37.6,28.2){\makebox(0,0)[c]{$\phi_3 dx \otimes dx$}}
\put(62.4,28.2){\makebox(0,0)[c]{$\phi_3 dx \otimes dy$}}
\put(87.1,28.2){\makebox(0,0)[c]{$\phi_3 dy \otimes dy$}}
\put(12.9,15.2){\makebox(0,0)[c]{$\phi_4$}}
\put(37.6,15.2){\makebox(0,0)[c]{$\phi_4 dx \otimes dx$}}
\put(62.4,15.2){\makebox(0,0)[c]{$\phi_4 dx \otimes dy$}}
\put(87.1,15.2){\makebox(0,0)[c]{$\phi_4 dy \otimes dy$}}
\put(12.9,2.1){\makebox(0,0)[c]{$\phi_5$}}
\put(37.6,2.1){\makebox(0,0)[c]{$\phi_5 dx \otimes dx$}}
\put(62.4,2.1){\makebox(0,0)[c]{$\phi_5 dx \otimes dy$}}
\put(87.1,2.1){\makebox(0,0)[c]{$\phi_5 dy \otimes dy$}}
  \end{overpic}
  \caption{\textbf{Spanning sets for symmetric 2-tensors.}
  We can use eigenfunctions $\phi_i$ of the Markov chain as a function basis (left column).
  The two ambient coordinates $x$ and $y$ are an embedding of the data, so we can use their gradients $\nabla x$ and $\nabla y$ to construct a basis for the space of vector fields by multiplying them by basis functions (middle and right columns).}
  \label{fig:2-tensor-basis}
\end{figure}

\subsubsection{Riemannian metric}

The metric of 2-tensors is given by the formula
\[
g(f dh_1 \otimes dh_2, f' dh'_1 \otimes dh'_2) = ff' \Gamma(h_1, h'_1)\Gamma(h_2, h'_2).
\]
We can compute the metric of $\p{i}dx_{j_1} \otimes dx_{j_2}$ and $\p{i'}dx_{j_1'} \otimes dx_{j_2'}$ at a point $p$ in the 7-tensor
\begin{equation}
\label{eq: g02}
\begin{split}
\textbf{g}^{(0,2)}_{pij_1j_2i'j_1'j_2'} 
&= \boldsymbol{\phi}_i(p)\boldsymbol{\phi}_{i'}(p) \boldsymbol\Gamma(\textbf{x}_{j_1},\textbf{x}_{j_1'})(p)
\boldsymbol\Gamma(\textbf{x}_{j_2},\textbf{x}_{j_2'})(p) \\
&= \textbf{U}_{pi}\textbf{U}_{pi'} \boldsymbol\Gamma_p(\textbf{x}_{j_1},\textbf{x}_{j_1'})
\boldsymbol\Gamma_p(\textbf{x}_{j_2},\textbf{x}_{j_2'}),\\
\end{split}
\end{equation}
which has dimension $n \times n_1 \times d \times d \times n_1 \times d \times d$.
We reshape $\textbf{g}^{(0,2)}$ into an $n \times n_1d^2 \times n_1d^2$ 3-tensor $\textbf{g}^{(0,2)}_{pIJ}$ so, at a fixed point $p$, $\textbf{g}^{(0,2)}_p$ is a symmetric $n_1d^2 \times n_1d^2$ matrix which represents the metric at $p$.
If $\boldsymbol\alpha,\boldsymbol\beta \in \R^{n_1d}$ are 2-tensors, their metric is given by
\[
(\boldsymbol\alpha^T \textbf{g}^{(0,2)}_p \boldsymbol\beta)_{p=1,...,n} \in \R^n.
\]

\subsubsection{Inner product}

The inner product of 2-tensors is given by the formula $\inp{\alpha}{\beta} = \int g(\alpha,\beta) d\mu$, so we can compute the $n_1d^2 \times n_1d^2$ Gram matrix $\mathbf{G}^{(0,2)}$ of 2-tensors as
\begin{equation}
\label{eq: G02}
\mathbf{G}^{(0,2)}_{IJ} = \sum_{p=1}^n \mathbf{g}^{(0,2)}_{pIJ} \boldsymbol{\mu}_p.
\end{equation}
If $\boldsymbol{\alpha}, \boldsymbol{\beta} \in \R^{n_1d}$ are 2-tensors, their inner product is given by $\boldsymbol{\alpha}^T \textbf{G}^{(0,2)} \boldsymbol{\beta}\in \R$.

\begin{computationalnote}
\label{note: symmetric 2 tensors}
We will often want to work with symmetric 2-tensors (for example, the Hessian is symmetric).
As with the alternating 2-tensors (2-forms), we can more efficiently represent these with fewer indices in the sums
\[
\boldsymbol\alpha = \sum_{i=1}^{n_1} \sum_{j_1 \leq j_2} \boldsymbol\alpha_{ij_1j_2} \p{i} dx_{j_1} \otimes dx_{j_2},
\]
so 2-tensors are represented by the $n_1 \times d(d+1)/2$ coefficient matrices $(\boldsymbol\alpha_{iJ})_{i\leq n_1, J\leq d(d+1)/2}$.
We flatten these into vectors of length $n_1d(d+1)/2$, and so identify the space of symmetric 2-tensors $\Omega^{1}(M)^{
\otimes2,\ \textnormal{sym}}$ with $\R^{n_1d(d+1)/2}$.
In this reduced basis, we need to correct the metric by doubling the off-diagonal entries
\[
\textbf{g}^{(0,2),\ \textnormal{sym}}_{pIJ}
= \begin{cases}
    \textbf{g}^{(0,2)}_{pIJ}
    & I = J \\
    2\textbf{g}^{(0,2)}_{pIJ}
    & I \neq J
\end{cases}
\]
which compensates for the fact that the off-diagonal entries are represented half as many times in $\Omega^{1}(\textbf{M})^{\otimes2,\ \textnormal{sym}}$ as they are in $\Omega^k(\textbf{M})^{\otimes2}$.
We make the corresponding change to the Gram matrix.
\end{computationalnote}

\subsubsection{Action on vector fields}
\label{sub: action of 2 tensors}

Just as a 1-form $\alpha \in \Omega^1(M)$ can act on a vector field $X \in \mathfrak{X}(M)$ to produce a function $\alpha(X) = g(\alpha, X^\flat) \in \A$, a $(0,2)$-tensor $\alpha \in \Omega^1(M)^{\otimes2}$ can \boldblue{act on a pair of vector fields} $X,Y \in \mathfrak{X}(M)$ to produce a function $\alpha(X,Y) \in \A$ via
\[\alpha(X,Y) = g(\alpha, X^\flat \otimes Y^\flat).\]
This offers another interpretation of the $(0,2)$ notation, as describing a map from two vector fields to zero vector fields (i.e. functions).
If 
\[
\boldsymbol{\alpha} = \sum_{i=1}^{n_1} \sum_{j_1, j_2 = 1}^{d} \boldsymbol\alpha_{ij_1j_2} \p{i} dx_{j_1} \otimes dx_{j_2}
\qquad
\textbf{X} = \sum_{s_1 = 1}^{n_1} \sum_{t_1 = 1}^{d} \textbf{X}_{s_1 t_1} \p{s_1} \nabla x_{t_1}
\qquad
\textbf{Y} = \sum_{s_2 = 1}^{n_1} \sum_{t_2 = 1}^{d} \textbf{Y}_{s_2 t_2} \p{s_2} \nabla x_{t_2},
\]
then we can compute the action $\boldsymbol{\alpha}(\textbf{X},\textbf{Y})$ at a point $p$ as
\begin{align*}
\boldsymbol{\alpha}(\textbf{X},\textbf{Y})_p
&= \sum_{i,s_1,s_2=1}^{n_1} \sum_{j_1,j_2,t_1,t_2 = 1}^{d} \boldsymbol\alpha_{ij_1j_2} \textbf{X}_{s_1 t_1} \textbf{Y}_{s_2 t_2} 
\p{i} \p{s_1} \p{s_2}
\boldsymbol\Gamma_p(\textbf{x}_{j_1},\textbf{x}_{t_1})\boldsymbol\Gamma_p(\textbf{x}_{j_2},\textbf{x}_{t_2}) \\
&= \sum_{i,s_1,s_2=1}^{n_1} \sum_{j_1,j_2,t_1,t_2 = 1}^{d} \boldsymbol\alpha_{ij_1j_2} \textbf{X}_{s_1 t_1} \textbf{Y}_{s_2 t_2} \textbf{U}_{pi} \textbf{U}_{ps_1} \textbf{U}_{ps_2}
\boldsymbol\Gamma_p(\textbf{x}_{j_1},\textbf{x}_{t_1})\boldsymbol\Gamma_p(\textbf{x}_{j_2},\textbf{x}_{t_2}).
\end{align*}


\subsubsection{Visualising 2-tensors}

We can visualise a 2-tensor $\alpha$ with the same approach as for forms (\ref{sec: visualising k forms}).
At each point, we compute a $d \times d$ matrix whose entries are given by the action of $\alpha$ on the coordinate vector fields
\[
\big(g(\alpha, dx_{j_1} \otimes dx_{j_2})\big)_{j_1,j_2=1,...,d} \in \R^{d \times d}.
\]
This is generally much harder to visualise, but if $\alpha \in \Omega^{k,\textnormal{sym}}(M)^{\otimes2}$ then the matrix is symmetric, and so can be represented as an ellipsoid (or ellipse in 2d).
This is still a crude representation, because each axis of the ellipsoid has a signed magnitude which cannot all be represented together by one colour.

\section{Frame theory and weak formulations}
\label{sec: frame_theory_weak_formulations}

The spaces of functions, vector fields, forms, and other tensors provide a general framework for representing calculus and geometry objects, but much of the interesting information is contained in maps between these spaces.
For example, given some function $f \in \A$, we would like to know its gradient $\nabla f \in \mathfrak{X}(M)$.
Section \ref{sec: overview} provides a direct way to visualise $\nabla f$ by computing its action on coordinate functions, but we would really like to represent $\nabla f$ in the spanning set for vector fields
\begin{equation}
\label{eq: weak formulation gradient}
\nabla f \approx \sum_{i=1}^{n_1}\sum_{j=1}^d \textbf{X}_{ij}\phi_i \nabla x_j.
\end{equation}
This would allow us to compute more things, like its Riemannian metric with other vector fields, or the divergence of $\nabla f$ to get the Laplacian of $f$.
However, the coefficients $\textbf{X}_{ij}$ are not immediately available from $f$, and depend on the choice of spanning set.

We can resolve this problem by solving an inverse problem using a \boldblue{weak formulation}.
If $\nabla f$ satisfies \eqref{eq: weak formulation gradient}, we can take the inner product of both sides with each spanning set element $\phi_{i'} \nabla x_{j'}$ to get
\begin{equation}
\label{eq: weak formulation inner product demo}
\inp{\phi_{i'} \nabla x_{j'}}{\nabla f}
= \sum_{i=1}^{n_1}\sum_{j=1}^d \textbf{X}_{ij} \inp{\phi_{i'} \nabla x_{j'}}{\phi_i \nabla x_j},
\end{equation}
which is the weak formulation of \eqref{eq: weak formulation gradient}.
If $\textbf{X} = (\textbf{X}_{ij})_{i \leq n_1, j \leq d} \in \R^{n_1 d}$ is the vector of coefficients we want to compute, then the right-hand side of \eqref{eq: weak formulation inner product demo} is just the product $\textbf{G}^{(1)} \textbf{X}$, where $\textbf{G}^{(1)}$ is the $n_1d \times n_1d$ Gram matrix of vector fields defined in \eqref{eq: G1}.
The left-hand side can be written using the carré du champ as
\[
\inp{\phi_{i'} \nabla x_{j'}}{\nabla f}
= \int_M g(\phi_{i'} \nabla x_{j'}, \nabla f) d\mu
= \int_M \phi_{i'} \Gamma(x_{j'}, f) d\mu,
\]
and so easily computed from $f$, as discussed below.
This vector, called the \boldblue{load vector} in numerical analysis, is denoted $(\boldsymbol{\nabla f})^\textnormal{weak} \in \R^{n_1 d}$, so \eqref{eq: weak formulation inner product demo} becomes
\begin{equation}
\label{eq: weak linear system grad demo}
(\boldsymbol{\nabla f})^\textnormal{weak} = \textbf{G}^{(1)} \textbf{X},
\end{equation}
which we can solve for $\textbf{X}$ by inverting $\textbf{G}^{(1)}$.
However, there are several questions of numerical stability:
\begin{enumerate}
    \item When can $\textbf{G}^{(1)}$ be inverted?
    \item When can $\textbf{G}^{(1)}$ be inverted stably?
    \item In practice, we will only have estimates of the Gram matrix and load vector from finite data.
    Will these estimates converge to the true solution as $n \to \infty$ and $n_1 \to \infty$?
\end{enumerate}
We will address these in this section using \boldblue{frame theory} \cite{mallat1999wavelet}.
However, the conclusion will be that weak formulations like \eqref{eq: weak linear system grad demo} are well-behaved, and can be solved by standard numerical methods for rank-deficient systems.
\boldblue{
If you are mainly interested in applying these methods, you can safely skip ahead.
}

\subsection{Bessel sequences and frames for Hilbert spaces}

We will address these questions in the setting of \boldblue{Hilbert spaces}, which will simultaneously cover the infinite-dimensional continuous case of $\mathfrak{X}(M)$, $\Omega^k(M)$ etc (and their Hilbert space completions $L^2\mathfrak{X}(M)$, $L^2\Omega^k(M)$ etc), the finite-dimensional discrete analogues $\mathfrak{X}(\textbf{M})$, $\Omega^k(\textbf{M})$ etc, and the limiting behaviour as $n, n_1 \to \infty$.

The computational approach described above involves representing vectors $x$ in a Hilbert space $H$ in terms of a spanning set $\{ u_i \}_{i \in \N}$.
This set is called a \boldblue{frame} if it satisfies two inequalities called the \boldblue{frame conditions}, which guarantee that these representations are well-behaved, in the sense that the coefficients of $x$ in the spanning set are bounded and we can solve weak formulations like \eqref{eq: weak linear system grad demo} stably.
We consider the two frame conditions separately, starting with the most important one.

\subsubsection{Bessel sequences}

A crucial property of a spanning set is that, given any vector $x \in H$, the sequence of inner products between $x$ and elements of the set remains bounded.

\begin{definition}
A sequence of vectors $ \{ u_i \}_{i \in \N} $ in a Hilbert space $H$ is called a 
\boldblue{$B$-Bessel sequence} if there is a constant $B > 0$ (called the \boldblue{Bessel bound}) such that
\begin{equation}
\label{eq: bessel condition}
\sum_{i = 1}^\infty |\langle x, u_i \rangle|^2 \leq B \|x\|^2
\end{equation}
for all $x \in H$.
\end{definition}

We can interpret this definition as saying that the sequence of inner products with a Bessel sequence $(\langle x, u_i \rangle)_{i \in \N}$ is in $\ell_2$, and that the linear map $x \mapsto (\langle x, u_i \rangle)_{i \in \N}$ is bounded.

\begin{definition}
If $\{ u_i \}_{i \in \N}$ is a $B$-Bessel sequence, then the \boldblue{analysis operator} is the linear map $T : H \to \ell_2$ given by $x \mapsto (\langle x, u_i \rangle)_{i \in \N}$, which is bounded with $\|T\| \leq \sqrt{B}$.
Its adjoint $T^* : \ell_2 \to H$ is called the \boldblue{synthesis operator}, which is also bounded with $\|T^*\| = \|T\|$.
\end{definition}

If $T$ is the analysis operator of a Bessel sequence $\{ u_i \}_{i \in \N}$, then $\|T\|^2$ is the smallest $B$ such that $\{ u_i \}_{i \in \N}$ is a $B$-Bessel sequence.
We can easily derive a formula for the synthesis operator because, if $a = (a_i) \in \ell_2$ and $x \in H$, then
\[
\inp{a}{Tx}_{\ell_2} = \sum_{i=1}^\infty a_i \inp{x}{u_i}_H = \inp{x}{\sum_{i=1}^\infty a_iu_i}_H,
\]
so $T^*a = \sum_{i=1}^\infty a_iu_i$.
If $x \in H$, we will say that a sequence $a \in \ell_2$ contains the \boldblue{sequence coefficients} of $x$ if $x = \sum_{i=1}^\infty a_i u_i = T^*a$.
As such, the analysis operator \textit{analyses} an element of $H$ by computing its inner product sequence, and the synthesis operator \textit{synthesises} an element of $H$ from such a sequence.

The fact that the sequence $\{ u_i \}_{i \in \N}$ is a spanning set means that $T$ is injective, and $T^*$ has a dense range in $H$.
Furthermore, the set $\{ u_i \}_{i \in \N}$ is linearly dependent if and only if  $T^*$ is non-injective, because they both happen exactly when there is some non-zero $a \in \ell_2$ where $T^*a = \sum a_iu_i = 0$.
In this case, any $x \in H$ can be represented by multiple coefficient sequences: if $x = T^*a$ then also $x = T^*(a + \ker(T^*))$.

\begin{definition}
If $\{ u_i \}_{i \in \N}$ is a Bessel sequence in a Hilbert space $H$ with analysis operator $T$ and synthesis operator $T^*$, the \boldblue{Gram operator} is the map $G = TT^*: \ell_2 \to \ell_2$, 
which can be written
\begin{equation}
(Ga)_i = \sum_{j=1}^\infty \langle u_i, u_j \rangle a_j
\end{equation}
for $a \in \ell_2$.
\end{definition}

If $\{u_i\}$ is finite, then $G$ is the familiar \boldblue{Gram matrix}.
The Gram operator $G$ is self-adjoint, positive, and bounded with $\|G\| = \|T\|^2 \leq B$.
Since $T$ is injective, $G$ satisfies $\ker(G) = \ker(T^*)$, so is injective if and only if $\{ u_i \}_{i \in \N}$ is linearly independent (so form a basis for $H$).

\subsubsection{Weak formulations in Bessel sequences}

We can use the language of Bessel sequences to describe weak formulations like \eqref{eq: weak linear system grad demo}.
In this setting, we have a Hilbert space $H$ (e.g. the complete space of vector fields $L^2\mathfrak{X}(M)$), and want to represent a vector $x \in H$ (e.g. $\nabla f$) as a linear combination of Bessel sequence elements.
As in \eqref{eq: weak linear system grad demo}, we will assume that we have access only to its sequence of inner products with sequence elements, i.e. $Tx$, and would like to recover $x$ from that.
If $a \in \ell_2$ contains sequence coefficients for $x$, i.e. $x = T^*a$, then we would like to solve
\begin{equation}
\label{eq: weak formulation}
Tx = TT^*a = Ga
\end{equation}
for $a$.
Recall that the Gram operator $G$ will have a non-zero kernel if the set $\{ u_i \}_{i \in \N}$ is overcomplete, and so will not be invertible.
This is not really a problem, because it just means that (\ref{eq: weak formulation}) will have multiple solutions, and we only need one.
A natural choice of solution is the one with the \boldblue{smallest $\ell_2$ norm}, which is given by the following construction.


\begin{definition}
If $H'$ is a Hilbert space and $S:H' \to H'$ is a bounded self-adjoint operator, then $S$ restricts to a bijection $\ker(S)^\perp \to \textnormal{ran}(S)$, where $\textnormal{ran}(S)$ denotes the range of $S$.
The \boldblue{Moore-Penrose pseudo-inverse} of $S$ is the map $S^+: \textnormal{ran}(S) \oplus \ker(S) \to H'$ given by $S^+ =  (S|_{\ker(S)^\perp})\inv \oplus 0|_{\ker(S)}$.
\end{definition}

The Moore-Penrose pseudo-inverse $S^+$ generalises the inverse, because $S^+ = S^{-1}$ when $S$ is bijective.
It is \boldblue{densely defined} because $\textnormal{ran}(S)$ is dense in $\ker(S)^\perp$, and generally \boldblue{unbounded}.
To see that the solution is the \textit{smallest}, notice that, if $y \in \textnormal{ran}(S)$ and $y = Sx$, then the coset $x + \ker(S)$ contains all the solutions of $y = Sx$.
Its minimal element is orthogonal to $\ker(S)$, and this is precisely the element obtained from $S^+$.

\begin{example}
If $H'$ has finite dimension $n$, then $S$ can be represented by an $n \times n$ matrix $\textbf{S}$.
If $\textbf{S} = \textbf{U}\boldsymbol\Sigma \textbf{V}^T$ is the singular value decomposition of $\textbf{S}$, then define $\boldsymbol{\Sigma}^+$ by inverting the non-zero singular values in $\boldsymbol\Sigma$, and leaving the zeros as zeros.
Then $\textbf{S}^+ = \textbf{U}\boldsymbol\Sigma^+ \textbf{V}^T$.
\end{example}

In our case, we can solve the weak formulation (\ref{eq: weak formulation}) by applying the pseudo-inverse of $G : \ell_2 \to \ell_2$ to $Tx \in \ell_2$.
If we suppose that $x = T^*a \in \textnormal{ran}(T^*)$, we know that $Tx \in \textnormal{ran}(TT^*) = \textnormal{ran}(G)$, so
$$
a = G^+Tx = (G|_{\ker(T^*)^\perp})\inv Tx \in \ker(T^*)^\perp,
$$
where we use that $\ker(G) = \ker(T^*)$.
In this sense, the pseudo-inverse of $G$ can guarantee a solution to a weak formulation in a Bessel sequence, which is optimal in the sense of $\ell_2$ norm.

However, there are two problems.
First, the fact that $G^+$ is \textit{generally unbounded} means that, in practice, using it to compute solutions to (\ref{eq: weak formulation}) might be numerically unstable, as tiny perturbations could be magnified uncontrollably.
Second, the fact that $\text{ran}(T^*)$ is only \textit{dense} in $H$ means that there might be some $x \in H$ where $Tx \notin \textnormal{ran}(G)$, so (\ref{eq: weak formulation}) has no solution at all.
These two problems are actually equivalent, as we now show.
First, we prove the following standard Lemma about general Moore-Penrose pseudo-inverses.

\begin{lemma}
\label{lemma: Moore penrose bounded conditions}
Let $H'$ be a Hilbert space and $S:H' \to H'$ be a bounded self-adjoint operator.
Then the following are equivalent:
\begin{enumerate}
    \item $S^+$ is defined everywhere on $H'$,
    \item $\textnormal{ran}(S)$ is closed,
    \item $S^+$ is bounded.
\end{enumerate}
\end{lemma}
\begin{proof}
The domain of $S^+$ is
\[
\textnormal{ran}(S) \oplus \ker(S) 
= \textnormal{ran}(S) \oplus \textnormal{ran}(S)^\perp \\
\subseteq \overline{\textnormal{ran}(S)} \oplus \textnormal{ran}(S)^\perp \\
= H',
\]
so we see that (1) and (2) are equivalent.
To see that (2) implies (3), if $\textnormal{ran}(S)$ is closed then the bijective linear map $S|_{\ker(S)^\perp} : \ker(S)^\perp \to \textnormal{ran}(S)$ is a map between Hilbert spaces, and so the bounded inverse (open mapping) theorem ensures that $(S|_{\ker(S)^\perp})\inv$ is bounded, and hence so is $S^+$.
To see that (3) implies (2), notice that if $S^+$ is bounded, then $S|_{\ker(S)^\perp} : \ker(S)^\perp \to \textnormal{ran}(S)$ is a bounded bijection with bounded inverse, so it is a homeomorphism.
Then $\textnormal{ran}(S) = S(\ker(S)^\perp)$ is homeomorphic to the complete space $\ker(S)^\perp$, so is itself complete and hence closed in $H'$.
\end{proof}

We now apply this Lemma to the Gram operator of a Bessel sequence to find the following standard properties.

\begin{corollary}
\label{cor: frame equivalence}
Let $\{ u_i \}_{i \in \N}$ be a Bessel sequence in $H$ with Gram operator $G$ and synthesis operator $T^*$, and assume the span of $\{ u_i \}$ is dense in $H$.
Then the following are equivalent:
\begin{enumerate}
    \item $G^+$ is bounded (and defined everywhere on $\ell_2$),
    \item $\text{ran}(G)$ is closed,
    \item there exists $A > 0$ such that $A\|x\|^2 \leq \sum_{i=1}^\infty |\langle x, u_i \rangle|^2$ for all $x \in H$,
    \item $\textnormal{ran}(T^*)$ is closed (and thus equals $H$).
\end{enumerate}
\end{corollary}

\begin{proof}
Notice that (1) and (2) are equivalent by Lemma \ref{lemma: Moore penrose bounded conditions}.
To see that (3) and (4) are equivalent, we write condition (3) as $\|Tx\|_{\ell_2} \geq \sqrt{A}\|x\|_H$, which states that the analysis operator $T$ is bounded below.
By the closed range theorem and the open mapping theorem, $T$ is bounded below if and only if its adjoint $T^*$ is surjective, so $\textnormal{ran}(T^*) = H$. 

We now show that (3) and (4) imply (2).
Condition (3) implies that $\text{ran}(T)$ is closed, because if $(Tx_n)$ is a sequence in $\text{ran}(T)$ converging to some $y \in \ell_2$, then $(x_n)$ is a Cauchy sequence in $H$ by the lower bound on $T$, and so converges to some $x \in H$.
By (4), $\text{ran}(G) = T(\text{ran}(T^*)) = T(H) = \text{ran}(T)$ is closed.

Finally, we show that (1) implies (4).
If $G^+$ is bounded, then $G$ is bounded below on $\ker(G)^\perp$, meaning there exists $C > 0$ such that $\|Ga\| \ge C\|a\|$ for all $a \in \ker(G)^\perp$.
Since $G$ is a positive self-adjoint operator, the spectral mapping theorem implies that $\inp{Ga}{a} \ge C\|a\|^2$ for all $a \in \ker(G)^\perp$.
Using the identity $\inp{Ga}{a} = \inp{TT^*a}{a} = \|T^*a\|^2$ and the fact that $\ker(T^*) = \ker(G)$, we have $\|T^*a\| \ge \sqrt{C}\|a\|$ for all $a \in \ker(T^*)^\perp$.
This lower bound implies that $\textnormal{ran}(T^*)$ is closed, by applying the same argument as above in the Hilbert space $\ker(T^*)^\perp$.
\end{proof}

Applying this Lemma to the Gram operator $G : \ell_2 \to \ell_2$ of a Bessel sequence $\{ u_i \}_{i \in \N}$, we see that the two problems with using $G^+$ to solve weak formulations are both resolved by requiring the condition (3) on $\{ u_i \}_{i \in \N}$, which we discuss next.

\subsubsection{Frames}

Corollary \ref{cor: frame equivalence} motivates the following definition of a special type of Bessel sequence $\{ u_i \}_{i \in \N}$.

\begin{definition}
A sequence of vectors $\{ u_i \}_{i \in \N}$ in a Hilbert space $H$ is called an \boldblue{$(A,B)$-frame} if there are constants $0 < A \leq B < \infty$ (called the \boldblue{frame bounds}) such that
\begin{equation}
\label{eq: frame condition}
A\|x\|^2 \leq \sum_{i=1}^\infty |\langle x, u_i \rangle|^2 \leq B\|x\|^2
\end{equation}
for all $x \in H$. If $A = B$, the frame is called a \boldblue{tight frame}, and if $A = B = 1$, it is called a \boldblue{Parseval frame}.
\end{definition}

The lower bound condition ensures that the bijective map $G|_{\ker(T^*)^\perp} : \ker(T^*)^\perp \to \textnormal{ran}(T)$ has a bounded inverse, with 
\[
\|(G|_{\ker(T^*)^\perp})\inv\| \leq \frac{1}{A},
\]
and so $\|G^+\| \leq 1/A$.
Following Corollary \ref{cor: frame equivalence}, it also means that $\textnormal{ran}(T^*) = H$, so that every $x \in H$ has at least one sequence of coefficients $a \in \ell_2$ where $x = T^*a$.
In particular, we can now solve the weak formulation \eqref{eq: weak formulation} for every $x \in H$ using $G^+$, and the solution will be numerically stable because $G^+$ is bounded.
We can quantify the stability of $G^+$ through its \boldblue{condition number} $B/A$, which measures the sensitivity of $G^+a$ to small changes in $a$.
We should try to pick a spanning set $\{ u_i \}_{i \in \N}$ that minimises $B/A$, with the best possible scenario being when $\{ u_i \}_{i \in \N}$ is a \textit{tight frame} so $B/A = 1$.

We can interpret the lower frame bound as forcing $G$ to have a positive \boldblue{spectral gap}, meaning its non-zero eigenvalues must be at least $A$.
Concretely, the spectrum of $G$ must be a subset of $\{0\} \cup [A,B]$, where $\{0\}$ is the spectrum of $G|_{\ker(T^*)} = 0$, and the spectrum of $G|_{\ker(T^*)^\perp}$ is contained in $[A,B]$.
Just as $\|G\|$ is the smallest possible upper bound $B$ for a given set $\{ u_i \}_{i \in \N}$, $1/\|G^+\|$ is the largest possible lower bound $A$.
As a very special case, if the spanning set forms a tight frame, then $G^+ = (1/A)I$ on $\textnormal{ran}(G)$.

\begin{example}
\label{eg: finite frame}
In the special case that $H$ is finite-dimensional and $\{ u_i \}_{i=1,...,N}$ is a finite spanning set, $\{ u_i \}_{i=1,...,N}$ is always a frame with optimal frame bounds given by the smallest and largest non-zero eigenvalues of the $N \times N$ Gram matrix $\textbf{G}$.
If $H$ is $d$-dimensional then $\textbf{G}$ has rank $d$, and so we can factorise $\textbf{G} = \textbf{P}^T \boldsymbol{\Sigma} \textbf{P}$ where $\textbf{P}$ is a $d \times N$ orthonormal matrix, and $\boldsymbol{\Sigma}$ is a $d \times d$ diagonal matrix containing the positive eigenvalues $\lambda_i$ of $\textbf{G}$.
We can then write $\textbf{G}^+ = \textbf{P}^T \boldsymbol{\Sigma}\inv \textbf{P}$, where $\boldsymbol{\Sigma}\inv$ is diagonal with entries $1/\lambda_i$.
If $A$ is the smallest eigenvalue of $\textbf{G}$ then $1/A$ is the largest eigenvalue (i.e. operator norm) of $\textbf{G}^+$.
\end{example}

\subsection{Frame bounds for the continuous spanning set}

Before considering the discrete spanning sets for $\Omega^1(\textbf{M})$, we will first analyse the frame conditions for a continuous spanning set for $\Omega^1(M)$.
Although the numerical stability concerns discussed above only apply to the discrete spanning set, we will need the following analysis of the continuous case to prove the convergence of the discrete case to the continuous case as $n, n_1 \to \infty$.
We would like to know the frame bounds for the spanning sets for all the different spaces of tensors, but it will suffice to consider only the case of 1-forms, as we explain below.

We first show that the global inner product on the space of 1-forms $\Omega^1(M)$ is given by integration of the pointwise inner product $g$, and so testing the global frame conditions reduces to testing their pointwise equivalents.
Given a 1-form $\alpha \in \Omega^1(M)$, we can evaluate the $\ell_2$ norm in the frame condition as follows.

\begin{lemma}
\label{lemma: frame inner product integral}
Suppose $\{ \p{i}dx_j \}_{i \in \N, j\leq d}$ is a spanning set of $\Omega^1(M)$, where $\p{i}$ is an orthonormal basis of $\A = L^2(\mu)$.
Then
\[
\sum_{i=1}^\infty \sum_{j=1}^d |\inp{\p{i}dx_j}{\alpha}|^2 = \int \Big( \sum_{j=1}^d g(dx_j, \alpha)^2 \Big)d\mu
\]
for all $\alpha \in \Omega^1(M)$.
\end{lemma}
\begin{proof}
We can rewrite
\[
\sum_{i=1}^\infty \sum_{j=1}^d |\inp{\p{i}dx_j}{\alpha}|^2
= \sum_{i=1}^\infty \sum_{j=1}^d \big|\int \p{i} g(dx_j, \alpha) d\mu \big|^2
= \sum_{j=1}^d \Big( \sum_{i=1}^\infty \big|\inp{\p{i}}{g(dx_j, \alpha)}_{L^2(\mu)}\big|^2 \Big).
\]
Since $\{ \p{i} \}_{i \in \N}$ is an orthonormal basis of $L^2(\mu)$ (or more generally, a Parseval frame\footnote{
if $\{ \p{i} \}_{i \in \N}$ were only a frame with bounds $A'$ and $B'$, we would have $A' \|g(dx_j, \alpha)\|^2 \leq
\sum_{i=1}^\infty \big|\inp{\p{i}}{g(dx_j, \alpha)}\big|^2
\leq B' \|g(dx_j, \alpha)\|^2$, and the following analysis would all still apply.
}), this is equal to
\[
\sum_{j=1}^d \|g(dx_j, \alpha)\|_{L^2(\mu)}^2
= \sum_{j=1}^d \int g(dx_j, \alpha)^2 d\mu
= \int \Big( \sum_{j=1}^d g(dx_j, \alpha)^2 \Big)d\mu
\]
for all $\alpha \in \Omega^1(M)$.
\end{proof}
This lets us test the frame conditions for $\{ \p{i}dx_j \}$ globally by testing them locally for $\{dx_j\}_{j\leq d}$.
In particular, if $\{dx_j\}$ satisfies the \boldblue{pointwise frame condition}
\begin{equation}
\label{eq: pointwise frame condition}
A g(\alpha,\alpha) \leq \sum_{j=1}^d g(dx_j, \alpha)^2 \leq Bg(\alpha,\alpha)
\end{equation}
$\mu$-almost everywhere for constants $A$ and $B$, then integration of \eqref{eq: pointwise frame condition} implies the global frame condition \eqref{eq: frame condition}.
In general, we have the following lemma.

\begin{lemma}
\label{lemma: pointwise frame eigenvalues}
Let $A(p)$ and $B(p)$ denote the smallest and largest non-zero eigenvalues of the carré du champ matrix $(\Gamma(x_i,x_j)(p))_{i,j=1,...d}$ at the point $p$.
Then $\alpha \in \Omega^1(M)$ satisfies
\begin{equation}
\label{eq: pointwise variable frame condition}
A(p) g(\alpha,\alpha)(p) \leq \sum_{j=1}^d g(dx_j, \alpha)(p)^2 \leq B(p) g(\alpha,\alpha)(p)
\end{equation}
for all $p$.
\end{lemma}
\begin{proof}
This is a direct application of Example \ref{eg: finite frame}.
At a point $p$, we can view $\alpha$ as an element of the semi-inner product space $(\Omega^1(M), g(p))$ (on a manifold, this is just the cotangent space $T_p^*M$) which is finite-dimensional and spanned by the set $\{dx_j\}_{j=1,...,d}$.
The gram matrix of this set is $(\Gamma(x_i,x_j)(p))_{i,j=1,...d}$, and so the set forms an $(A(p), B(p))$-frame.
\end{proof}

We can use Lemma \ref{lemma: pointwise frame eigenvalues} to obtain the pointwise frame condition (\ref{eq: pointwise frame condition}), and hence the global one (\ref{eq: frame condition}), by showing that $A(p)$ is uniformly bounded away from zero $\mu$-almost everywhere and $B(p)$ is uniformly bounded above $\mu$-almost everywhere.

Notice that this Lemma also gives the conditions under which the frame bounds will be satisfied for all the different spaces of forms and tensors.
Since the Riemannian metric in these cases is given by compound matrices and tensor products formed from the base metric on 1-forms (or vector fields), its eigenvalues can also be obtained from it.
For example, the Riemannian metric on $k$-forms is, at each point $p$, given by the $k$-th compound matrix of the metric on 1-forms, so its eigenvalues are all possible products of $k$ eigenvalues of the metric on 1-forms.
Thus, if the frame bounds hold for 1-forms with constants $A$ and $B$, then they will hold for $k$-forms with constants $A^k$ and $B^k$.
Something analogous holds for all tensor spaces, so it suffices to consider only the frame bounds for 1-forms.

\subsubsection{Upper frame condition}

It is easy to find an essentially uniform upper bound on $B(p)$ in the general case.

\begin{prop}
\label{prop: upper frame bound theoretical}
We have $B(p) \leq \|\sum_{i=1}^d \Gamma(x_i,x_i)\|_\infty$, and so the set $\{ \p{i}dx_j \}$ satisfies the upper bound
$$
\sum_{i=0}^\infty \sum_{j=1}^d |\inp{\p{i}dx_j}{\alpha}|^2 \leq B \|\alpha\|^2
$$
for $B = \|\sum_{i=1}^d \Gamma(x_i,x_i)\|_\infty$.
\end{prop}
\begin{proof}
The eigenvalues of a positive semi-definite matrix are bounded above by the trace, so 
\[
B(p) \leq \textnormal{tr}(\Gamma(x_i,x_j)) = \sum_{i=1}^d \Gamma(x_i,x_i)
\]
for all $p$.
Each $\Gamma(x_i,x_i) \in \A$ is bounded, so $B(p) \leq \|\sum_{i=1}^d \Gamma(x_i,x_i)\|_\infty$ uniformly.
\end{proof}

\subsubsection{Lower frame condition}

Obtaining a lower bound on $A(p)$ is harder.
We know that $A(p) > 0$ for all $p$, but if $A(p) \to 0$ then the lower frame condition will fail, since it requires $A(p) > A > 0$ $\mu$-almost everywhere for a uniform constant $A$.
We will not fully address this question here, besides the following observations:
\begin{enumerate}
    \item If $M$ is an \boldblue{isometrically embedded} manifold then the eigenvalues of $\Gamma(x_i,x_j)$ are either 0 or 1, so $A(p) = B(p) = 1$ everywhere, and the set $\{ \p{i}dx_j \}$ forms a Parseval frame.
    Although we have not yet developed a theory of isometric embeddings for general Markov triples in diffusion geometry, we expect that the same will be true in the general case.
    Since we almost always work with the induced metric from the ambient space, we will usually be in this situation.
    \item Suppose that $M$ is \boldblue{compact} with respect to some topology and the functions $\Gamma(x_i,x_j)$ are continuous in that topology. Then $A(p)$ is a continuous positive function on a compact space, so $\inf_p A(p) > 0$ and the lower frame condition is satisfied.
    In particular, this holds for any compact manifold, even if the embedding is not isometric.
\end{enumerate}
It would be interesting to investigate and formalise these conditions in future work.
Note that, if we have an exact expression for the carré du champ operator $\Gamma$ (for example, if we are working with a known manifold) instead of a data-driven approximation, then these results are sufficient to guarantee stable solutions to the weak formulation using the spanning set $\{ \p{i}dx_j \}$.

\subsection{Frame bounds for the discrete spanning set}

The frame bounds that matter practically are the bounds on the set $\{ \boldsymbol{\phi}_{i}d\textbf{x}_j \}$, because this is what we actually use for computation.
Since the spaces involved are now finite-dimensional, any spanning set is automatically a frame (Example \ref{eg: finite frame}), and so the first two of the three questions posed at the start of this section are trivially true.
The important question is \boldblue{how the frame bounds depend on $n$}, and whether the solution of the weak formulation \eqref{eq: weak formulation} remains numerically stable as $n$ increases.
We now perform the same analysis as above, swapping integrals for sums to obtain a discrete analogue of Lemma \ref{lemma: frame inner product integral}.

\begin{lemma}
\label{lemma: frame inner product sum}
The spanning set $\{ \boldsymbol{\phi}_{i}d\textbf{x}_j \}_{i\leq n_1, j\leq d}$ of $\Omega^1(\textbf{M}) = \R^{n_1d}$ satisfies
\[
\sum_{i=1}^{n_1} \sum_{j=1}^d |\inp{\boldsymbol{\phi}_{i}d\textbf{x}_j}{\boldsymbol\alpha}|^2 
\leq \sum_{p=1}^n \Big( \sum_{j=1}^d g(d\textbf{x}_j, \boldsymbol\alpha)(p)^2 \Big)\boldsymbol\mu_p
\]
for all $\boldsymbol\alpha \in \Omega^1(\textbf{M})$, with equality when $n_1 = n$.
\end{lemma}
\begin{proof}
As before, we can evaluate the $\ell^2$ norm in the frame condition to get
\[
\sum_{i=1}^{n_1} \sum_{j=1}^d |\inp{\boldsymbol{\phi}_{i}d\textbf{x}_j}{\boldsymbol\alpha}|^2
= \sum_{i=1}^{n_1} \sum_{j=1}^d \big|\sum_{p=1}^n \boldsymbol{\phi}_{i}(p) g(d\textbf{x}_j, \boldsymbol\alpha)(p) \boldsymbol\mu_p \big|^2
= \sum_{j=1}^d \Big( \sum_{i=1}^{n_1} \big|\inp{\boldsymbol{\phi}_{i}}{g(d\textbf{x}_j, \boldsymbol\alpha)}_{L^2(\R^n, \boldsymbol\mu)}\big|^2 \Big).
\]
Since $n_1 \leq n$, this is bounded above by
\[
\sum_{j=1}^d \Big( \sum_{i=1}^{n} \big|\inp{\boldsymbol{\phi}_{i}}{g(d\textbf{x}_j, \boldsymbol\alpha)}_{L^2(\R^n, \boldsymbol\mu)}\big|^2 \Big),
\]
which, by Parseval's identity\footnote{in finite dimensions Parseval's identity is just Pythagoras' theorem.} in $L^2(\R^n, \boldsymbol\mu)$, is equal to
\[
\sum_{j=1}^d \|g(d\textbf{x}_j, \boldsymbol\alpha)\|_{L^2(\R^n, \boldsymbol\mu)}^2
= \sum_{j=1}^d \sum_{p=1}^n g(d\textbf{x}_j, \boldsymbol\alpha)(p)^2 \boldsymbol\mu_p
= \sum_{p=1}^n \Big( \sum_{j=1}^d g(d\textbf{x}_j, \boldsymbol\alpha)(p)^2 \Big)\boldsymbol\mu_p
\]
for all $\boldsymbol\alpha \in \Omega^1(\textbf{M}) = \R^{n_1d}$.
\end{proof}
As in the continuous case, this directly relates the global inner product on $\Omega^1(\textbf{M})$ to the $\boldsymbol\mu$-weighted sum of the pointwise inner product given by the Riemannian metric $g$.

\subsubsection{Upper bound}

We can find an upper frame bound (Bessel condition) by finding a uniform bound
\[
\sum_{j=1}^d g(d\textbf{x}_j, \boldsymbol\alpha)(p)^2 \leq B g(\boldsymbol\alpha,\boldsymbol\alpha)(p),
\]
that holds for all $p$.
This would then imply that 
\[
\sum_{i=1}^{n_1} \sum_{j=1}^d |\inp{\boldsymbol{\phi}_{i}d\textbf{x}_j}{\boldsymbol\alpha}|^2
\leq
\sum_{p=1}^n B g(\boldsymbol\alpha,\boldsymbol\alpha)(p) \boldsymbol\mu_p
= B \| \boldsymbol\alpha \|^2.
\]
We can find such a bound by the same argument as Proposition \ref{prop: upper frame bound theoretical}, here adapted to the discrete setting.

\begin{prop}
\label{prop: upper frame bound estimated}
We have
\[
\sum_{j=1}^d g(d\textbf{x}_j, \boldsymbol\alpha)(p)^2 
\leq \max_p \Big(\sum_{i=1}^d \boldsymbol\Gamma_p(\textbf{x}_i,\textbf{x}_i) \Big) g(\boldsymbol\alpha, \boldsymbol\alpha)(p)
\]
for all $p$, and so the spanning set $\{ \boldsymbol{\phi}_{i}d\textbf{x}_j \}$ satisfies the upper frame bound
\[
\sum_{i=1}^{n_1} \sum_{j=1}^d |\inp{\boldsymbol{\phi}_{i}d\textbf{x}_j}{\boldsymbol\alpha}|^2
\leq B \|\boldsymbol\alpha\|^2
\]
for $B = \max_p \big(\sum_{i=1}^d \boldsymbol\Gamma_p(\textbf{x}_i,\textbf{x}_j) \big)$.
\end{prop}
\begin{proof}
By the same argument as Lemma \ref{lemma: pointwise frame eigenvalues}, the pointwise sum
\[
\sum_{j=1}^d g(d\textbf{x}_j, \boldsymbol\alpha)(p)^2
\]
is bounded above by $B(p)$, the largest eigenvalue of the carré du champ matrix $(\boldsymbol\Gamma(\textbf{x}_i,\textbf{x}_j)(p))_{i,j = 1,...,d}$.
The eigenvalues of a positive semi-definite matrix are bounded above by the trace, so
\[
B(p) 
\leq \textnormal{tr}(\boldsymbol\Gamma_p(\textbf{x}_i,\textbf{x}_j))
= \sum_{i=1}^d \boldsymbol\Gamma_p(\textbf{x}_i,\textbf{x}_i)
\leq \max_p \Big(\sum_{i=1}^d \boldsymbol\Gamma_p(\textbf{x}_i,\textbf{x}_i)\Big)
\]
uniformly.
\end{proof}

We can use this estimate to assess the numerical stability of the spanning set as we vary $n$.
Recall from \eqref{eq: cdc covariance variable bandwidth} that we defined the carré du champ operator $\boldsymbol\Gamma$ to be
\[
\boldsymbol\Gamma_p(f,h)
= \frac{1}{2\rho_p^2} \, \textnormal{Cov}\!\left[f(X), h(X)\,:\, X \sim \textbf{P}_p\right].
\]
When $f = h = \textbf{x}_i$ is a coordinate function, the upper frame bound in Proposition \ref{prop: upper frame bound estimated} is equal to 
\begin{equation}
\label{eq: variance expression for coords cdc}
\max_p \frac{1}{2\rho_p^2} \E \big[ \|X - \E(X)\|^2 : X \sim P_p \big].
\end{equation}
This is just an {empirical variance estimate}, and so, if we fix the bandwidth function $\rho$, we can expect this to be stable as $n$ increases.
In particular, we expect a stable upper bound for the spectrum of the Gram matrix and the norm of the weak formulation vectors, so equations like \eqref{eq: weak linear system grad demo} will remain \boldblue{stable for all $n$}.

\begin{computationalnote}
It is hard to make this sort of claim precise, because the kernel bandwidth $\rho_p$ we use in practice is self-tuned, and so will depend on the density, and that depends on $n$.
To obtain a stronger guarantee, we can use the fact that the kernel $\textbf{P}_p$ is truncated to be supported on just its $k_{\mathrm{nn}}$ nearest neighbours. 
We can rewrite \eqref{eq: variance expression for coords cdc} as
\[
\max_p \frac{1}{2\rho_p^2}\,\mathrm{Var}\!\big[ X : X \sim \textbf{P}_p \big]
= \max_p \frac{1}{2\rho_p^2}\sum_{j=1}^n \textbf{P}_{pj} \Big\| \textbf{x}_j - \sum_{k=1}^n \textbf{P}_{pk}\,\textbf{x}_k \Big\|^2 .
\]
If $\mathcal{N}_p$ contains the indices of the $k_{\mathrm{nn}}$ nearest neighbours of $p$, let $r_p := \max_{j \in \mathcal{N}_p} \|\textbf{x}_j - \textbf{x}_p\|$.
Using the fact that the variance minimises the second moment about any centre,
\[
\sum_{j \in \mathcal{N}_p} \textbf{P}_{pj}\Big\|\textbf{x}_j - \sum_{k}\textbf{P}_{pk}\,\textbf{x}_k\Big\|^2
\leq \sum_{j \in \mathcal{N}_p} \textbf{P}_{pj}\|\textbf{x}_j - \textbf{x}_p\|^2
\leq r_p^2,
\]
so we obtain the strict upper bound
\[
\max_p \sum_{i=1}^d \boldsymbol\Gamma_p(\textbf{x}_i,\textbf{x}_i)
\leq \max_p \frac{r_p^2}{2\rho_p^2} .
\]
The adaptive bandwidth $\rho_p$ at least $\ord(r_p)$, so $r_p^2 / 2\rho_p^2$ remains of at most \boldblue{constant order}.
Again, this is not a precise claim, because the bandwidth selection method is arbitrary and the upper bound for the frame is given by the \textit{worst outlier} point $p$, but using the standard method in our implementation, we find a perfectly stable upper bound between 1 and 2 in every case.
\end{computationalnote}

\subsubsection{Lower bound}

As before, obtaining a lower frame bound is more difficult.
We need to find a uniform lower bound
\[\sum_{j=1}^d g(d\textbf{x}_j, \boldsymbol\alpha)(p)^2 \geq A g(\boldsymbol\alpha,\boldsymbol\alpha)(p)\]
that holds for all $p$.
Even if we had access to the exact carré du champ operator $\Gamma$ and could compute the exact Gram matrix, there would still be the problem of spectral pollution when working with a finite truncation of the spanning set.
We observe this in practice, where our carré du champ matrices will \boldblue{almost never have exact zero eigenvalues}, due to spectral pollution and the inherent uncertainty in finite data.
Even when the data are densely sampled from a compact manifold, the eigenvalues of the carré du champ matrices tend to exist on a continuous spectrum, and so there is no clear distinction between \q{zero} and \q{non-zero} eigenvalues.

The easy practical solution is to compute the upper bound $B$ (the largest eigenvalue of the Gram matrix) and then \boldblue{pick a desired lower bound $A$} that would ensure the condition number $B/A$ is low enough for reasonable stability (e.g. less than $10^5$).
The idea is to restrict to the \textit{effective support} of the Gram matrix, but in practice, it is just an arbitrary choice.
We can then project the spanning set onto the span of the Gram matrix eigenvectors whose eigenvalues are at least $A$.
This approach of \boldblue{spectral cutoff} is discussed further below.
While it is theoretically unsatisfying, it is likely the best practical solution, because we can never know what the \q{true} frame bounds are \q{meant} to be when presented with finite data from an unknown source.

\begin{computationalnote}
Fortunately, the choice of lower bound $A$ seems to have little effect, even taking the condition number as low as 5 or as high as $10^{12}$.
This is perhaps because, by picking a higher $A$ (i.e. lower condition number), we exclude more \q{bad} vectors from the space, but, if we leave them in, they have little effect anyway because their norm (square root of their eigenvalue) is low.
\end{computationalnote}

\subsection{Weak formulations of operators}

We can apply the method of weak formulations to compute discrete representations of linear operators as matrices (or other tensors, where appropriate).
In general terms, let $A: \R^p \to \R^q$ be a linear operator represented by an unknown $q \times p$ matrix $\textbf{A}$.
Let $e_i$ and $f_i$ be the basis elements for $\R^p$ and $\R^q$ respectively, and let $\textbf{G}^{(q)}_{ij} = \inp{f_i}{f_j}$ be the Gram matrix for an inner product on $\R^q$.
If we can compute
$$
\textbf{A}^{\textnormal{weak}}_{ij} = \inp{f_i}{A(e_j)} = \inp{f_i}{\textbf{A}e_j} = f_i^T \textbf{G}^{(q)} \textbf{A}e_j  = (\textbf{G}^{(q)} \textbf{A})_{ij},
$$
then we can recover $\textbf{A} = (\textbf{G}^{(q)}\big)^+ \textbf{A}^{\textnormal{weak}}$ using the Moore--Penrose pseudoinverse $(\textbf{G}^{(q)})^+$.
This is simply solving the weak formulation for each basis vector $e_j$ of $\R^p$.
In the language of the finite-element method, $\textbf{G}^{(q)}$ is the \boldblue{stress matrix} and the columns of $\textbf{A}^{\textnormal{weak}}$ are \boldblue{load vectors}.

We can recycle the same weak formulation to compute the \boldblue{adjoint} of $A$.
Suppose that $\textbf{G}^{(p)}_{ij} = \inp{e_i}{e_j}$ is the Gram matrix for an inner product on $\R^p$, so $A$ has an adjoint $A^* : \R^q \to \R^p$.
Then
$$
\textbf{A}^{\textnormal{weak}}_{ij} 
= \inp{f_i}{A(e_j)}
= \inp{A^*(f_i)}{e_j}
= (\textbf{A}^*)^{\textnormal{weak}}_{ji},
$$
meaning $(\textbf{A}^*)^{\textnormal{weak}} = (\textbf{A}^{\textnormal{weak}})^T$, so $\textbf{A}^* = (\textbf{G}^{(p)})^+ (\textbf{A}^{\textnormal{weak}})^T$.

\subsection{Convergence to continuous differential operators}

In practice, we do not work with the continuous spanning set $\{ \p{i}dx_j \}_{i \in \N, j\leq d}$, but rather the discrete spanning set $\{ \boldsymbol{\phi}_{i}d\textbf{x}_j \}_{i\leq n_1, j\leq d}$ computed from finite data.
This is a \boldblue{finite-dimensional truncation} of the continuous spanning set, which is \boldblue{approximated} from finite data.
In this section, we analyse the convergence of the computed weak formulations to the true continuous operators as both the truncation size $n_1$ increases and the approximation error decreases.
The approximation error depends on the number of data points $n$ and the kernel bandwidth, and so we can combine all these factors to get a convergence result in the limit of large $n$, small bandwidth, and large $n_1$.

\subsubsection{Convergence of weak formulations in truncated frames}

The frame condition guarantees a stable solution to weak formulations in the Hilbert space $H$ given the entire sequence.
An important practical question is whether, if we restrict to the first $k$ elements $\mathcal{U}_k = \{ u_1, \dots, u_k \}$ and solve the corresponding finite-dimensional weak formulation \eqref{eq: weak formulation}, we obtain a good approximation to the true solution as $k \to \infty$.
We show here that this is indeed the case in the \boldblue{noise-free} setting, where the Gram matrix and load vector are computed exactly.

Let $H_k = \textnormal{span}(\mathcal{U}_k) \subset H$ be the finite-dimensional subspace spanned by $\mathcal{U}_k$.
We saw in Example \ref{eg: finite frame} that $\mathcal{U}_k$ is automatically a frame for the subspace $H_k$, with frame bounds $A_k$ and $B_k$ given by the smallest and largest non-zero eigenvalues of the $k \times k$ Gram matrix $\mathbf{G}_k$.
We can define an approximate coefficient vector $a^{(k)} \in \mathbb{K}^k$ by solving the finite system
\begin{equation}
\label{eq: finite weak formulation}
a^{(k)} = \mathbf{G}_k^+ (Tx)_k,
\end{equation}
where $(Tx)_k = (\langle x, u_1 \rangle, \dots, \langle x, u_k \rangle)^\top$ is the truncation of the load vector, which corresponds to
\begin{equation}
x_k = \sum_{j=1}^k a^{(k)}_j u_j \in H_k.
\end{equation}
The following observation guarantees that $x_k$ converges to $x$ as $k \to \infty$.

\begin{lemma}
The approximation $x_k$ is the orthogonal projection of $x$ onto $H_k$.
\end{lemma}
\begin{proof}
For each $i = 1, \dots, k$, we have
\[
\langle x_k, u_i \rangle_H
= \langle T^* \mathbf{G}_k^+ (Tx)_k, u_i \rangle_H
= \langle \mathbf{G}_k^+ (Tx)_k, Tu_i \rangle_{\ell_2},
\]
which is the $i^{th}$ entry of
\[
\big( \mathbf{G}_k \mathbf{G}_k^+ (Tx)_k \big)_i 
= \big( \mathbf{G}_k \mathbf{G}_k^+ (Tx)_k \big)_i 
= \big( (Tx)_k \big)_i 
= \big( Tx \big)_i 
= \langle x, u_i \rangle
\]
so $\langle x - x_k, u_i \rangle = 0$.
\end{proof}

It matters here that $\{ u_i \}_{i \in \N}$ is a frame for $H$, because then $\text{ran}(T^*) = H$ (by Proposition \ref{cor: frame equivalence}), so the union of the subspaces $H_k$ is dense in $H$, and so $\|x - x_k\| \to 0$ as $k \to \infty$.
This means that solving the finite matrix system associated with the first $k$ frame elements is numerically stable (by the frame property of $\mathcal{U}_k$) and converges to the true solution (by the frame property of the infinite sequence).

\subsubsection{The problem of stable convergence in noisy truncated frames}

While we know that each $\mathcal{U}_k$ is a frame for $H_k$, the weak formulation may become asymptotically unstable.
The upper bound $B_k$ is well-behaved and is a non-decreasing sequence converging to $B$, but the lower bound $A_k$ is more complicated.
The frame could be overcomplete and contain linear dependencies, but those dependencies may only appear at $k = \infty$ (or, for practical purposes, at some very large $k$).
This leads to elements with very small norm for finite $k$, but which disappear in the limit $k \to \infty$, and so $A_k \to 0$.
This means that $\kappa(\mathbf{G}_k) = B_k/A_k \to \infty$ as $k \to \infty$, so the \boldblue{finite-dimensional solutions are asymptotically unstable}.

In our setting, we do not have access to the exact entries of the Gram matrix $\mathbf{G}_k$ or the load vector $(Tx)_k$, and instead compute empirical estimates $\tilde{\mathbf{G}}_k$ and $(\tilde{T}x)_k$.
We therefore solve the perturbed system
\begin{equation}
\label{eq: perturbed matrix system}
\tilde{\mathbf{G}}_k \tilde{a}^{(k)} = (\tilde{T}x)_k,
\end{equation}
which leads to the error
\begin{equation}
\label{eq: total error decomposition}
\| a - \tilde{a}^{(k)} \|
\leq \underbrace{\| a - a^{(k)} \|}_{\text{truncation error}} + \underbrace{\| a^{(k)} - \tilde{a}^{(k)} \|}_{\text{approximation error}}.
\end{equation}
The above analysis shows that the truncation error vanishes as $k \to \infty$, but the approximation error depends on the condition number of the Gram matrix $\mathbf{G}_k$.
Standard perturbation theory for linear systems tells us that the relative error is bounded by the condition number $\kappa(\mathbf{G}_k)$
\begin{equation}
\label{eq: perturbation bound}
\frac{\|\tilde{a}^{(k)} - a^{(k)}\|}{\|a^{(k)}\|} \leq \frac{\kappa(\mathbf{G}_k)}{1 - \kappa(\mathbf{G}_k)\frac{\|\delta \mathbf{G}_k\|}{\|\mathbf{G}_k\|}} \left( \frac{\|\delta \mathbf{G}_k\|}{\|\mathbf{G}_k\|} + \frac{\|\delta (Tx)_k\|}{\|(Tx)_k\|} \right),
\end{equation}
where $\delta \mathbf{G}_k = \tilde{\mathbf{G}}_k - \mathbf{G}_k$ and $\delta (Tx)_k = (\tilde{T}x)_k - (Tx)_k$ represent the approximation errors due to finite sampling and carré du champ estimation\footnote{assuming $\kappa(\mathbf{G}_k) \|\delta \mathbf{G}_k\|/\|\mathbf{G}_k\| < 1$}.
If the true Gram matrix $\mathbf{G}_k$ has a uniformly bounded condition number $\kappa(\mathbf{G}_k)$, then, provided that the perturbations $\|\delta \mathbf{G}_k\|$ and $\|\delta (Tx)_k\|$ are sufficiently small, we can guarantee that the data-driven solution $\tilde{a}^{(k)}$ converges to the true solution $a^{(k)}$ as the sample size $N \to \infty$.
However, if $\kappa(\mathbf{G}_k) \to \infty$ as $k \to \infty$, then the error bound \eqref{eq: perturbation bound} only applies when $\|\delta \mathbf{G}_k\|$ and $\|\delta (Tx)_k\|$ are zero.

This is a standard problem in numerical analysis, which can be resolved using \boldblue{regularisation}.
We replace the operator $G$ with a modified operator $G_\epsilon$, where $\epsilon > 0$ is a small regularisation parameter, chosen so that $G_\epsilon$ has a uniformly bounded condition number, and $\|G_\epsilon - G\| \to 0$ as $\epsilon \to 0$.
This leads to the modified error decomposition
\begin{equation}
\| a - \tilde{a}_\epsilon^{(k)} \|
\leq 
\underbrace{\| a - a_\epsilon \|}_{\text{regularisation bias}} +
\underbrace{\| a_\epsilon - a_\epsilon^{(k)} \|}_{\text{truncation error}} + 
\underbrace{\| a_\epsilon^{(k)} - \tilde{a}_\epsilon^{(k)} \|}_{\text{approximation error}},
\end{equation}
where the perturbation bound \eqref{eq: perturbation bound} now applies to the approximation error term and ensures that it vanishes as $N \to \infty$ for fixed $n$ and $\epsilon$.
The operator-norm convergence $\|G_\epsilon - G\| \to 0$ also ensures that the regularisation bias vanishes as $\epsilon \to 0$.
The challenge is now to show that the truncation errors for the modified operators also vanish as $k \to \infty$, which would then ensure the convergence
\[
\lim_{\epsilon \to 0} \lim_{k \to \infty} \lim_{N \to \infty} \| a - \tilde{a}_\epsilon^{(k)} \| = 0.
\]
We will consider two regularisation methods here and discuss convergence results for both.
They both involve picking some small $\epsilon > 0$ and modifying the finite-dimensional weak formulation to enforce a uniform lower bound $A_k \geq \epsilon$ for all $k$.
\begin{enumerate}
  \item \boldblue{Tikhonov regularisation} replaces $\mathbf{G}_k$ with $\mathbf{G}_{k, \epsilon} = \mathbf{G}_k + \epsilon \mathbf{I}_k$.
  \item \boldblue{Spectral cutoff} diagonalises $\mathbf{G}_k$ and then defines $\mathbf{G}_{k, \epsilon}$ by replacing any eigenvalues smaller than $\epsilon$ with zero.
\end{enumerate}
Tikhonov regularisation is simpler to analyse and implement computationally, and has condition number $\kappa(\mathbf{G}_{k, \epsilon}) \leq (B + \epsilon)/\epsilon$, but leads to a biased solution with $\|a - a_\epsilon\| \neq 0$ for any fixed $\epsilon > 0$.
Spectral cutoff has several interesting advantages: it has an improved condition number $\kappa(\mathbf{G}_{k, \epsilon}) \leq B/\epsilon$ and produces an unbiased solution with $a_\epsilon = a$ for all $\epsilon < A$.
It can also reduce the computational cost by projecting onto the smaller subspace spanned by the large eigenvectors of $\mathbf{G}_k$.
However, showing that the truncation error $\| a_\epsilon - a_\epsilon^{(k)} \|$ vanishes as $k \to \infty$ is non-trivial and does not hold for general frames.

We will first show, by a standard argument, that both of these regularisation methods have uniformly bounded pseudoinverses $\|\mathbf{G}_{k, \epsilon}^+\| \leq 1/\epsilon$, ensuring numerical stability and vanishing approximation error, and also have vanishing regularisation biases as $\epsilon \to 0$.
We then recall classical results that Tikhonov regularisation has vanishing truncation error, and that, under a strong assumption, so does spectral cutoff.

\subsubsection{Bias convergence and stability of spectral regularisations}

Tikhonov regularisation and spectral cutoff are both examples of \boldblue{spectral regularisation}, because the operators $G_\epsilon$ and $\mathbf{G}_{n, \epsilon}$ are obtained by applying a \boldblue{spectral map} $g_\epsilon : \R_+ \to \R_+$ to define $G_\epsilon = g_\epsilon(G)$ and $\mathbf{G}_{n, \epsilon} = g_\epsilon(\mathbf{G}_{n})$.

\begin{enumerate}
  \item Tikhonov regularisation uses $g_{\epsilon}(\lambda) = \lambda + \epsilon$, and
  \item spectral cutoff uses $g_{\epsilon}(\lambda) = \lambda \mathbb{I}_{\lambda > \epsilon}$.
\end{enumerate}

The following Lemma shows that, for a general class of spectral regularisations, the Moore-Penrose pseudoinverse $G_\epsilon ^+$ is bounded, and the regularisation bias $\| a - a_\epsilon \|$ vanishes as $\epsilon \to 0$.

\begin{lemma}
Let $g_\epsilon : \R_+ \to \R_+$ be a piecewise continuous function. 
Define the pseudo-inverse function $g_\epsilon^+(\lambda) = 1/g_\epsilon(\lambda)$ where $g_\epsilon(\lambda) \neq 0$ and $0$ otherwise. 
Assume:
\begin{enumerate}
  \item $g_\epsilon(\lambda) : \R_+ \to \{0\} \cup [\epsilon, \infty)$,
  \item $\lim_{\epsilon \to 0} \lambda g_\epsilon^+(\lambda) = 1$ for all $\lambda > 0$, and
  \item there exists a constant $C$ such that $\sup_{\epsilon, \lambda \ge 0} \lambda g_\epsilon^+(\lambda) \le C$.
\end{enumerate}
Let $G = TT^*$ be the Gram operator of a frame for a Hilbert space $H$, which is bounded and self-adjoint, so we can apply the spectral mapping theorem to define $G_\epsilon = g_\epsilon(G)$. 
Then $G_\epsilon$ satisfies
\begin{enumerate}
  \item $\| G_\epsilon ^+ \| \leq 1/\epsilon$, and
  \item if $x \in H$, then $\| G_\epsilon ^+ Tx - G^+ Tx \| \to 0$ as $\epsilon \to 0$.
\end{enumerate}
\end{lemma}
\begin{proof}
The fact that $g_\epsilon(\lambda) : \R_+ \to \{0\} \cup [\epsilon, \infty)$ means that $\sigma(G_\epsilon) \subseteq \{0\} \cup [\epsilon, \infty)$, so $\| G_\epsilon ^+ \| \leq 1/\epsilon$.
To prove the second claim, we apply the spectral theorem to $G$.
Let $E$ be the spectral measure of $G$. 
Let $a = G^+ Tx$ be the true solution, which lies in $\ker(G)^\perp = \textnormal{ran}(G^*) = \textnormal{ran}(G)$.
We can express the regularised solution as
\[
a_\epsilon = G_\epsilon ^+ Tx= G_\epsilon ^+ G a.
\]
Notice that the Moore-Penrose pseudoinverse is itself defined by a spectral map $\lambda \mapsto \frac{1}{\lambda}\mathbb{I}_{\lambda > 0}$.
Using the fact that $g_\epsilon(\lambda) : \R_+ \to \{0\} \cup [\epsilon, \infty)$, we can express $G_\epsilon ^+$ as the application of the spectral map
\[
\lambda \mapsto 
\frac{1}{g_\epsilon(\lambda)}\mathbb{I}_{g_\epsilon(\lambda) > 0}
\]
to $G$.
We can then use spectral calculus to write
\[
\| a - a_\epsilon \|^2 
 = \| (I - G_\epsilon ^+ G) a \|^2
 = \int_{\sigma(G)} \left( 1 - \frac{\lambda}{g_\epsilon(\lambda)}\mathbb{I}_{g_\epsilon(\lambda) > 0} \right)^2 d\mu_a(\lambda).
\]
Since $\sigma(G) \subseteq \{0\} \cup [A, B]$ and $a \in \ker(G)^\perp$ we can restrict the integral to $\sigma(G) \setminus \{0\} \subset [A, B]$ to get
\[
\| a - a_\epsilon \|^2 
 = \int_A ^B \left( 1 - \frac{\lambda}{g_\epsilon(\lambda)}\mathbb{I}_{g_\epsilon(\lambda) > 0} \right)^2 d\mu_a(\lambda).
\]
The integrand is uniformly bounded for small enough $\epsilon$ (by the third assumption) and converges pointwise to 0 as $\epsilon \to 0$ (by the second assumption), so we can apply the dominated convergence theorem to get
\begin{align*}
\lim_{\epsilon \to 0} \| a - a_\epsilon \|^2 
& = \lim_{\epsilon \to 0} \int_A ^B \left( 1 - \frac{\lambda}{g_\epsilon(\lambda)}\mathbb{I}_{g_\epsilon(\lambda) > 0} \right)^2 d\mu_a(\lambda) \\
& = \int_A ^B \lim_{\epsilon \to 0} \left( 1 - \frac{\lambda}{g_\epsilon(\lambda)}\mathbb{I}_{g_\epsilon(\lambda) > 0} \right)^2 d\mu_a(\lambda) \\
& = 0
\end{align*}
which proves the result.
\end{proof}

\subsubsection{Truncation convergence of Tikhonov Regularisation}

We now show, by a standard argument, that solutions of the truncated Tikhonov-regularised formulations converge.
Tikhonov regularisation is special because the finite-dimensional matrix $\mathbf{G}_{k, \epsilon} = \mathbf{G}_k + \epsilon \mathbf{I}_k$ is just the restriction of $G_\epsilon = G + \epsilon I$ to the first $k$ components.
In other words, Tikhonov regularisation commutes with projection onto the frame, which we can exploit to prove convergence of the truncations in arbitrary frames.

Let $X : \ell_2 \times \ell_2 \to \R$ be a bilinear form.
We say that $X$ is \boldblue{bounded} if $|X(u, v)| \leq M \|u\| \|v\|$, and \boldblue{coercive} if $X(u, u) \geq c \|u\|^2$.
Given a bounded linear functional $f$, we can define an abstract weak formulation as
\begin{equation}
\label{eq: abstract weak formulation}
X(a, v) = f(v)
\end{equation}
for all $v \in \ell_2$.
If $X$ is bounded and coercive, then the \boldblue{Lax-Milgram theorem} guarantees the existence of a unique solution $a \in \ell_2$.
If \eqref{eq: abstract weak formulation} is truncated to a finite-dimensional weak formulation by evaluating $v$ in a frame, we can recover a solution $a^{(k)}$ in the span of the truncated frame.
\boldblue{Céa's Lemma} bounds the truncation error as
\[
\| a - a^{(k)} \| 
\leq \frac{M}{c} \inf_{b \in E_k} \| a - b \|,
\]
where $E_k$ is the $k$-dimensional subspace generated by the first $k$ basis elements in $\ell_2$.

We can apply this to Tikhonov regularisation by setting $X(a, v) = \langle G_\epsilon a, v \rangle$ and $f(v) = \langle Tx, v \rangle$, which gives $M = B + \epsilon$ and $c = \epsilon$.
As we noted above, the finite-dimensional matrix $\mathbf{G}_{k, \epsilon} = \mathbf{G}_k + \epsilon \mathbf{I}_k$ is exactly the matrix representation of $X$ restricted to the first $k$ components, and so the solution to the matrix problem for a given $k$ is governed by Céa's Lemma.
The union of the subspaces $E_k$ is dense in $\ell_2$, so the best approximation error $\inf_{b \in E_k} \| a_\epsilon - b \|$ converges to zero as $k \to \infty$, and hence
\[
\lim_{k \to \infty} \| a_\epsilon - a^{(k)}_\epsilon \| = 0.
\]

\subsubsection{Truncation convergence of spectral cutoff}

General spectral regularisations have no guarantee of truncation convergence, because we cannot know in general whether the spectra of the finite sections $\mathbf{G}_k$ converge to the spectrum of $G$.
Truncation of general frames may exhibit \boldblue{spectral pollution}, where eigenvalues of the finite sections $\mathbf{G}_k$ converge to points outside the spectrum of $G$.

\begin{definition}
\label{def: no spectral pollution}
Let $G$ be the Gram operator of an $(A,B)$-frame for a Hilbert space $H$, and let $\mathbf{G}_k$ be its finite sections.
We say that the frame has \boldblue{no spectral pollution in the gap} if the set of eigenvalues of $\mathbf{G}_k$ has no limit points in $(0,A)$.
\end{definition}

Proving that a given frame has no spectral pollution in the gap is highly non-trivial, and we will not attempt to do so here.
This assumption means that spectral cutoff regularisation has vanishing truncation error, provided that the cutoff parameter $\epsilon$ is chosen in the gap $(0, A)$.
In this case, the spectral cutoff on the true operator is exact, so there is no regularisation bias, and we can show that the regularised and truncated solutions converge directly to the true solution.

\begin{theorem}
Let $G$ be the Gram operator of a frame with no spectral pollution in the gap $(0,A)$, and let $\mathbf{G}_k$ be its finite sections (extended by zero to $\ell_2$).
Fix $\epsilon \in (0,A)$, and let $a = G^+ Tx$ be the true solution for some $x \in H$, and let
\[
a_\epsilon^{(k)} = \mathbf{G}_{k,\epsilon}^+ (Tx)_k
\]
denote the finite-dimensional spectral-cutoff solution, where $\mathbf{G}_{k,\epsilon}^+$ is obtained by applying the spectral map $h_\epsilon(\lambda) = \lambda^{-1}\mathbb{I}_{\{\lambda > \epsilon\}}$ to $\mathbf{G}_k$.
Then
\[
\lim_{k\to\infty} \|a_\epsilon^{(k)} - a\| = 0.
\]
\end{theorem}

\begin{proof}
Since $\sigma(G) \subset \{0\} \cup [A,B]$ and $\epsilon < A$, we have $h_\epsilon(G) = G^+$, so $a = G^+ Tx = h_\epsilon(G) Tx$.
By the assumption of no spectral pollution, the eigenvalues of $\mathbf{G}_k$ have no limit points in $(0,A)$.
If we pick $\delta$ with $\epsilon < \delta < A$, then there exists an $N \in \mathbb{N}$ such that
\[
\sigma(\mathbf{G}_k) \cap [\epsilon,\delta] = \emptyset
\]
for all $k \ge N$.
We also have $\sigma(G) \cap [\epsilon,\delta] = \emptyset$ because $\sigma(G) \subset \{0\}\cup[A,B]$.
We now let $f \in C([0,B])$ be the continuous function such that
\[
f(\lambda) = 0 \quad\text{for } 0 \le \lambda \le \epsilon,
\qquad
f(\lambda) = \frac{1}{\lambda} \quad\text{for } \lambda \ge \delta,
\]
and which is linear on $[\epsilon, \delta]$ connecting $(\epsilon, 0)$ to $(\delta, 1/\delta)$.
Then for all $\lambda \in \sigma(G) \cup \sigma(\mathbf{G}_k)$ and all $k \ge N$ we have $f(\lambda) = h_\epsilon(\lambda)$, so by spectral calculus
\[
f(G) = h_\epsilon(G) = G^+
\qquad
f(\mathbf{G}_k) = h_\epsilon(\mathbf{G}_k) = \mathbf{G}_{k,\epsilon}^+
\]
for all $k \ge N$.
Let $y = Tx \in \ell_2$, and let $P_k$ be the orthogonal projection onto the first $k$ coordinates, so $(Tx)_k = P_k y$.
For $k \ge N$ we have
\[
a_\epsilon^{(k)} = \mathbf{G}_{k,\epsilon}^+ (Tx)_k = f(\mathbf{G}_k) P_k y
\qquad
a = G^+ Tx = f(G) y
\]
so
\[
\|a_\epsilon^{(k)} - a\|
= \|f(\mathbf{G}_k) P_k y - f(G)y\|
\le \|f(\mathbf{G}_k)(P_k - I)y\| + \|(f(\mathbf{G}_k) - f(G))y\|.
\]
The operators $f(\mathbf{G}_k)$ are uniformly bounded by $1/\delta$ because $f \le 1/\delta$ on $[0, B]$, so
\[
\|a_\epsilon^{(k)} - a\|
\le \frac{1}{\delta} \|(P_k - I)y\| + \|(f(\mathbf{G}_k) - f(G))y\|.
\]
The first term tends to zero because $P_k \to I$ strongly.
Since $\mathbf{G}_k \to G$ strongly and $f$ is continuous, the continuity of functional calculus for bounded self-adjoint operators means $f(\mathbf{G}_k) \to f(G)$ strongly, so the second term also tends to zero.
Therefore $\|a_\epsilon^{(k)} - a\| \to 0$ as $k \to \infty$.
\end{proof}

\subsubsection{Towards overall convergence results}

In this section, we have shown that the solutions to suitably regularised solutions to weak formulations will converge to the true solutions as long as the entries of the relevant matrices and vectors converge.
This reduces the problem to the question of convergence for the carré du champ and its associated integral expressions.
For example, the entries in our Gram matrices and load vectors are of the form
\[
\langle \phi_i dx_j, \alpha \rangle
= \int \phi_i g(dx_j, \alpha) d\mu
\]
for some 1-form $\alpha$, and it remains to show that these converge.

Although the computational framework here applies to any carré du champ estimator, the approach we laid out in Section \ref{sec: functions_vector_fields_forms_tensors} uses heat kernel estimates, for which we have several convergence guarantees in the special case of manifolds.
Our estimates involve several steps:
\begin{enumerate}
  \item \boldblue{Eigenfunction estimation}: we estimate the eigenfunctions $\phi_i$ by diagonalising the heat kernel Markov chain. This is essentially \boldblue{diffusion maps} eigenfunction estimation \cite{COIFMAN20065}, and requires \boldblue{spectral convergence} of the Markov chain \cite{trillos2018error, calder2022improved, dunson2021spectral, belkin2006convergence, trillos2018variational, shi2015convergence}.
  \item \boldblue{Carré du champ estimation}: we estimate the inner products $g(x_j, \alpha)$ using the carré du champ operator, which only requires \boldblue{pointwise convergence} of the Markov chain.
  See Theorem 3.1 in \cite{jones2024manifold} for an example result: such results are easily inherited from pointwise convergence of the Markov chain \cite{berry2016variable, belkin2008towards, gine2006empirical}.
  An essential point is that we can only reliably estimate the carré du champ of functions that are sufficiently smooth.
  Like most methods for discrete geometry, we use local inter-point variability as a proxy for the derivative, and if the function varies too quickly, this is not a good estimate.
  \item \boldblue{Monte Carlo integration}: we estimate the integrals using a weighted sum over the data points. 
  This is easily justified by \boldblue{uniform convergence} of the carré du champ estimates, but can be justified under weaker conditions depending on the carré du champ convergence rates and the empirical measure estimation.
\end{enumerate}

However, there is a significant gap between theory and practice in the use of these methods.
First, the whole motivation for diffusion geometry is its use on \boldblue{general data sets}, which may not lie on a manifold, but nothing is known about the above convergence in this setting\footnote{an interesting development is \cite{smale2009geometry}, which proves spectral convergence of eigenfunctions for a fixed-bandwidth diffusion kernel on a general probability space, but not in the small-bandwidth limit}.
Second, in practice, diffusion methods are usually constructed with variable bandwidth kernels, and even for manifold data, we only have formal guarantees of pointwise convergence of the Markov chain \cite{berry2016variable}, and do not currently know whether the eigenfunctions converge correctly.
Third, the spectral cutoff regularisation method we proposed is common in numerical analysis (for example, it is used in Spectral Exterior Calculus \cite{berry2020spectral}), but is only justified under the strong assumption of no spectral pollution in the gap.

\begin{table*}[h!]
\centering
\begin{tabular}{ccc}
& theoretically justified & used in practice \\
\hline
sample space & manifold & general geometry \\
kernel bandwidths & fixed & variable \\
regularisation & Tikhonov & spectral cutoff \\
\end{tabular}
\end{table*}

The lag of theory behind practice presents an open problem, where empirical performance (such as on the non-manifold examples in this paper) demonstrates that much stronger theoretical results could be obtained than those currently known.
For now, we simply summarise that, at least under overly idealised conditions, the methods presented here can be shown to converge, and we leave the study of non-manifold carré du champ convergence for future work.

\begin{computationalnote}
We defined spanning sets like $\{\phi_i dx_j\}_{i \le n_1, j \le d}$.
It is essential that the immersion coordinates $x_j$ are \textit{smooth enough} because of the issue of {carré du champ estimation}.
The largest possible set we could work with is $\{e_i de_j\}_{i, j \le n}$ for pointwise basis functions $e_i$ (or, by a basis change, $\{\phi_i d\phi_j\}_{i, j \le n}$), but this would include computing the derivatives of functions that are too irregular.
We could interpret the smaller spanning set $\{\phi_i dx_j\}_{i \le n_1, j \le d}$ merely as a computational cost reduction, or as an opportunity to inject inductive bias through the choice of $x_j$, but this analysis reveals that \boldblue{picking smooth $x_j$ is essential to maintain sufficient accuracy}.

In a similar vein, if we set $n_1 = n$, then computations like $d(\phi_i dx_j) = d\phi_i \wedge dx_j$ would involve the derivatives of overly irregular $\phi_i$.
This means that, for any operators that involve derivatives, the choice of $n_1 \ll n$ is also essential in constructing a kind of \q{empirical Sobolev space} of forms on which the derivative is reasonably accurate.
\end{computationalnote}

\section{Differential operators}
\label{sec: differential_operators}

Differential operators, like the exterior derivative and Lie bracket, generalise the standard vector calculus methods from $\R^3$ to more general spaces.
They are fundamental in the theory of Riemannian geometry on manifolds, and diffusion geometry generalises them to more general spaces.
We can apply the method of weak formulations to compute discrete representations of differential operators from diffusion geometry, in order to apply them in data analysis.

\subsection{Gradient and divergence}

The \boldblue{gradient} is a map $\nabla : \A \to \mathfrak{X}(M)$.
If $f\in\A$ is a function, its gradient $\nabla f$ is a vector field that measures the direction and rate of steepest increase of $f$. 
We discretise it as an $n_1d \times n_0$ matrix $\boldsymbol{\nabla}$ that maps functions in $\R^{n_0}$ to vector fields in $\R^{n_1d}$.
We obtain this via a weak formulation $\boldsymbol{\nabla}^{\textnormal{weak}}$ that computes the inner product of $\nabla(\p{i})$ and $\p{i'}\nabla x_{j'}$, in the 3-tensor
\begin{equation}
\label{eq: weak gradient}
\boldsymbol{\nabla}^{\textnormal{weak}}_{i' j' i}
= \inp{\p{i'}\nabla x_{j'}}{\nabla \p{i}}
= \int \p{i'} \Gamma(x_{j'}, \p{i}) d\mu
= \sum_{p=1}^n \textbf{U}_{pi'}
\boldsymbol\Gamma_p(\textbf{x}_{j'}, \boldsymbol{\phi}_{i})
\boldsymbol{\mu}_p,
\end{equation}
which has dimension $n_1 \times d \times n_0$.
We reshape $\boldsymbol{\nabla}^{\textnormal{weak}}$ into an $n_1d \times n_0$ matrix $\boldsymbol{\nabla}^{\textnormal{weak}}_{I' i}$, the weak gradient operator, and recover the strong operator $\boldsymbol{\nabla} : \R^{n_0} \to \R^{n_1d}$ as $\boldsymbol{\nabla} = (\textbf{G}^{(1)})^+ \boldsymbol{\nabla}^{\textnormal{weak}}$.
We illustrate the action of the gradient operator in Figure \ref{fig:gradient}.

\begin{figure}[h!]
  \vspace{1em}
  \centering
  \begin{overpic}[width=\linewidth,grid=false]{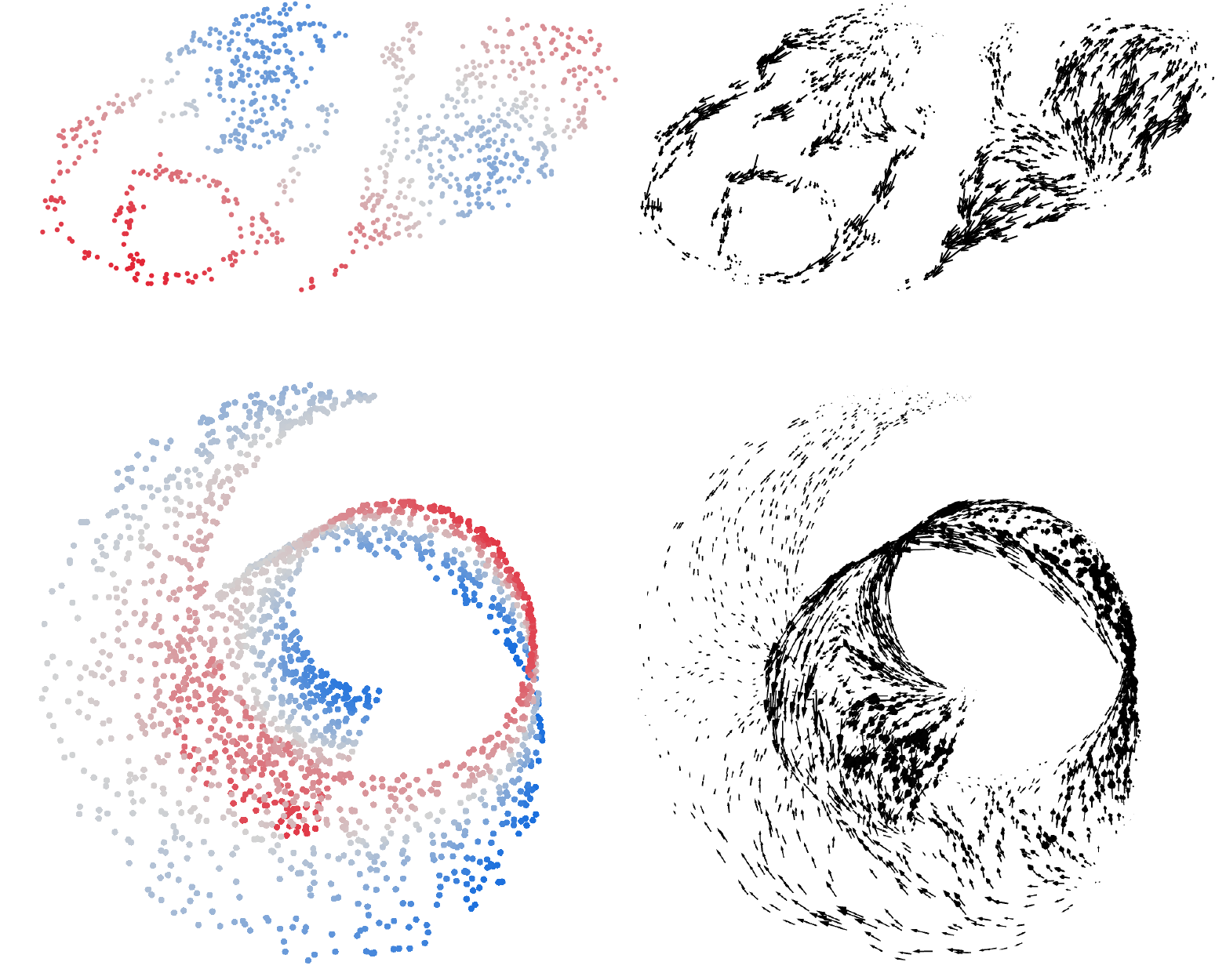}
\put(25.0,52.8){\makebox(0,0)[c]{$f$}}
\put(75.0,52.8){\makebox(0,0)[c]{$\nabla f$}}
\put(25.0,-1.2){\makebox(0,0)[c]{$f$}}
\put(75.0,-1.2){\makebox(0,0)[c]{$\nabla f$}}
  \end{overpic}
  \vspace{1em}
  \caption{\textbf{Gradient of functions.}
  A function $f$ is mapped to its gradient vector field $\nabla f$.
  The data in the top row is in 2d, and the bottom row is in 3d.
  }
  \label{fig:gradient}
\end{figure}

The \boldblue{divergence} is defined as the negative adjoint $- \nabla^* : \mathfrak{X}(M) \to \A$, and $-\nabla^*X$ measures the local expansion and contraction of space under the flow of $X$.
We can compute the divergence $-\boldsymbol{\nabla}^* : \R^{n_1d} \to \R^{n_0}$ as $-\boldsymbol{\nabla}^* = -(\textbf{G}^{(0)})^+ (\boldsymbol{\nabla}^{\textnormal{weak}})^T$.
When the function basis is orthonormal, $\textbf{G}^{(0)} = \textbf{I}_{n_0}$ so the divergence operator is represented by the matrix $-(\boldsymbol{\nabla}^{\textnormal{weak}})^T$.
We illustrate the divergence operator in Figure \ref{fig:divergence}.

\begin{figure}[h!]
  \vspace{1em}
  \centering
  \begin{overpic}[width=\linewidth,grid=false]{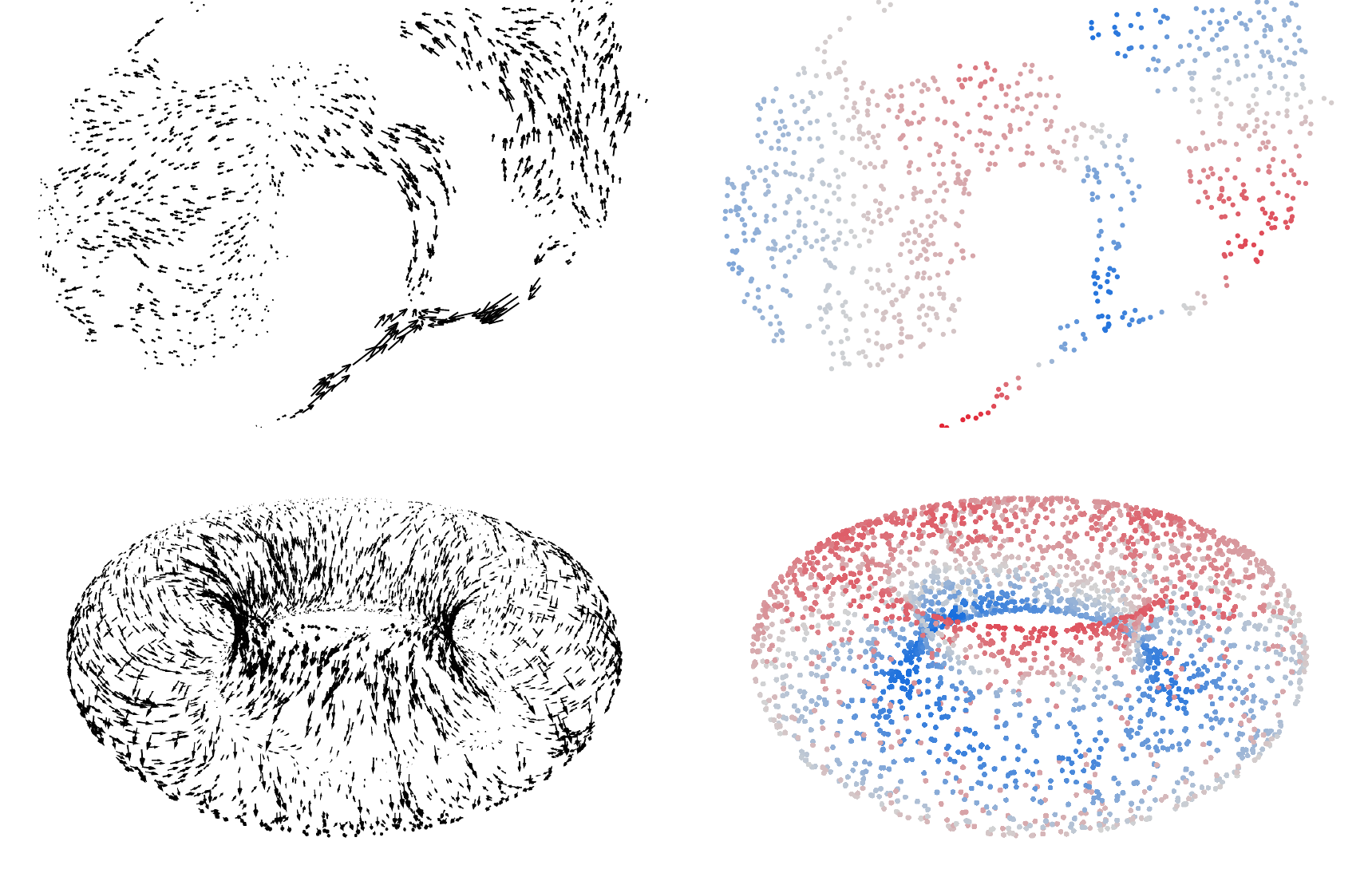}
\put(25.0,31.5){\makebox(0,0)[c]{$X$}}
\put(75.0,31.5){\makebox(0,0)[c]{$-\nabla^* X$}}
\put(25.0,-0.9){\makebox(0,0)[c]{$X$}}
\put(75.0,-0.9){\makebox(0,0)[c]{$-\nabla^* X$}}
  \end{overpic}
    \vspace{0.2em} 
  \caption{\textbf{Divergence of vector fields.}
  A vector field $X \in \mathfrak{X}(M)$ is mapped to its divergence $-\nabla^* X \in \A$, which is positive where $X$ expands space and negative where $X$ contracts space.
  The data in the top row is in 2d, and the bottom row is in 3d.
  }
  \label{fig:divergence}
\end{figure}
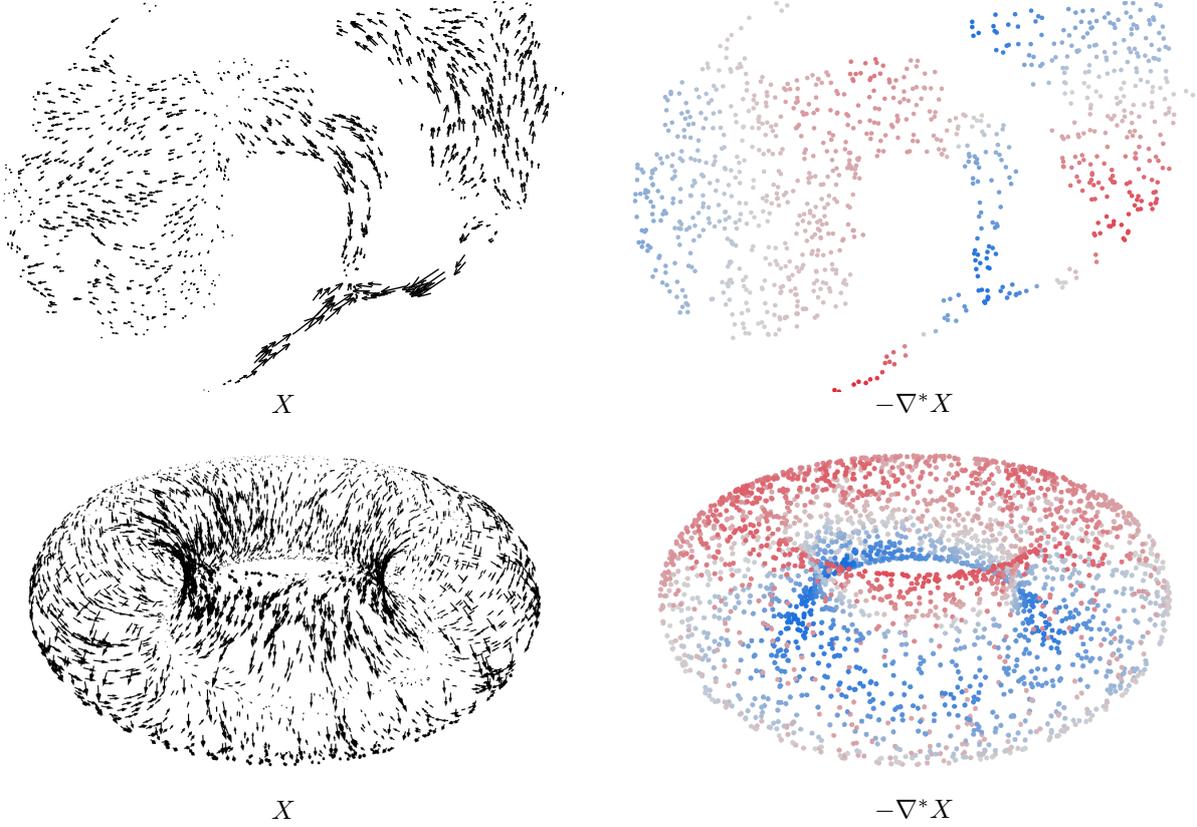

\subsection{Exterior derivative and codifferential}

The \boldblue{exterior derivative} is a map $d^{(k)} : \Omega^k(M) \to \Omega^{k+1}(M)$.
When $k=0$, $d^{(0)} : \A \to \Omega^{1}(M)$ is simply the dual to the gradient operator, i.e. $d^{(0)}f = (\nabla f)^\flat$, and $d^{(k)}$ provides a generalisation of the gradient to higher-degree forms for $k > 1$.
We discretise it as an $n_1\binom{d}{k+1} \times n_1 \binom{d}{k}$ matrix $\textbf{d}^{(k)}$ that maps $k$-forms in $\R^{n_1\binom{d}{k}}$ to $(k+1)$-forms in $\R^{n_1\binom{d}{k+1}}$.
We obtain this via a weak formulation $\textbf{d}^{(k),\textnormal{weak}}$ that computes the inner product of $d^{(k)}(\p{i}dx_{(J)})$ and $\p{i'}dx_{(J')}$, where $J=(j_1<\dots<j_k)$ and $J'=(j'_1<\dots<j'_{k+1})$ are multi-indices, in the 4-tensor
\begin{equation}
\label{eq: weak d_k}
\begin{split}
	\textbf{d}^{(k),\textnormal{weak}}_{i' J' i J}
&= \inp{\p{i'}dx_{(J')}}{d^{(k)}(\p{i}dx_{(J)})} \\
&= \inp{\p{i'}dx_{(J')}}{d\p{i} \wedge dx_{(J)}} \\
&= \int \p{i'}
\det
\begin{bmatrix}
\Gamma(x_{j_1'}, \p{i}) & \Gamma(x_{j_1'}, x_{j_1}) & \dots & \Gamma(x_{j_1'}, x_{j_k}) \\
\Gamma(x_{j_2'}, \p{i}) & \Gamma(x_{j_2'}, x_{j_1}) & \dots & \Gamma(x_{j_2'}, x_{j_k}) \\
\dots & \dots  & \dots & \dots  \\
\Gamma(x_{j_{k+1}'}, \p{i}) & \Gamma(x_{j_{k+1}'}, x_{j_1}) & \dots & \Gamma(x_{j_{k+1}'}, x_{j_k}) \\
\end{bmatrix} 
d\mu \\
&= \sum_{p=1}^n \textbf{U}_{pi'}
\det
\begin{bmatrix}
\boldsymbol\Gamma_p(\textbf{x}_{j_1'}, \boldsymbol{\phi}_{i}) & \boldsymbol\Gamma_p(\textbf{x}_{j_1'}, \textbf{x}_{j_1}) & \dots & \boldsymbol\Gamma_p(\textbf{x}_{j_1'}, \textbf{x}_{j_k}) \\
\boldsymbol\Gamma_p(\textbf{x}_{j_2'}, \boldsymbol{\phi}_{i}) & \boldsymbol\Gamma_p(\textbf{x}_{j_2'}, \textbf{x}_{j_1}) & \dots & \boldsymbol\Gamma_p(\textbf{x}_{j_2'}, \textbf{x}_{j_k}) \\
\dots & \dots  & \dots & \dots  \\
\boldsymbol\Gamma_p(\textbf{x}_{j_{k+1}'}, \boldsymbol{\phi}_{i}) & \boldsymbol\Gamma_p(\textbf{x}_{j_{k+1}'}, \textbf{x}_{j_1}) & \dots & \boldsymbol\Gamma_p(\textbf{x}_{j_{k+1}'}, \textbf{x}_{j_k}) \\
\end{bmatrix} 
\boldsymbol{\mu}_p,
\end{split}
\end{equation}
which has dimension $n_1 \times \binom{d}{k+1} \times n_1 \times \binom{d}{k}$.
We reshape $\textbf{d}^{(k),\textnormal{weak}}$ into an $n_1\binom{d}{k+1} \times n_1\binom{d}{k}$ matrix $\textbf{d}^{(k),\textnormal{weak}}_{I' I}$, the weak exterior derivative, and recover the strong map $\textbf{d}^{(k)} : \R^{n_1\binom{d}{k}} \to \R^{n_1\binom{d}{k+1}}$ as $\textbf{d}^{(k)} = (\textbf{G}^{(k+1)})^+ \textbf{d}^{(k),\textnormal{weak}}$.
Since $d^{(0)}$ is just the dual to the gradient $\nabla$ (so the matrix represenations $\textbf{d}^{(0)} = \boldsymbol{\nabla}$ agree), it is also visualised in Figure \ref{fig:gradient}.
We illustrate the action of the exterior derivative on 1-forms in Figure \ref{fig:derivative-1forms}, where it measures the \boldblue{signed vorticity} of the dual vector field.
When the ambient dimension is 3, the vorticity 2-form represents, at each point, an oriented 2d subspace with a unique normal vector (its \textit{Hodge dual}) called the \boldblue{curl}.
Specifically, given a vector field $X \in \mathfrak{X}(M)$ in ambient dimension 3, we can define $\textnormal{curl}(X) = (* d(X^\flat))^\sharp$.


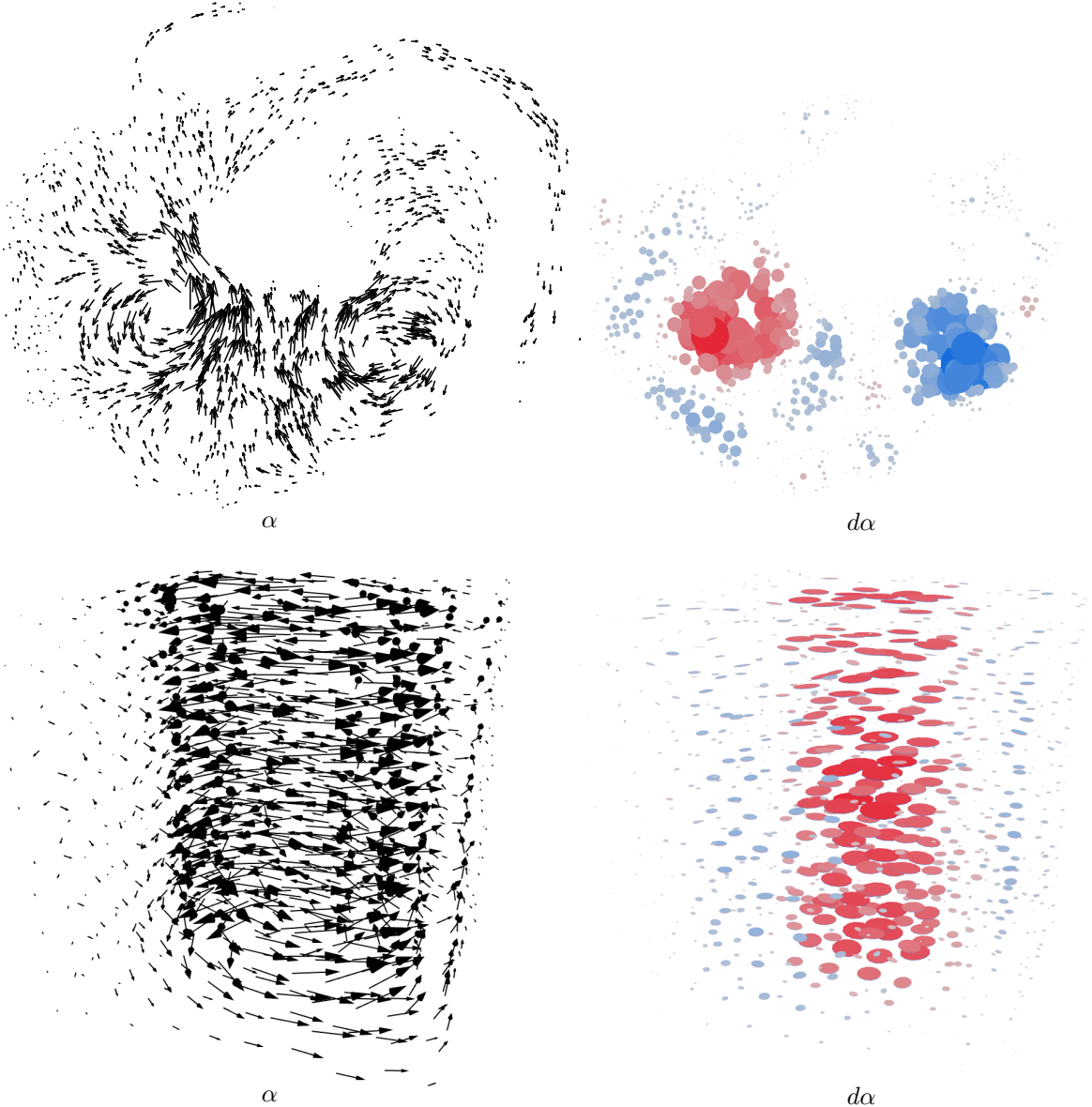
\begin{figure}[h!]
  \centering
  \begin{overpic}[width=\linewidth,grid=false]{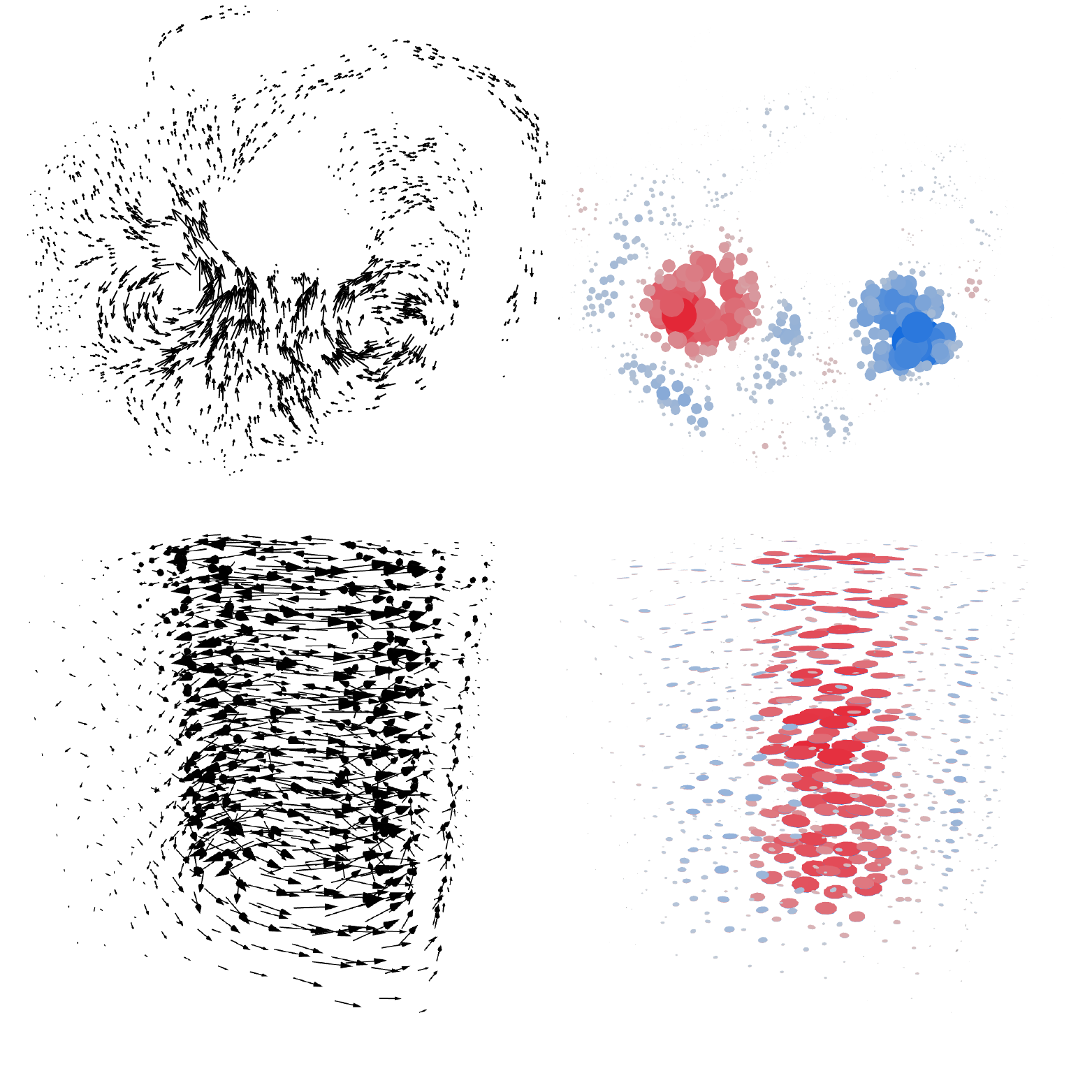}
\put(25.0,55.3){\makebox(0,0)[c]{$\alpha$}}
\put(75.0,55.3){\makebox(0,0)[c]{$d \alpha$}}
\put(25.0,6.4){\makebox(0,0)[c]{$\alpha$}}
\put(75.0,6.4){\makebox(0,0)[c]{$d \alpha$}}
  \end{overpic}
  \vspace{-2em}
  \caption{\textbf{Derivative of 1-forms.}
  If $\alpha \in \Omega^1(M)$ is a 1-form, its derivative $d\alpha \in \Omega^2(M)$ measures the signed vorticity of $\alpha^\sharp \in \mathfrak{X}(M)$ (clockwise in red, anticlockwise in blue).
  The data in the top row is in 2d, and the bottom row is a dense cube in 3d.
  }
  \label{fig:derivative-1forms}
\end{figure}

\begin{computationalnote}
The bottleneck in \eqref{eq: weak d_k} is evaluating the $(k\!+\!1)\!\times\!(k\!+\!1)$ determinants indexed by $(p,J',i,J)$.
We can ease this by exploiting the matrix structure and using a Laplace expansion down the first column, giving
\begin{equation}\label{eq:laplace-exp}
\det[\Gamma_p(x_{j'_r},\p{i}) \;|\; \Gamma_p(x_{j'_r},x_{j_m})]
= \sum_{r=1}^{k+1} (-1)^{r+1} 
\Gamma_p(x_{j'_r},\p{i}) 
\det[\Gamma_p(x_{j'_s},x_{j_m})]_{s\ne r}.
\end{equation}
The minors 
$\det[\Gamma_p(x_{j'_s},x_{j_m})]_{s\ne r}$,
depend only on $(p,J',J)$ but not $i$, so we can cheaply precompute them in $\mathcal{O}\left( k^3\binom{d}{k}\binom{d}{k+1}\right)$ time. Taking the sum in Equation (\ref{eq:laplace-exp}) then takes the total complexity to $\mathcal{O}\left( \binom{d}{k}\binom{d}{k+1}nk(k^2+n_1^2)\right)$.
We can also reorder the operators and perform the sum over $p$ before the sum over $r$, which yields another small improvement.
\end{computationalnote}

We can compute the \boldblue{codifferential} $\boldsymbol\partial^{(k)} = (\textbf{d}^{(k-1)})^* : \R^{n_1\binom{d}{k}} \to \R^{n_1\binom{d}{k-1}}$ for minimal extra effort, because
\[
\boldsymbol\partial^{(k)} = (\textbf{G}^{(k-1)})^+ (\textbf{d}^{(k-1),\textnormal{weak}})^T
\]
for $k \geq 1$.
Since $\partial^{(1)} : \Omega^1(M) \to \A$ is just the negative of the dual to the divergence $-\nabla^* : \mathfrak{X}(M) \to \A$, we have $\boldsymbol\partial^{(1)} = - \boldsymbol{\nabla}^*$, which is visualised in Figure \ref{fig:divergence}.
We illustrate the action of the codifferential on 2-forms in Figure \ref{fig:codifferential-2forms}, where it defines a 1-form whose dual vector field is a \boldblue{rotational flow} around the boundaries of the support of the 2-form.
The direction of rotation reflects the orientation of the 2-form.

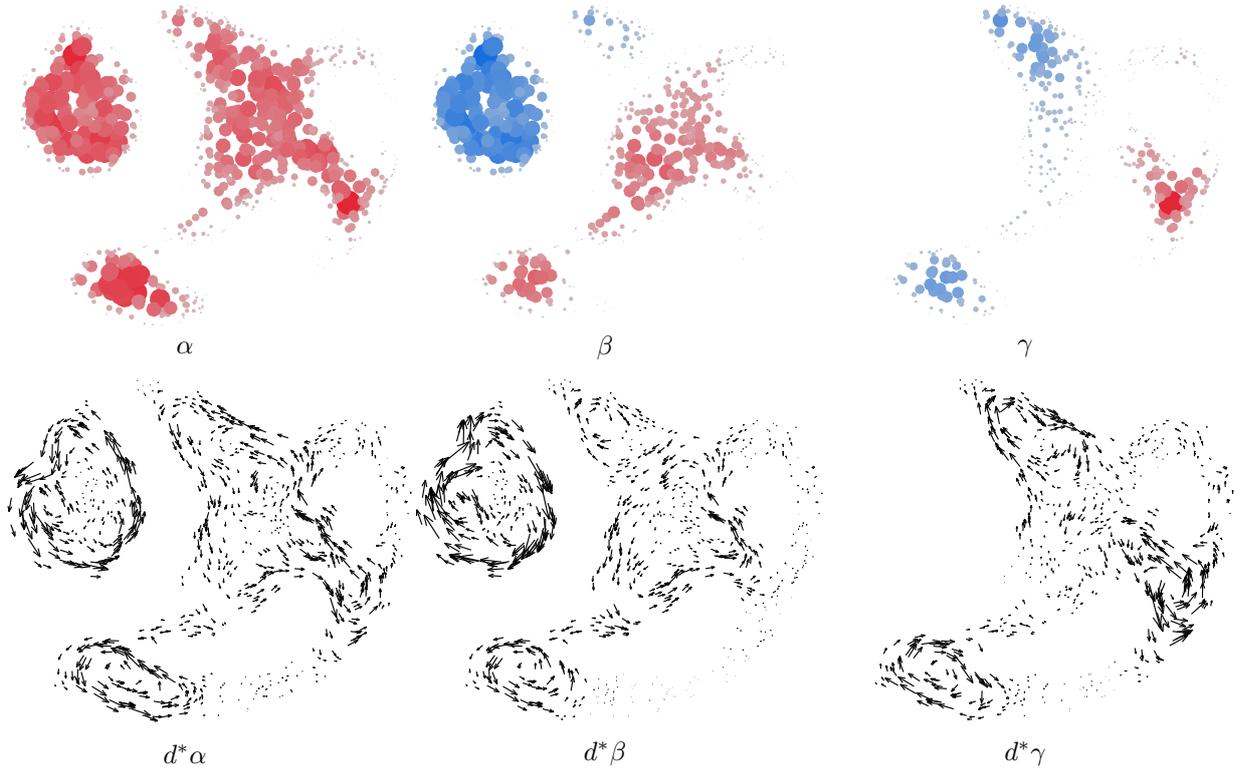
\begin{figure}[h!]
  \vspace{1em}
  \centering
  \begin{overpic}[width=\linewidth,grid=false]{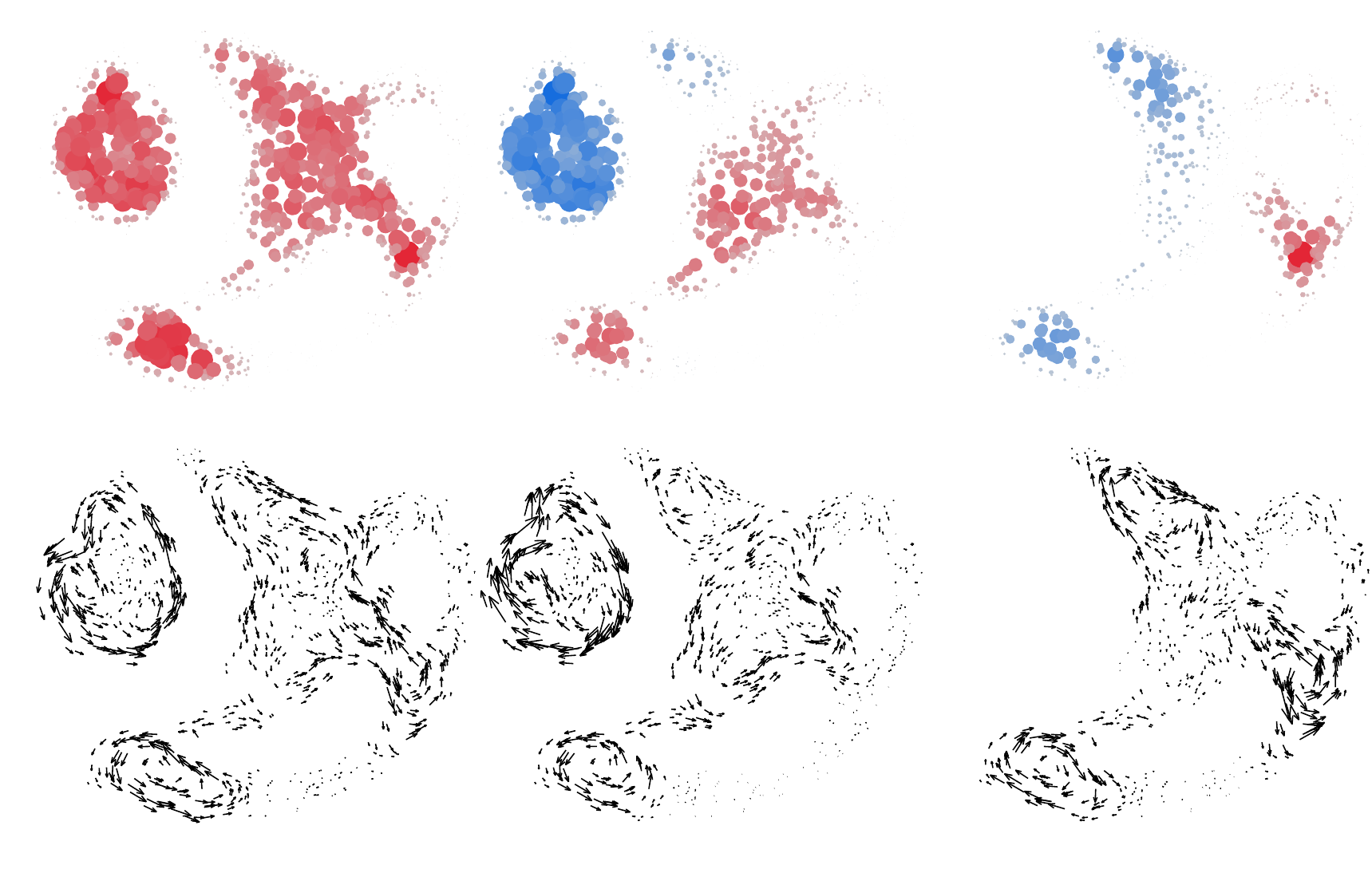}
\put(16.7,33.5){\makebox(0,0)[c]{$\alpha$}}
\put(50.0,33.5){\makebox(0,0)[c]{$\beta$}}
\put(83.3,33.5){\makebox(0,0)[c]{$\gamma$}}
\put(16.7,1.1){\makebox(0,0)[c]{$d^* \alpha$}}
\put(50.0,1.1){\makebox(0,0)[c]{$d^* \beta$}}
\put(83.3,1.1){\makebox(0,0)[c]{$d^* \gamma$}}
  \end{overpic}
  \vspace{0.5em}
  \caption{\textbf{Codifferential of 2-forms.}
  A 2-form $\alpha$ is mapped to its codifferential 1-form $d^* \alpha$, which defines a purely rotational flow around the boundaries of support of $\alpha$.
  We plot three 2-forms defined on data in 2d.
  }
  \label{fig:codifferential-2forms}
\end{figure}

\subsection{Directional derivatives}

\label{sub: directional derivatives}

Vector fields are first-order differential operators and act on functions by $X(f) = g(X, \nabla f)$, for $f \in \A$ and $X \in \mathfrak{X}(M)$, which we can interpret as the \boldblue{directional derivative} of $f$ in the direction of $X$.
If $f\nabla h \in \mathfrak{X}(M)$ is a vector field and $f' \in \A$ is a function, then we can evaluate
\[
(f \nabla h)(f') = f g(\nabla h, \nabla f') = f \Gamma(h,f') \in \A.
\]
Now taking
\[
\textbf{X} = \sum_{i=1}^{n_1} \sum_{j=1}^d \textbf{X}_{ij} \p{i} \nabla x_j
\]
we see that, if $f \in \A$ is a function, then
\[
\textbf{X}(f) = \sum_{i=1}^{n_1} \sum_{j=1}^d \textbf{X}_{ij} \p{i} \Gamma(x_j, f).
\]
We previously used this fact to visualise vector fields by applying them to the coordinate functions $x_j$ (subsection \ref{sub: visualising vector fields}), but will now compute the action of a vector field as an operator on functions via a weak formulation.
Any vector field $\textbf{X} \in \R^{n_1d}$ will correspond to a linear operator on functions, which is represented by an $n_0 \times n_0$ matrix $\textbf{X}^\textnormal{op}$, which we can access via its weak formulation $\textbf{X}^\textnormal{op, weak}$.
We compute the action of $\textbf{X}$ on the function $\p{t}$ via its inner product with $\p{s}$ in the $n_0 \times n_0$ matrix
\begin{equation*}
\begin{split}
\textbf{X}_{st}^\textnormal{op, weak}
&= \inp{\p{s}}{X(\p{t})} \\
&= \sum_{i=1}^{n_1} \sum_{j=1}^d \textbf{X}_{ij} \inp{\p{s}}{(\p{i} \nabla x_j)(\p{t})} \\
&= \sum_{i=1}^{n_1} \sum_{j=1}^d \textbf{X}_{ij} \inp{\p{s}}{\p{i} \Gamma(x_j, \p{t})} \\
&= \sum_{i=1}^{n_1} \sum_{j=1}^d \textbf{X}_{ij} \int \p{s} \p{i} \Gamma(x_j, \p{t}) d\mu \\
&= \sum_{i=1}^{n_1} \sum_{j=1}^d \sum_{p=1}^n \textbf{X}_{ij} \textbf{U}_{ps} \textbf{U}_{pi} \boldsymbol\Gamma_p(\textbf{x}_j, \boldsymbol{\phi}_t) \boldsymbol{\mu}_p.
\end{split}
\end{equation*}
We recover $\textbf{X}^\textnormal{op}$ as $\textbf{X}^\textnormal{op} = (\textbf{G}^{(0)})^+ \textbf{X}^\textnormal{op, weak}$.
In the standard case where the function basis is orthonormal, $\textbf{G}^{(0)} = \textbf{I}_{n_0}$ so $\textbf{X}^\textnormal{op} = \textbf{X}^\textnormal{op, weak}$, but this general approach will work for any function basis.
So, if $\textbf{f} \in \R^{n_0}$ is a function and $\textbf{X} \in \R^{n_1d}$ is a vector field, then $\textbf{X}(\textbf{f}) = \textbf{X}^\textnormal{op} \textbf{f} \in \R^{n_0}$.
We illustrate the action of vector fields on functions in Figure \ref{fig:directional-derivative}.

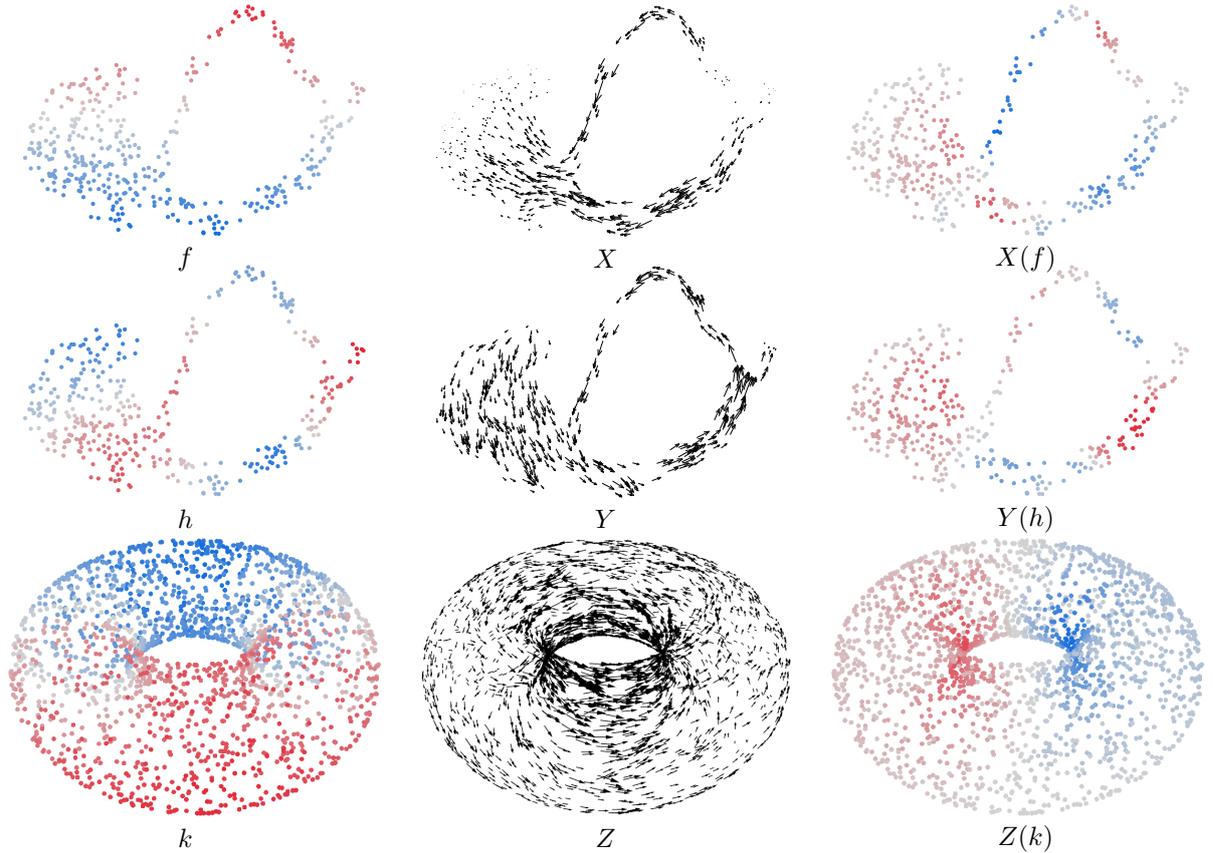
\begin{figure}[h!]
  \vspace{1em}
  \centering
  \begin{overpic}[width=\linewidth,grid=false]{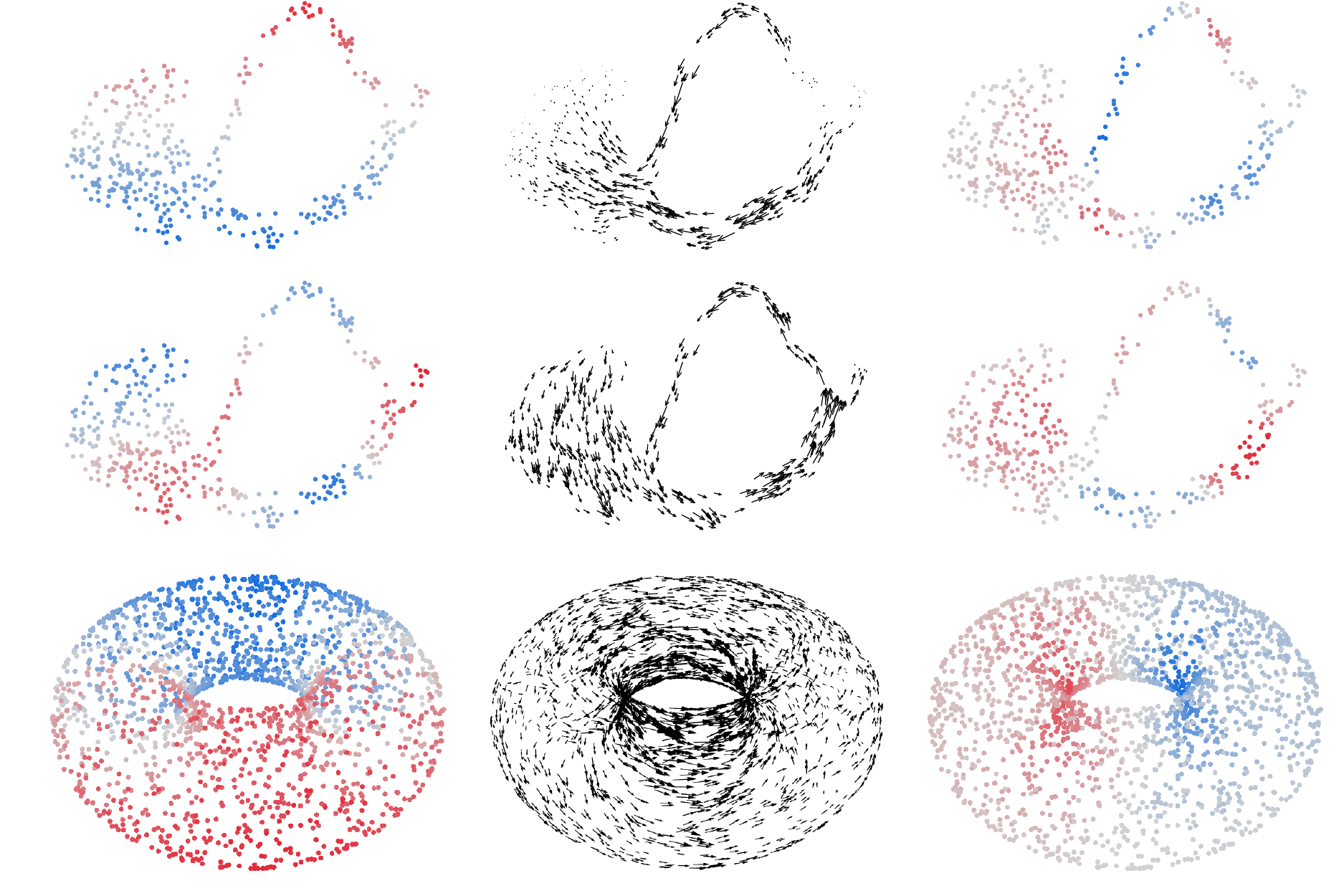}
\put(17.7,46.3){\makebox(0,0)[c]{$f$}}
\put(51.0,46.3){\makebox(0,0)[c]{$X$}}
\put(84.3,46.3){\makebox(0,0)[c]{$X(f)$}}
\put(17.7,25.5){\makebox(0,0)[c]{$h$}}
\put(51.0,25.5){\makebox(0,0)[c]{$Y$}}
\put(84.3,25.5){\makebox(0,0)[c]{$Y(h)$}}
\put(17.7,0.){\makebox(0,0)[c]{$k$}}
\put(51.0,0.){\makebox(0,0)[c]{$Z$}}
\put(84.3,0.){\makebox(0,0)[c]{$Z(k)$}}
  \end{overpic}
  \vspace{0.5em}
  \caption{\textbf{Directional derivative of functions along vector fields.}
  A function $f$ is mapped by a vector field $X$ to its derivative $X(f)$ in the direction of $X$.
  The data in the top two rows is in 2d, and the bottom row is in 3d.
  }
  \label{fig:directional-derivative}
\end{figure}

\subsection{Hodge Laplacian}
\label{sub: hodge laplacian}

The \boldblue{Hodge Laplacian} is the map $\Delta^{(k)} = \partial^{(k+1)}d^{(k)} + d^{(k-1)}\partial^{(k)} : \Omega^k(M) \to \Omega^k(M)$ which generalises the Laplacian on functions.
It has the powerful property of measuring the topology of the space through its spectrum: see Section \ref{sec: TDA}.
We will discretise $\Delta^{(k)}$ as a $n_1\binom{d}{k} \times n_1 \binom{d}{k}$ matrix $\boldsymbol\Delta^{(k)}$ that acts on $k$-forms in $\R^{n_1\binom{d}{k}}$.
We will compute the weak formulation (which is called the \boldblue{Hodge energy})
\[
\inp{\alpha}{\Delta^{(k)}\beta} = 
\inp{d^{(k)}\alpha}{d^{(k)}\beta} + \inp{\partial^{(k)}\alpha}{\partial^{(k)}\beta}
\]
as an $n_1\binom{d}{k} \times n_1 \binom{d}{k}$ matrix $\boldsymbol\Delta^{(k),\textnormal{weak}}$ and recover $\boldsymbol\Delta^{(k)} = (\textbf{G}^{(k)})^+ \boldsymbol\Delta^{(k),\textnormal{weak}}$.
If we only want to know the spectrum of $\boldsymbol\Delta^{(k)}$, we can efficiently solve the generalised eigenproblem
\begin{equation}
\label{eq: hodge generalised eigenproblem}
\boldsymbol\Delta^{(k),\textnormal{weak}}\boldsymbol\alpha = \lambda \textbf{G}^{(k)}\boldsymbol\alpha,
\end{equation}
which is fully symmetric because $\Delta^{(k)}$ is self-adjoint.
To compute the Hodge energy $\boldsymbol\Delta^{(k),\textnormal{weak}}$, we address the \q{up} and \q{down} energy terms separately.
We can write the \q{down} energy $\inp{\partial^{(k)}\alpha}{\partial^{(k)}\beta}$ as
\begin{equation*}
\begin{split}
\inp{\partial^{(k)}\alpha}{\partial^{(k)}\beta} 
&= \inp{\alpha}{d^{(k-1)}\partial^{(k)}\beta} \\
&= \boldsymbol\alpha \textbf{G}^{(k)} \textbf{d}^{(k-1)}\boldsymbol\partial^{(k)}\boldsymbol\beta \\
&= \boldsymbol\alpha \textbf{G}^{(k)} \big( (\textbf{G}^{(k)})^+ \textbf{d}^{(k-1),\textnormal{weak}} \big) \big( (\textbf{G}^{(k-1)})^+ (\textbf{d}^{(k-1),\textnormal{weak}})^T \big) \boldsymbol\beta \\
&= \boldsymbol\alpha \textbf{d}^{(k-1),\textnormal{weak}} (\textbf{G}^{(k-1)})^+ (\textbf{d}^{(k-1),\textnormal{weak}})^T \boldsymbol\beta, \\
\end{split}
\end{equation*}
so the \q{down} energy has matrix representation 
$$
\textbf{Down}^{(k),\textnormal{weak}} = \textbf{d}^{(k-1),\textnormal{weak}} (\textbf{G}^{(k-1)})^+ (\textbf{d}^{(k-1),\textnormal{weak}})^T
$$
which is the weak formulation of the \boldblue{down Laplacian} $d^{(k-1)}\partial^{(k)}$.
We could compute the \q{up} energy in the same way as 
$$
	\textbf{Up}^{(k)} = (\textbf{d}^{(k),\textnormal{weak}})^T (\textbf{G}^{(k+1)})^+ \textbf{d}^{(k),\textnormal{weak}}
$$
but this involves mapping explicitly from $\R^{n_1\binom{d}{k}}$ into the possibly larger space $\R^{n_1\binom{d}{k+1}}$ and inverting the $n_1\binom{d}{k+1} \times n_1\binom{d}{k+1}$ matrix $\textbf{G}^{(k+1)}$ there, which would be computationally expensive.
Instead, we employ the \boldblue{kernel trick} and directly evaluate the \q{up} energy only on $k$-forms in $\R^{n_1\binom{d}{k}}$.
We compute the inner product of $d^{(k)}(\p{i}dx_{(J)})$ and $d^{(k)}(\p{i'}dx_{(J')})$ (with $J=(j_1<\dots<j_k)$ and $J'=(j'_1<\dots<j'_k)$) and collect these into the 4-tensor
\begin{equation}
\label{eq: up hodge}
\begin{split}
	\textbf{Up}^{(k),\textnormal{weak}}_{i' J' i J}
&= \inp{d^{(k)}(\p{i'} dx_{(J')})}
{d^{(k)}(\p{i} dx_{(J)})} \\
&= \inp{d\p{i'} \wedge dx_{j'_1} \wedge \dots \wedge dx_{j'_k}}
{d\p{i} \wedge dx_{j_1} \wedge \dots \wedge dx_{j_k}} \\
&= \int
\det
\begin{bmatrix}
\Gamma(\p{i'}, \p{i}) & \Gamma(\p{i'}, x_{j_1}) & \dots & \Gamma(\p{i'}, x_{j_k}) \\
\Gamma(x_{j_1'}, \p{i}) & \Gamma(x_{j_1'}, x_{j_1}) & \dots & \Gamma(x_{j_1'}, x_{j_k}) \\
\dots & \dots  & \dots & \dots  \\
\Gamma(x_{j_k'}, \p{i}) & \Gamma(x_{j_k'}, x_{j_1}) & \dots & \Gamma(x_{j_k'}, x_{j_k}) \\
\end{bmatrix} 
d\mu \\
&= \sum_{p=1}^n
\det
\begin{bmatrix}
\boldsymbol\Gamma_p(\boldsymbol{\phi}_{i'}, \boldsymbol{\phi}_{i}) & \boldsymbol\Gamma_p(\boldsymbol{\phi}_{i'}, \textbf{x}_{j_1}) & \dots & \boldsymbol\Gamma_p(\boldsymbol{\phi}_{i'}, \textbf{x}_{j_k}) \\
\boldsymbol\Gamma_p(\textbf{x}_{j_1'}, \boldsymbol{\phi}_{i}) & \boldsymbol\Gamma_p(\textbf{x}_{j_1'}, \textbf{x}_{j_1}) & \dots & \boldsymbol\Gamma_p(\textbf{x}_{j_1'}, \textbf{x}_{j_k}) \\
\dots & \dots  & \dots & \dots  \\
\boldsymbol\Gamma_p(\textbf{x}_{j_k'}, \boldsymbol{\phi}_{i}) & \boldsymbol\Gamma_p(\textbf{x}_{j_k'}, \textbf{x}_{j_1}) & \dots & \boldsymbol\Gamma_p(\textbf{x}_{j_k'}, \textbf{x}_{j_k}) \\
\end{bmatrix} 
\boldsymbol{\mu}_p, \\
\end{split}
\end{equation}
which has dimension $n_1 \times \binom{d}{k} \times n_1 \times \binom{d}{k}$.
We then reshape $\textbf{Up}^{(k)}$ into an $n_1\binom{d}{k} \times n_1\binom{d}{k}$ matrix $\textbf{Up}^{(k)}_{IJ}$, which is the weak \boldblue{up Laplacian} $\partial^{(k+1)}d^{(k)}$.
Summing the \q{up} and \q{down} terms we obtain the Hodge energy
$$
\boldsymbol\Delta^{(k),\textnormal{weak}} = \textbf{Down}^{(k),\textnormal{weak}} + \textbf{Up}^{(k),\textnormal{weak}}.
$$
Computing the spectrum of the Hodge Laplacian via the weak eigenproblem (\ref{eq: hodge generalised eigenproblem}) leads to powerful new tools for topological data analysis, which we will explore in Section \ref{sec: TDA}.

We illustrate the Hodge Laplacian acting on functions and 1-forms in Figure \ref{fig:laplacian}, where we separate out the \q{up} and \q{down} components of the Laplacian.
When $f \in \A$ is a function, $\Delta^{(0)} f = \partial^{(1)} d^{(0)} f$ is the classical Laplacian of $f$, which measures how $f$ differs from its local average.
If we rewrite $\Delta^{(0)} f = \nabla^* \nabla f$, we see that it measures how much the gradient vector field $\nabla f$ expands and contracts space, which is visualised in the first two rows of Figure \ref{fig:laplacian}.
When $\alpha \in \Omega^1(M)$ is a 1-form, $\Delta^{(1)} \alpha = d^{(0)} \partial^{(1)} \alpha + \partial^{(2)} d^{(1)} \alpha$.
The \q{down} term $d^{(0)} \partial^{(1)} \alpha$ measures how and where the dual vector field $\alpha^\sharp$ expands and contracts space (like the Laplacian of a function), while the \q{up} term $\partial^{(2)} d^{(1)} \alpha$ measures how much $\alpha^\sharp$ rotates, and this is visualised in the last two rows of Figure \ref{fig:laplacian}.

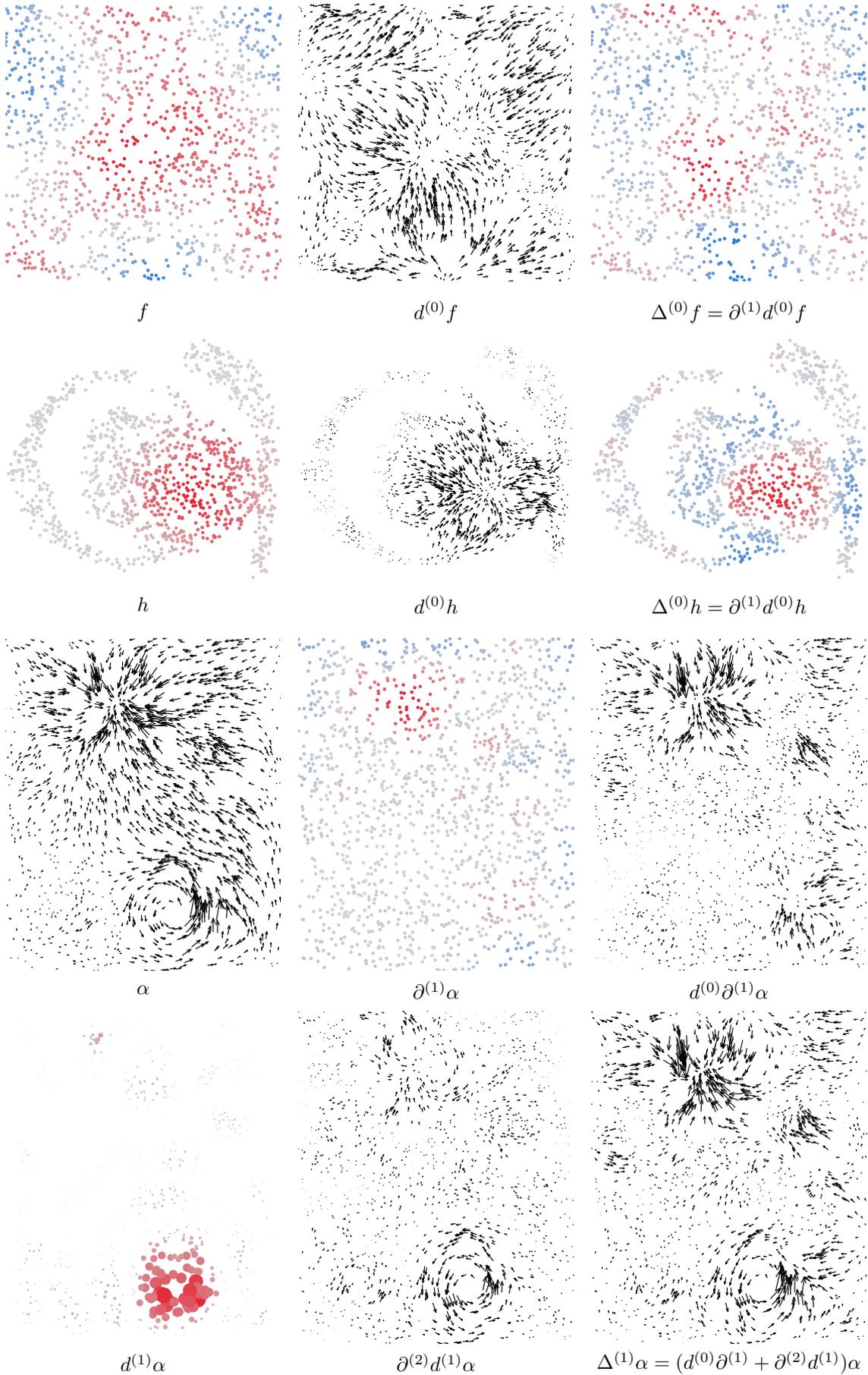
\begin{figure}
  \vspace{-1em}
  \centering
  \begin{overpic}[width=0.9\linewidth,grid=false]{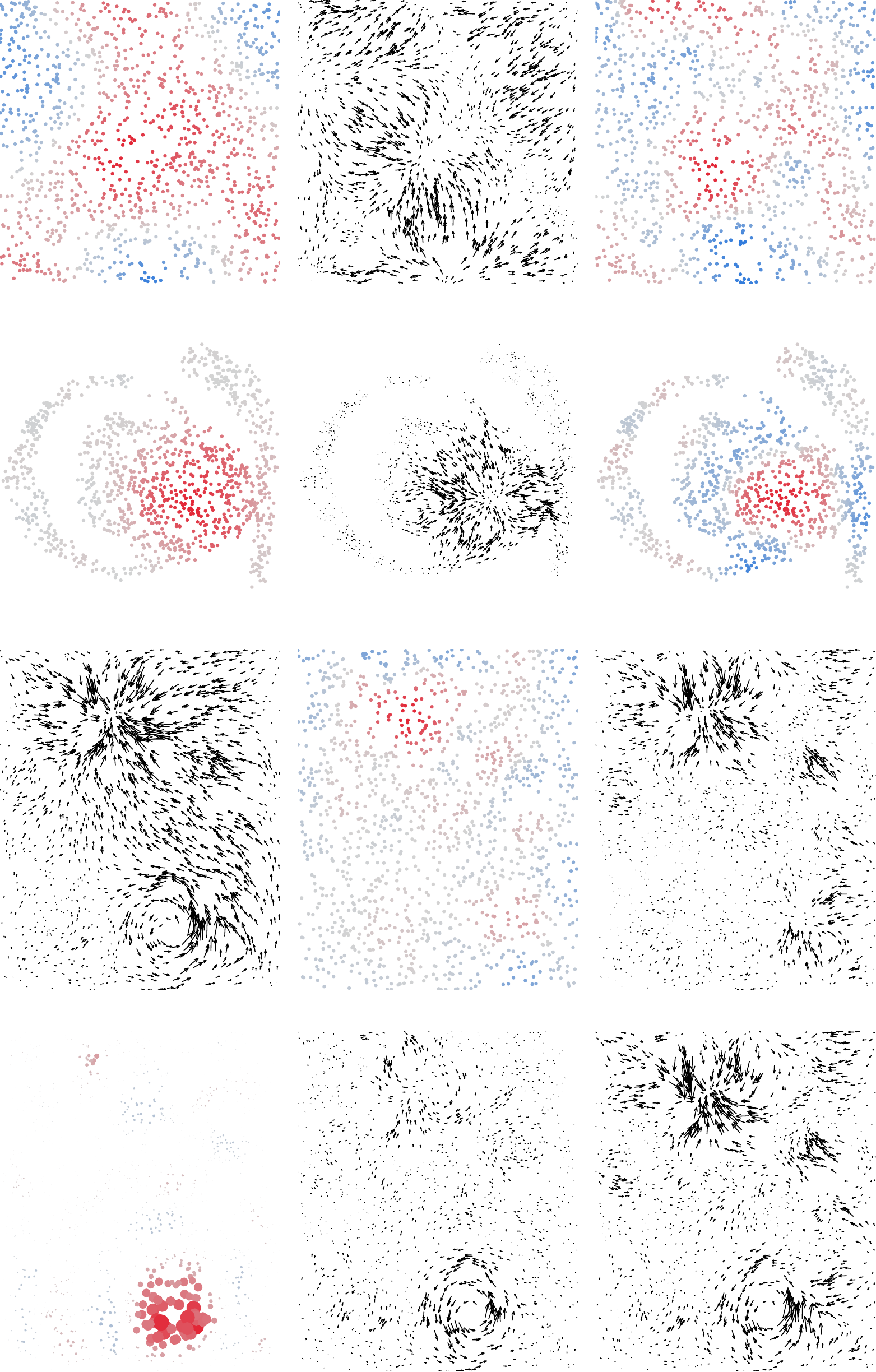}
\put(10.2,77.0){\makebox(0,0)[c]{$f$}}
\put(31.9,77.0){\makebox(0,0)[c]{$d^{(0)} f$}}
\put(53.6,77.0){\makebox(0,0)[c]{$\Delta^{(0)} f = \partial^{(1)} d^{(0)} f$}}
\put(10.2,55.2){\makebox(0,0)[c]{$h$}}
\put(31.9,55.2){\makebox(0,0)[c]{$d^{(0)} h$}}
\put(53.6,55.2){\makebox(0,0)[c]{$\Delta^{(0)} h = \partial^{(1)} d^{(0)} h$}}

\put(10.2,26.4){\makebox(0,0)[c]{$\alpha$}}
\put(31.9,26.4){\makebox(0,0)[c]{$\partial^{(1)} \alpha$}}
\put(53.6,26.4){\makebox(0,0)[c]{$d^{(0)} \partial^{(1)} \alpha$}}
\put(10.2,-1.4){\makebox(0,0)[c]{$d^{(1)} \alpha$}}
\put(31.9,-1.4){\makebox(0,0)[c]{$\partial^{(2)} d^{(1)} \alpha$}}
\put(53.6,-1.4){\makebox(0,0)[c]{$\Delta^{(1)} \alpha = (d^{(0)} \partial^{(1)} + \partial^{(2)} d^{(1)}) \alpha$}}
  \end{overpic}
  \vspace{2em}
  \caption{\textbf{Hodge Laplacian of functions and 1-forms.}
  All data is 2d.
  If $f \in \A$ is a function, then $\Delta^{(0)} f = \partial^{(1)} d^{(0)} f$ measures how much $d^{(0)} f = (\nabla f)^\flat$ expands and contracts space, i.e. the curvature of $f$. If $\alpha \in \Omega^1(M)$ is a 1-form, the down Laplacian $d^{(0)} \partial^{(1)} \alpha$ measures the failure of $\alpha$ to preserve volumes, and the up Laplacian $\partial^{(2)} d^{(1)} \alpha$ measures how locally rotational $\alpha$ is.
  }
  \label{fig:laplacian}
\end{figure}

\begin{computationalnote}
The bottleneck in \eqref{eq: up hodge} is evaluating the $(k\!+\!1)\!\times\!(k\!+\!1)$ determinants indexed by $(p,i',J',i,J)$.
We can exploit the block matrix structure using the Schur determinant identity.
Let $\textbf{D}_p(J',J) = [\boldsymbol\Gamma_p(\textbf{x}_{j'_r}, \textbf{x}_{j_c})]_{r,c\leq k}$ be the $k\!\times\!k$ bottom-right submatrix,
$\textbf{b}_p(i',J) = [\boldsymbol\Gamma_p(\boldsymbol{\phi}_{i'}, \textbf{x}_{j_c})]_{c\leq k}$ be the $1\!\times\!k$ top-right row vector,
and $\textbf{c}_p(J',i) = [\boldsymbol\Gamma_p(\textbf{x}_{j'_r}, \boldsymbol{\phi}_{i})]_{r\leq k}^T$ be the $k\!\times\!1$ bottom-left column vector.
The determinant is then given by
\[
\det
\begin{bmatrix}
\boldsymbol\Gamma_p(\boldsymbol{\phi}_{i'}, \boldsymbol{\phi}_{i}) & \textbf{b}_p(i',J) \\
\textbf{c}_p(J',i) & \textbf{D}_p(J',J)
\end{bmatrix} 
= \det(\textbf{D}_p) 
\left( \boldsymbol\Gamma_p(\boldsymbol{\phi}_{i'}, \boldsymbol{\phi}_{i}) - \textbf{b}_p \textbf{D}_p^{-1} \textbf{c}_p \right).
\]
The $k\!\times\!k$ matrix $\textbf{D}_p(J',J)$ and its determinant $\det(\textbf{D}_p)$ depend only on $(p,J',J)$ but not on $i$ or $i'$, so we can cheaply precompute them.
We can then perform the batched linear solves $\textbf{D}_p \textbf{x} = \textbf{c}_p(J',i)$ for all $(p,i',J',J)$ with no dependence on $i$, and so form the quadratic forms $\textbf{b}_p \textbf{D}_p^{-1} \textbf{c}_p$ efficiently for all $(p,i',i,J',J)$.
As before, we also reorder the operators and perform the sum over $p$ first, i.e.
\[
\sum_{p=1}^n\det(\textbf{D}_p) 
\left( \boldsymbol\Gamma_p(\boldsymbol{\phi}_{i'}, \boldsymbol{\phi}_{i}) - \textbf{b}_p \textbf{D}_p^{-1} \textbf{c}_p \right)
=
\sum_{p=1}^n\det(\textbf{D}_p) 
\boldsymbol\Gamma_p(\boldsymbol{\phi}_{i'}, \boldsymbol{\phi}_{i}) - 
\sum_{p=1}^n\det(\textbf{D}_p) \textbf{b}_p \textbf{D}_p^{-1} \textbf{c}_p,
\]
which yields another small improvement.
\end{computationalnote}

\subsection{Hodge decomposition}

We can perform the \boldblue{Hodge decomposition} of a $k$-form $\alpha \in \Omega^k(M)$ into
\[
\alpha = d_{k-1}\beta + \gamma + \partial_{k+1}\delta,
\]
where $\beta \in \Omega^{k-1}(M)$, $\gamma \in \Omega^k(M)$, and $\delta \in \Omega^{k+1}(M)$, with $d_k \gamma = 0$ and $\partial_k \gamma = 0$.
We call $d_{k-1}\beta$ the \boldblue{exact part} of $\alpha$, and $\gamma$ the \boldblue{harmonic part}, and $\partial_{k+1}\delta$ the \boldblue{coexact part}.
We will refer to $\beta$ as the \boldblue{exact potential} of $\alpha$, and $\delta$ as the \boldblue{coexact potential}.
The Hodge decomposition is unique, and the three parts are orthogonal.
The harmonic part $\gamma$ lies in the kernel of the Hodge Laplacian $\Delta_k$.
Intuitively, the harmonic and coexact parts of $\alpha$ both capture its rotational behaviour: the coexact part corresponds to vortices whose centre is in the space, and the harmonic part corresponds to global circular flows around holes in the space.
We can solve for the exact potential $\beta \in \Omega^{k-1}(M)$ using the fact that
\[
\Delta_{k-1}^{\textnormal{up}}\beta 
= \partial_k d_{k-1} \beta
= \partial_k (d_{k-1} \beta + \gamma + \partial_{k+1} \delta)
= \partial_k \alpha,
\]
and so $\beta = \big(\Delta_{k-1}^{\textnormal{up}}\big)^+ \partial_k \alpha$, where $\big(\Delta_{k-1}^{\textnormal{up}}\big)^+$ is the pseudoinverse of $\Delta_{k-1}^{\textnormal{up}}$.
In our discretisation, we can rewrite this equation using only weak formulations as
\begin{align*}
\boldsymbol{\beta} 
&= \big(\boldsymbol{\Delta}_{k-1}^{\textnormal{up}}\big)^+ \boldsymbol{\partial}_k \boldsymbol{\alpha} \\
&= \big((\mathbf{G}^{(k-1)})^+\boldsymbol{\Delta}_{k-1}^{\textnormal{up, weak}}\big)^+ 
(\textbf{G}^{(k-1)})^+ (\textbf{d}^{(k-1),\textnormal{weak}})^T \boldsymbol{\alpha} \\
&= \big(\boldsymbol{\Delta}_{k-1}^{\textnormal{up, weak}}\big)^+ \mathbf{G}^{(k-1)}
(\textbf{G}^{(k-1)})^+ (\textbf{d}^{(k-1),\textnormal{weak}})^T \boldsymbol{\alpha} \\
&= \big(\boldsymbol{\Delta}_{k-1}^{\textnormal{up, weak}}\big)^+
(\textbf{d}^{(k-1),\textnormal{weak}})^T \boldsymbol{\alpha},
\end{align*}
which we can easily solve.
We likewise compute the coexact potential $\delta \in \Omega^{k+1}(M)$ as
\[
\boldsymbol{\delta} = \big(\boldsymbol{\Delta}_{k+1}^{\textnormal{down, weak}}\big)^+
(\textbf{d}^{(k+1),\textnormal{weak}})^T \boldsymbol{\alpha},
\]
and obtain the harmonic part of $\boldsymbol{\alpha}$ as $\boldsymbol{\gamma} = \boldsymbol{\alpha} - \textbf{d}^{(k-1)}\boldsymbol{\beta} - \boldsymbol\partial^{(k+1)}\boldsymbol{\delta}$.
We illustrate the Hodge decomposition of functions and 1-forms in Figure \ref{fig:hodge-decomp}.
When $f \in \A$ is a function, its harmonic part is locally constant, and its coexact part captures the local variability of $f$.
The exact part is always zero because there are no $(-1)$-forms.
When $\alpha \in \Omega^1(M)$ is a 1-form, the exact potential is a function $f \in \A$ that measures the displacement of the flow along $\alpha^\sharp$, and the coexact part contains its vortices.
A 1-form $\beta \in \Omega^1(M)$ on a torus can have a harmonic part that captures a circular flow around the holes in the space.

\begin{computationalnote}
In practice, the harmonic part $\boldsymbol{\gamma}$ is always nearly zero.
Any forms that \q{look} harmonic will have very low Hodge energy, but will not be exactly in the kernel of $\Delta^{(k)}$ due to data uncertainty and numerical error.
When we perform their Hodge decomposition, the harmonic behaviour is absorbed into the coexact part, with a very large coexact potential $\boldsymbol{\delta}$.
This applies in the \q{harmonic part} of the flow on the torus in Figure \ref{fig:hodge-decomp}, where the visually harmonic part is actually computed as coexact.
On the other hand, the harmonic part of the function in the top row is genuinely computed as harmonic because of the k-nearest neighbour kernel construction, which makes the computational graph genuinely disconnected.
\end{computationalnote}

\begin{figure}[h!]
  \vspace{-1em}
  \centering
  \begin{overpic}[width=\linewidth,grid=false]{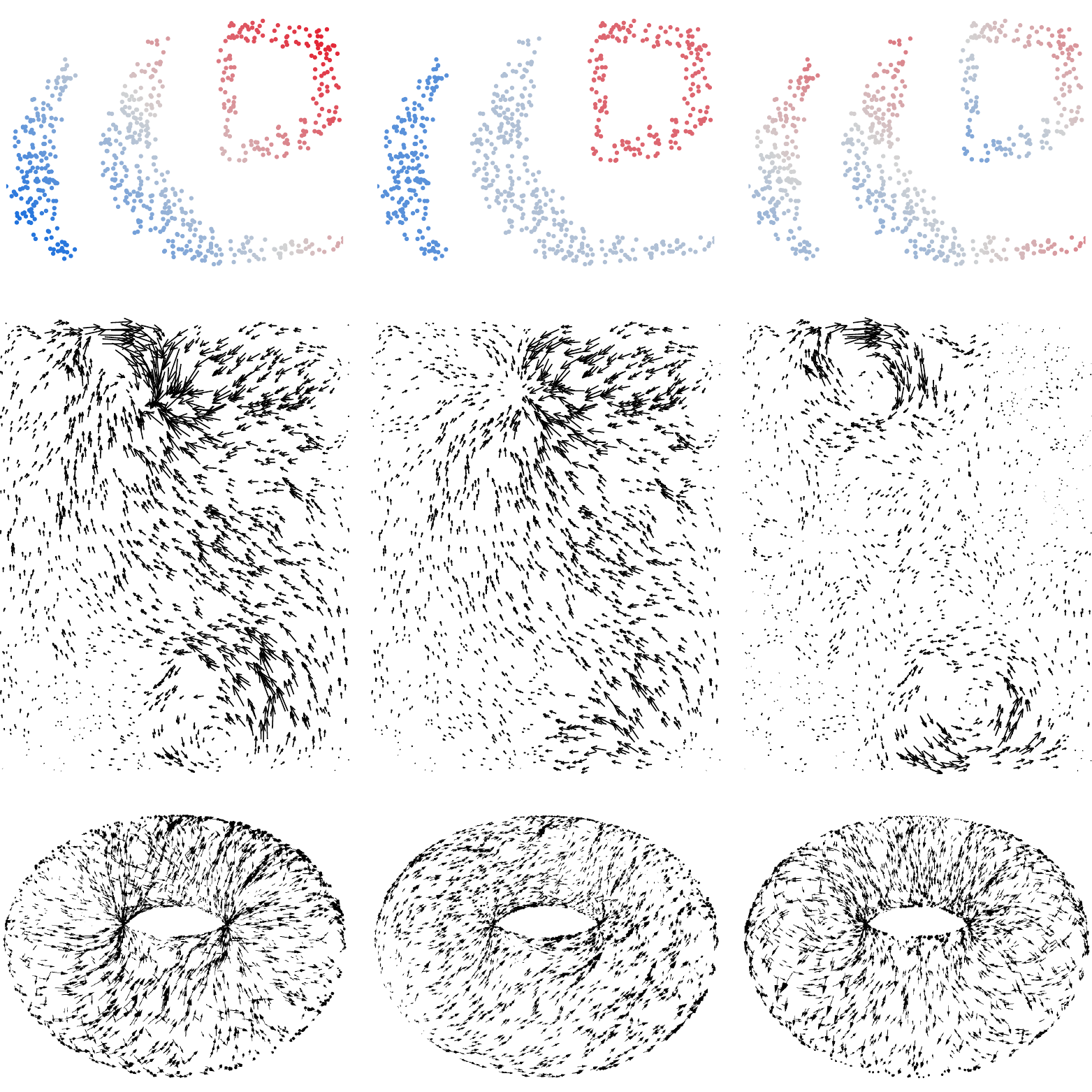}
\put(16.0,73.8){\makebox(0,0)[c]{$f$}}
\put(50.0,73.8){\makebox(0,0)[c]{harmonic part of $f$}}
\put(84.0,73.8){\makebox(0,0)[c]{coexact part of $f$}}
\put(16.0,-0.8){\makebox(0,0)[c]{$\beta$}}
\put(50.0,-0.8){\makebox(0,0)[c]{exact part of $\beta$}}
\put(84.0,-0.8){\makebox(0,0)[c]{harmonic part of $\beta$}}

\put(16.0,27.5){\makebox(0,0)[c]{$\alpha$}}
\put(50.0,27.5){\makebox(0,0)[c]{exact part of $\alpha$}}
\put(84.0,27.5){\makebox(0,0)[c]{coexact part of $\alpha$}}
  \end{overpic}
  \vspace{1em}
  \caption{\textbf{Hodge decomposition of functions and 1-forms.}
  If $f \in \A$ is a function, then its harmonic part is locally constant, and its coexact part captures the local variability of $f$.
  The fact that they are orthogonal means that the coexact part integrates to 0, i.e. it has measures \textit{only} variability but has mean 0.
  If $\alpha \in \Omega^1(M)$ is a 1-form, the exact part has the form $d f$ and measures the displacement of the flow along $\alpha$.
  The coexact part is purely rotational, and captures just the vortices of $\alpha$.
  The data in the top two rows are 2d, and the bottom row is 3d.
  }
  \label{fig:hodge-decomp}
\end{figure}
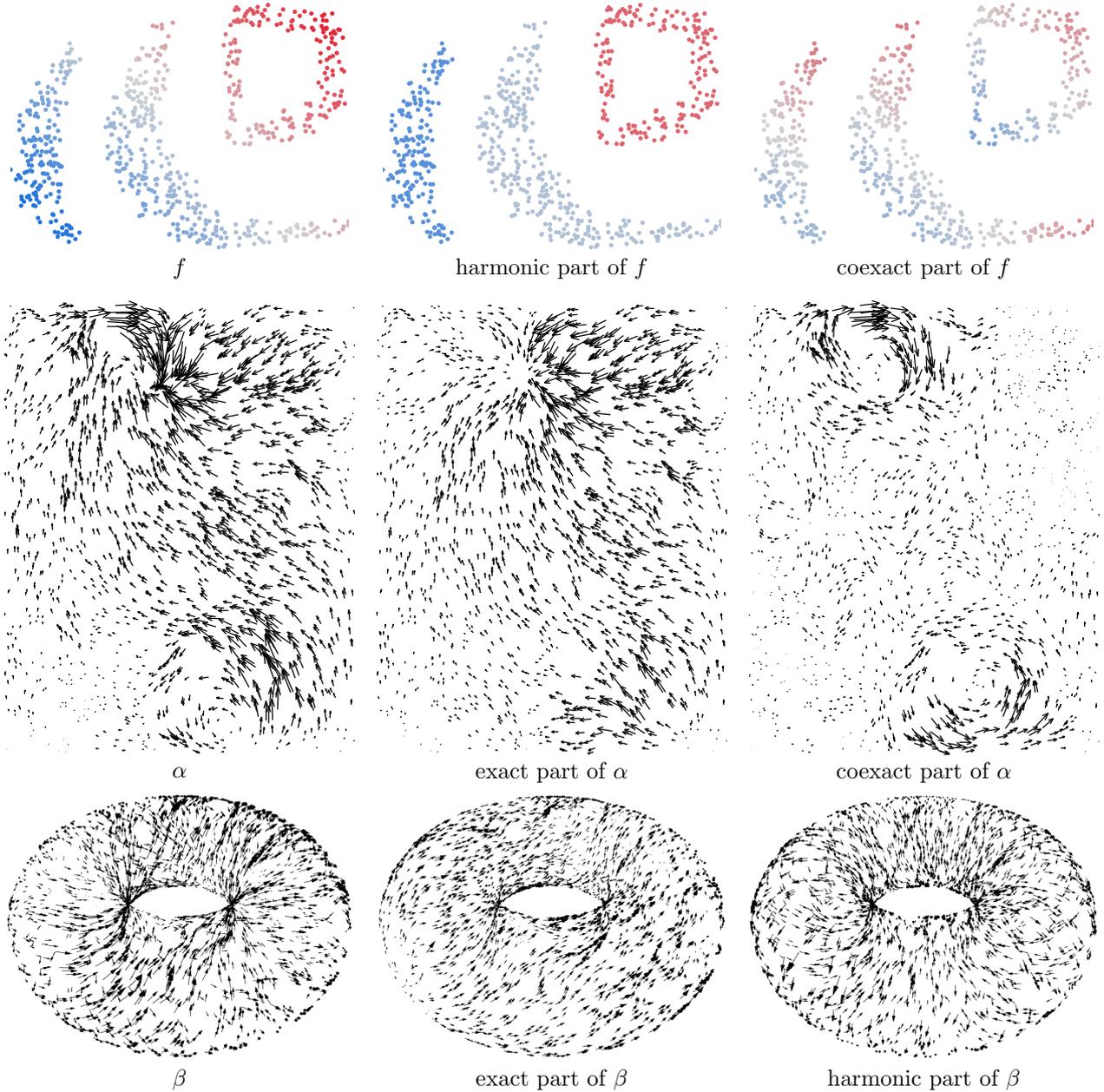

\subsection{Lie bracket}

The directional derivative of a function $f \in \A$ along a vector field $X \in \mathfrak{X}(M)$ is also called the \boldblue{Lie derivative} and is denoted $\mathcal{L}_X f$.
This notion extends to vector fields: if $X, Y \in \mathfrak{X}(M)$ are vector fields, then the Lie derivative of $Y$ along $X$ is another vector field $\mathcal{L}_X Y \in \mathfrak{X}(M)$ defined by
\begin{equation}
\label{eq: lie derivative of vector fields}
\mathcal{L}_X Y (f) = X(Y(f)) - Y(X(f))
\end{equation}
for any $f \in \A$.
This operation is also called the \boldblue{Lie bracket} of $X$ and $Y$, denoted $[X,Y] = \mathcal{L}_X Y$, and defines a bilinear map 
\begin{equation}\label{eq: def of lie bracket}
[\cdot,\cdot] : \mathfrak{X}(M) \times \mathfrak{X}(M) \to \mathfrak{X}(M).
\end{equation}

\begin{wrapfigure}{r}{0.3\textwidth}
\centering
\begin{tikzpicture}[scale=2,>=stealth,thick]

\coordinate (p1) at (0,0);
\coordinate (p2) at (1,0.1);
\coordinate (p3) at (0.9,1.1);
\coordinate (p4) at (-0.1,1.);
\coordinate (p5) at (-0.2,0.3);

\draw[->,smooth,shorten >=2pt] (p1) .. controls (0.2,-0.2) and (0.8,0.2)  .. (p2)
   node[pos=0.4,below right] {$\Phi^X_{\varepsilon}$};
\draw[->,smooth,shorten >=2pt] (p2) .. controls (1.2,0.3) and (0.75,0.7)  .. (p3)
   node[pos=0.7,below right] {$\Phi^Y_{\varepsilon}$};
\draw[->,smooth,shorten >=2pt] (p3) .. controls (0.5,1.2) and (0.4,0.9)  .. (p4)
   node[pos=0.8,above right, xshift=-0.25em, yshift=0.2em] {$\Phi^X_{-\varepsilon}$};
\draw[->,smooth,shorten >=2pt] (p4) .. controls (-0.2,0.6) and (0,0.6)  .. (p5)
node[pos=0.15,below left] {$\Phi^Y_{-\varepsilon}$};

\draw[dotted,->,shorten >=2.5pt] (p1) -- (p5)
   node[pos=0.3,below left, yshift=7pt] {$\varepsilon^2 [X,Y](p)$};

\fill (p1) circle (0.04) node[below,yshift=-0.5em ] {$p$};
\fill (p2) circle (0.03);
\fill (p3) circle (0.03);
\fill (p4) circle (0.03);
\fill (p5) circle (0.03);
  

\end{tikzpicture}
\caption{Non-commuting flows of two vector fields $X$ and $Y$ lead to non-trivial Lie bracket $[X,Y]$.} \label{fig:flowparallelogram}
\end{wrapfigure}
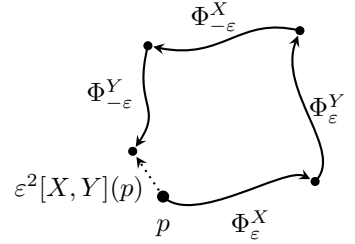

We can get some intuition for the Lie bracket by considering what happens on a classical manifold, where it quantifies the failure of the flows generated by $X$ and $Y$ to commute. 
Let $\Phi_\epsilon^X,\Phi_\epsilon^Y:M\to M$ denote the infinitesimal flows generated by $X$ and $Y$, respectively\footnote{%
The infinitesimal flow $\Phi_\epsilon^X(p)$ is the point reached by following the integral curve of $X$ starting at $p$ for a time $\epsilon$, satisfying $\frac{d}{d\epsilon}\Phi_\epsilon^X(p)=X(\Phi_\epsilon^X(p))$ and $\Phi_0^X(p)=p$.}. 
Composing them in the opposite order does not, in general, return to the same point:
\[
\Phi^Y_{-\varepsilon}\,\Phi^X_{-\varepsilon}\,\Phi^Y_{\varepsilon}\,\Phi^X_{\varepsilon}(p)
  = p + \varepsilon^2 [X,Y](p) + O(\varepsilon^3).
\]
The infinitesimal discrepancy $\epsilon^2 [X,Y](p)$ represents the small closing edge of the flow diagram spanned by $X$ and $Y$ at a point $p\in M$: see Figure \ref{fig:flowparallelogram}.

In our framework, we discretise the Lie bracket \eqref{eq: def of lie bracket} as an $n_1d \times n_1d \times n_1d$ 3-tensor $\textbf{Lie}$ that maps a pair of vector fields in $\R^{n_1d}$ to a vector field in $\R^{n_1d}$.
We compute the weak formulation $\textbf{Lie}^{\textnormal{weak}}$ as the inner product of $[\p{i_1}\nabla x_{j_1}, \p{i_2}\nabla x_{j_2}]$ and $\p{i'}\nabla x_{j'}$ in the 6-tensor
\begin{equation}
\label{eq: weak lie}
\begin{split}
\textbf{Lie}^{\textnormal{weak}}_{i' j' i_1j_1 i_2j_2}
&= \inp{\p{i'}\nabla x_{j'}}
{[\p{i_1}\nabla x_{j_1}, \p{i_2}\nabla x_{j_2}]} \\
&= \int \p{i'}
[\p{i_1}\nabla x_{j_1}, \p{i_2}\nabla x_{j_2}](x_{j'})
d\mu \\
&= \int \p{i'}
\big(
\p{i_1}\nabla x_{j_1} (\p{i_2}\Gamma(x_{j_2}, x_{j'}))
- \p{i_2}\nabla x_{j_2} (\p{i_1}\Gamma(x_{j_1}, x_{j'}))
\big)
d\mu \\
&= \int \p{i'}
\big(
\p{i_1}\Gamma (x_{j_1}, \p{i_2}\Gamma(x_{j_2}, x_{j'}))
- \p{i_2}\Gamma (x_{j_2}, \p{i_1}\Gamma(x_{j_1}, x_{j'}))
\big)
d\mu \\
&= \sum_{p=1}^n \textbf{U}_{pi'}
\big(
\textbf{U}_{pi_1}
\boldsymbol\Gamma_p(\textbf{x}_{j_1}, \boldsymbol{\phi}_{i_2} \boldsymbol\Gamma(\textbf{x}_{j_2}, \textbf{x}_{j'}))
- \textbf{U}_{pi_2}
\boldsymbol\Gamma_p(\textbf{x}_{j_2}, \boldsymbol{\phi}_{i_1} \boldsymbol\Gamma(\textbf{x}_{j_1}, \textbf{x}_{j'}))
\big)
\boldsymbol{\mu}_p,
\end{split}
\end{equation}
which has dimension $n_1 \times d \times n_1 \times d \times n_1 \times d$.
We then reshape $\textbf{Lie}^{\textnormal{weak}}$ into an $n_1d \times n_1d \times n_1d$ 3-tensor, and recover the strong formulation $\textbf{Lie} : \R^{n_1d} \times \R^{n_1d} \to \R^{n_1d}$ by computing $\textbf{Lie} = (\textbf{G}^{1})^+ \textbf{Lie}^{\textnormal{weak}}$, which represents the Lie bracket as a bilinear map $\R^{n_1d} \times \R^{n_1d} \to \R^{n_1d}$ given by $(X,Y) \mapsto [X, Y]$.

We visualise the Lie bracket in Figure~\ref{fig:lie-bracket}.
To verify the computation, we compute two examples on spaces where the Lie bracket can be evaluated analytically.
On a square, we expect $[x\nabla y, y\nabla x] = x\nabla x-y\nabla y$.
On a sphere, the latitude and longitude angles $\theta$ and $\phi$ define vector fields $\partial_\theta$ and $\partial_\phi$, and we expect $[\partial_\theta, \frac{1}{\sin(\theta)}\partial_\phi] = -\mathrm{cot}(\theta)\partial_\phi$.
In both cases, we recover the expected results computationally.

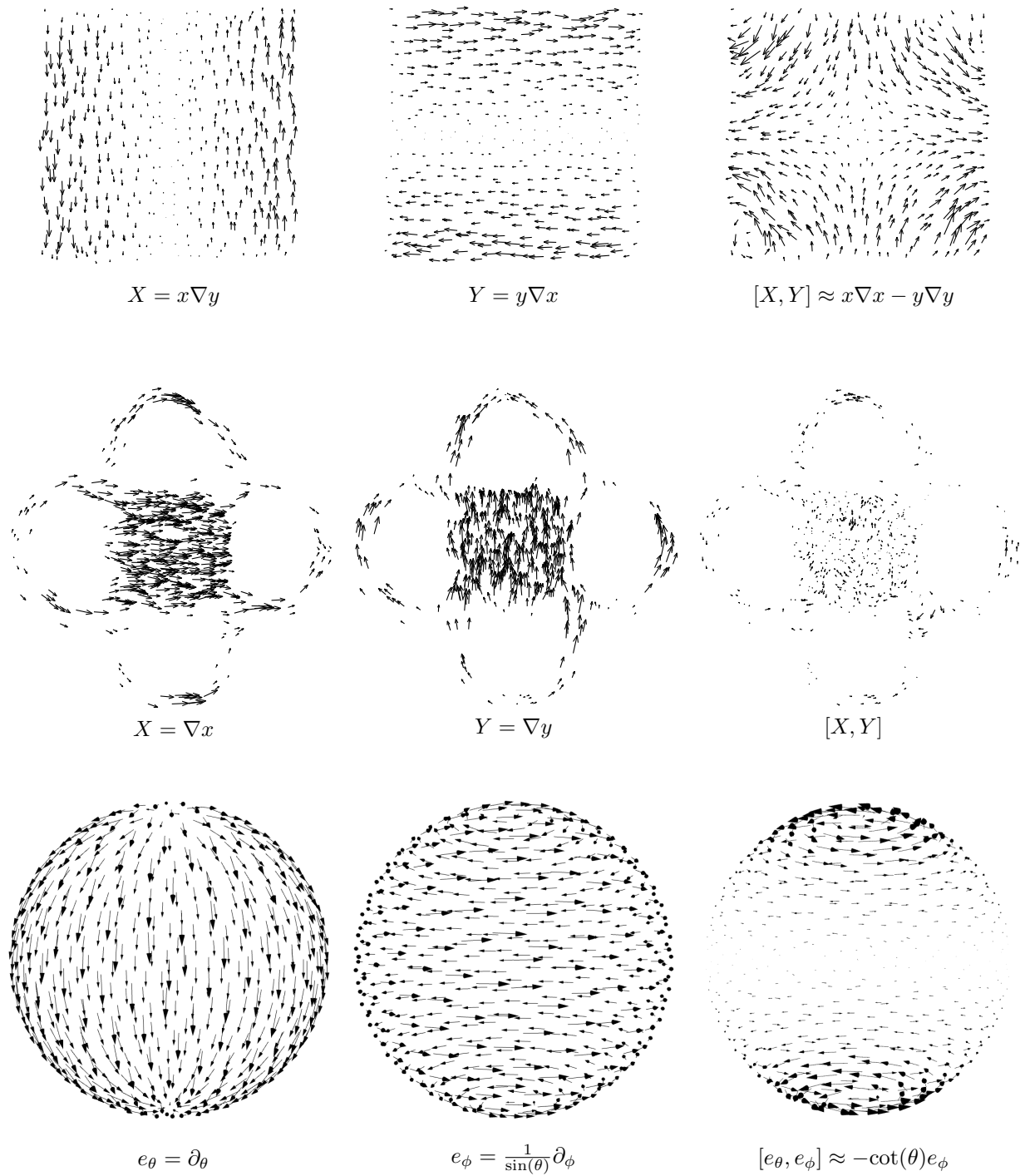
\begin{figure}[]
  \centering
\begin{overpic}[width=\linewidth,grid=false, yshift=2cm]{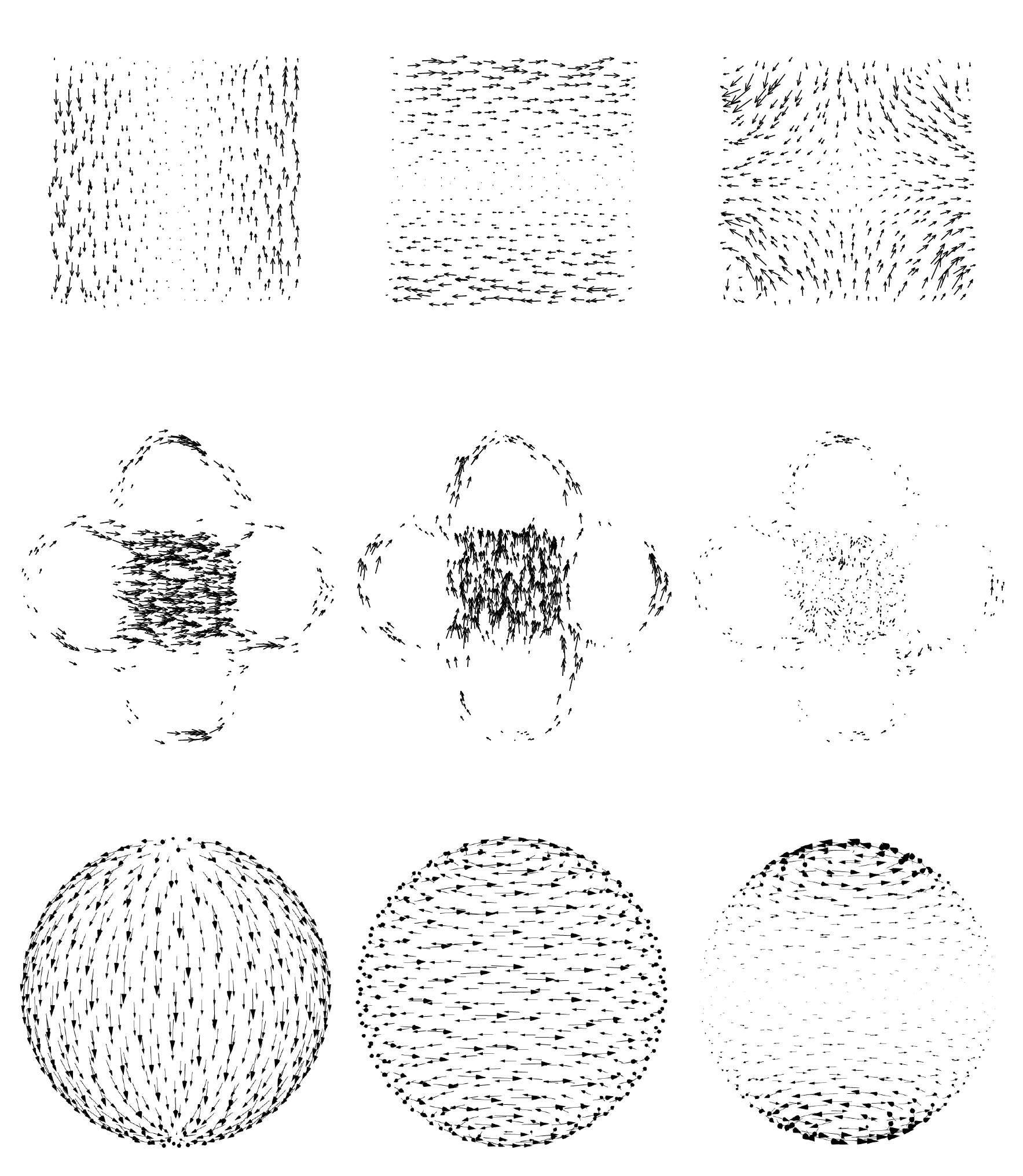}
\put(15.3,71.1){\makebox(0,0)[c]{$X=x\nabla y$}}
\put(43.5,71.1){\makebox(0,0)[c]{$Y=y\nabla x$}}
\put(71.7,71.1){\makebox(0,0)[c]{$[X, Y] \approx x\nabla x-y\nabla y$}}
\put(15.3,35.0){\makebox(0,0)[c]{$X=\nabla x$}}
\put(43.5,35.0){\makebox(0,0)[c]{$Y=\nabla y$}}
\put(71.7,35.0){\makebox(0,0)[c]{$[X, Y]$}}
\put(15.3,-1.0){\makebox(0,0)[c]{$e_\theta= \partial_\theta$}}
\put(43.5,-1.0){\makebox(0,0)[c]{$e_\phi=\frac{1}{\sin(\theta)}\partial_\phi$}}
\put(71.7,-1.0){\makebox(0,0)[c]{$[e_\theta, e_\phi] \approx -\mathrm{cot}(\theta)e_\phi$}}
\end{overpic}
\vspace{2em}
\caption{
\textbf{Lie bracket of vector fields.}
The Lie bracket measures the derivative of a vector field along the flow of another.
The first and third rows correctly recover the ground truth, which can be computed analytically.
The first row illustrates the Leibniz rule for the Lie bracket, where non-constant coefficients cause the flows of the vector fields to fail to commute.
The vector fields in the second row do commute, leading to (approximately) zero Lie bracket.
The vector fields in the third row form an \textit{orthonormal frame} on the sphere $S^2$ (meaning they are pointwise orthonormal), and their non-zero value captures the intrinsic geometry.}
\label{fig:lie-bracket}
\end{figure}

\subsection{Hessian}

If $f: \R^d \to \R$ is a smooth function, its \boldblue{Hessian} at a point $p \in \R^d$ is the square matrix of second derivatives
\[H(f)_p = 
\begin{bmatrix}
\frac{\partial^2 f}{\partial x_1 \partial x_1}(p) & \dots & \frac{\partial^2 f}{\partial x_1 \partial x_d}(p) \\
\vdots & \ddots & \vdots \\
\frac{\partial^2 f}{\partial x_d \partial x_1}(p) & \dots & \frac{\partial^2 f}{\partial x_d \partial x_d}(p) \\
\end{bmatrix},\]
which we can think of as a rank-2 tensor field (i.e. a field of matrices) on $\R^d$.
This is generalised in Riemannian and diffusion geometry to the Hessian operator $H : \A \to \Omega^1(M)^{\otimes 2}$, which maps functions to $(0,2)$-tensors.
We would like to discretise it as an $n_1d^2 \times n_0$ matrix $\textbf{H}$ that maps functions in $\R^{n_0}$ to 2-tensors in $\R^{n_1d^2}$.
We can access the Hessian using the carré du champ as
\[
H(f)(\nabla h_1, \nabla h_2) =
\frac{1}{2}
\Big(
\Gamma(h_1, \Gamma(h_2, f))
+ \Gamma(h_2, \Gamma(h_1, f))
- \Gamma(f, \Gamma(h_1, h_2))
\Big),
\]
which evaluates the action of the Hessian $H(f)$ on the pair of vector fields $\nabla h_1$ and $\nabla h_2$ as described in Section \ref{sub: action of 2 tensors}.
We will use the weak formulation $\textbf{H}^{\textnormal{weak}}$, which we compute as the inner product of $\p{i'}dx_{j'_1} \otimes dx_{j'_2}$ and $H(\p{i})$ in the 4-tensor
\begin{equation*}
\label{eq: weak Hessian}
\begin{split}
\textbf{H}^{\textnormal{weak}}_{i'j'_1j'_2i}
&= \inp{\p{i'}dx_{j'_1} \otimes dx_{j'_2}}
{H(\p{i})}\\
&= \int \p{i'} g(H(\p{i}), dx_{j'_1} \otimes dx_{j'_2}) d\mu \\
&= \int \p{i'} H(\p{i})(\nabla x_{j'_1}, \nabla x_{j'_2}) d\mu \\
&= \frac{1}{2}\int \p{i'}
\big[ \Gamma(x_{j'_1}, \Gamma(x_{j'_2}, \p{i}))
+ \Gamma(x_{j'_2}, \Gamma(x_{j'_1}, \p{i}))
- \Gamma(\p{i}, \Gamma(x_{j'_1}, x_{j'_2}))
\big] d\mu \\
&= \frac{1}{2}\sum_{p=1}^n
\textbf{U}_{pi'}
\big[ 
\boldsymbol\Gamma_p(\textbf{x}_{j'_1}, \boldsymbol\Gamma(\textbf{x}_{j'_2},\boldsymbol{\phi}_{i}))
 + \boldsymbol\Gamma_p(\textbf{x}_{j'_2}, \boldsymbol\Gamma(\textbf{x}_{j'_1},\boldsymbol{\phi}_{i}))
 - \boldsymbol\Gamma_p(\boldsymbol{\phi}_{i}, \boldsymbol\Gamma(\textbf{x}_{j'_1}, \textbf{x}_{j'_2}))
\big] 
\boldsymbol{\mu}_p,
\end{split}
\end{equation*}
which has dimension $n_1 \times d \times d \times n_0$.
We reshape $\textbf{H}^{\textnormal{weak}}$ into an $n_1d^2 \times n_0$ matrix, and recover the strong formulation $\textbf{H} : \R^{n_0} \to \R^{n_1d^2}$ as $\textbf{H} = (\textbf{G}^{(0,2)})^+ \textbf{H}^{\textnormal{weak}}$.

\begin{computationalnote}
The Hessian is a symmetric tensor field, and so we can implement an optimised version that only computes the $d(d+1)/2$ unique components of the $d \times d$ Hessian matrix at each point, and maps directly into the smaller space of symmetric 2-tensors $\Omega^{1}(\textbf{M})^{\otimes2,\ \textnormal{sym}}$ described in \compnote{note: symmetric 2 tensors}.
\end{computationalnote}

We can compute the action of the Hessian $\textbf{H} \textbf{f} \in \R^{n_1d^2}$ on any pair of vector fields $\textbf{X}, \textbf{Y} \in \R^{n_1d}$ using the method described in Section \ref{sub: action of 2 tensors}, which represents the second derivative of $\textbf{f}$ in the directions $\textbf{X}$ and $\textbf{Y}$.
We illustrate the Hessian of functions on 2d and 3d data in Figure \ref{fig:hessian}.
In the bottom row, we compute the Hessian of the $z$ coordinate function, which measures the curvature of the surface in the vertical direction.
This is explicitly connected to the curvature of the space through the second fundamental form, which we exploited to compute the various curvature tensors in \cite{jones2024manifold}.

\begin{computationalnote}
This approach to computing the Hessian as a bilinear operator $\mathfrak{X}(M) \times \mathfrak{X}(M) \rightarrow \A$ is indirect, because it passes via an expansion in the $\phi_i dx_{j_1} \otimes dx_{j_2}$ basis for the $(0,2)$-tensors.
Alternatively, we could directly implement the bilinear operator form, which we describe in the Appendix.
This latter approach may be more direct, but the indirect approach we favour here is more extrinsic and creates the opportunity for regularisation through the immersion coordinate functions.
\end{computationalnote}

One application of the Hessian operator is when we are given $k$ vector fields $X_1, \dots, X_k$ and compute the $k \times k$ matrix of second derivatives $[H(f)(X_i, X_j)]_{i,j\leq k}$.
This can be used for classifying the critical points of a function, as we discuss in Section \ref{sec: morse}, where we use this method for Morse theory.

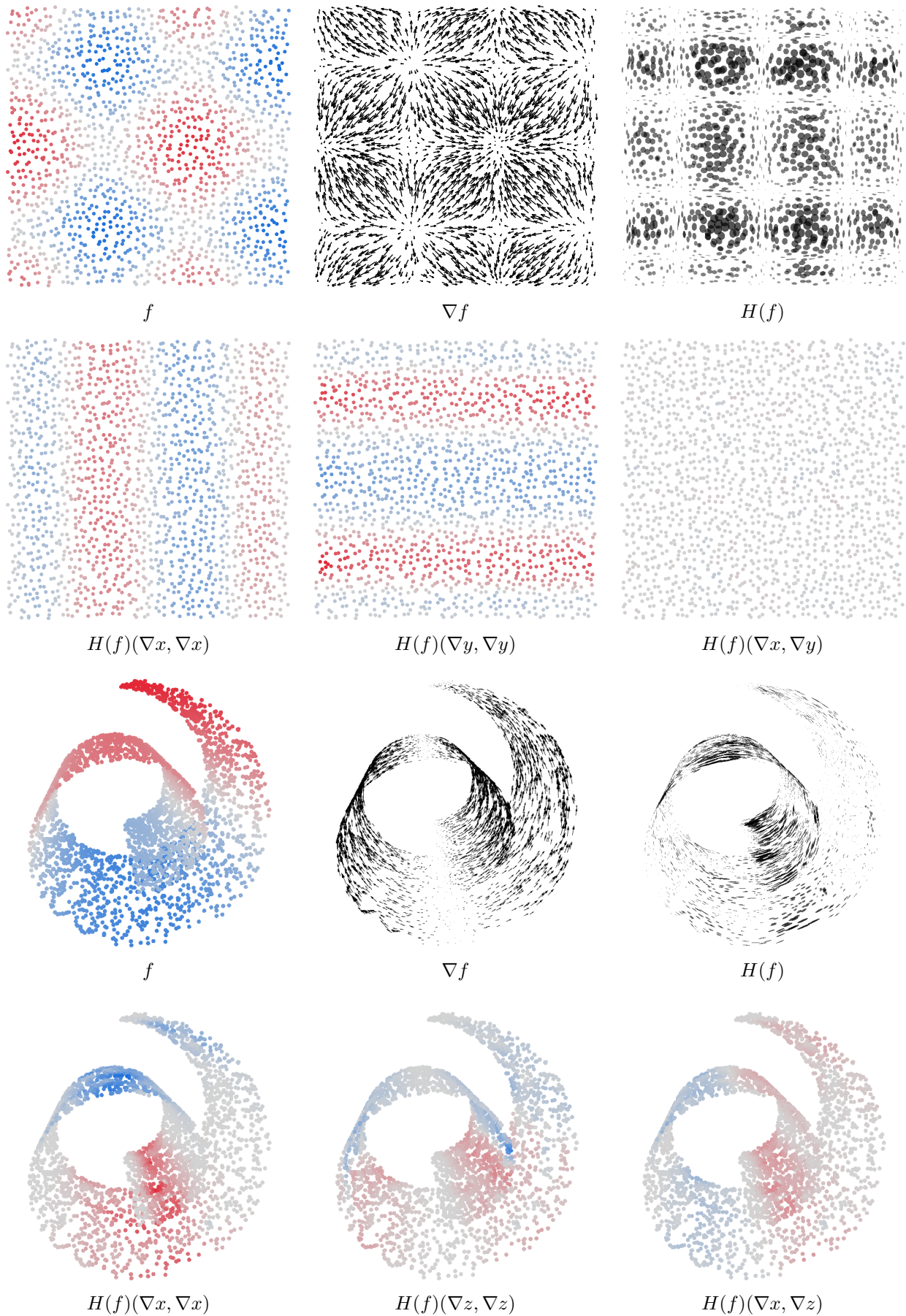
\begin{figure}
  \vspace{-1em}
  \centering
  \begin{overpic}[width=\linewidth,grid=false]{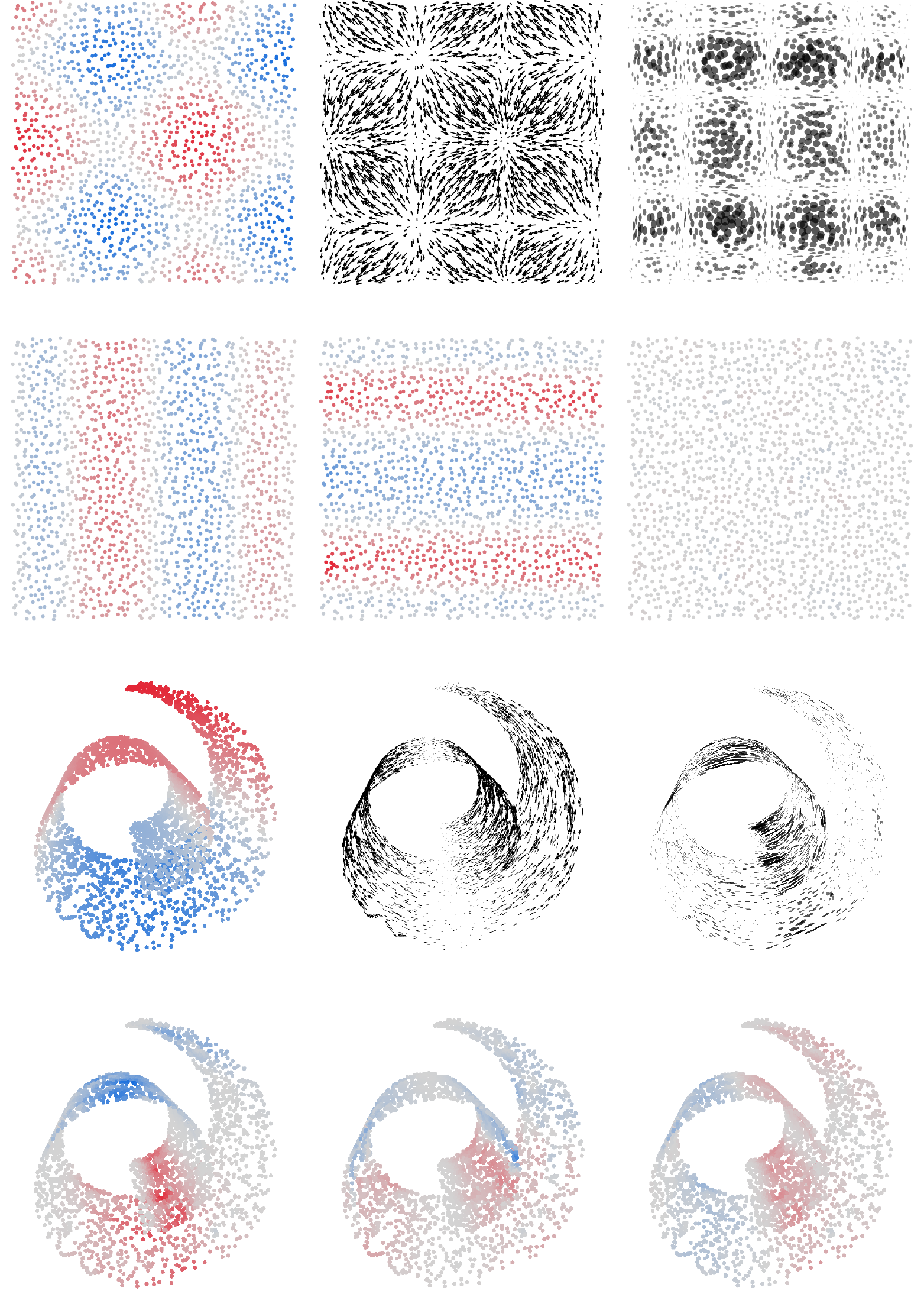}
\put(11.9,76.0){\makebox(0,0)[c]{$f$}}
\put(35.7,76.0){\makebox(0,0)[c]{$\nabla f$}}
\put(59.5,76.0){\makebox(0,0)[c]{$H(f)$}}
\put(11.9,50.0){\makebox(0,0)[c]{$H(f)(\nabla x, \nabla x)$}}
\put(35.7,50.0){\makebox(0,0)[c]{$H(f)(\nabla y, \nabla y)$}}
\put(59.5,50.0){\makebox(0,0)[c]{$H(f)(\nabla x, \nabla y)$}}
\put(11.9,24.5){\makebox(0,0)[c]{$f$}}
\put(35.7,24.5){\makebox(0,0)[c]{$\nabla f$}}
\put(59.5,24.5){\makebox(0,0)[c]{$H(f)$}}
\put(11.9,-1.5){\makebox(0,0)[c]{$H(f)(\nabla x, \nabla x)$}}
\put(35.7,-1.5){\makebox(0,0)[c]{$H(f)(\nabla z, \nabla z)$}}
\put(59.5,-1.5){\makebox(0,0)[c]{$H(f)(\nabla x, \nabla z)$}}
  \end{overpic}
  \vspace{0.4em}
  \caption{\textbf{Hessian of functions.}
  If $f \in \A$ is a function, then $H(f) \in \Omega^1(M)^{\otimes 2}$ measures the second derivatives of $f$ (its acceleration) in the direction of a pair of vector fields.
  The function on 2d data in the top two rows has no \q{diagonal} acceleration, meaning $H(f)(\nabla x, \nabla y) = 0$.
  The function on 3d data in the bottom two rows does not vary in the $y$ (forwards-backwards) direction, and so the only non-zero components of $H(f)$ are those depending on $x$ (left-right) and $z$ (up-down), which we plot here.
  }
  \label{fig:hessian}
\end{figure}

\subsection{Operator form of $(0,2)$ tensors}
\label{sub: operator form of 2 tensors}

In Section \ref{sub: action of 2 tensors} we defined the action of a $(0,2)$-tensor $\alpha \in \Omega^1(M)^{\otimes2}$ on a pair of vector fields $X, Y \in \mathfrak{X}(M)$ as the function $\alpha(X,Y) \in \A$ given by
\[
\alpha(X,Y) = g(\alpha, X^\flat \otimes Y^\flat).
\]
We can use this action to define a linear map $\alpha: \mathfrak{X}(M) \to \mathfrak{X}(M)$ via
\[
\inp{\alpha(X)}{Y} = \int \alpha(X,Y) d\mu,
\]
which we call the \boldblue{operator form} of $\alpha$, and has several applications, such as in defining the Levi-Civita connection below.
Given
\[
\boldsymbol{\alpha} = \sum_{i=1}^{n_1} \sum_{j_1,j_2=1}^d \boldsymbol{\alpha}_{ij_1j_2} \phi_i dx_{j_1} \otimes dx_{j_2} \in \Omega^1(M)^{\otimes2},
\]
we can discretise its operator form $\boldsymbol{\alpha}^{\textnormal{op}} : \R^{n_1d} \to \R^{n_1d}$ as an $n_1d \times n_1d$ matrix via its weak formulation $\boldsymbol{\alpha}^{\textnormal{op, weak}}$.
We compute the inner product of $\phi_{s_1} \nabla x_{t_1}$ and $\alpha(\phi_{s_2} \nabla x_{t_2})$ in the 4-tensor
\begin{align*}
\boldsymbol{\alpha}^{\textnormal{op, weak}}_{s_1 t_1 s_2 t_2}
&= \inp{\phi_{s_1} \nabla x_{t_1}}{\alpha(\phi_{s_2} \nabla x_{t_2})} \\
&= \inp{\alpha(\phi_{s_2} \nabla x_{t_2})}{\phi_{s_1} \nabla x_{t_1}} \\
&= \int \alpha(\phi_{s_2} \nabla x_{t_2}, \phi_{s_1} \nabla x_{t_1})\, d\mu \\
&= \sum_{i=1}^{n_1} \sum_{j_1,j_2=1}^d \boldsymbol{\alpha}_{i j_1 j_2}
\int \phi_i \phi_{s_1} \phi_{s_2}
g(dx_{j_1} \otimes dx_{j_2}, d x_{t_2} \otimes d x_{t_1})\, d\mu \\
&= \sum_{i=1}^{n_1} \sum_{j_1,j_2=1}^d \boldsymbol{\alpha}_{i j_1 j_2}
\int \phi_i \phi_{s_1} \phi_{s_2}
\Gamma(x_{j_1}, x_{t_2}) \Gamma(x_{j_2}, x_{t_1}) d\mu \\
&= \sum_{p=1}^n \sum_{i=1}^{n_1} \sum_{j_1,j_2=1}^d
\boldsymbol{\alpha}_{i j_1 j_2}
\textbf{U}_{p i} \textbf{U}_{p s_1} \textbf{U}_{p s_2}
\boldsymbol{\Gamma}_p(x_{j_1}, x_{t_2})
\boldsymbol{\Gamma}_p(x_{j_2}, x_{t_1})
\boldsymbol{\mu}_p
\end{align*}

which has dimension $n_1 \times d \times n_1 \times d$.
We reshape $\boldsymbol{\alpha}^{\textnormal{op, weak}}$ into an $n_1d \times n_1d$ matrix, and recover the strong formulation $\boldsymbol{\alpha}^{\textnormal{op}} : \R^{n_1d} \to \R^{n_1d}$ as $\boldsymbol{\alpha}^{\textnormal{op}} = (\textbf{G}^{(1)})^+ \boldsymbol{\alpha}^{\textnormal{op, weak}}$.

\subsection{Levi-Civita connection}

If $f$ is a function on a manifold, and $X \in \mathfrak{X}(M)$ is a vector field, let $\gamma: (-\epsilon,\epsilon)\to M$ be a smooth curve with $\gamma(0)=x$ and $\gamma'(0)=X(x)$.
We can define the directional derivative of $f$ along $X$ as
\[
X(f) = \nabla_X(f) := \lim_{t\to 0} \frac{f(\gamma(t))-f(\gamma(0))}{t},
\]
which depends only on the values of $f$ along the path $\gamma$, and not on its neighbouring values.
If we try to naively extend this definition to a directional derivative of a vector field $Y$, we find that
\[
``\nabla_X(Y) := \lim_{t\to 0} \frac{Y(\gamma(t))-Y(\gamma(0))}{t}"
\]
is \textit{not} well-defined. 
Indeed, $Y(\gamma(t))$ and $Y(\gamma(0))$ belong to different tangent spaces, and generally, there is no canonical way to identify tangent spaces. 
\boldblue{Covariant derivative operators} remedy this by providing a consistent way to compare tangent spaces, and there are many possible choices of covariant derivative that all encode different geometric information, such as curvature.

\begin{wrapfigure}{r}{0.4\textwidth}

\vspace{-2em}
\begin{center}
\begin{tikzpicture}[>=stealth,scale=1.2]

\coordinate (A) at (0,0);

\coordinate (xshift) at (0.5,0);
\coordinate (B) at ($(3,0)+(xshift)$);

\draw[
    thick,
    decoration={
        markings,
        mark=at position 0.52 with {\arrow[scale=1.5]{stealth}}
    },
    postaction=decorate
]
    (A)
    .. controls (1.5,0.5) and (1.5,-0.5) ..
    (B)
    node[midway,below, yshift=-0.1cm, xshift=0.15cm] {$\gamma$};

\node[below] at (A) {$\gamma(0)$};
\fill (A) circle(2pt);
\node[below] at (B) {$\gamma(t)$};
\fill (B) circle(2pt);


\coordinate (T0) at (A);

\fill[customgrey]
    ($(T0)+(-0.4,1.9)$) 
        -- ++(0.8,0.25)
        -- ++(0,-1)
        -- ++(-0.8,-0.25)
        -- cycle;

\fill[customgrey]
    ($(T0)+(2.6,1.9)+(xshift)$) 
        -- ++(0.9,0.15)
        -- ++(0,-1)
        -- ++(-0.9,-0.15)
        -- cycle;

\draw[dashed, ->, >=stealth] (-0,0.85) -- ($(A) + (0,0.15)$);
\draw[dashed, ->, >=stealth] ($(3,0.85)+(xshift)$) -- ($(B) + (0,0.15)$);

\draw[->, >=stealth] ($(T0)+(0,1.525)$) -- ($(T0)+(0,1.525)+(0.2,0.3)$);
\draw[->, >=stealth, customred] ($(T0)+(3,1.525)+(xshift)$) -- ($(T0)+(3,1.525)+(0.26,-0.22)+(xshift)$);
\draw[->, >=stealth, customred,  dash pattern=on 2pt off 2pt,] ($(T0)+(0,1.525)$) -- ($(T0)+(0,1.525)+(0.26,-0.22)$);

\node at (-0,2.4) {$T_{\gamma(0)}M$};
\node at ($(3,2.4)+(xshift)$) {$T_{\gamma(t)}M$};

\coordinate (shift) at (0.1,-0.08);

\draw[->, >=stealth, dashed] ($(T0)+(3,1.525)+(-0.2,0)+(shift)+(xshift)$) -- ($(T0)+(0,1.525)+(shift)+(0.16,0)$) node[midway, above] {parallel transport};;

\end{tikzpicture}
\end{center}
\vspace{-3em}
\end{wrapfigure}

Among all covariant derivatives, the \boldblue{Levi-Civita connection} is distinguished: it is the metric-compatible, torsion-free connection naturally associated with a Riemannian metric $g$, and the one we will compute in diffusion geometry. 
The Levi-Civita connection $\nabla_X Y$ is a vector field that measures the infinitesimal change of $Y$ in the direction of $X$ by \boldblue{parallel transporting} $Y$ back along the flow of $X$, preserving the lengths and angles defined by the metric $g$.

In diffusion geometry, the Levi-Civita connection is defined as an operator $\nabla : \mathfrak{X}(M) \to \Omega^1(M)^{\otimes 2}$ that maps vector fields to $(0,2)$-tensors.
We define $\nabla_X Y$ as the operator form of the $(0,2)$-tensor $\nabla Y$ applied to the vector field $X$ (see Section \ref{sub: operator form of 2 tensors}).
Given a vector field $X=\sum_{i,j} \phi_i\nabla x_j$, we define $\nabla X$ as
\[
\nabla X := \sum_{i,j} ( d\phi_i \otimes d x_j + \phi_i H(x_j)).
\]
We can discretise $\nabla$ as an $n_1d^2 \times n_1d$ matrix $\boldsymbol{\nabla}$ that maps vector fields in $\R^{n_1d}$ to 2-tensors in $\R^{n_1d^2}$.
We will use the weak formulation $\boldsymbol{\nabla}^{\textnormal{weak}}$, which we compute as the inner product of $\nabla(\p{i'}\nabla x_j)$ and $\p{i}dx_{j_1} \otimes dx_{j_2}$ in the 5-tensor
\begin{align*} 
\boldsymbol{\nabla}_{ij_1j_2i'j'}^{\textnormal{weak}} 
&= \langle \phi_{i'} dx_{j'_1}\otimes dx_{j'_2}, \nabla(\p{i}\nabla x_{j})\rangle \\
&= \langle \phi_{i'} dx_{j'_1}\otimes dx_{j'_2}, d\phi_{i} \otimes dx_{j} + \phi_{i} H(x_{j})\rangle \\
&= \int \phi_{i'} \Gamma(x_{j'_1},\phi_{i}) \Gamma(x_{j'_2}, x_{j}) d\mu 
+ \int \phi_{i'} \phi_{i} H(x_{j})(\nabla x_{j'_1}, \nabla x_{j'_2}) d\mu \\
&= \sum_{p=1}^n \textbf{U}_{pi'} \boldsymbol\Gamma_p(\textbf{x}_{j'_1}, \boldsymbol\phi_{i}) \boldsymbol\Gamma_p(\textbf{x}_{j'_2}, \textbf{x}_{j}) \boldsymbol{\mu}_p \\
& \qquad
+ \frac{1}{2}\sum_{p=1}^n
\textbf{U}_{pi} \textbf{U}_{pi'}
\big[ 
\boldsymbol\Gamma_p(\textbf{x}_{j'_1}, \boldsymbol\Gamma(\textbf{x}_{j'_2}, \textbf{x}_{j}))
 + \boldsymbol\Gamma_p(\textbf{x}_{j'_2}, \boldsymbol\Gamma(\textbf{x}_{j'_1}, \textbf{x}_{j}))
 - \boldsymbol\Gamma_p(\textbf{x}_{j}, \boldsymbol\Gamma(\textbf{x}_{j'_1}, \textbf{x}_{j'_2}))
\big] 
\boldsymbol{\mu}_p,
\end{align*}
using an adaptation of the weak formulation of the Hessian computed above.
This weak formulation 5-tensor has shape $n_1 \times d \times d \times n_1 \times d$, which we reshape into a $n_1d^2 \times n_1d$ matrix.
We can compute the strong formulation as $\boldsymbol{\nabla} = \big(\textbf{G}^{(0,2)}\big)^+ \boldsymbol{\nabla}^{\textnormal{weak}}$, which represents $\nabla$ as a linear map $\R^{n_1d} \to \R^{n_1d^2}$.
We can compute the operator form of the Levi-Civita connection $Y \mapsto \nabla_X Y$ using the method described in Section \ref{sub: operator form of 2 tensors}, which represents $\nabla_X$ as a linear operator $\R^{n_1d} \to \R^{n_1d}$.

We visualise the Levi-Civita connection in Figure \ref{fig:levi_civita_connection}.
To verify the computation, we compute examples on spaces where the connection can be evaluated analytically.
In flat 2d space, if $X = \nabla x$ and $Y = x \nabla y$, we expect $\nabla_X(Y) = \nabla y$, and if $\mathrm{rot}= x \nabla y - y\nabla x$, then $\nabla_{\mathrm{rot}}(\mathrm{rot}) = -\frac{1}{2}\nabla(x^2 + y^2)$.
We recover both the expected results computationally, as well as a similar calculation on a torus.

\begin{figure}[h!]
  \centering
\begin{overpic}[width=\linewidth,grid=false, yshift=2cm]{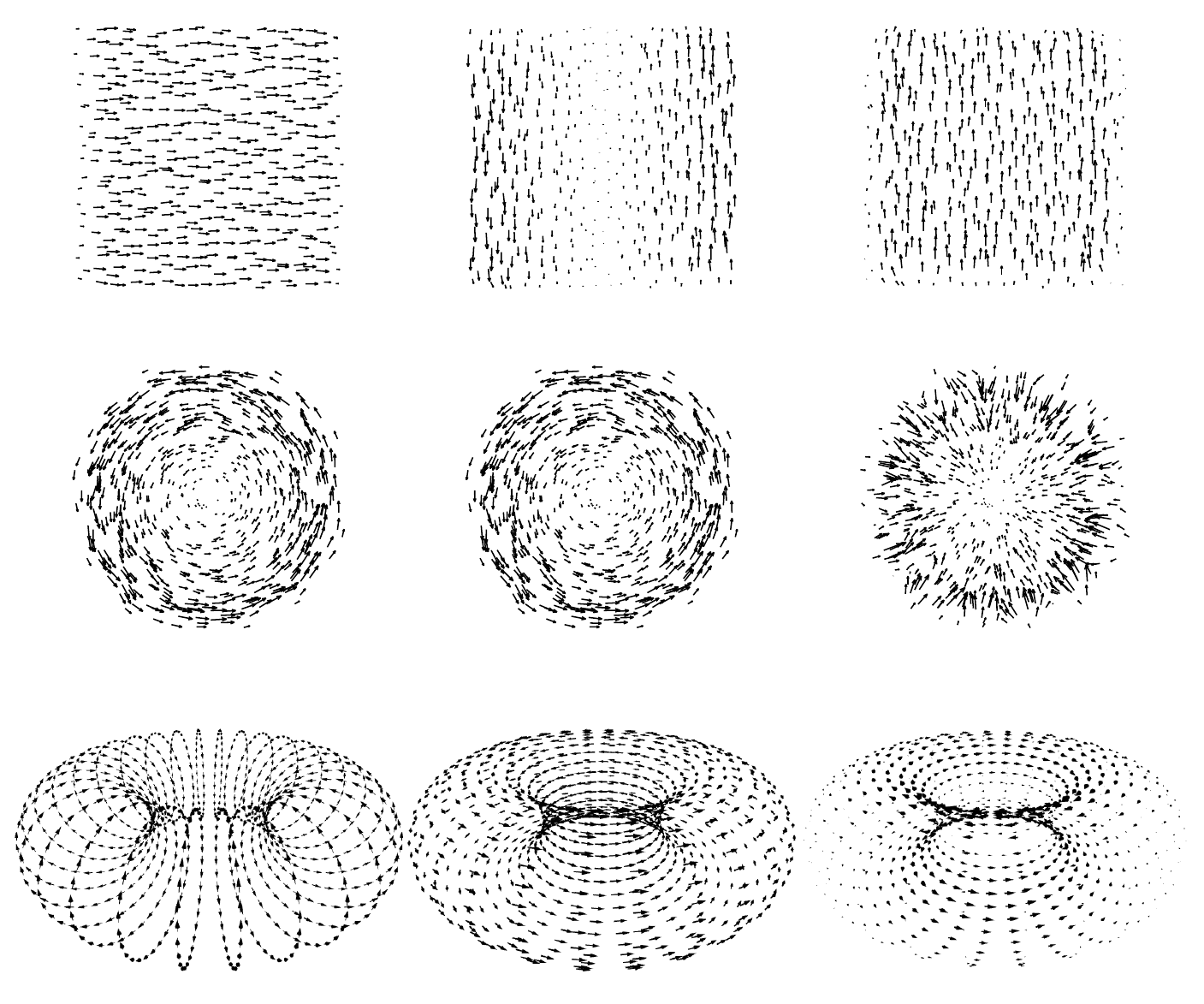}
\put(17.7,56.8){\makebox(0,0)[c]{$X_1=\nabla x$}}
\put(50.0,56.8){\makebox(0,0)[c]{$Y_1=x\nabla y$}}
\put(82.3,56.8){\makebox(0,0)[c]{$\nabla_{X_1}(Y_1)\approx\nabla y$}}
\put(17.7,27.7){\makebox(0,0)[c]{$\mathrm{rot}= -y\nabla x + x \nabla y$}}
\put(50.0,27.7){\makebox(0,0)[c]{$\mathrm{rot}$}}
\put(82.3,27.7){\makebox(0,0)[c]{$\nabla_{\mathrm{rot}}(\mathrm{rot})\approx-(x\nabla x + y\nabla y)$}}
\put(17.7,-1.5){\makebox(0,0)[c]{$e_{\theta}=\frac{1}{r}\partial_\theta$}}
\put(50.0,-1.5){\makebox(0,0)[c]{$e_{\phi}=\frac{1}{R+r\cos(\theta)}\partial_\phi$}}
\put(82.3,-1.5){\makebox(0,0)[c]{$\nabla_{e_\theta}(e_\phi)\approx\frac{-\sin(\theta)}{R+r\cos(\theta)}e_{\phi}$}}
\end{overpic}
\vspace{2em}
\caption{
\textbf{Levi-Civita connection of vector fields.}
The first row showcases the Leibniz rule for the covariant derivative. The second row illustrates that the covariant derivative of the rotational vector field with itself correctly recovers the inward-pointing vector field $-(x\nabla x + y\nabla y) = -\frac{1}{2}\nabla(x^2 + y^2)$. 
The third row highlights that in flat space, the covariant derivatives of $\nabla x$ and $\nabla y$ vanish. 
The last row computes the covariant derivative of an orthonormal frame on the torus (where $r$ and $R$ are the minor and major radius), which correctly recovers the ground truth.}
\label{fig:levi_civita_connection}
\end{figure}
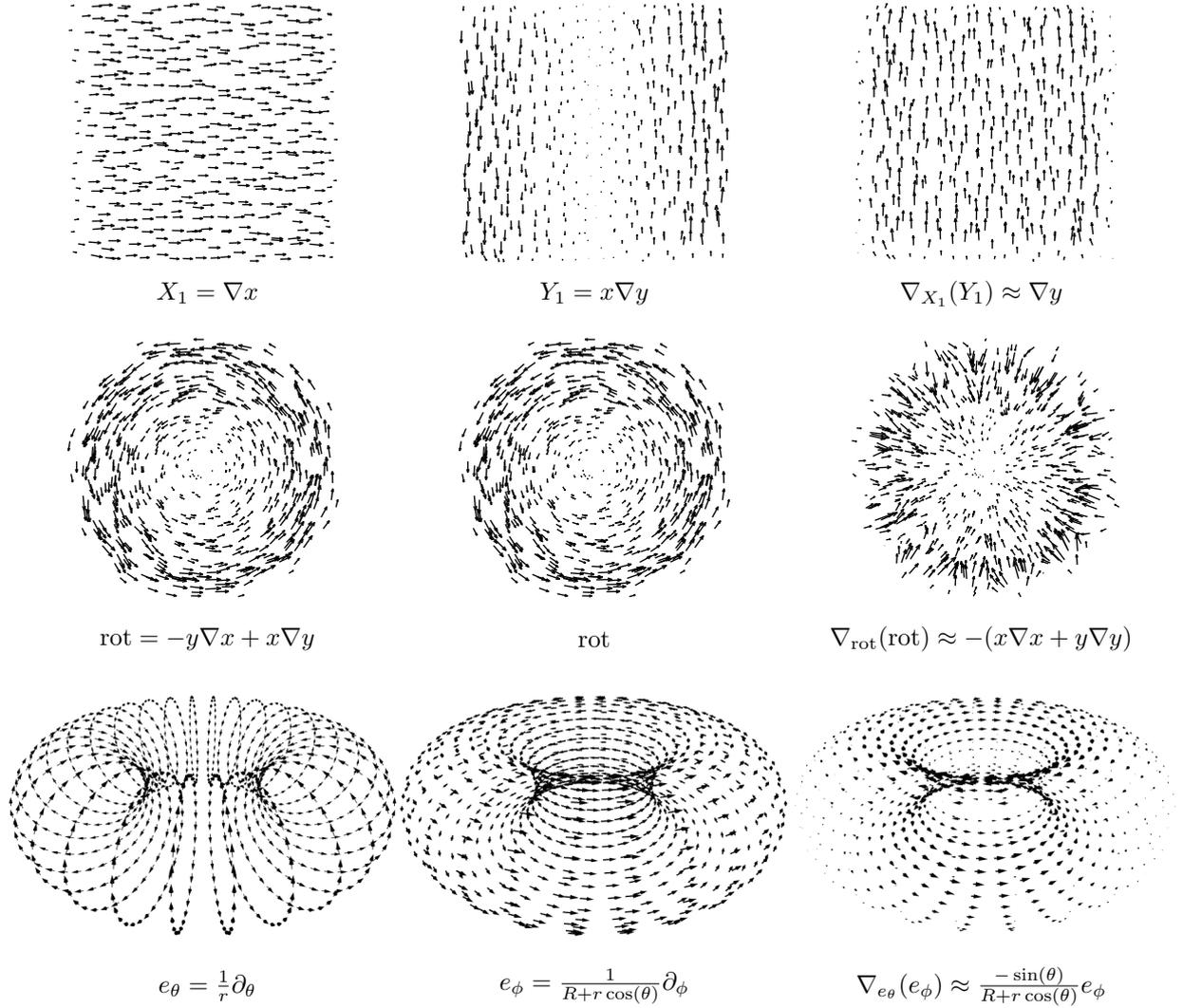

\begin{computationalnote}
As with the Hessian, we could directly implement the operator form of the Levi-Civita connection as a bilinear operator $\mathfrak{X}(M) \times \mathfrak{X}(M) \to \mathfrak{X}(M)$, which we describe in the Appendix.
In the case of the Hessian, this avoids a loss of information when computing the action (which returns a function, so it is more sensitive to high-frequency errors).
However, the operator conversion of a $(0,2)$-tensor (subsection \ref{sub: operator form of 2 tensors}) uses a weak formulation expressed as a global integral, which is less sensitive to high-frequency errors in the integrand.
This makes the difference between the two implementations harder to assess, but, as with the operator form of the Hessian, we favour the approach of expanding out the $(0,2)$-tensor and then applying Section \ref{sub: operator form of 2 tensors}, in order to better inherit the regularity properties of the immersion coordinates.
\end{computationalnote}

\subsection{Comparison between Lie bracket and Levi-Civita connection}
The Lie derivative $\mathcal{L}_X = [X, \cdot]$ and the Levi-Civita connection $\nabla_X$ both define a notion of the \q{derivative of one vector field with respect to another}, and they both generalise the directional derivative of functions $f \mapsto X(f)$ to vector fields.
However, they do so in different ways.

The Levi-Civita connection $\nabla_XY$ is a \textit{point operator} in $X$, meaning the operator $\nabla_X$ at a point $p$ only depends on the pointwise evaluation $X_p$. 
In particular, if $X'$ is a different vector field such that $X'_p=X_p$, then $(\nabla_XY)_p = (\nabla_{X'}Y)_p$ for any vector field $Y$. 
This insensitivity to the neighbouring values of $X$ makes the Levi-Civita connection the true generalisation of the directional derivative (which is also a point operator) to vector fields.

On the other hand, the Lie derivative measures the \textit{failure of commutativity of the flows}.
It therefore depends on the behaviour of the fields in a neighbourhood of the point. 
This sensitivity arises from the fact that the Lie bracket satisfies the Leibniz rule in both arguments, so its value at a point cannot be determined from pointwise data alone. 
In Figure \ref{fig:comp-lie-cov}, we compute examples of the difference.

\begin{figure}[h]
  \centering
\begin{overpic}[width=\linewidth,grid=false, yshift=0cm]{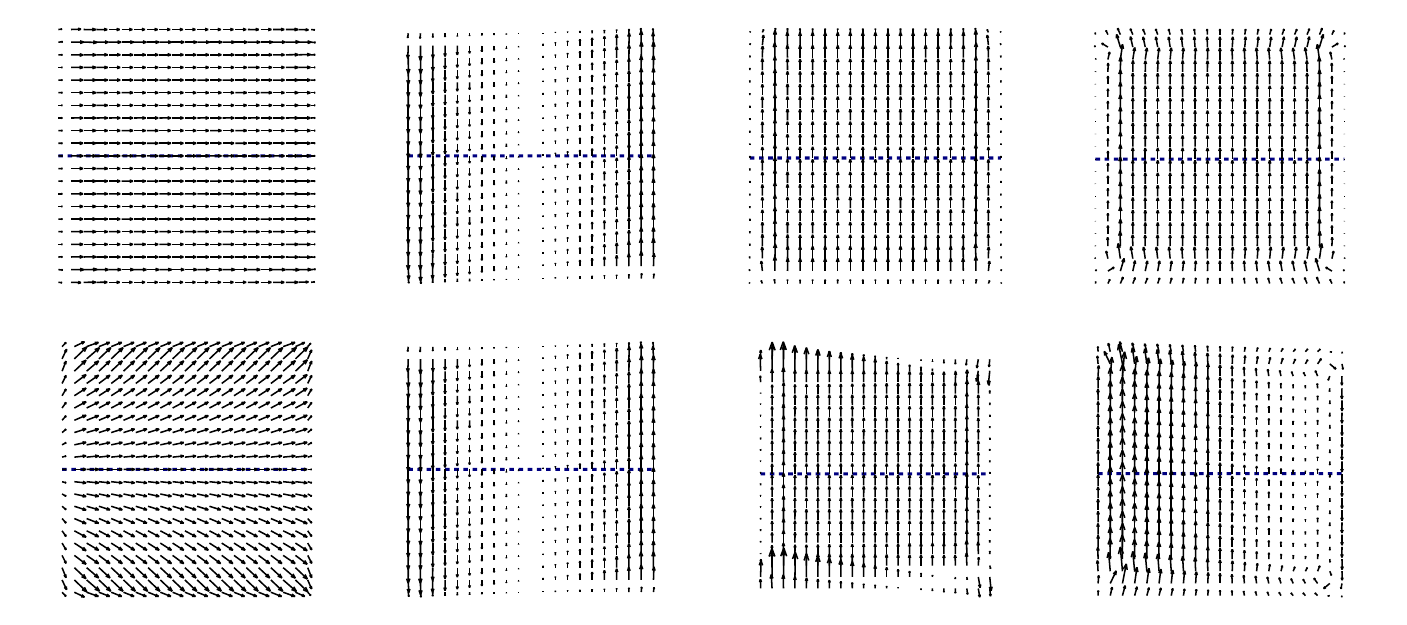}
\put(13.6,22.0){\makebox(0,0)[c]{$X_1=\nabla x$}}
\put(37.9,22.0){\makebox(0,0)[c]{$Y=x\nabla y$}}
\put(62.1,22.0){\makebox(0,0)[c]{$\nabla_{X_1}(Y)\approx\nabla y$}}
\put(86.4,22.0){\makebox(0,0)[c]{$[X_1, Y]\approx\nabla y$}}
\put(13.6,-1.0){\makebox(0,0)[c]{$X_2=\nabla x+\frac{1}{2}y\nabla y$}}
\put(37.9,-1.0){\makebox(0,0)[c]{$Y=x\nabla y$}}
\put(62.1,-1.0){\makebox(0,0)[c]{$\nabla_{X_2}(Y)\approx\nabla y$}}
\put(86.4,-1.0){\makebox(0,0)[c]{$[X_2, Y]\approx(1-\frac{1}{2}x)\nabla y$}}
\end{overpic}
\vspace{2em}
\caption{\textbf{Difference between the Levi-Civita connection and the Lie bracket.}
The Levi-Civita connection at a point depends only on the chosen direction of differentiation, whereas the Lie bracket is sensitive to how the vector fields behave in a neighbourhood of the point. The dashed blue line represents the line $y=0$ on which $X_1$ and $X_2$ agree.
In the top row, as we flow from left to right along $X_1$, the field $Y$ is increasing at a constant rate of $\nabla y$, which is observed by both the Levi-Civita connection and Lie bracket.
} \label{fig:comp-lie-cov}
\end{figure}

\section{Differential equations}
\label{sec: differential_equations}

Partial differential equations (PDEs) are ubiquitous in mathematical modelling for flows and dynamical processes.
When the domain of those equations is \textit{spatial} or otherwise \textit{geometric}, the geometry of the underlying space will affect the dynamics.
For example, an electrical wave travelling along the surface of the brain will speed up and slow down depending on the curvature of the surface at each point.
Spatial PDEs are usually solved using mesh-based methods, where the space is known in advance and explicitly triangulated, but these are hard to apply to \textit{observed} spatial data in the form of point clouds.
We can apply diffusion geometry as a \textit{mesh-free} method for solving a wide variety of spatial PDEs, by using our new computational scheme to represent differential operators.



\subsection{First-order differential equations}

We first consider differential equations that are first-order in $t$, like
\begin{equation}
\label{eq: DE first-order}
\frac{\partial u_t}{\partial t} = T(u_t)
\end{equation}
for a linear differential operator $T$, with some initial condition $u_0 = f$.
In our discretisation, functions $u_t, f \in \A$ are represented by vectors $\textbf{u}_t, \textbf{f} \in \R^{n_0}$, and linear operators $T : \A \to \A$ by $n_0 \times n_0$ matrices $\textbf{T}$.
Equation \ref{eq: DE first-order} now becomes
\begin{equation}
\label{eq: DE first-order matrix}
\frac{\partial \textbf{u}_t}{\partial t} = \textbf{T}\textbf{u}_t,
\end{equation}
with $\textbf{u}_0 = \textbf{f}$.
If we diagonalise $\textbf{T} = \textbf{Q}\inv \diag(\textbf{W}) \textbf{Q}$, then the solution to (\ref{eq: DE first-order matrix}) is given by
$$
\textbf{u}_t = \exp(t\textbf{T})\textbf{f}
$$
where $\exp(t\textbf{T}) = \textbf{Q}\inv \diag(\exp(t\textbf{W})) \textbf{Q}$.

\subsubsection{Heat equation}
We can solve the heat equation
\[
\frac{\partial\textbf{u}_t}{\partial t} = - \boldsymbol\Delta \textbf{u}_t,
\]
with initial condition $\textbf{u}_0 = \textbf{f}$, by setting $\textbf{T} = - \boldsymbol\Delta$.
We visualise these solutions in Figures \ref{fig:heat} and \ref{fig:heat_wave_tissue}.

\begin{figure}[h!]
  \centering
  \begin{overpic}[width=\linewidth,grid=false]{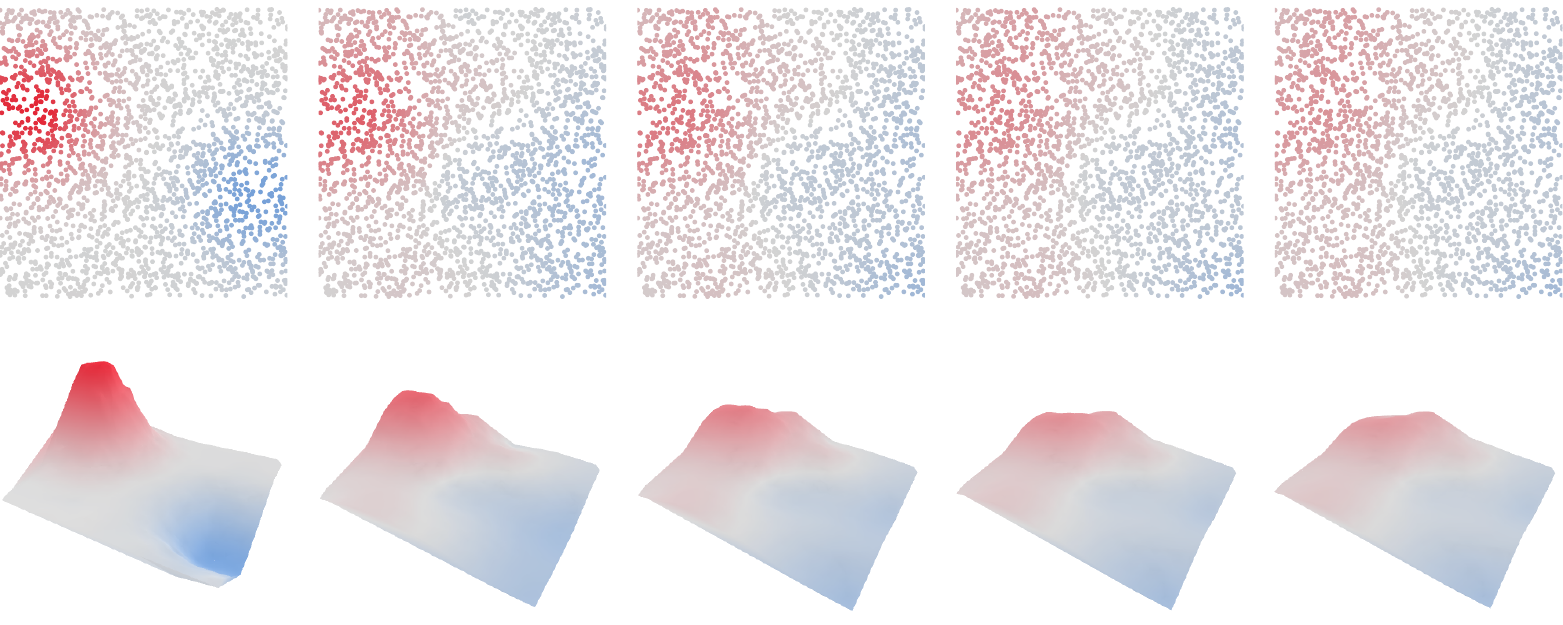}
\put(9.2,-1.2){\makebox(0,0)[c]{$t = 0.00$}}
\put(29.6,-1.2){\makebox(0,0)[c]{$t = 0.05$}}
\put(50.0,-1.2){\makebox(0,0)[c]{$t = 0.10$}}
\put(70.4,-1.2){\makebox(0,0)[c]{$t = 0.15$}}
\put(90.8,-1.2){\makebox(0,0)[c]{$t = 0.20$}}
  \end{overpic}
  \vspace{0.5em}
  \caption{\textbf{Heat equation.}
  If $u_0 \in \A$ is a function, we can solve the heat equation $\dot{u}_t = -\Delta u_t$ for $u_t \in \A$.
  We plot the evolution of $u_t$ on the data in the top row, and show the graph of $u_t$ in the bottom row.
  }
  \label{fig:heat}
\end{figure}

\subsubsection{Vector field flows}
\label{sub: pde: vector field flows}

If $\textbf{X} \in \R^{n_1d}$ is a vector field, then $\textbf{X}^{\text{op}} \in \R^{n_0 \times n_0}$ is the square matrix that describes the action of $\textbf{X}$ on functions $\textbf{f} \in \R^{n_0}$ (i.e. $\textbf{X}(\textbf{f}) = \textbf{X}^{\text{op}}\textbf{f}$, see \ref{sub: directional derivatives}).
The equation
\begin{equation}
\label{eq: vector field flow}
\frac{\partial\textbf{u}_t}{\partial t} = \textbf{X}(\textbf{u}_t)
\end{equation}
with initial condition $\textbf{u}_0 = \textbf{f}$ describes the \q{flow} of $\textbf{f}$ along $\textbf{X}$, and is solved by setting $\textbf{T} = \textbf{X}^{\text{op}}$.

On smooth (not necessarily Riemannian) manifolds, a vector field $X$ induces a one-parameter group of diffeomorphisms $\Phi_t$, which flow points along the integral curves of $X$.
In this manifold setting, equation (\ref{eq: vector field flow}) describes the pushforward of a function along $\Phi_t$.

\begin{relatedwork}
The idea of using the operator form of a vector field to define flows was introduced on meshes in \cite{azencot2013operator}.
It has been applied to point cloud data as a form of non-parametric dynamical systems forecasting \cite{berry2015nonparametric, berry2016forecasting,giannakis2019data}.
\end{relatedwork}

\subsection{Second-order differential equations}

We can also solve differential equations that are second-order in $t$, like
\begin{equation}
\label{eq: DE second-order}
\frac{\partial^2 u_t}{\partial t^2} = T(u_t) + S\left( \frac{\partial u_t}{\partial t}\right)
\end{equation}
for linear differential operators $T, S$, with the pair of initial conditions $u_0 = f$ and $\frac{\partial u_t}{\partial t}|_{t=0} = h$.
In our discretisation, this equation becomes
\begin{equation}
\label{eq: DE second-order matrix}
\frac{\partial^2 \textbf{u}_t}{\partial t^2} = \textbf{T}\textbf{u}_t + \textbf{S}\frac{\partial \textbf{u}_t}{\partial t},
\end{equation}
with $\textbf{u}_0 = \textbf{f}$ and $\frac{\partial \textbf{u}_t}{\partial t}|_{t=0} = \textbf{h}$, and $n_0 \times n_0$ matrices $\textbf{T}$ and $\textbf{S}$.
We can apply a standard trick to turn this second-order equation into a first-order one by defining $\textbf{v}_t \in \R^{2n_0}$ as
$$
\textbf{v}_t = 
\begin{bmatrix}
\textbf{u}_t \\
\frac{\partial \textbf{u}_t}{\partial t}
\end{bmatrix}
$$
so now
$$
\frac{\partial \textbf{v}_t}{\partial t}
=
\begin{bmatrix}
\frac{\partial \textbf{u}_t}{\partial t} \\
\textbf{T}\textbf{u}_t + \textbf{S}\frac{\partial \textbf{u}_t}{\partial t}
\end{bmatrix}
=
\begin{bmatrix}
0 & \textbf{I}_{n_0} \\
\textbf{T} & \textbf{S}
\end{bmatrix}
\textbf{v}_t.
$$
We can then apply the first-order method described above to the $2n_0 \times 2n_0$ block matrix
$$
\begin{bmatrix}
0 & \textbf{I}_{n_0} \\
\textbf{T} & \textbf{S}
\end{bmatrix}
$$
with initial condition $\textbf{v}_0 = (\textbf{f},\ \textbf{h}) \in \R^{2n_0}$ to obtain a solution $\textbf{v}_t \in \R^{2n_0}$, and then project onto the first $n_0$ coordinates to find $\textbf{u}_t$.

\subsubsection{Wave equation}
We can solve the wave equation
$$
\frac{\partial^2 \textbf{u}_t}{\partial t^2} = -\boldsymbol\Delta \textbf{u}_t
$$
by setting $\textbf{T} = - \boldsymbol\Delta$ and $\textbf{S} = 0$.
The \textit{damped} wave equation introduces a friction coefficient $\gamma > 0$, giving
$$
\frac{\partial^2 \textbf{u}_t}{\partial t^2} = -\boldsymbol\Delta \textbf{u}_t - \gamma \frac{\partial \textbf{u}_t}{\partial t},
$$
which we can solve with $\textbf{T} = - \boldsymbol\Delta$ and $\textbf{S} = - \gamma \textbf{I}_{n_0}$.
We visualise these solutions in Figures \ref{fig:wave} and \ref{fig:heat_wave_tissue}.

\begin{figure}[h!]
  \centering
  \begin{overpic}[width=\linewidth,grid=false]{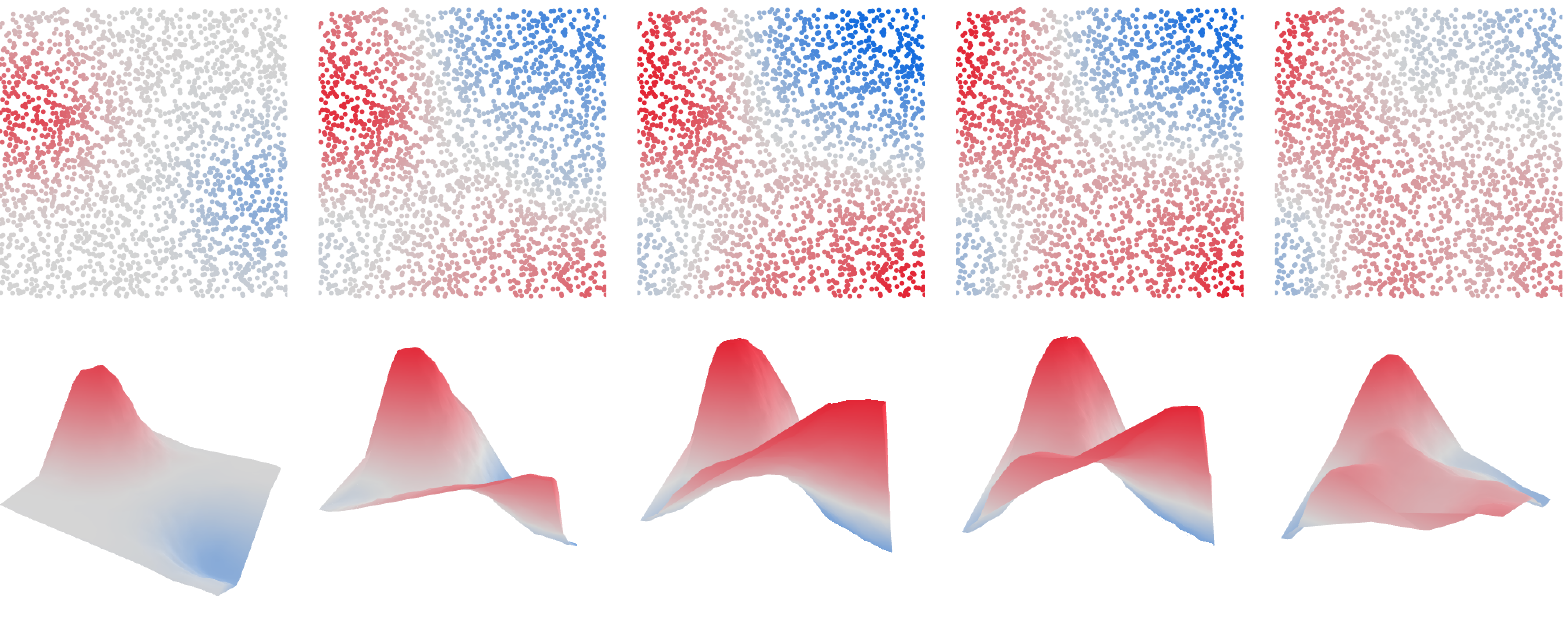}
\put(9.2,-1.2){\makebox(0,0)[c]{$t = 0.00$}}
\put(29.6,-1.2){\makebox(0,0)[c]{$t = 0.05$}}
\put(50.0,-1.2){\makebox(0,0)[c]{$t = 0.10$}}
\put(70.4,-1.2){\makebox(0,0)[c]{$t = 0.15$}}
\put(90.8,-1.2){\makebox(0,0)[c]{$t = 0.20$}}
  \end{overpic}
  \vspace{0.5em}
  \caption{\textbf{Wave equation.}
  If $u_0 \in \A$ is a function and $\dot{u}_0 = 0$, we can solve the heat equation $\ddot{u}_t = -\Delta u_t$ for $u_t \in \A$.
  We plot the evolution of $u_t$ on the data in the top row, and show the graph of $u_t$ in the bottom row.
  }
  \label{fig:wave}
\end{figure}

\begin{figure}[]
  \centering
  \begin{overpic}[width=\linewidth,grid=false]{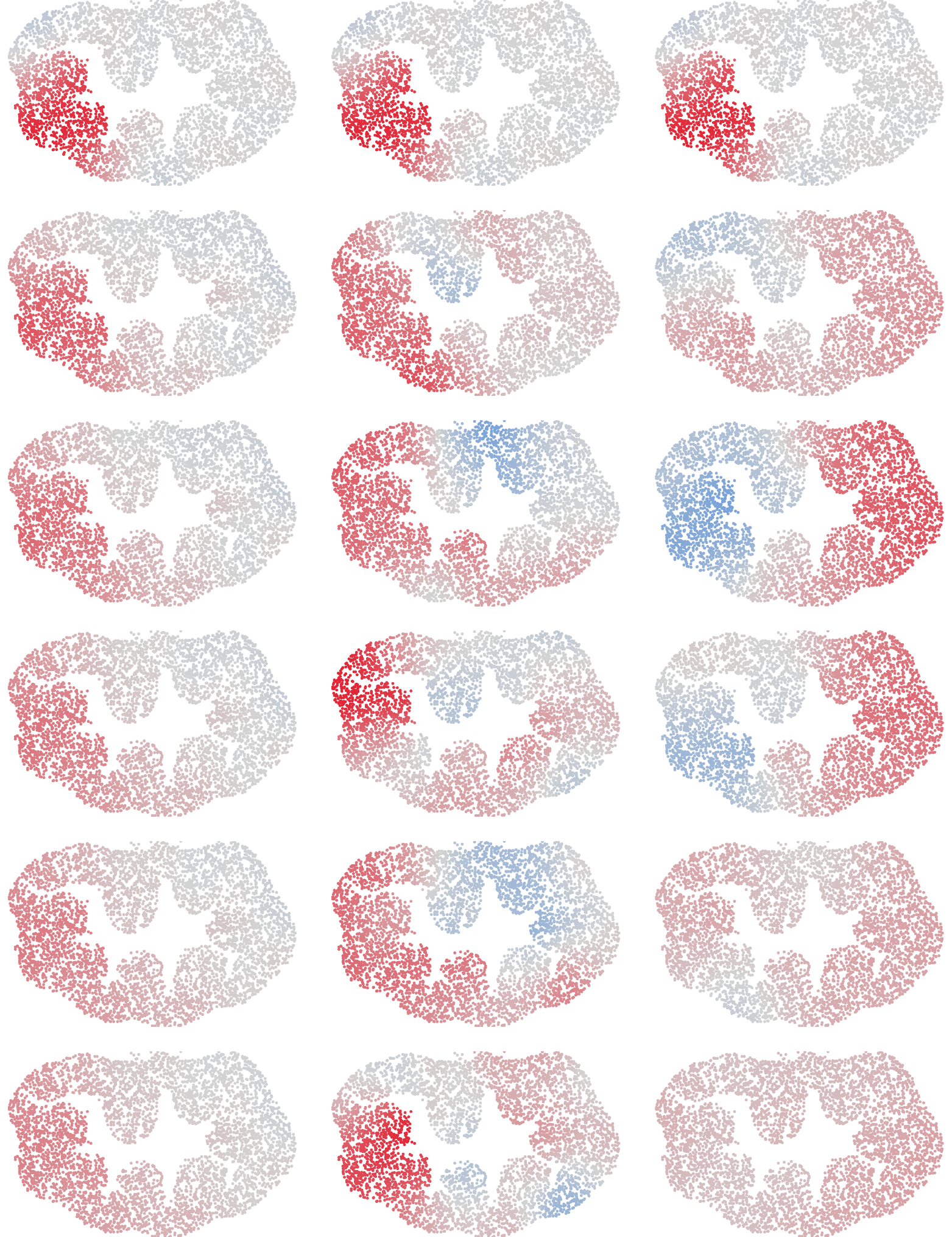}
\put(12.3,102.5){\makebox(0,0)[c]{$\frac{\partial u_t}{\partial t} = - \Delta u_t$}}
\put(38.5,102.5){\makebox(0,0)[c]{$\frac{\partial^2 u_t}{\partial t^2} = - \Delta u_t$}}
\put(64.6,102.5){\makebox(0,0)[c]{$\frac{\partial^2 u_t}{\partial t^2} = - \Delta u_t - 0.1 \frac{\partial u_t}{\partial t}$}}
  \end{overpic}
  \vspace{1em}
  \caption{\textbf{Heat, wave, and damped wave equations.}
  The data is a 2d histology image of a slice of colon tissue.
  Starting from an initial function $u_0\in\A$, we plot the evolution of the signal $u_t$ under the heat, wave, and damped wave equations.
  The heat equation (left column) models the diffusion of a chemical throughout the tissue.
  In the wave equations (which are second order), we use the initial speed $\dot{u}_t = 0$.
  In the damped wave equation (right column), the small friction coefficient $0.1$ has a smoothing effect that dissipates the wave over time.
  }
  \label{fig:heat_wave_tissue}
\end{figure}

\subsection{Integral curves}\label{sec:int-curv}

If $X \in \mathfrak{X}(M)$ is a vector field, then it induces a \textit{flow} defined by the equation
\begin{equation}
\label{eq: integral curve local}
\frac{d p_t}{d t} = X(p_t),
\end{equation}
where $p_t \in M$ is the position of a point at time $t$. 
The solution to this ordinary differential equation (ODE) for an initial point $p_0$ is a path $p_t$, called an \textit{integral curve}, that traces out the vector field from that starting point.
We can compute these curves for all points in $M$ simultaneously by solving the PDEs
\begin{equation}
\label{eq: integral curve global}
\frac{d x_{i,t}}{d t} = X(x_{i,t}),
\end{equation}
where $x_{i,t}$ denotes the $i\thupper$ coordinate function at time $t$, for each $i$.
This then reduces to solving the vector field flow of $X$ using the method described above in \ref{sub: pde: vector field flows}.
Evolving the coordinate functions forward and backwards in $t$ allows us to visualise the one-parameter group of diffeomorphisms induced by $X$.
We visualise these integral curves in Figure \ref{fig:integral_curve}.

\begin{figure}[t]
\centering
\begin{overpic}[width=0.8\linewidth]{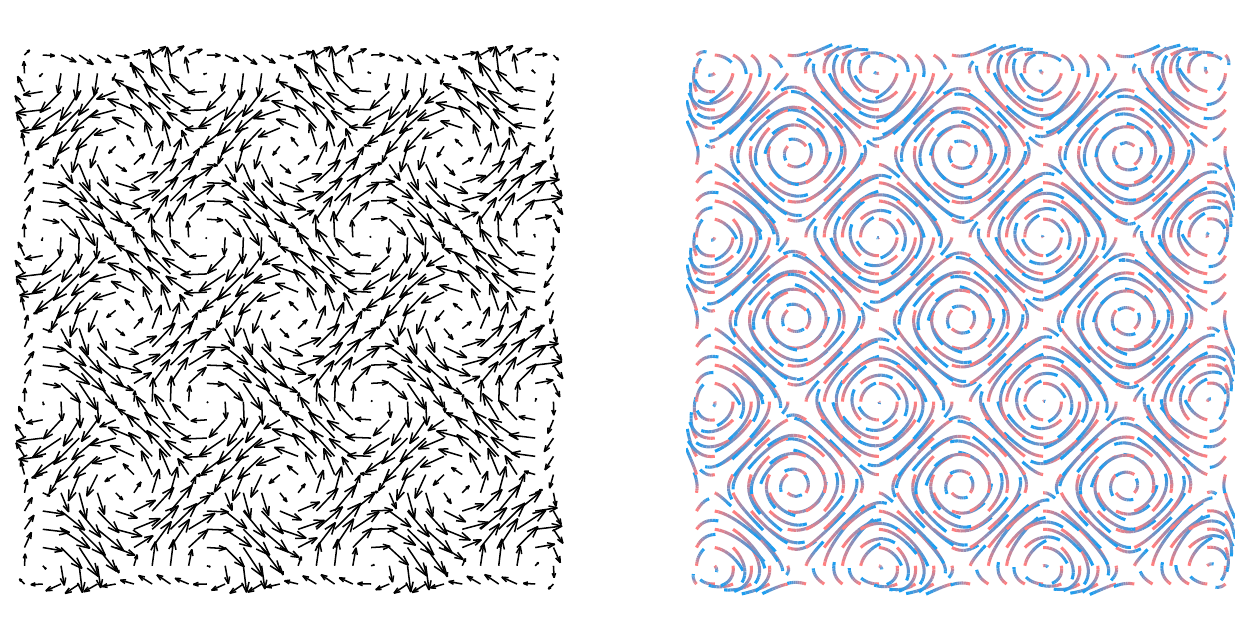}
\end{overpic}
\caption{
\textbf{Flow lines of a vector field.}
The colour represents flow time.}
\label{fig:integral_curve}
\end{figure}

\begin{computationalnote}
Changing the local equation (\ref{eq: integral curve local}) into the global one (\ref{eq: integral curve global}) is essential for this method to work.
It would be possible to directly solve equation (\ref{eq: integral curve local}) by representing the point $p_i$ as the Dirac delta function $e_i$, but this would be extremely unstable and inaccurate.
Recall that the error of the Laplacian estimation is bounded by the Dirichlet energy of the function it is operating on, and a delta function has \q{infinite energy}.
The global approach of solving for all points simultaneously (\ref{eq: integral curve global}), while superficially less efficient, has a regularising effect from the smoother coordinate functions $x_i$, which have much lower energy.
\end{computationalnote}






\section{Geometric data analysis}
\label{sec: geometric data analysis}

\subsection{Geodesic distances}

We can apply these tool to approximate the function of geodesic distances from a point $p$, $f_p := d(p,\cdot)$.
Notice that $f_p$ is 1-Lipschitz, meaning that $\Gamma(f_p,f_p) \leq 1$ $\mu$-almost everywhere, and that $f_p(p) = 0$.
Among all the functions that meet these criteria, $f_x$ has the largest \textit{mean value} $\int f_x d\mu$.
We can therefore characterise the function of geodesic distances from a point $x$ as
\[
d(p,\cdot) = \text{argmax}\{f \mapsto \int f d\mu : f(p) = 0 \textnormal{ and } \Gamma(f,f) \leq 1 \ \mu\textnormal{-a.e.} \}.
\]
In our discretisation, we represent functions in the eigenfunction basis as vectors $\textbf{f} \in \R^{n_0}$, and, in a general choice of function basis, we here suppose that the first function in the basis $\phi_0 = 1$.
The mean value then satisfies
$$
\int f d\mu = \inp{f}{1} = \inp{f}{\phi_0} = \textbf{f}_0
$$
and that, if $f = \sum_{i=1}^{n_0} \textbf{f}_i \phi_i$, then
$$
f(p_i) = \sum_{i=j}^{n_0} \textbf{f}_j \phi_j(p_i) = \sum_{j=1}^{n_0} \textbf{f}_j \textbf{U}_{ij} = \textbf{U}_i\textbf{f},
$$
where $\textbf{U}_i$ is the $1 \times n_0$ row vector representing \q{evaluation at $p_i$}.
The 1-Lipschitz condition corresponds to $\Gamma(f,f)(p_j) \leq 1$ for all $j = 1,...,n$, and so is discretised as the collection of $n$ quadratic constraints $\textbf{f}^T \boldsymbol\Gamma_j \textbf{f} \leq 1$ for $j=1,...,n$, where $\boldsymbol\Gamma_j$ is the $n_0 \times n_0$ matrix representing the carré du champ at $p_j$.
We can therefore approximate $d(x_p,\cdot)$ in our eigenfunction basis as the solution to the convex optimisation problem
\begin{equation}
\label{eq: geodesic distance optimisation}
\begin{split}
\textnormal{argmax}\ & \textbf{f}_0 \\
\textnormal{subject to}\ & \textbf{U}_p\textbf{f} = 0 \\
&\textbf{f}^T \boldsymbol\Gamma_j \textbf{f} \leq 1 \textnormal{ for } j=1,...,n,
\end{split}
\end{equation}
which can be efficiently solved with standard conic programming methods.

\begin{computationalnote}
Our compressed function space $\A$ contains only smooth, low-frequency functions, but the geodesic distance function $d(p,\cdot)$ is not smooth at $p$ (it looks like a cone with a tip at $p$).
This means that directly approximating $d(p,\cdot)$ in $\A$ as stated in \eqref{eq: geodesic distance optimisation} will lead to errors near $p$.
Instead, we model $d(p,q) = \|q-p\| + f$, where $\|q-p\|$ is the distance in the ambient space $\R^d$ and $f \in \A$ is a non-negative correction.
\end{computationalnote}

\begin{figure}[h!]
  \centering
  \begin{overpic}[width=\linewidth,grid=false]{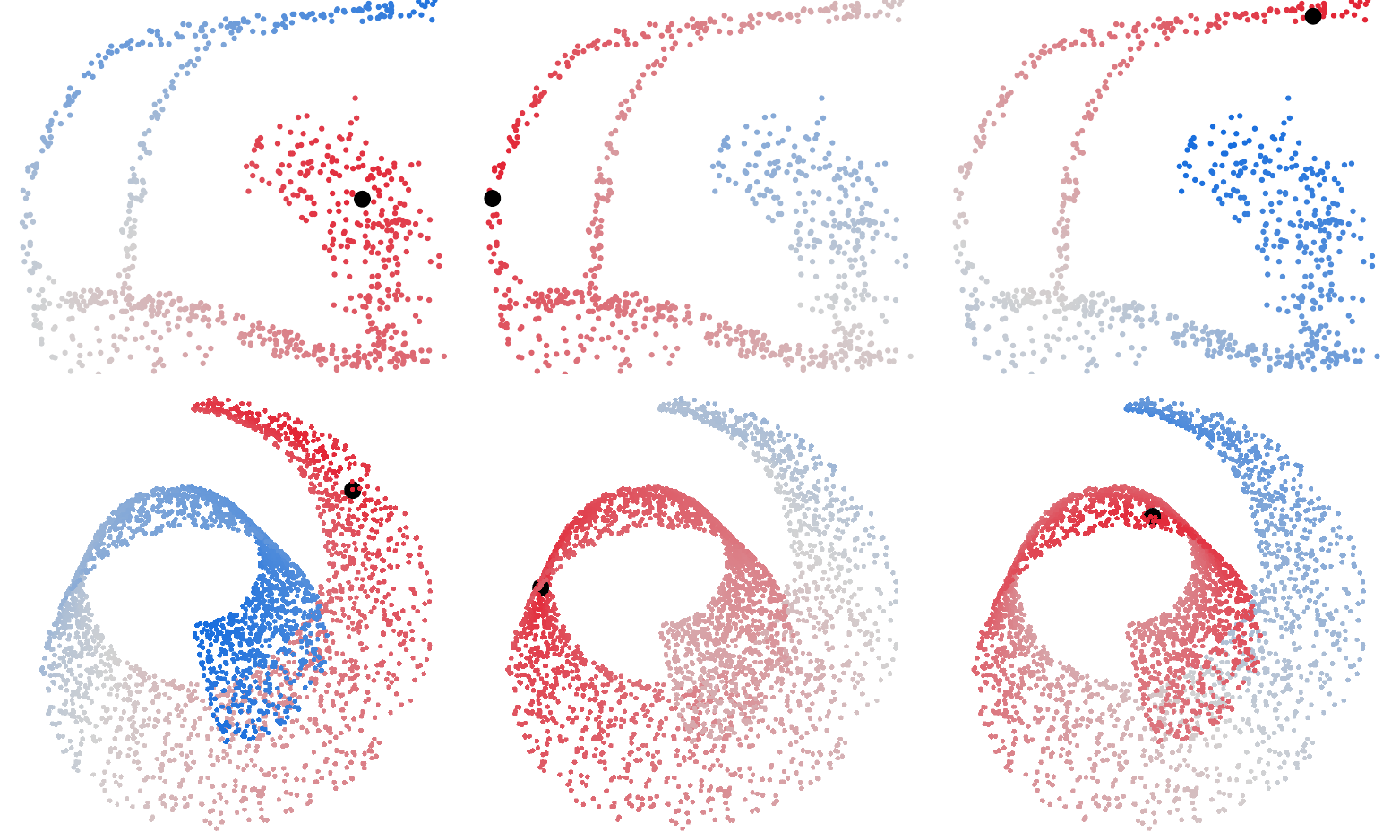}
  \end{overpic}
  \vspace{0.5em}
  \caption{\textbf{Geodesic distances.}
  If $u_0 \in \A$ is a function, we can solve the heat equation $\dot{u}_t = -\Delta u_t$ for $u_t \in \A$.
  We plot the evolution of $u_t$ on the data in the top row, and show the graph of $u_t$ in the bottom row.
  }
  \label{fig:geodesics}
\end{figure}


\begin{relatedwork}
This definition of geodesic distance functions via Lipschitz constraints is inspired by the Bakry-Émery characterisation \cite{bakry2014analysis}.
It is dual to the notion of geodesic distances as minimisers of path integrals.
\end{relatedwork}




\subsection{Curvature}

We can measure the curvature of a space with the \boldblue{Riemann curvature tensor}, which is the map
\begin{align*}
R(X,Y,Z,W)
&= g( \nabla_X \nabla_Y Z, W) - g( \nabla_Y \nabla_X Z, W) - g( \nabla_{[X,Y]} Z, W) \\
&= \nabla(\nabla_Y Z)(X, W) - \nabla(\nabla_X Z)(Y, W) - \nabla(Z)([X,Y], W),
\end{align*}
where the second line directly invokes the action of the Levi-Civita connection as a $(0,2)$-tensor.
Given a pair of vector fields $X,Y \in \mathfrak{X}(M)$, the \boldblue{sectional curvature} $K(X,Y)$ measures the curvature of the 2-dimensional surface spanned by $X$ and $Y$, defined by
\[
K(X,Y) = \frac{R(X,Y,X,Y)}{\|X\|^2\|Y\|^2 - g(X,Y)^2}.
\]
It is a symmetric functional $K : \mathfrak{X}(M) \times \mathfrak{X}(M) \to \A$.
We illustrate the sectional curvature in Figure \ref{fig:sectional_curvature}.

\begin{figure}[ht]
    \centering
    \includegraphics[width=\linewidth]{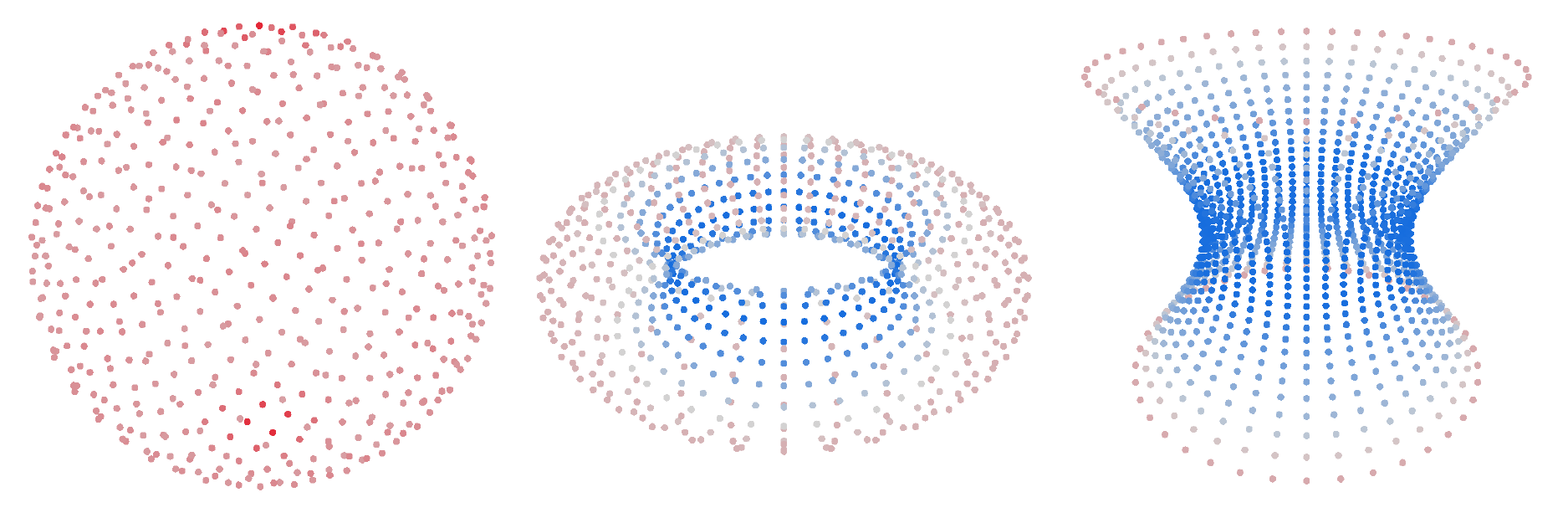}
    \caption{\textbf{Sectional curvature} of a sphere, torus, and hyperboloid.}
    \label{fig:sectional_curvature}
\end{figure}

\begin{relatedwork}
Even for data on manifolds, curvature estimation is quite underexplored.
In \cite{jones2024manifold}, we computed the scalar, Ricci, and Riemann curvatures for manifold data using the Hessian method described above.
This led to state-of-the-art results for curvature estimation on point clouds, due to the noise robustness of diffusion geometry.
This explicitly used the manifold assumption by working with the second fundamental form, an approach also used in \cite{sritharan2021computing} to compute the scalar curvature via local quadratic regression.
The scalar curvature was also computed in \cite{hickok2023intrinsic} via an intrinsic method based on volume comparison.
The use of intrinsic volume estimation makes this suitable for very high ambient dimensions, but less robust and sample efficient.
The sectional curvature presented here is, to our knowledge, the first attempt to compute any of the classical curvature tensors without the manifold assumption.
\end{relatedwork}

\section{Topological data analysis}
\label{sec: TDA}

We now introduce several tools for computing \boldblue{topological features} of data.
The topology of a space is a more abstract notion than geometry, and encodes the aspects of the space that are preserved under continuous deformations like stretching or bending.
For example, a sphere and a cube are topologically equivalent, because one can be continuously deformed into the other, even though they are geometrically distinct.

\subsection{de Rham cohomology via Hodge theory}\label{sec:deRahm}

A principal tool for measuring the topology of a space $M$ is \boldblue{cohomology}, which counts the number of \q{holes} of different dimensions.
The zeroth cohomology group $\mathcal{H}^0(M,\R)$ represents the connected components of $M$, the first cohomology group $\mathcal{H}^1(M,\R)$ represents the holes or loops, and the second cohomology group $\mathcal{H}^2(M,\R)$ represents the enclosed spaces or voids inside $M$.
The sequence continues with higher-dimensional \q{generalised holes} represented by $\mathcal{H}^k(M,\R)$ for all $k$, and, together, represents a kind of coarse-grained topological summary of the space $M$.

Formally, the \boldblue{de Rham cohomology} groups are defined to be
\[
\mathcal{H}^k(M,\R) = \frac{\ker(d^{(k)})}{\text{im}(d^{(k-1)})},
\]
where $d^{(k)} : \Omega^k(M) \to \Omega^{k+1}(M)$ is the exterior derivative on $k$-forms.
If $\alpha \in \ker(d^{(k)})$, then it is called a \boldblue{closed} $k$-form, and if $\alpha \in \text{im}(d^{(k-1)})$, then it is called an \boldblue{exact} $k$-form.
Since the carré du champ gives an inner product on forms, we can identify elements of $\mathcal{H}^k(M,\R)$, which are cosets like $\alpha + \text{im}(d^{(k-1)})$, with their unique \textit{minimal element}\footnote{the \textit{uniqueness} of the minimal element comes from the fact that we have an inner product and so the norm is strictly convex, and we also assume that the derivative is a closed linear operator so $\text{im}(d^{(k-1)})$ is a closed subspace}.
This minimal element will be orthogonal to $\text{im}(d^{(k-1)})$, or, equivalently, it will be in the kernel of the adjoint of $d^{(k-1)}$, which is the codifferential $\partial^{(k)} : \Omega^k(M) \to \Omega^{k-1}(M)$.
We will call such an element a \boldblue{coclosed} $k$-form, and forms that are both closed and coclosed are called \boldblue{harmonic}.
We can then identify
\[
\mathcal{H}^k(M,\R) \cong \ker(d^{(k)}) \cap \ker(\partial^{(k)})
= \ker(d^{(k-1)}\partial^{(k)} + \partial^{(k+1)}d^{(k)})
= \ker(\Delta^{(k)}),
\]
where $\Delta^{(k)}$ is the \boldblue{Hodge Laplacian}.
This statement is the \boldblue{Hodge theorem}, and tells us that each \q{generalised hole} in $\mathcal{H}^k(M,\R)$ corresponds to a unique harmonic form in $\Omega^{k}(M)$, and these can be found by computing the kernel of the Hodge Laplacian.
Although there is not yet any formal guarantee that this result extends to more general settings than manifolds, we conjecture that such a connection does exist and will use it to motivate statistics for cohomology based on Hodge theory.

In Section \ref{sub: hodge laplacian} we described a method for discretising the Hodge Laplacian $\Delta^{(k)}$ as a matrix $\boldsymbol\Delta^{(k)}$ that acts on $k$-forms in $\R^{n_1\binom{d}{k}}$.
The Hodge Laplacian is self-adjoint and positive definite, so it has positive real eigenvalues.
We can compute its spectrum via a generalised eigenproblem (\ref{eq: hodge generalised eigenproblem}) to find harmonic (or approximately harmonic) differential forms, which should have zero (or nearly zero) eigenvalue.
In practice, the eigenvalues are never exactly zero, but we observe that eigenvalues that are closer to zero correspond to clearer or larger cohomology classes in the data, and so can be used as a proxy for confidence or size.
Recall that our approach of weak formulations is in the Hilbert space completion $L^2\Omega^k(M)$, and so we are working with the projection of the Hodge Laplacian $\Delta_k$ into that space.
As such, we can only hope to approximate the $L^2$ cohomology of the space, and will miss out any harmonic forms with unbounded norm.
For example, a punctured disk has a cohomology class representing the missing point at the origin, but this is represented by a harmonic form whose pointwise norm grows like $\|x\|^{-1}$, and so is not in $L^2\Omega^k(M)$ and will not be approximated by this method\footnote{this is no great loss, because no statistic could distinguish the punctured disk (with a hole) and a disk (with no hole) from a finite sample of data}.
We visualise harmonic forms computed from data in Figures \ref{fig:tda1} and \ref{fig:tda2}.

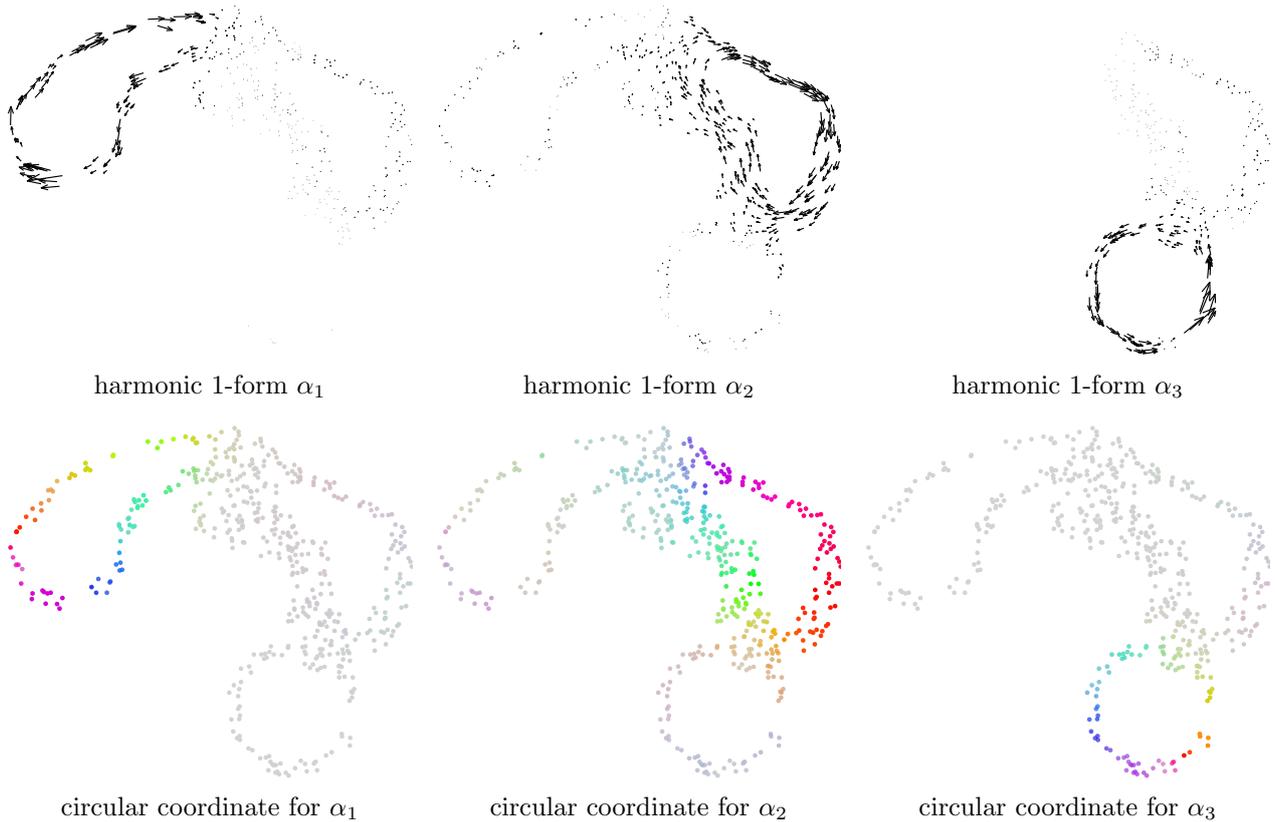
\begin{figure}[h!]
  \centering
  \begin{overpic}[width=\linewidth,grid=false]{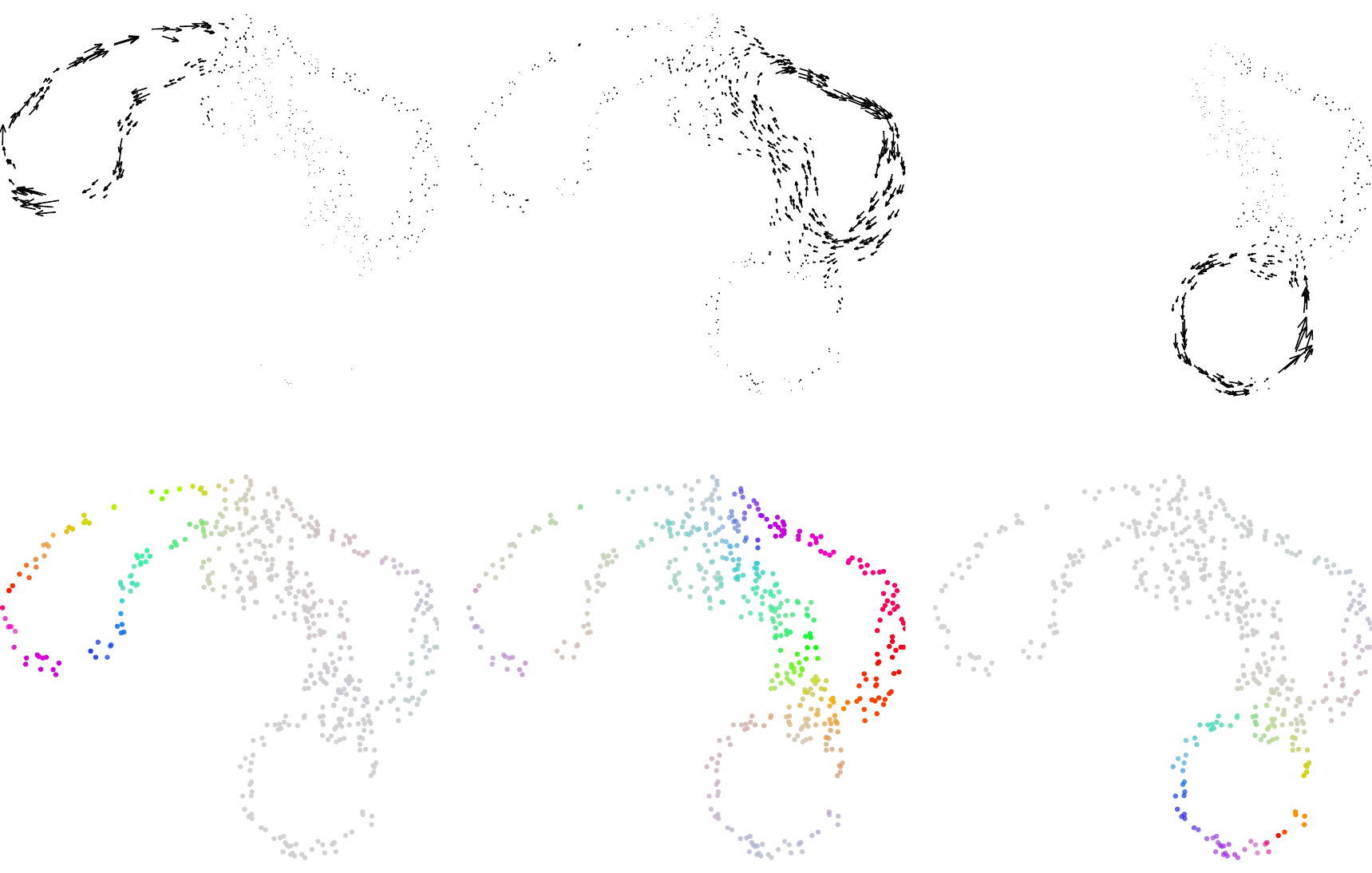}
\put(16.0,32.2){\makebox(0,0)[c]{harmonic 1-form $\alpha_1$}}
\put(50.0,32.2){\makebox(0,0)[c]{harmonic 1-form $\alpha_2$}}
\put(84.0,32.2){\makebox(0,0)[c]{harmonic 1-form $\alpha_3$}}
\put(16.0,-1.5){\makebox(0,0)[c]{circular coordinate for $\alpha_1$}}
\put(50.0,-1.5){\makebox(0,0)[c]{circular coordinate for $\alpha_2$}}
\put(84.0,-1.5){\makebox(0,0)[c]{circular coordinate for $\alpha_3$}}
  \end{overpic}
  \vspace{0.4em}
  \caption{\textbf{de Rham cohomology.}
  The first three eigenforms of the Hodge Laplacian $\Delta^{(1)}$ are approximately harmonic, and capture the three \q{holes} in this 2d data (top row).
  We can parametrise these holes by circular coordinates, which measure the \q{angle} of each point around the holes (bottom row).
  For clarity, the colour intensity of the circular coordinates is proportional to the pointwise norm of the corresponding harmonic form.
  }
  \label{fig:tda1}
\end{figure}

\begin{relatedwork}
The most direct relation to the method presented here is from Spectral Exterior Calculus (SEC) \cite{berry2020spectral}, which used spectral information from the diffusion maps Laplacian \cite{COIFMAN20065} to formulate the Hodge Laplacian on 1-forms, and which we applied for diffusion geometry in \cite{jones2024diffusion}.
We find that the formulation presented here, with the carré du champ and differential forms generated by the ambient coordinates, leads to more stable eigenforms than SEC.
See the discussion of SEC in Section \ref{sec: background}, and further comparison with persistent homology below.
\end{relatedwork}

The dimension of the $k\thupper$ cohomology group $\mathcal{H}^k(M,\R)$ is called the $k\thupper$ \boldblue{Betti number} $\beta_k$, and counts the number of independent $k$-dimensional holes in $M$.
The Betti numbers give a topological summary of $M$, but do not capture everything.
For example, the Betti numbers of the torus are $\beta_0 = 1$, $\beta_1 = 2$, and $\beta_2 = 1$, since the torus is connected, has two independent 1-dimensional holes (the \q{long} and \q{short} ways around), and has one enclosed 2-dimensional void inside.
Alternatively, a sphere with two circles attached will also have Betti numbers $\beta_0 = 1$, $\beta_1 = 2$, and $\beta_2 = 1$, since it is also connected, has two independent 1-dimensional holes (the two circles), and has one enclosed 2-dimensional void inside the sphere.
However, these two spaces are topologically distinct, because the two 1-dimensional holes in the torus \textit{interact} to enclose the 2-dimensional void, but the two 1-dimensional holes in the sphere with circles do not interact at all.

To measure this interaction, we can take two cohomology classes $\alpha + \text{im}(d^{(k-1)}) \in \mathcal{H}^k(M,\R)$ and $\beta + \text{im}(d^{(l-1)}) \in \mathcal{H}^l(M,\R)$, and compute the wedge product $\alpha \wedge \beta + \text{im}(d^{(k+l-1)}) \in \mathcal{H}^{k+l}(M,\R)$.
This new cohomology class is independent of the choice of representatives $\alpha$ and $\beta$, and so induces a well-defined bilinear map on cohomology groups called the \boldblue{cup product}
\[\cup : \mathcal{H}^k(M,\R) \times \mathcal{H}^l(M,\R) \to \mathcal{H}^{k+l}(M,\R).\]
If $\alpha \in \mathcal{H}^k(M,\R)$ and $\beta \in \mathcal{H}^l(M,\R)$ are cohomology classes represented by harmonic forms, then their cup product $\alpha \cup \beta \in \mathcal{H}^{k+l}(M,\R)$ is represented by the harmonic form that is the orthogonal projection of the wedge product $\alpha \wedge \beta$ onto $\ker(\Delta^{(k+l)})$.
We can use this to measure the interaction between cohomology classes and tell apart spaces that have the same Betti numbers but different topologies, as we illustrate in Figure \ref{fig:tda2}.

\begin{figure}[h!]
  \centering
  \begin{overpic}[width=\linewidth,grid=false]{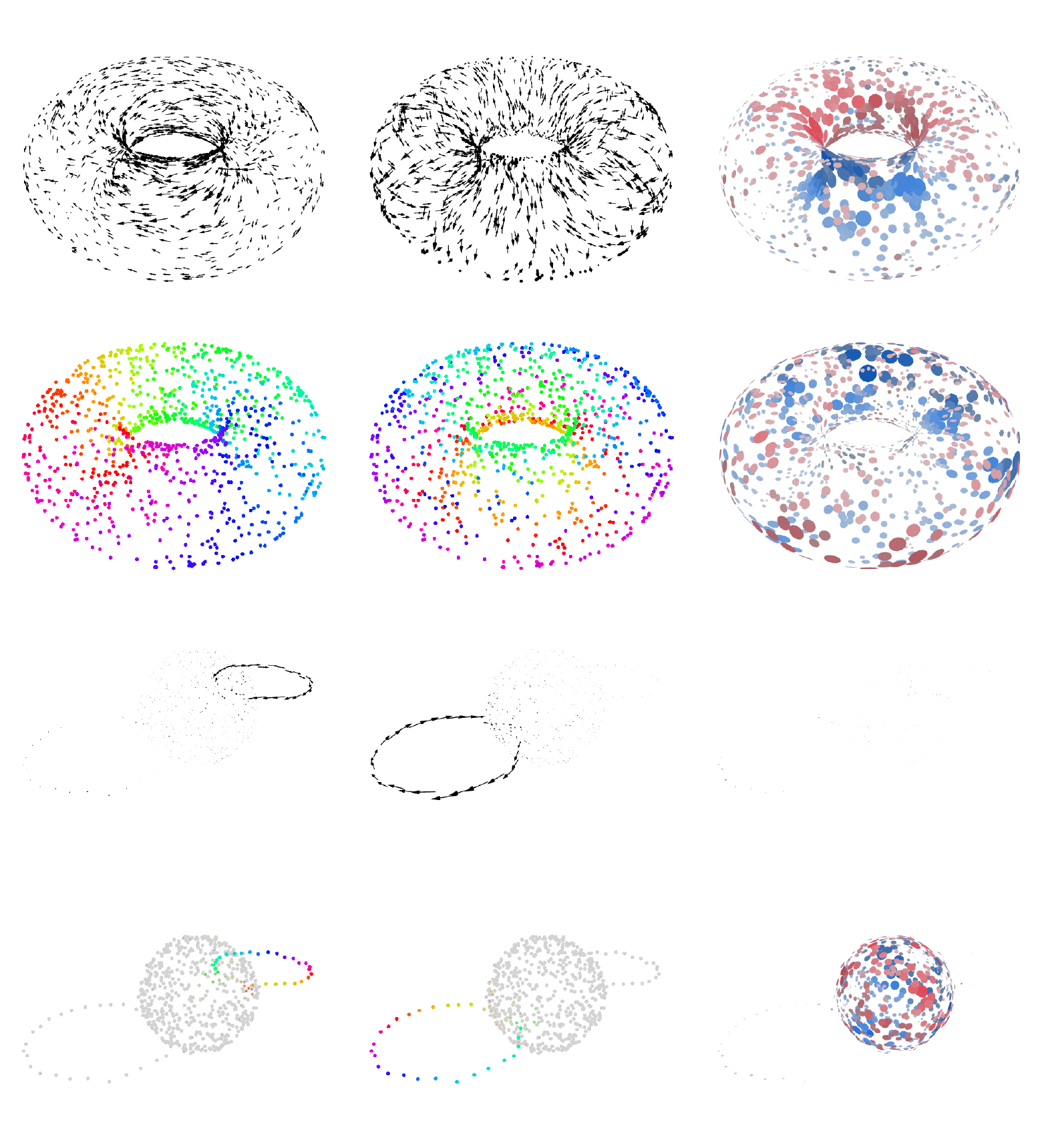}
\put(15.2,73.0){\makebox(0,0)[c]{harmonic 1-form $\alpha_1$, $\|\alpha_1\| = 1$}}
\put(45.5,73.0){\makebox(0,0)[c]{harmonic 1-form $\alpha_2$, $\|\alpha_2\| = 1$}}
\put(75.8,73.0){\makebox(0,0)[c]{$\alpha_1 \wedge \alpha_2$, $\|\alpha_1 \wedge \alpha_2\| = 0.93$}}
\put(15.2,48.0){\makebox(0,0)[c]{circular coordinate for $\alpha_1$}}
\put(45.5,48.0){\makebox(0,0)[c]{circular coordinate for $\alpha_2$}}
\put(75.8,48.0){\makebox(0,0)[c]{harmonic 2-form $\beta$, $\langle \alpha_1 \wedge \alpha_2, \beta \rangle = 0.54$}}
\put(15.2,27.0){\makebox(0,0)[c]{harmonic 1-form $\alpha_1$, $\|\alpha_1\| = 1$}}
\put(45.5,27.0){\makebox(0,0)[c]{harmonic 1-form $\alpha_2$, $\|\alpha_2\| = 1$}}
\put(75.8,27.0){\makebox(0,0)[c]{$\alpha_1 \wedge \alpha_2$, $\|\alpha_1 \wedge \alpha_2\| = 0.05$}}
\put(15.2,2.0){\makebox(0,0)[c]{circular coordinate for $\alpha_1$}}
\put(45.5,2.0){\makebox(0,0)[c]{circular coordinate for $\alpha_2$}}
\put(75.8,2.0){\makebox(0,0)[c]{harmonic 2-form $\beta$, $\langle \alpha_1 \wedge \alpha_2, \beta \rangle = 0.002$}}
  \end{overpic}
  \vspace{0.4em}
  \caption{\textbf{Cup product of cohomology classes.}
  A torus (top two rows) and a sphere with two circles attached (bottom two rows) both have Betti numbers $\beta_0 = 1$, $\beta_1 = 2$, and $\beta_2 = 1$, so they have two harmonic 1-forms $\alpha_1, \alpha_2$ and one harmonic 2-form $\beta$, but their topologies are different.
  We can tell them apart by computing the cup product $\alpha_1 \cup \alpha_2$, which is represented by the harmonic part of the wedge product $\alpha_1 \wedge \alpha_2$.
  We can measure this in the inner product $\langle \alpha_1 \wedge \alpha_2, \beta \rangle$, which is large for the torus (0.54) but small for the sphere with circles (0.002).
  All the harmonic forms are normalised to have unit norm, and so the inner products are a geometric measure of \textit{correlation}.
  For clarity, the colour intensity of the circular coordinates is proportional to the pointwise norm of the corresponding harmonic form.
  }
  \label{fig:tda2}
\end{figure}
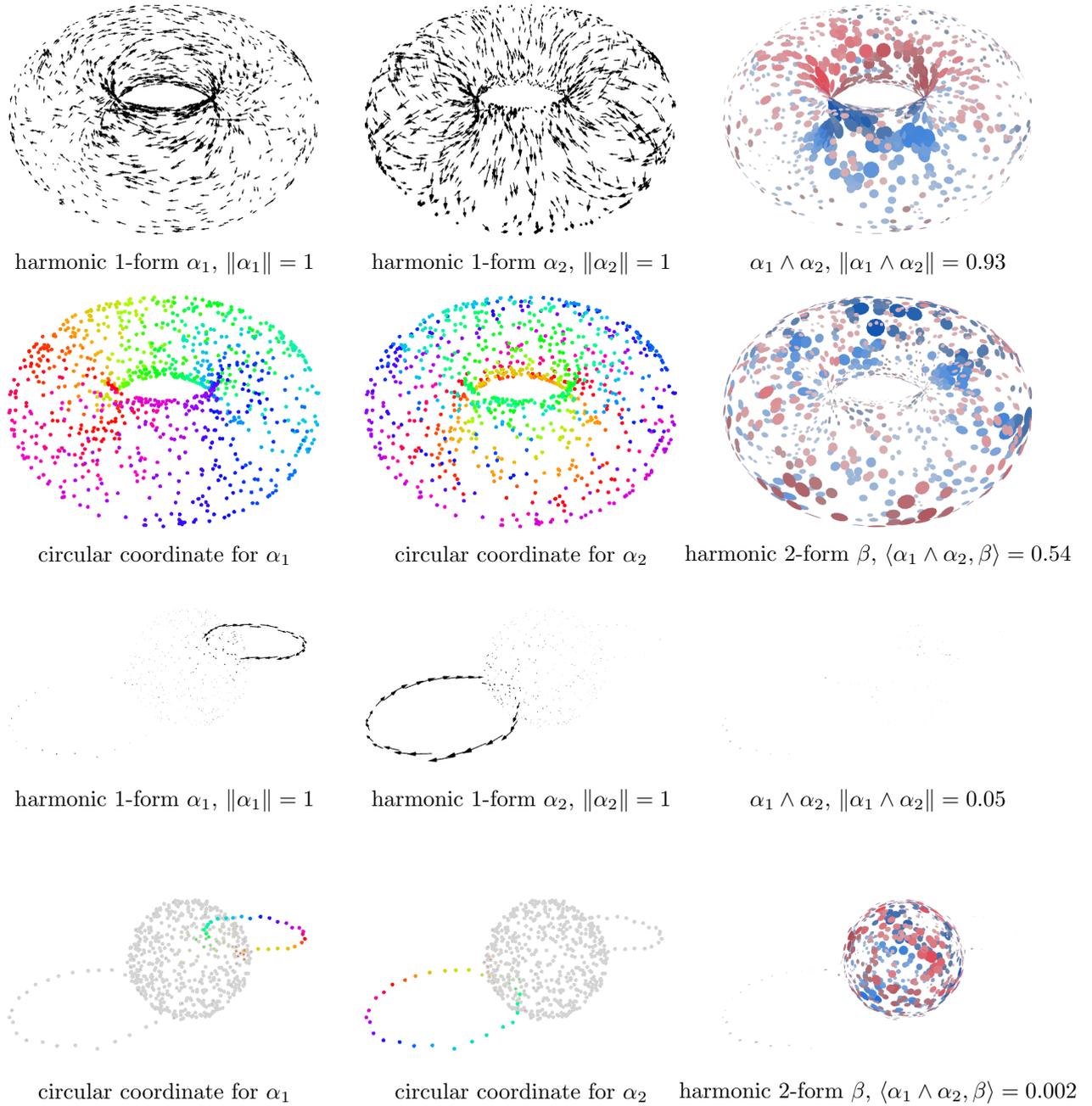

\subsection{Comparison with persistent homology}
\label{sec: ph comparison}

The most common method for computing the topology of data is \boldblue{persistent homology} \cite{robins1999towards, edelsbrunner2002topological, zomorodian2004computing, carlsson2009topology}.
Persistent homology turns a point cloud into a combinatorial object called a \boldblue{filtered simplicial complex}, whose topology can reveal information about the geometry and topology of the points across different scales.

A standard construction is the \boldblue{Vietoris–Rips filtration}, which (approximately) considers balls of radius $\epsilon$ centred at each data point, which will progressively join up as $\epsilon$ increases.
By computing the cohomology of the complex at scale $\epsilon$, we can measure how many topological features exist that have a \textit{radius} of at least $\epsilon$.
As we increase $\epsilon$ from 0 to $\infty$, each of these features will appear at some \boldblue{birth time} $b$ and disappear at a \boldblue{death time} $d$.
The total lifetime $d - b$ is called the \boldblue{persistence} of the feature and is typically interpreted as a kind of topological \q{signal strength}.
The pairs $(b, d)$ can be plotted in a 2d scatter plot called a \boldblue{persistence diagram} which summarises the overall geometry and topology of the point cloud \cite{zomorodian2004computing}.
We illustrate this construction in Figure~\ref{fig:vietoris-rips}.

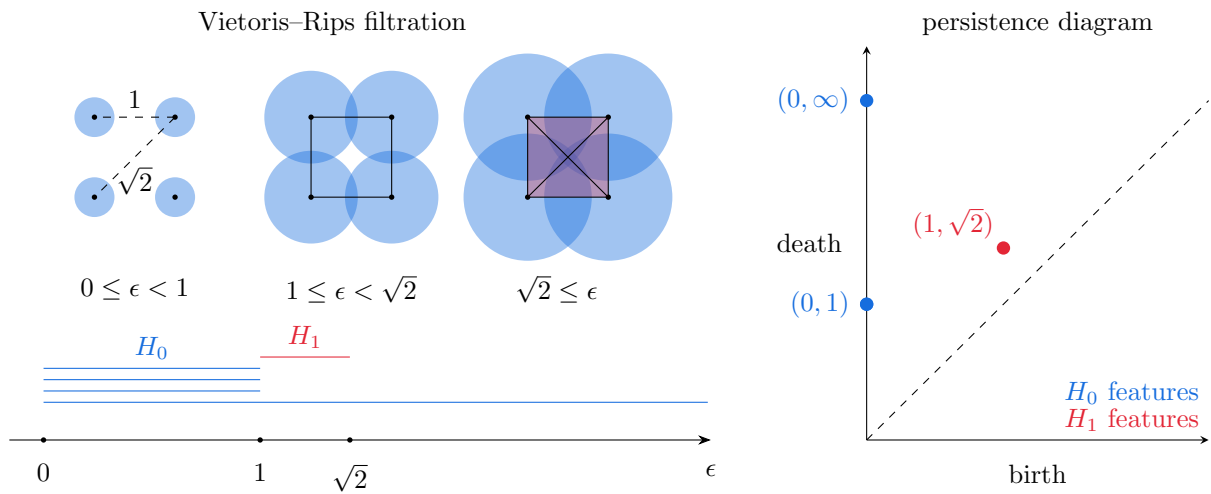
\begin{figure}[h!]
\centering

\begin{tikzpicture}[scale=0.5]

\definecolor{dotred}{HTML}{e32636}
\definecolor{dotblue}{HTML}{166dde}
\def\mylinewidth{1pt}

\def\numVertices{4}
\def\radius{1.5}
\def\figuredistance{3.8}
\def\annotationY{-3.5}


  \foreach \i in {1,...,\numVertices} {
    \pgfmathsetmacro\angle{360/\numVertices * (\i - 1) + 45}
    \pgfmathsetmacro\x{\radius * cos(\angle)}
    \pgfmathsetmacro\y{\radius * sin(\angle)}
    \coordinate (V\i) at (\x , \y);
  }

  \def\radius1{15pt}
    \foreach \i in {1,...,\numVertices} {
        \coordinate (V\i) at ($(V\i) + (\figuredistance+0.5, 0)$);
        \fill[dotblue, opacity=0.4] (V\i) circle (\radius1);
        \fill[black] (V\i) circle (2pt); 
    }
    \draw[dashed](V1) -- (V2) node[pos=0.5, above] {1};
    \draw[dashed](V1) -- (V3) node[pos=0.5, below] {$\sqrt{2}$};
    \node at (\figuredistance+0.5,\annotationY) {$0\leq \epsilon < 1$};

    \def\radius2{35pt}
    \foreach \i in {1,...,\numVertices} {
        \coordinate (V\i) at ($(V\i) + (1.5*\figuredistance, 0)$);
        \fill[dotblue, opacity=0.4] (V\i) circle (\radius2);
        \fill[black] (V\i) circle (2pt); 
    }
    \draw (V1)
    \foreach \i in {2,...,\numVertices} {
      -- (V\i)
    }
    -- cycle;

    \def\angle2{120}
    \pgfmathsetmacro\x{1 * cos(\angle2)}
    \pgfmathsetmacro\y{1 * sin(\angle2)}


    \node at (0.5+2.5*\figuredistance,\annotationY) {$1\leq \epsilon < \sqrt{2}$};

    \def\radius3{48pt}
    \foreach \i in {1,...,\numVertices} {
        \coordinate (V\i) at ($(V\i) + (1.5*\figuredistance, 0)$);
        \fill[dotblue, opacity=0.4] (V\i) circle (\radius3);
        \fill[black] (V\i) circle (2pt); 
    }

    \fill[dotred,opacity=0.3] (V1)
    \foreach \i in {2,...,\numVertices} {
      -- (V\i)
    }
    -- cycle;

    \draw (V1)
    \foreach \i in {2,...,\numVertices} {
      -- (V\i)
    }
    -- cycle;

    \draw (V1) -- (V3);
    \draw (V2) -- (V4);

    \node at (0.2+4*\figuredistance,\annotationY) {$\sqrt{2}\leq \epsilon $};

\def\linedistance{-4}
\def\annotationshift{-0.2cm}

    \draw[-Stealth](1,\annotationY+\linedistance) -- (5.1213*\figuredistance,\annotationY+\linedistance) node[pos=1.0, below, yshift=\annotationshift]{$\epsilon$};

    \coordinate (P1) at (0.5*\figuredistance,\annotationY+\linedistance);
    \coordinate (P2) at (2*\figuredistance,\annotationY+\linedistance);
    \coordinate (P3) at (2.6213*\figuredistance,\annotationY+\linedistance);

    \fill (P1) circle (2pt) node[below, yshift=0.2\annotationshift] {0};
    \fill (P2) circle (2pt) node[below, yshift=0.2\annotationshift] {1};
    \fill (P3) circle (2pt) node[below, yshift=0.2\annotationshift] {$\sqrt{2}$};

\def\baselineoffset{1}
\def\inlineoffset{0.3}

\draw[dotblue] ($(P1) + (0,\baselineoffset)$) -- ($(P2) + (3.1*\figuredistance,\baselineoffset)$);
\draw[dotblue] ($(P1) + (0,\baselineoffset+\inlineoffset)$) -- ($(P2) + (0,\baselineoffset+\inlineoffset)$);
\draw[dotblue] ($(P1) + (0,\baselineoffset+2*\inlineoffset)$) -- ($(P2) + (0,\baselineoffset+2*\inlineoffset)$);
\draw[dotblue] ($(P1) + (0,\baselineoffset+3*\inlineoffset)$) -- ($(P2) + (0,\baselineoffset+3*\inlineoffset)$);
\draw[dotred] ($(P2) + (0,\baselineoffset+4*\inlineoffset)$) -- ($(P3) + (0,\baselineoffset+4*\inlineoffset)$) node[pos=0.5, above] {$H_1$};


\node[dotblue, above, yshift=0.3cm] at ($0.5*($(P1) + (0,\baselineoffset+\inlineoffset)$) + 0.5*($(P2) + (0,\baselineoffset+\inlineoffset)$)$) {$H_0$};



\def\pdX{6.2*\figuredistance} 
\def\pdY{-7.5}                
\def\pdScale{1.8}             
\def\pointSize{5pt}           

\draw[-stealth] (\pdX, \pdY) -- ++(5*\pdScale,0) node[midway, below, yshift=\annotationshift] {birth};
\draw[-stealth] (\pdX, \pdY) -- ++(0, 5.8*\pdScale) node[midway, left, xshift=\annotationshift] {death};

\draw[dashed] (\pdX, \pdY) -- ++(5*\pdScale, 5*\pdScale);


\fill[dotblue] (\pdX + 0*\pdScale, \pdY + 2*\pdScale) circle (\pointSize);
\fill[dotblue] (\pdX + 0*\pdScale, \pdY + 2*\pdScale) circle (\pointSize); 
\fill[dotblue] (\pdX + 0*\pdScale, \pdY + 5*\pdScale) circle (\pointSize);

\fill[dotred] (\pdX + 2*\pdScale, \pdY + 2.8284*\pdScale) circle (\pointSize);

\node[dotblue, left, xshift=-0.1cm] at (\pdX + 0*\pdScale, \pdY + 2*\pdScale) {$(0,1)$};
\node[dotblue, left, xshift=-0.1cm] at (\pdX + 0*\pdScale, \pdY + 5*\pdScale) {$(0,\infty)$};
\node[dotred, above left, xshift=-0.cm] at (\pdX + 2*\pdScale, \pdY + 2.8284*\pdScale) {$(1,\sqrt{2})$};

\node[dotblue, left] at (\pdX + 5*\pdScale, \pdY + 0.7*\pdScale) {$H_0$ features};
\node[dotred, left] at (\pdX + 5*\pdScale, \pdY + 0.3*\pdScale) {$H_1$ features};


\node[anchor=south] at (2.5*\figuredistance, 3) {Vietoris–Rips filtration};
\node[anchor=south] at (\pdX + 2.5*\pdScale, 3) {persistence diagram};


\end{tikzpicture}

\captionof{figure}{\textbf{Vietoris-Rips filtration and persistent homology} of four points in a square. 
At $\epsilon=0$, four connected components appear, merging at $\epsilon=1$. 
At this scale, a 1-cycle emerges that persists until $\epsilon=\sqrt{2}$. 
The \textit{persistence diagram} summarises these topological features. 
Note that three points overlap in the $H_0$ diagram.}
\label{fig:vietoris-rips}
\end{figure}

We note several differences between our approach and persistent homology.

\begin{enumerate}
    \item Diffusion geometry is defined on the \textit{underlying probability space} that the data are sampled from.
    Our computed objects are estimators for their continuous counterparts, and so our results, such as the harmonic forms, are \boldblue{inferential statistics}.
    They are inferring things about the underlying space.
    Conversely, persistent homology is only defined for a given filtration on a point cloud, and so is a \boldblue{descriptive statistic}.
    It is describing the data and is not necessarily an estimator for something about the space.
    \item In diffusion geometry, the harmonic forms provide unique, canonical representatives for each cohomology class.
    This means they can be easily \boldblue{visualised}, and also \boldblue{vectorised} for use in statistics and machine learning.
    Conversely, the points in the persistence diagram correspond to \textit{multiscale classes} of topological features and cannot be canonically visualised or vectorised \cite{ali2023survey}.
    This can make diffusion geometry easier to interpret: compare the harmonic forms in Figure \ref{fig:stability_and_robustness} to the persistence diagrams.
    \item When the filtration in persistent homology arises from a \boldblue{length scale}, such as in Vietoris–Rips, the birth and death times are all lengths in the original data units.
    The explicit link between topology and length has led to significant applications in areas where the data's scale has an essential meaning \cite{saadatfar2017pore,arns2004effect,herring2019topological}.
    While the harmonic forms in diffusion geometry offer a rich description of the space and can be combined with other calculus tools, they do not assess topological scale in the ambient units.
    \item The wedge product in diffusion geometry leads directly to the \boldblue{cup product} structure, as described above.
    Conversely, the multiscale aspect of persistent homology comes at the cost of not easily representing the cup product, which does not factor through the persistence diagram decomposition.
    There have been analogues of the cup product developed in this setting, like the \boldblue{persistent cup-length} and \boldblue{persistent cup-modules} \cite{yarmola2010persistence, contessoto2021persistent,memoli2024persistent}, but these come at a high computational cost and do not preserve the clean algebraic structure of the continuous setting.
    \item A practical limitation of topological methods is their \boldblue{computational complexity}.
    Persistent homology is rarely applied to datasets larger than $\approx 10^3$ and in degrees bigger than 1.
    Diffusion geometry is more efficient in both computational time and memory, by several orders of magnitude, as we explore in Section \ref{sec: computational_complexity}.
\end{enumerate}

\subsubsection{Stability and robustness}

Two fundamental properties of geometric and topological statistics are their \boldblue{stability to perturbation} and their \boldblue{robustness to noise}.
We say that a statistic of a point cloud is \boldblue{stable} if small changes to the input point cloud correspond to small changes in the statistic (with respect to some metric on each).
Conversely, a statistic is \boldblue{robust} if its value only changes slightly when the input data are noisy.
This is not a precise notion because it entirely depends on the type of noise, but it is generally understood to include large perturbations and outliers.

The \textit{stability theorem} guarantees the stability of persistent homology with respect to the \textit{bottleneck distance} on the diagrams, which treats the diagram as a multiset \cite{cohen2005stability}.
The computations in this paper all involve matrices and tensors, whose entries depend smoothly on their input, and so are smooth.
The solutions to weak formulations will therefore vary smoothly.
The solutions to the generalised eigenproblems (which we use to find harmonic forms) are more subtle.
In general, the eigenvalues will vary continuously, which is the most direct analogue of the persistent homology stability theorem.
The eigenvectors, however, are not stable, because when two eigenvalues change places, the eigenvectors will jump discontinuously from one to the other (a problem mirrored in the computation of \textit{cocycle representatives} in persistent homology).
However, provided the eigenvalues are all distinct, then, for sufficiently small perturbations, the eigenvalues and eigenvectors are both smooth.

In diffusion geometry, the underlying computation uses a heat kernel which smoothly aggregates the different data points, so our harmonic forms are extremely robust to both large perturbations and outliers.
Conversely, the Vietoris-Rips filtration is much more easily distorted, and is not significantly robust: see Figure \ref{fig:stability_and_robustness}.

\begin{computationalnote}
When the Gram matrix has small eigenvalues, and we employ regularisation (see Section \ref{sec: frame_theory_weak_formulations}), the stability of the solutions will depend on the regularisation.
They will be stable for Tikhonov regularisation, but not generally for spectral cutoff.
\end{computationalnote}


\begin{figure}[htbp]
    \centering
    \makebox[\textwidth][c]{
        \newcommand{\imgwidth}{0.31\textwidth} 
        \setlength{\tabcolsep}{2pt}
        \begin{tabular}{ccc}
             clean data 
             & $+$ perturbation
             & $+$ outliers \\
             \\
            
            \includegraphics[width=\imgwidth, valign=m, clip=true, trim=0 20 0 20]{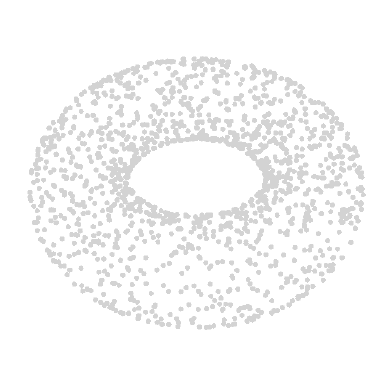} &
            \includegraphics[width=\imgwidth, valign=m, clip=true, trim=0 20 0 20]{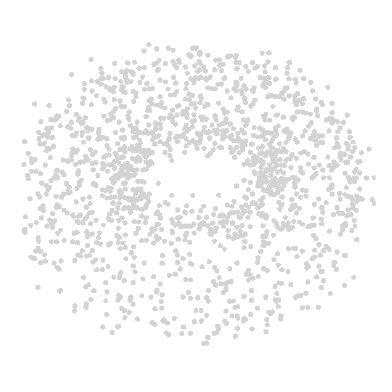} &
            \includegraphics[width=\imgwidth, valign=m, clip=true, trim=0 20 0 20]{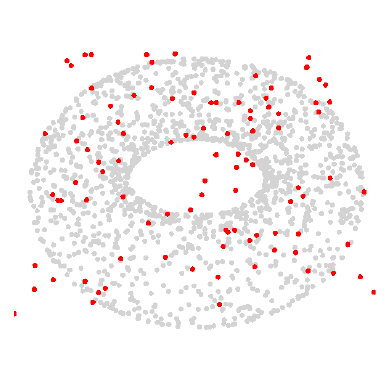} \\

            \includegraphics[width=\imgwidth, valign=m, clip=true, trim=0 20 0 20]{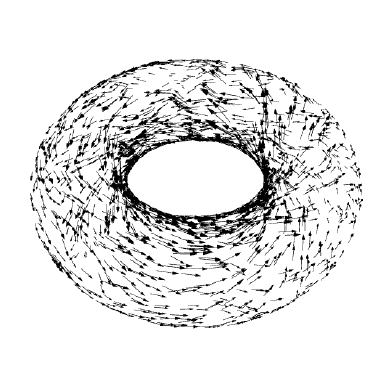} &
            \includegraphics[width=0.34\textwidth, valign=m, clip=true, trim=0 20 0 20]{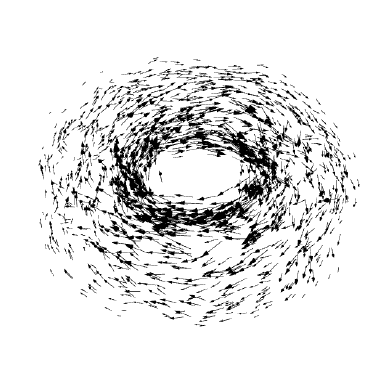} &
            \includegraphics[width=\imgwidth, valign=m, clip=true, trim=0 20 0 20]{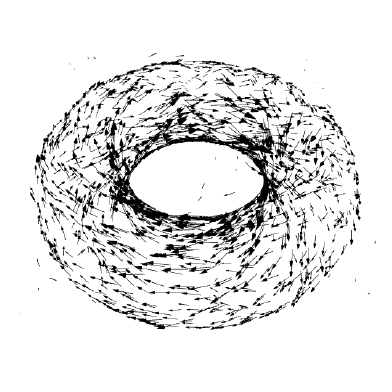} \\
            
            \includegraphics[width=\imgwidth, valign=m, clip=true, trim=0 20 0 20]{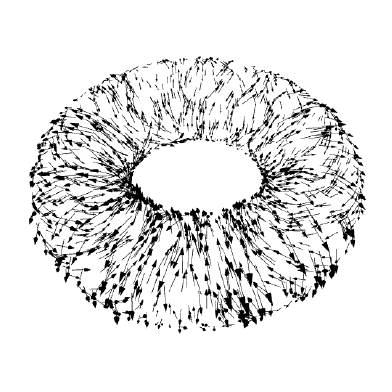} &
            \includegraphics[width=0.34\textwidth, valign=m, clip=true, trim=0 20 0 20]{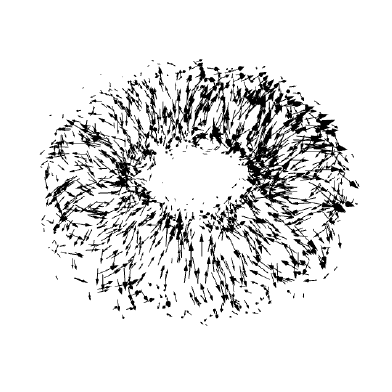} &
            \includegraphics[width=\imgwidth, valign=m, clip=true, trim=0 20 0 20]{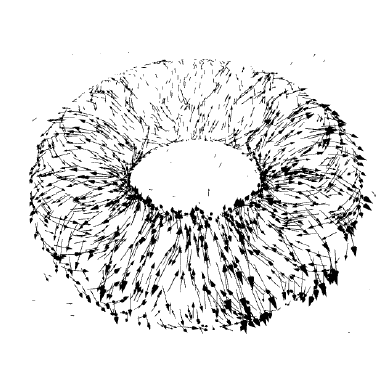} \\
            
            \\
            \includegraphics[width=\imgwidth, valign=m]{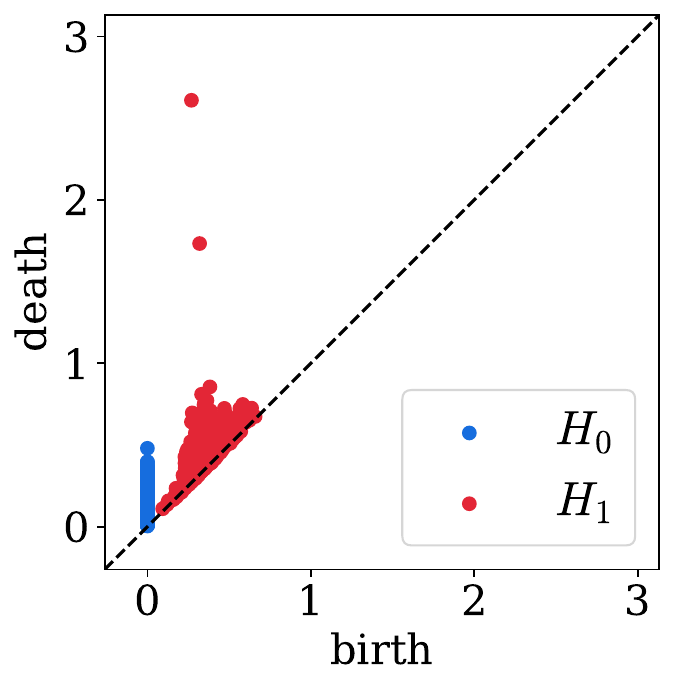} &
            \includegraphics[width=\imgwidth, valign=m]{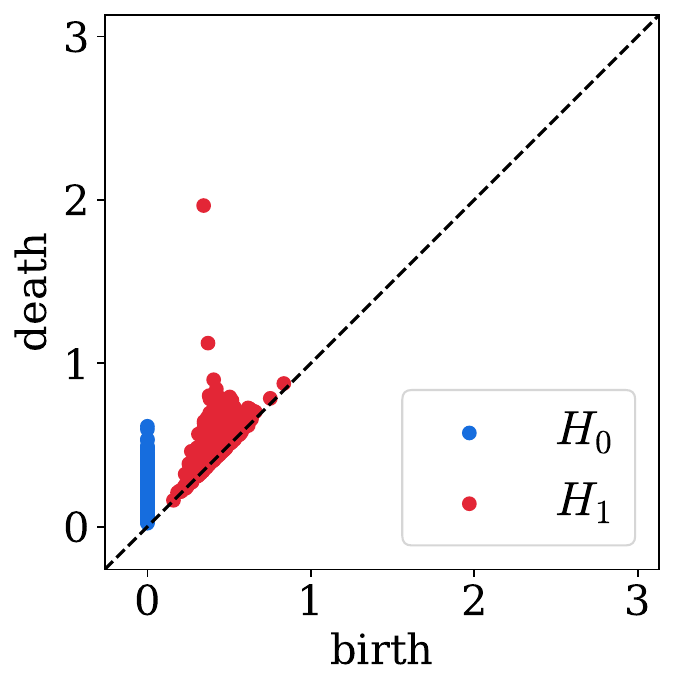} &
            \includegraphics[width=\imgwidth, valign=m]{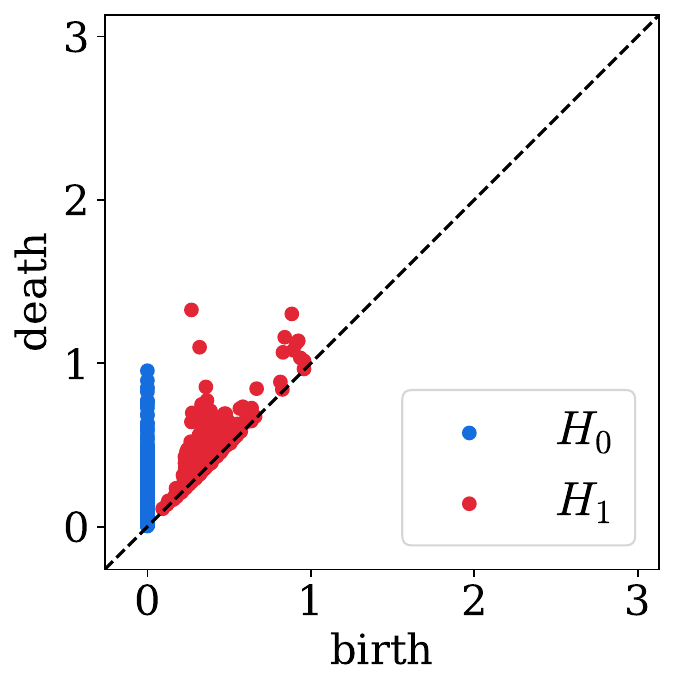} \\
        \end{tabular}
    }
    \caption{\textbf{Diffusion geometry vs persistent homology: stability and robustness.}
    Geometric and topological statistics should be \textit{stable} under small perturbations, and, ideally, \textit{robust} to large perturbations and more diverse noise like outliers.
    We sample points uniformly from a torus (left), with a small perturbation (middle), and with outliers (right).
    In the second and third rows, we use diffusion geometry to compute the first two eigenforms of the Hodge Laplacian, which stably and robustly reveal the topology of the torus.
    In the bottom row, we show the Vietoris-Rips persistence diagrams.
    These are less clear to interpret, but for the clean data, the $H_1$ persistence diagram has two highly persistent points, which we expect to correspond to the two holes in the torus.
    These are affected by perturbation but do not completely disappear (corresponding to stability), and are lost when we add outliers (indicating a lack of robustness).
    }
    \label{fig:stability_and_robustness}
\end{figure}

\subsection{Circular coordinates for 1-cohomology classes}

Given a harmonic 1-form $\alpha \in \Omega^1(M)$, we can ask how far a given point $p \in M$ is around the corresponding 1-dimensional hole.
When $M$ is a manifold, then, under certain conditions\footnote{the 1-form $\alpha$ represents a cohomology class $[\alpha] \in \mathcal{H}^1(M,\R)$, and has a circular coordinate precisely when $[\alpha]$ is an \textit{integral} cohomology class in $\mathcal{H}^1(M,\Z) \subset \mathcal{H}^1(M,\R)$}
on $\alpha$, this can be measured by a \boldblue{circular coordinate}, which assigns an angle $\theta(p) \in S^1$ to each point $p$.
Locally, $\theta$ satisfies $\alpha = 2\pi d\theta$, and so we can interpret the flow generated by the dual vector field $\alpha^\sharp$ as pushing points around the hole at a constant angular speed.
We will compute candidate circular coordinates for a 1-form $\alpha$ using the heuristic that the flow generated by $\alpha^\sharp$ is \textit{purely rotational}.


\begin{lemma}
If $\alpha \in \Omega^1(M)$ is a 1-form, then its dual vector field $\alpha^\sharp$ has adjoint $(\alpha^\sharp)^* = \partial\alpha - \alpha^\sharp$ with respect to the inner product on $L^2(\mu)$.
In particular, \boldblue{$\partial\alpha = 0$ if and only if $\alpha^\sharp$ is skew-adjoint}, so it has a purely imaginary spectrum.
\end{lemma}
\begin{proof}
We calculate
\[
\inp{\alpha^\sharp(f)}{h}_{L^2(\mu)}
=\inp{g(\alpha,df)}{h}_{L^2(\mu)} 
=\int h g(\alpha,df) d\mu 
=\int g(h\alpha,df) d\mu 
=\inp{h\alpha}{df}_{\Omega^1(\M)}.
\]
Using the fact that $\partial = d^*$, this equals
\[
\inp{\partial(h\alpha)}{f}_{L^2(\mu)}
=\inp{h\partial(\alpha) - g(\alpha,dh)}{f}_{L^2(\mu)}
=\inp{h\partial(\alpha) - \alpha^\sharp(h)}{f}_{L^2(\mu)}
\]
for any $f,h \in \A$.
\end{proof}

If $\partial\alpha = 0$ (or, equivalently, its Hodge decomposition has no exact part), we say that $\alpha \in \Omega^1(M)$ is \boldblue{coclosed}.
Harmonic forms are coclosed, so their dual vector fields are skew-adjoint, and their flows are rotational.
Given a harmonic 1-form $\boldsymbol{\alpha} \in \R^{n_1d}$, we can compute $\boldsymbol{\alpha}^{\textnormal{op}}$, the operator form of its dual vector field, as described in Section \ref{sub: directional derivatives}.
If we diagonalise this matrix, we expect it to have imaginary eigenvalues $\lambda_i$ and complex eigenfunctions $z_i \in \C^{n_0}$ that evolve under the flow of $\boldsymbol{\alpha}^{\textnormal{op}}$ as $z_i \mapsto e^{\lambda_i t}z_i$.
Using the heuristic that the flow of $\alpha^\sharp$ rotates uniformly around the hole, we will use the circular functions $\arg(z_i) \in S^1$ with the smallest eigenvalue magnitudes $|\lambda_i|$ as candidate circular coordinates for $\alpha$.


In practice, we regularise the vector field $\boldsymbol{\alpha}^{\textnormal{op}}$ by adding a Laplacian term to get
\[
L_\epsilon = \boldsymbol{\alpha}^{\textnormal{op}} - \epsilon \boldsymbol{\Delta},
\]
where higher values of $\epsilon$ give smoother eigenvectors.
When $\epsilon$ is too high, the eigenvalues of $L_\epsilon$ eventually become purely real valued, and when $\epsilon$ is too low, the eigenfunctions become too irregular to be useful.
In all our examples, we set $\epsilon = 1$ and choose the complex eigenfunction with the smallest eigenvalue magnitude.
We plot these coordinates next to their corresponding harmonic forms in Figures \ref{fig:tda1} and \ref{fig:tda2}.
We stress that the idea of using the flow of a coclosed 1-form to generate circular coordinates is purely heuristic, and we do not offer any formal guarantees.
The fact that it works robustly in all our examples raises several interesting theoretical questions for future work.

\begin{relatedwork}
The idea of using 1-forms to generate circular coordinates on data was introduced in \cite{de2009persistent}, which used persistent cohomology to find integral 1-forms and then solved a least-squares problem to find circular coordinates, and was developed in \cite{perea2020sparse}.
The idea of smoothing the vector field by adding a diffusion term is a standard trick in dynamical systems analysis, such as \cite{giannakis2019data}.
\end{relatedwork}

\subsection{Morse theory}
\label{sec: morse}

Another method from differential topology that we can access with diffusion geometry is \boldblue{Morse theory}, which uses the critical points of functions to measure the topology of a space.
If $f : M \to \R$ is a smooth function on a manifold $M$, then a point $p \in M$ is a \boldblue{critical point} of $f$ if $\nabla f$ vanishes there.
At a critical point $p$, the \boldblue{Hessian} $H(f)(p) : \mathfrak{X}(M) \times \mathfrak{X}(M) \to \R$ is a bilinear form that measures the second derivatives of $f$ at $p$, in the direction of a pair of vector fields.
If $M$ is $D$-dimensional, then $H(f)(p)$ has rank at most $D$.
We call a critical point $p$ \boldblue{non-degenerate} if $H(f)(p)$ has rank $D$, and a function $f$ is a \boldblue{Morse function} if all its critical points are non-degenerate.
If $p$ is a non-degenerate critical point of $f$, then the \boldblue{Morse index} of $f$ at $p$ is the number of negative eigenvalues of $H(f)(p)$, which is an integer between $0$ and $D$.
As such, the Morse index measures the number of independent directions around $p$ along which $f$ is locally concave, and so is 0 at local minima, $D$ at local maxima, and some integer in between at saddle points.
The Morse indices, and, more generally, the eigenvectors of the Hessian, give a complete topological description of the underlying manifold $M$.

We can compute these indices using our diffusion geometry framework.
First, we can find critical points of a function $f \in \A$ from the norm of its gradient $\|\nabla f\|$, which is (approximately) zero at critical points.
We can also compute the Hessian $H(f) \in \Omega^1(M)^{\otimes 2}$, and illustrate these all in Figure \ref{fig:morse_derivatives}.

\begin{figure}[h!]
  \centering
  \begin{overpic}[width=\linewidth,grid=false]{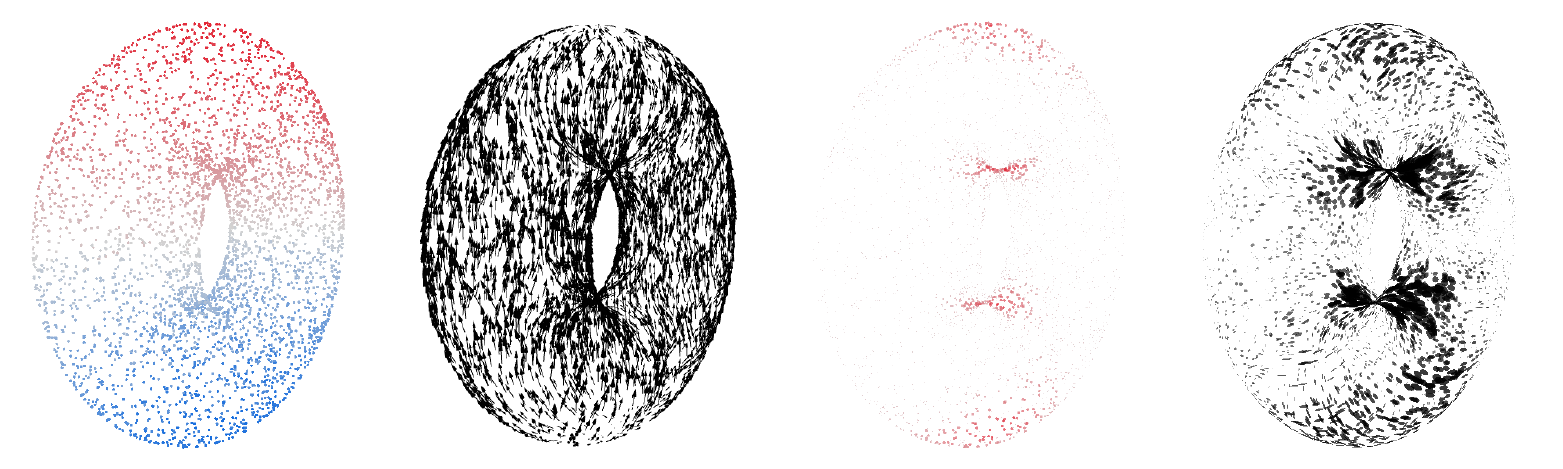}
  \end{overpic}
  \caption{\textbf{Function derivatives and critical points.}
  If $f \in \A$ is a function, then $\nabla f \in \mathfrak{X}(M)$ measures the velocity of $f$ and is zero at critical points.
  In the third panel, we plot an indicator function $\exp(-\|\nabla f\|)$ that highlights the critical points of $f$.
  The Hessian $H(f) \in \Omega^1(M)^{\otimes 2}$ measures the second derivatives of $f$ (its acceleration) in the direction of a pair of vector fields.
  }
  \label{fig:morse_derivatives}
\end{figure}

If $p$ is a critical point of $f$ and $X_1,...,X_D \in \mathfrak{X}(M)$ are any $D$ linearly independent vector fields at $p$, then we can represent $H(f)(p)$ as a $D \times D$ symmetric matrix with entries $H(f)(p)_{ij} = H(f)(X_i,X_j)(p)$.
Let $G(p) = (g(X_i,X_j)(p))_{i,j=1}^D$ be the Gram matrix of the vector fields at $p$, which is strictly positive definite because the $X_i$ are linearly independent.
Then the eigenvalues and eigenvectors of $H(f)(p)$ are solutions to the generalised eigenproblem
\[H(f)(p)v = \lambda G(p)v.
\]
Notice that these solutions are independent of the choice of vector fields $X_i$, since changing the vector fields corresponds to a change of basis in $\R^D$.
This eigenproblem is fully symmetric, and so has real eigenvalues and orthogonal eigenvectors with respect to $G(p)$.
The positive eigenvalues correspond to directions along which $f$ is locally convex, and the negative eigenvalues correspond to directions along which $f$ is locally concave.
We can therefore compute the Morse index of $f$ at $p$ by counting the number of negative eigenvalues of this generalised eigenproblem.
This process is illustrated in Figure \ref{fig:morse_indices}.

\begin{figure}[h!]
  \centering
  \begin{overpic}[width=\linewidth,grid=false]{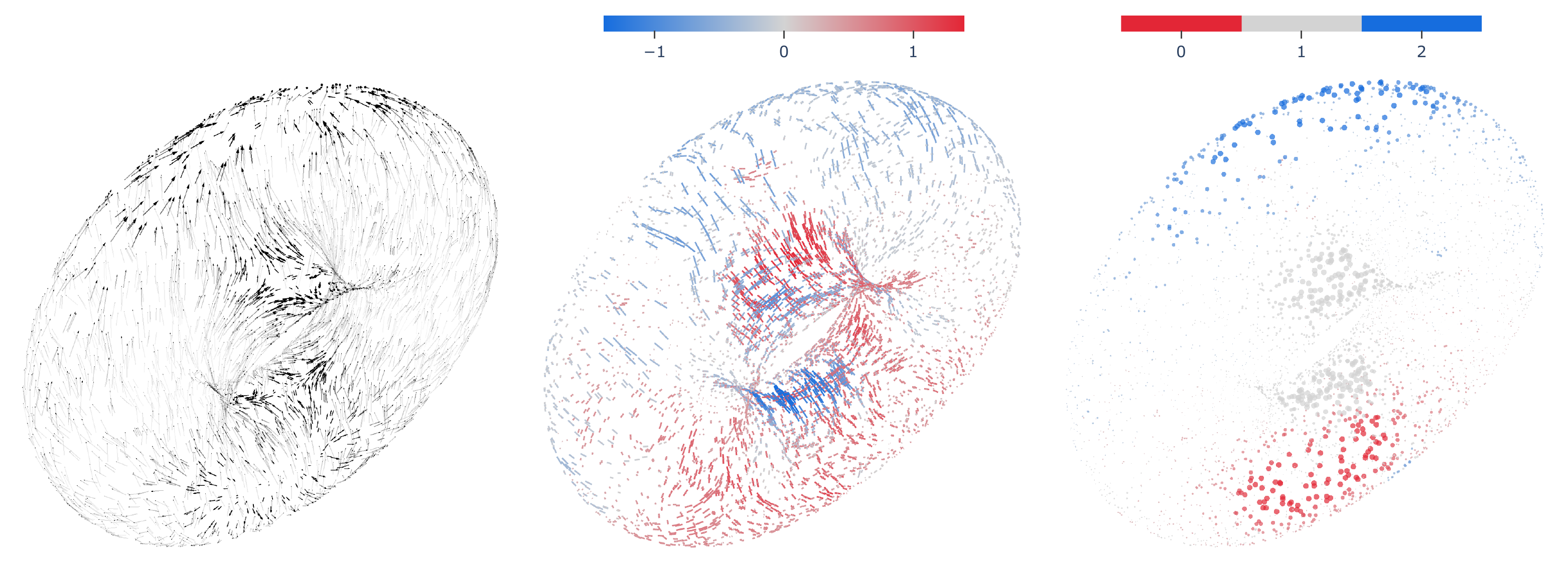}
\put(16.7,-1.5){\makebox(0,0)[c]{$\nabla f$ near critical points}}
\put(50.0,-1.5){\makebox(0,0)[c]{eigenvectors of $H(f)$}}
\put(83.3,-1.5){\makebox(0,0)[c]{Morse indices near critical points}}
  \end{overpic}
  \vspace{0.4em}
  \caption{\textbf{Morse indices of a function.}
    We illustrate the flow of $\nabla f$ near critical points (left), using the indicator function $\exp(-\|\nabla f\|)$ to highlight them.
    We can diagonalise the Hessian $H(f)$ locally to find its eigenvectors (middle), which are coloured by their eigenvalues (blue for negative, red for positive).
    Counting the number of negative eigenvalues gives the Morse index at each critical point (right).
    Near local minima, $\nabla f$ is expanding in all directions, so both Hessian eigenvalues are positive and the Morse index is $0$ (red).
    Near local maxima, $\nabla f$ is contracting in all directions, so both Hessian eigenvalues are negative and the Morse index is $2$ (blue).
    Near saddle points, $\nabla f$ is expanding in one direction and contracting in the other, so the Hessian has one positive and one negative eigenvalue and the Morse index is $1$ (grey).
  }
  \label{fig:morse_indices}
\end{figure}

\begin{relatedwork}
A notable prior attempt to infer Morse indices from manifold point clouds is \cite{yim2021local}.
\end{relatedwork}

\section{Computational complexity and scaling}
\label{sec: computational_complexity}

We now summarise the computational complexity of the methods introduced in this paper, and measure empirical scaling laws.

\subsection{Complexity}

In our proposed framework, the complexity depends most significantly on the number of data points $n$ and the immersion dimension $d$.
As we discussed in \compnote{note: immersion dimension and coords}, the immersion coordinates we use are arbitrary, and their number need not be the ambient dimension.
We therefore assume that $d$ is of the order of the \textit{intrinsic dimension} and so, like $n$, is an essential attribute of the data.

Given these constraints, we have several tools to control the computational complexity, which can be tuned to suit the available hardware and desired numerical precision.
These parameters are the nearest-neighbour kernel sparsity $k$, the dimension of the compressed function space $n_0$, and the number of coefficient functions $n_1$ used to compute the vector fields, forms, and tensors.
We summarise the complexities of the main operations in the following table.

\begin{table}[h!]
\label{table:complexity}
\centering
\renewcommand{\arraystretch}{1.3}
\begin{tabularx}{\textwidth}{llX}
Object 
& Time complexity 
& Remarks \\
\hline

kernel matrix $K$
& $\mathcal{O}(n\log(n)+d)$ 
& Takes advantage of sparsity by only considering $k$-NN. \\

carré du champ $\Gamma(f,h)$ 
& $\mathcal{O}(kn)$ 
& Takes advantage of sparsity through $k$. \\

eigenfunctions of $P$ 
& $\mathcal{O}(n n_0^2)$ 
& Using an iterative solver for sparse symmetric matrices: there is a factor of the number of iterations.
\\

product of $m$ pairs of functions 
& $\mathcal{O}(n n_0 m)$ 
& \\

Gram matrix of $k$-forms
& $\mathcal{O}(n n_1^2 \binom{d}{k}^2)$ 
& \\

wedge product $\alpha \in \Omega^a(\mathcal{M})$, $\beta\in\Omega^b(\mathcal{M})$ 
& $\mathcal{O}\!\left(n n_1 \binom{d}{a}\binom{d-a}{b}\right)$ 
& \\

weak exterior derivative $d^{(l),\mathrm{weak}}$ 
& $\mathcal{O}\left( \binom{d}{l}\binom{d}{l+1}nl(l^2+n_1^2)\right)$
& Assuming $\Gamma$-tensors have been precomputed.\\

\end{tabularx}
\end{table}

\subsection{Scalability}

In Section \ref{sec:deRahm}, we described a method for topological data analysis based on de Rham cohomology, which involves computing and diagonalising the Hodge Laplacian (see Section \ref{sub: hodge laplacian}).
We will use this as a test of the scalability of diffusion geometry by measuring empirical scaling laws against the dataset size $n$ and dimension $d$.
As a reference, we will compare the scaling performance of diffusion geometry with \textit{Ripser} \cite{bauer2021ripser}, which is the state-of-the-art software for Vietoris-Rips persistent homology (see Section \ref{sec: ph comparison} for further comparisons).

We note two differences between our method and persistent homology.
\begin{enumerate}
    \item Computing persistent homology in degree $k$ requires also computing it in all the degrees $0,...,k-1$.
    Diffusion geometry can isolate degrees.
    \item The complexity of persistent homology in degree $k$ is (in the worst case) $\mathcal{O}(n^{3k+3})$.
    In diffusion geometry, the Hodge Laplacian on $k$-forms is an $n_1 \binom{d}{k} \times n_1 \binom{d}{k}$ matrix, and computing it costs $\mathcal{O}(n n_1^2 \binom{d}{k}^2)$.
    In particular, the complexity does not depend significantly on the degree $k$ \textit{as a function of $n$}, even though it will be greater when $\binom{d}{k}$ is large.
\end{enumerate}

To give a realistic assessment of these methods' usability in practical data science contexts, we perform all the tests on standard consumer hardware (a 2024 MacBook Pro with 24GB memory).

\subsubsection{Scaling $n$}

We observe the scaling laws in $n$ by fixing a data distribution (here a torus in $\R^3$), sampling $n$ points from it, and timing both methods.
Since persistent homology can only compute degrees cumulatively, a fair test is to measure the time taken to calculate cohomology in \boldblue{degrees up to $k$}, for $k = 0,1,2$.
We plot the scaling behaviour in Figure~\ref{fig:n_scaling}, fixing $n_1=50$ for the diffusion geometry Hodge Laplacian computation.
The speed could be increased by reducing $n_1$ further (and, for a torus, $n_1 = 10$ would still give perfectly accurate harmonic form estimates).

\begin{figure}[h!]
\centering
\includegraphics[width=0.75\textwidth, clip=true, trim=0 0 0 100]{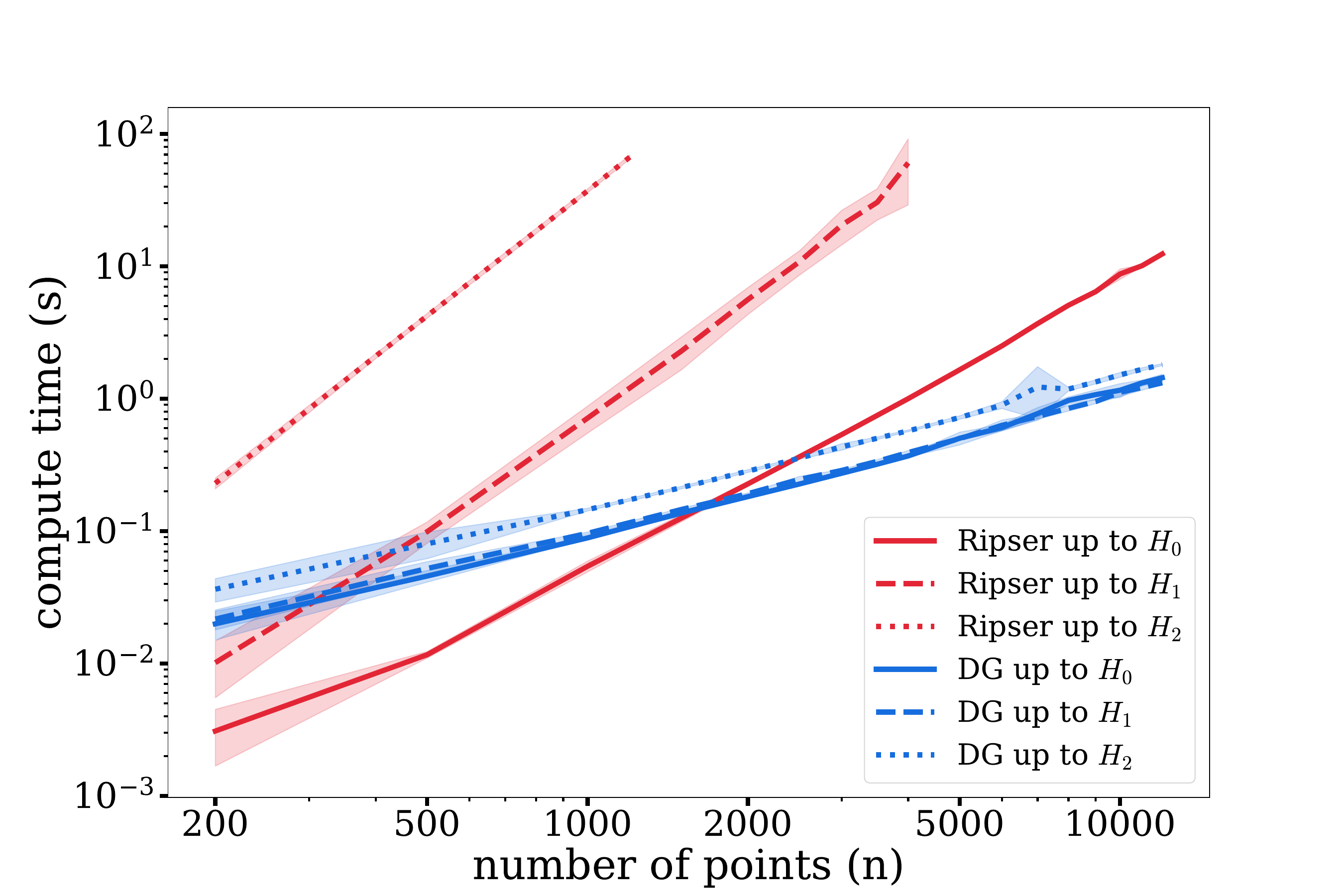}
\caption{
\textbf{Scalability of diffusion geometry vs persistent homology across dataset size.}
Benchmark results for cohomology computation in degrees up to $k$, for $k = 0,1,2$, using diffusion geometry (DG) and Ripser. 
The diffusion geometry Hodge Laplacian uses fixed parameters $k=32$ and $n_0=n_1=50$.
}
\label{fig:n_scaling}
\end{figure}

We notice that diffusion geometry exhibits substantially better scaling laws than persistent homology (the slope of the lines in Figure~\ref{fig:n_scaling}).
Importantly, the scaling laws for diffusion geometry do not depend on the degree $k$ (the lines for $H_0$, $H_1$, and $H_2$ have the same slope), whereas the complexity of Vietoris-Rips persistent homology is exponential in $k$ (the slope of the lines increases as $k$ increases).
These laws demonstrate that diffusion geometry is highly scalable, and takes only $\approx 1$ second to compute homology in all degrees for $n = 12,000$ on standard hardware.



\subsubsection{Scaling $d$}

We observe the scaling laws in $d$ by fixing a number of samples $n$, and sampling $n$ points from data distributions of increasing dimension (here, a $d$-sphere in $\R^{d+1}$), and timing both methods. 
Both diffusion geometry and persistent homology depend on $d$, but this dependence is more complicated than that on $n$.
The complexity of computing persistent homology in Ripser depends sensitively on both the intrinsic dimension and the degree.
The complexity of diffusion geometry depends very predictably on the immersion dimension, but, as discussed above, this depends somewhat arbitrarily on the choice of immersion coordinates.
Given this uncertainty, and the fact that persistent homology can only compute degrees cumulatively, we compute homology in every degree from $0$ to $d$, which we plot in the left panel of Figure~\ref{fig:d_scaling}.
The Vietoris-Rips complex becomes very large in high degrees, so we fix $n=100$ to allow computation up to degree 4.

\begin{figure}[h!]
\hspace*{-1cm}
\centering
\includegraphics[width=1.\textwidth]{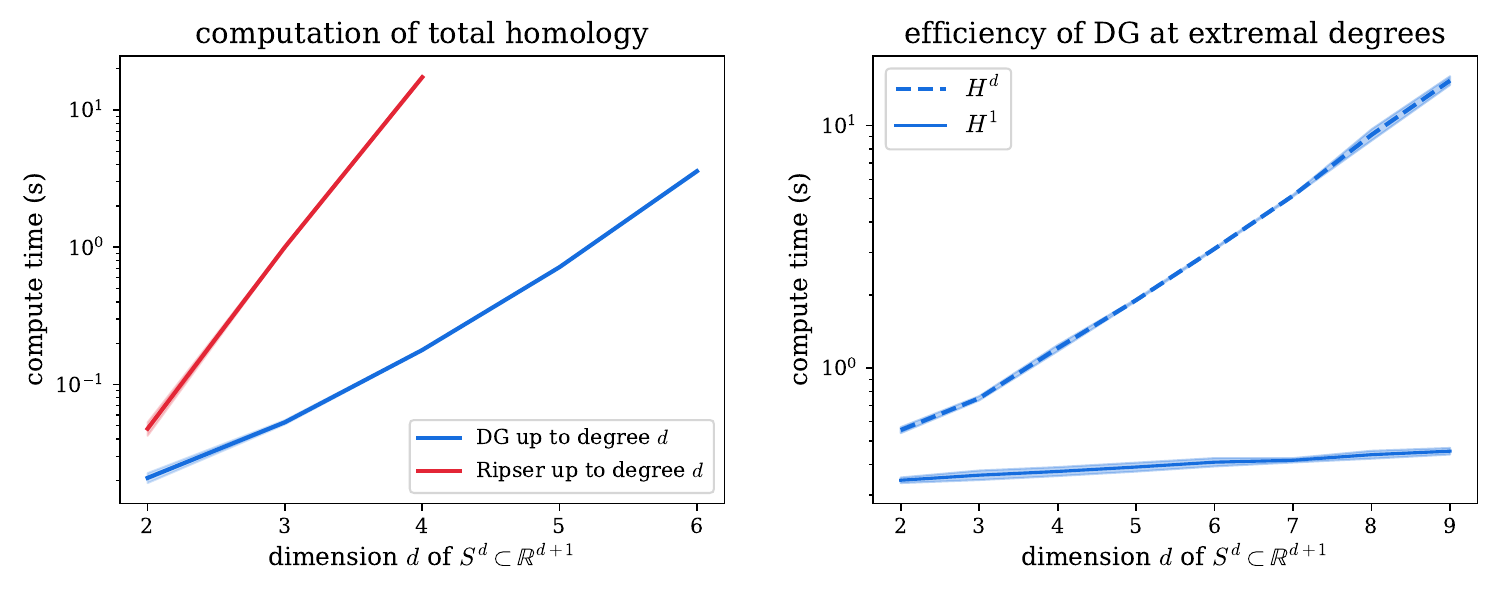}
\caption{
\textbf{Scalability of diffusion geometry vs persistent homology across dimensions.}
Left: computing cohomology in degrees up to $d$, using diffusion geometry (DG) and Ripser. 
A sample size of $n=100$ is chosen to account for the combinatorial explosion of simplices for Ripser.
Right: diffusion geometry benefits significantly from the combinatorial symmetry of the $k$-form space, so computation is exceptionally fast at extremal degrees (here $k=1$ and $k=d$).
This allows for high-dimensional feature recovery with minimal overhead, even here where we fix $n=5000$. 
For both plots, the diffusion geometry Hodge Laplacian uses fixed parameters $k=32$ and $n_0=n_1=50$.
}
\label{fig:d_scaling}
\end{figure}

Due to the $\binom{d}{k}$ scaling of the $k$-form space in diffusion geometry, computational complexity is minimal at the extremal degrees. This allows for very fast inference of topological features when $k$ is near 0 or the ambient dimension, as shown in the right panel of Figure~\ref{fig:d_scaling}.
The efficiency and scalability of diffusion geometry means we can set $n=5000$, and still obtain manageable computation times ($\approx 10$ seconds) up to degree 9.

\section{Background and context}
\label{sec: background}

Since this paper covers a wide range of topics, we have discussed some related work in each section as the need arises.
Here, we provide some general background for diffusion geometry in the context of computational geometry and topology.

\subsection{Theory}
Diffusion geometry is formulated using the language of \boldblue{Bakry-Émery $\Gamma$-calculus} \cite{bakry1985seminaire,bakry2014analysis}, which studies Markov diffusion processes through a bilinear operator called the \boldblue{carré du champ} $\Gamma$.
In diffusion geometry, the carré du champ is interpreted as a generalisation of the Riemannian metric, allowing Riemannian geometry to be generalised to any space with a diffusion process.
There have been many other approaches to generalising Riemannian geometry beyond manifolds, which are mostly purely theoretical, but a notable exception is \cite{bartholdi2012hodge}, which introduces a Hodge theory on metric measure spaces that can be computed using a simplicial complex on the data (see the following discussion of combinatorial models).
For a general theoretical background of diffusion geometry, see \cite{jones2024diffusion}.

\subsection{Topological models}
The dominant approach to computational geometry and topology is to construct a \boldblue{combinatorial structure} on the data, such as a \boldblue{graph}, \boldblue{simplicial complex}, \boldblue{hypergraph}, or related object.
In this model, functions are represented as values on vertices, vector fields and differential 1-forms as values on edges, and so on.
The combinatorial structure can be used to compute discrete analogues of differential operators, such as the graph Laplacian or the combinatorial exterior derivative.

These ideas originate from Whitney \cite{whitney2012geometric}, who used simplicial complexes to approximate differential forms on manifolds, and Dodziuk \cite{dodziuk1976finite, dodziuk1996riemannian}, who introduced combinatorial Hodge theory.
They later became widely applied in computational geometry and topology, with applications in computer graphics and simulation, and were formalised in the theory of \boldblue{discrete exterior calculus} \cite{hirani2003discrete, desbrun2005discrete}.
This framework can represent many of the standard objects from vector calculus and Riemannian geometry, such as vector fields, differential forms, the exterior derivative, Hodge star, and Laplacian.
There has been some progress towards representing more general tensors like the Hessian \cite{chen2022discrete}, connection \cite{berwick2021discrete}, and Lie bracket \cite{de2014discrete}, although these require special assumptions and fit less naturally into the framework.
Although these methods were originally developed for manifolds, they do not require any kind of tangent space basis\footnote{a useful property that is sometimes called being \q{coordinate-free} in the literature: we will avoid this term here to prevent confusion with the other sorts of coordinates used in this work}, so they can be used on any combinatorial complex.
A notable application is in \boldblue{topological data analysis}, which we will address in Section \ref{sec: TDA}.
As we discussed above, the fundamental objects of interest in geometry and calculus are differential operators, and a central question in discrete exterior calculus is how to represent these operators in a way that is stable under refinements of the complex and will eventually converge to the true operators on the underlying space.
This was resolved by \boldblue{finite element exterior calculus} (FEEC) \cite{arnold2006finite,arnold2010finite,christiansen2018generalized}, which used ideas from numerical analysis, specifically applying the \boldblue{Galerkin method} to solve PDEs using weak formulations.
Although our model is not combinatorial, we will take exactly this approach to define differential operators in Section \ref{sec: differential_operators}.

However, there are several limitations to applying combinatorial models to point cloud data.
First, it is hard to construct a suitable complex on the data in a statistically robust way.
These methods were generally developed for applications where the complex is given, and the results tend to depend sensitively on the choice of complex if the input data is just a point cloud, with small changes in the data leading to large changes in the complex (see the discussion in \cite{jones2024diffusion,jones2024manifold}).
Second, combinatorial complexes grow very quickly in size as the number of data points increases, making them computationally expensive.
Third, and most fundamentally, combinatorial models are essentially \boldblue{topological} and not geometric.
They have an exterior derivative that obeys the Leibniz rule, leading to meaningful cohomology groups with a cup product, and so are good models for topology. 
However, to model Riemannian geometry, we would like to define a Riemannian metric $g$, whose integral over the space defines a global inner product.
If $g$ is a pointwise inner product, then it must satisfy
\[
fg(X,Y)
= g(f \cdot X,Y)
= g(X \cdot f,Y)
= g(X,f \cdot Y)
= g(X,Y \cdot f)
\]
for all functions $f$ and vector fields $X,Y$, where $\cdot$ denotes the multiplication of functions and vector fields.
If $\langle X, Y \rangle = \int g(X,Y)$ is an inner product (for some notion of integral, such as a weighted sum over the vertices), then multiplication must be both associative and commutative\footnote{for example, 
$\langle fX - Xf, Y \rangle
= \int g(fX - Xf, Y) 
= \int fg(X,Y) - \int fg(X,Y)
= 0$ for all $Y$, so $fX = Xf$}, which in these combinatorial models it is not.
As such, combinatorial models cannot have a Riemannian metric.

\subsection{Geometric models}
The alternative class of models are those that admit a \boldblue{Riemannian metric}, which we will call \boldblue{geometric} models.
Most geometric models represent vector fields, forms, and tensors as pointwise vectors, and so are more aligned with the continuous notion of vector bundles.
For example, given $n$ data points in $\R^d$, a vector field could be represented as an $n \times d$ matrix, where each row is a $d$-dimensional vector representing the vector field at that data point.
This system easily admits a Riemannian metric by defining a pointwise inner product of these vectors.
Having a Riemannian metric means that geometric models must have an associative and commutative product, and so the vector fields, forms, and tensors are bimodules over the function space.
However, this means they cannot also obey the Leibniz rule, and so are weaker models for topology.

\subsubsection{Geometric models on a combinatorial complex}
Many geometric models still use a combinatorial complex.
When the data lie on a manifold $\M$ of dimension $d'$, the tangent space $T_p \M$ at each point $p$ is a $d'$-dimensional space, and given some choice of basis for each tangent space, we can represent vector fields as $nd'$-dimensional vectors.
This approach was applied to represent vector fields on \boldblue{triangulated manifolds} in computer graphics \cite{pedersen1995decorating, grimm1995modeling}, where the tangent space at each vertex is approximated by the plane of the adjacent faces.
Provided the tangent spaces at different points can be somehow aligned, this approach can be used to define a discrete parallel transport and vector field diffusion \cite{turk2001texture, wei2001texture, crane2010trivial, knoppel2013globally}, as well as the divergence and curl \cite{polthier2000variational, zhang2006vector, fisher2007design}.

An alternative approach to representing a vector field $X$ is to consider its \boldblue{operator form} as a directional derivative of functions, i.e. $f \mapsto X(f)$.
Given a basis for the function space, the vector fields can be identified with the space of square matrices.
This was introduced in \cite{azencot2013operator} for triangulated manifolds, with a function basis given by eigenfunctions of the graph Laplacian.
This choice of basis is a discrete analogue of the \boldblue{Fourier transform} of functions, and can be truncated to just the first $k$ eigenfunctions to obtain a low-rank \boldblue{compressed function space} of only the smoothest functions.
This approach avoids the need to estimate tangent spaces and leads straightforwardly to certain differential operators like the Lie bracket $[X,Y] = XY - YX$.
We classify this as a geometric model because the function multiplication is associative and commutative, even though there is not necessarily a Riemannian metric (since $X(f) = df(X)$ does not depend on a metric).
Given a metric, this approach can be used to define the Levi-Civita connection \cite{azencot2015discrete}.

This combinatorial setting has also seen the application of Bakry-Émery $\Gamma$-calculus to define a \boldblue{discrete carré du champ} operator using the graph Laplacian \cite{lin2010ricci, lin2011ricci}.
This has led to the development of a combinatorial Bakry-Émery theory that addresses the discrete analogues of functional inequalities \cite{jungel2017discrete, fathi2016entropic} and geometric bounds \cite{liu2018bakry, munch2018li, cushing2020bakry}.

\subsubsection{Diffusion methods on point clouds}

If the data are given only as a point cloud, then, even if they can safely be assumed to lie on a manifold, it is hard to construct a combinatorial complex on them in a way that the discrete geometry on that complex is statistically robust.
The dominant approach to robust geometry on point clouds is to use methods based on a \boldblue{kernel} evaluated on the data, such as a \boldblue{Gaussian kernel}.
These approximate the \boldblue{heat diffusion} on the underlying space, and approximate the \boldblue{Laplacian operator} when that space is a manifold \cite{belkin2003laplacian,COIFMAN20065,von2008consistency,hein2005graphs,belkin2006convergence,berry2016variable,garcia2020error}.
While the theoretical justification for these methods has currently only been developed for functions on manifolds, they have been extremely successful in practice when applied to non-manifold data such as images \cite{coifman2005geometric} and single-cell RNA data \cite{moon2019visualizing}.
Diffusion methods have mostly been applied to functions on the data, and less so for geometry (see the diffusion below for some exceptions), but their excellent statistical robustness and computational efficiency have made them popular tools for data analysis.
A notable application is in \boldblue{dimensionality reduction}\footnote{this is often called \qq{manifold learning}, but, like the \qq{manifold hypothesis}, really refers to low-dimensionality: the successful methods in this area are the ones that do not assume the data actually lie on a manifold (such as methods based on diffusion)}, where the eigenfunctions of the Laplacian are used to construct low-dimensional embeddings of the data \cite{coifman2005geometric,belkin2003laplacian,moon2019visualizing}.

\subsubsection{Geometric models on point clouds}

Given the limitations of constructing a combinatorial complex on point clouds, a major question is the development of \boldblue{complex-free geometric models} for vector fields, forms, and tensors on point clouds.

The most popular approach is to assume that the data lie on a manifold $\M$ of dimension $d'$, and explicitly estimate the tangent spaces $T_p \M$ at each point $p$ using \boldblue{local principal component analysis} (local PCA).
This gives a representation of vector fields as points in $\R^{nd'}$, where each $d'$-dimensional vector represents the vector field at that data point \cite{farmer1987predicting, kambhatla1997dimension, roweis2000nonlinear, donoho2003hessian, little2009estimation}.
This same idea can be used to represent the whole tensor algebra, including differential forms and general tensors, by taking the appropriate tensor products of the tangent spaces.
However, it is not clear how to define general differential operators in this framework\footnote{if we only want to compute these operators on the coordinate functions $x_1, ... , x_d$, then the picture is simpler: for example, at a point $p$, the $d' \times d$ matrix of local covariance eigenvectors that defines the tangent space is precisely the gradient operator $x_i \mapsto \nabla_p x_i$ applied to the coordinates}.
An important step in this direction was \boldblue{vector diffusion maps} (VDM) \cite{singer2012vector}, which uses a local alignment of the local PCA tangent spaces to define a parallel transport operator, and combines it with a kernel diffusion between points to define a vector field diffusion process.
The generator of this process approximates the connection Laplacian on vector fields when the data lie on a manifold \cite{singer2012vector,singer2017spectral}.
The local PCA tangent spaces can also be used to define other differential operators on functions, such as the Hessian \cite{donoho2003hessian}.

Just as vector fields on triangulated manifolds can be represented in operator form given a choice of function basis (typically eigenfunctions of the graph laplacian) \cite{azencot2013operator}, this idea can also be applied to point clouds (typically using eigenfunctions of the diffusion maps laplacian).
This is particularly popular in dynamical systems, where dynamical vector fields act as directional derivatives on functions \cite{giannakis2019data}.

A significant breakthrough in complex-free geometry was \boldblue{Spectral Exterior Calculus} (SEC), which builds a discrete representation of vector fields and forms on a manifold using only the eigenvalues and eigenfunctions of the Laplacian \cite{berry2020spectral}.
The method is inherently \textit{spectral}, and can be applied to the \textit{diffusion maps} Laplacian on point clouds \cite{COIFMAN20065} to compute the geometry of manifolds from data.
The authors observed that, if $\{ \p{i} \}_{i \le n_0}$ are the first $n_0$ eigenfunctions of the Laplacian, and we compute the $n_0 \times n_0 \times n_0$ multiplication tensor $c_{ijk} = \inp{\p{i} \p{j}}{\p{k}}$ described in \compnote{note: multiplication tensor}, then we can approximate the Riemannian metric $g(\nabla \p{i}, \nabla \p{j})$ using the carré du champ formula as
\begin{equation}
\label{eq: SEC cdc formula}
g(\nabla \p{i}, \nabla \p{j})
= \frac{1}{2} \left( \p{i} \Delta \p{j} + \p{j} \Delta \p{i} - \Delta (\p{i} \p{j}) \right)
\approx \frac{1}{2} \sum_{k=1}^{n_0} (\lambda_i + \lambda_j - \lambda_k) c_{ijk} \p{k},
\end{equation}
where $\lambda_i$ is the eigenvalue corresponding to $\p{i}$.
The approximation becomes exact when $n_0 = n$, and in this case the formula is equivalent to applying the discrete carré du champ operator of \cite{lin2010ricci} with the diffusion maps Laplacian \cite{COIFMAN20065} to the eigenfunctions $\p{i}, \p{j}$.
They showed that certain expressions from Riemannian geometry can be expressed via similar formulas involving only the eigenvalues, eigenfunctions, and the multiplication tensor.
This motivated the idea of representing vector fields and forms in an \boldblue{overcomplete spanning set} given by the eigenfunctions, such as $\{ \p{i} \nabla \p{j} \}_{i\le n_1, j \le n_2}$ for vector fields.
This allows the computation of differential operators via a weak formulation in global spaces, just like in FEEC, and is a significant departure from the pointwise construction of tangent spaces in local PCA.
In local PCA, the space of all ambient vector fields on the data $\R^{nd}$ is restricted to the $nd'$-dimensional subspace of vector fields that lie in the estimated tangent spaces (thereby excluding normal components), which forms a basis.
By contrast, the overcomplete spanning set in SEC does not exclude any normal vectors, but still gives meaningful results because those normal components are simply assigned a very small norm when computing inner products, and so do not contribute significantly in the weak formulations.
In other words, the locally low-dimensional structure of the data is \textit{implicit} in the Riemannian metric, rather than \textit{explicit} in the construction of a basis.
The authors applied this method to compute eigenforms of the Hodge Laplacian on 1-forms for data sampled from manifolds.
The SEC demonstrated, among many other things, two important ideas that we build on here.
First, it showed that the discrete carré du champ constructed from a diffusion kernel on the data gives a meaningful estimate of the Riemannian metric.
Second, it showed that, for data on manifolds, local tangent space estimates are not necessary and \textit{weak formulations are all you need} to compute differential operators, because they use the inner product and that already encodes the necessary information about the geometry.

These properties mean that SEC can be used directly as a tool to compute diffusion geometry, which removes the manifold assumption entirely and defines everything using only the carré du champ.
In \cite{jones2024diffusion}, we demonstrated this on various non-manifold datasets, and extended the SEC to compute higher order forms, and operators including the metric, derivative, codifferential, wedge product, and Lie bracket, and the SEC has been subsequently applied to represent dynamical systems \cite{das2025learning}.

However, the SEC has several limitations, which we have sought to address in this work.
First, it relies on computing the full multiplication tensor $c_{ijk}$ in $\ord(nn_0^3)$ time.
This cost can be constrained by taking a small $n_0$, but then the approximation error of multiplication increases.
Since the spectral approach involves expanding each expression in the eigenfunction basis, this error compounds and leads to the geometric expressions becoming inaccurate.
For example, the inner products are not positive definite, with Gram matrices often having negative eigenvalues of the same order of magnitude as the positive ones, leading to numerical instability.
This can be resolved by taking $n_0 = n$, but then the computational cost of $c_{ijk}$ becomes $\ord(n^4)$, which is prohibitively large.

Second, the SEC uses the eigenfunction gradients $\nabla \p{j}$ as generators for the spanning set $\{ \p{i} \nabla \p{j} \}_{i\le n_1, j \le n_2}$ for vector fields (and the tensor algebra generally).
All we require for $\nabla f_1, ..., \nabla f_m$ to generate a spanning set is that the functions $f_1, ..., f_m$ define an immersion into $\R^m$, and choosing the eigenfunction gradients $\nabla \p{j}$ has several drawbacks.
One is that the number of eigenfunctions required to create an immersion depends on the geometric complexity of the data, which we cannot know in advance, and could be arbitrarily large (for example, given multiple connected components).
A second is that the sequence of eigenfunctions will generally cover \q{larger} geometric features multiple times before addressing smaller ones.
Consider, for example, the unit circle in $\R^2$ with a small interval attached as $[1,1+\epsilon]\times\{0\}$.
The first three eigenfunctions will look like $\{1,\ \sin(\theta),\ \cos(\theta)\}$ and their gradients $\{0,\ \cos(\theta)\nabla\theta,\ \sin(\theta)\nabla\theta\}$ will generate the vector fields over the circle, but, if $\epsilon$ is small, it will take a very large number of eigenfunctions before one appears that varies along the small interval.
This will lead to high computational complexity, but also numerical instability, as the space of vector fields over the circle is highly redundant.
By comparison, using the ambient coordinates $\nabla x$ and $\nabla y$ is guaranteed to be always sufficient regardless of the geometric complexity.
In general, there are often better choices of immersion than the eigenfunctions of the Laplacian, such as the original ambient coordinates of the data in $\R^d$, a learned lower-dimensional immersion, or something domain-specific.

Third, the SEC is constrained to access the carré du champ by the spectral formula \eqref{eq: SEC cdc formula}, which depends on the accurate spectral estimation of the Laplacian.
In this work, we have followed the theory of diffusion geometry in treating the carré du champ as a general bilinear operator $\A\times\A\to\A$, which can be more easily composed to compute second derivatives, and also provides the flexibility to use different choices of operator.
For instance, we have favoured the mean-centred covariance formula for the carré du champ in \ref{sub: cdc}.
In general, we could obtain carré du champ operators from other kernel methods or using deep learning to scale these tools to high-dimensional data.
In the theoretical context, the convergence of the weak formulations in this work depends only on the convergence of the carré du champ, which exists in greater generality than the Laplacian, and so may be easier to prove in more general, non-manifold settings.

\subsection{Where does diffusion geometry fit?}

Our approach to computing diffusion geometry is a \boldblue{geometric model on a point cloud}, meaning it has a Riemannian metric, and the function-tensor multiplication is associative and commutative.
We follow the SEC in using a spanning set that is overcomplete, rather than a basis constructed pointwise, to represent vector fields, forms, and tensors.
We follow FEEC and SEC in using weak formulations to define differential operators.
Unlike the SEC, we compute the geometry using a diffusion-based \boldblue{carré du champ operator}, with which we can directly evaluate the Riemannian metric, instead of using spectral information from the Laplacian.
In particular, this operator is pointwise positive definite, and so we can define true inner products for vector fields, forms, and tensors.
It also allows us to use \boldblue{any choice of immersion} to generate the spanning set, which can encode inductive biases and reduce computational complexity.
By default, we use the original data coordinates in $\R^d$ as the immersion, which we found in \cite{jones2024manifold} to have a strong regularising effect on the geometry, improving robustness to noise and leading to state-of-the-art estimators for curvature.
The central role of the carré du champ operator mirrors its role in the theory of diffusion geometry \cite{jones2024diffusion}, and so we can directly apply methods like the Bakry-Émery Hessian formula \cite{bakry2014analysis} to compute geometric quantities like second derivatives, the Levi-Civita connection, and sectional curvature, as well as higher order tensors.

We compare the properties of our model to other geometric and topological models in the following table.
It is hard to fully compare the computational complexity, but as a simple comparison, we show the complexity of computing the product of a function and a vector field.
When models use a compressed function space, we denote its dimension by $n_0$.
We denote the ambient dimension by $d$, and then the immersion dimension (which may be $d$) by $\tilde{d}$.

\begin{table}[h!]
\renewcommand{\arraystretch}{1.3}
\newcolumntype{C}[1]{>{\centering\arraybackslash}p{#1}}

\begin{tabular}{l
                C{2em}
                C{2em}
                C{2em}
                C{2em}
                C{2em}
                C{2em}
                C{2em}
                C{2em}}
 & \rotatebox{60}{Riemannian metric}
 & \rotatebox{60}{Leibniz rule}
 & \rotatebox{60}{Manifold-specific}
 & \rotatebox{60}{Positive metric/inner products}
 & \rotatebox{60}{Directional derivatives}
 & \rotatebox{60}{Whole tensor algebra}
 & \rotatebox{60}{Uses immersion coords}
 & \rotatebox{60}{Product complexity} \\
\hline

Topological model
  & \boldblue{no} & \boldred{yes} & varies & \boldred{yes} & \boldblue{no} & \boldblue{no} & \boldblue{no} & varies \\

Geometric model on a complex
  & \boldred{yes} & \boldblue{no} & varies & \boldred{yes} & \boldblue{no} & \boldred{yes} & \boldblue{no} & varies \\

Operator representation on a complex
  & \boldblue{no} & \boldblue{no} & \boldred{yes} & \boldblue{no} & \boldred{yes} & \boldblue{no} & \boldblue{no} & $\ord(nn_0^3)$ \\

Local PCA
  & \boldred{yes} & \boldblue{no} & \boldred{yes} & \boldred{yes} & \boldblue{no} & \boldred{yes} & \boldred{yes} & $\ord(nd^3)$ \\

Operator representation on a point cloud
  & \boldblue{no} & \boldblue{no} & \boldblue{no} & \boldblue{no} & \boldred{yes} & \boldblue{no} & \boldblue{no} & $\ord(nn_0^3)$ \\

SEC (with truncated multiplication)
  & \boldred{yes} & \boldblue{no} & \boldblue{no} & \boldblue{no} & \boldred{yes} & \boldred{yes} & \boldblue{no} & $\ord(nn_0^3)$ \\

SEC (no truncation)
  & \boldred{yes} & \boldblue{no} & \boldblue{no} & \boldred{yes} & \boldred{yes} & \boldred{yes} & \boldblue{no} & $\ord(n^4)$ \\

Diffusion geometry
  & \boldred{yes} & \boldblue{no} & \boldblue{no} & \boldred{yes} & \boldred{yes} & \boldred{yes} & \boldred{yes} & $\ord(nn_0\tilde d)$ \\

\end{tabular}
\end{table}

\section{Conclusions}

Diffusion geometry reformulates classical Riemannian geometry in terms of a Markov process through its carré du champ operator, leading to a simpler and more general theory.
In this work, we have shown that the same principle applies to computation.
We used a heat kernel to compute a carré du champ operator on point clouds, and applied the method of weak formulations to estimate differential operators.
This allowed us to directly translate the theories of vector calculus, Riemannian geometry, geometric analysis, and differential topology into simple, scalable data analysis tools.

\section*{Acknowledgements}

This project has greatly benefited from conversations with Jacob Bamberger and Kelly Maggs.
We are also grateful to Jeff Giansiracusa and Yue Ren for their helpful comments.

\bibliographystyle{plain}
\bibliography{bibliography}

@book{bakry2014analysis,
  title={Analysis and geometry of Markov diffusion operators},
  author={Bakry, Dominique and Gentil, Ivan and Ledoux, Michel and others},
  volume={103},
  year={2014},
  publisher={Springer}
}

@article{carlsson2009topology,
  title={Topology and data},
  author={Carlsson, Gunnar},
  journal={Bulletin of the American Mathematical Society},
  volume={46},
  number={2},
  pages={255--308},
  year={2009}
}

@article{cayton2005algorithms,
  title={Algorithms for manifold learning},
  author={Cayton, Lawrence and others},
  journal={Univ. of California at San Diego Tech. Rep},
  volume={12},
  number={1-17},
  pages={1},
  year={2005}
}

@article{saadatfar2017pore,
  title={Pore configuration landscape of granular crystallization},
  author={Saadatfar, Mohammad and Takeuchi, Hiroshi and Robins, Vanessa and Francois, Nicolas and Hiraoka, Yasuaki},
  journal={Nature communications},
  volume={8},
  number={1},
  pages={15082},
  year={2017},
  publisher={Nature Publishing Group UK London}
}

@article{arns2004effect,
  title={Effect of network topology on relative permeability},
  author={Arns, Ji-Youn and Robins, Vanessa and Sheppard, Adrian P and Sok, Robert M and Pinczewski, Wolf Val and Knackstedt, Mark A},
  journal={Transport in Porous media},
  volume={55},
  number={1},
  pages={21--46},
  year={2004},
  publisher={Springer}
}

@inproceedings{cohen2005stability,
  title={Stability of persistence diagrams},
  author={Cohen-Steiner, David and Edelsbrunner, Herbert and Harer, John},
  booktitle={Proceedings of the twenty-first annual symposium on Computational geometry},
  pages={263--271},
  year={2005}
}

@article{herring2019topological,
  title={Topological persistence for relating microstructure and capillary fluid trapping in sandstones},
  author={Herring, AL and Robins, Vanessa and Sheppard, AP},
  journal={Water Resources Research},
  volume={55},
  number={1},
  pages={555--573},
  year={2019},
  publisher={Wiley Online Library}
}

@article{jones2024manifold,
  title={Manifold Diffusion Geometry: Curvature, Tangent Spaces, and Dimension},
  author={Jones, Iolo},
  journal={arXiv preprint arXiv:2411.04100},
  year={2024}
}

@inproceedings{little2009estimation,
  title={Estimation of intrinsic dimensionality of samples from noisy low-dimensional manifolds in high dimensions with multiscale SVD},
  author={Little, Anna V and Lee, Jason and Jung, Yoon-Mo and Maggioni, Mauro},
  booktitle={2009 IEEE/SP 15th Workshop on Statistical Signal Processing},
  pages={85--88},
  year={2009},
  organization={IEEE}
}

@article{singer2017spectral,
  title={Spectral convergence of the connection Laplacian from random samples},
  author={Singer, Amit and Wu, Hau-Tieng},
  journal={Information and Inference: A Journal of the IMA},
  volume={6},
  number={1},
  pages={58--123},
  year={2017},
  publisher={Oxford University Press}
}

@incollection{dodziuk1996riemannian,
  title={Riemannian structures and triangulations of manifolds},
  author={Dodziuk, Jozef and Patodi, Vijay Kumar},
  booktitle={Collected Papers of VK Patodi},
  pages={232--283},
  year={1996},
  publisher={World Scientific}
}

@book{whitney2012geometric,
  title={Geometric integration theory},
  author={Whitney, Hassler},
  year={2012},
  publisher={Courier Corporation}
}

@book{hirani2003discrete,
  title={Discrete exterior calculus},
  author={Hirani, Anil Nirmal},
  year={2003},
  publisher={California Institute of Technology}
}

@book{mallat1999wavelet,
  title={A wavelet tour of signal processing},
  author={Mallat, Stephane},
  year={1999},
  publisher={Academic Press}
}

@article{farmer1987predicting,
  title={Predicting chaotic time series},
  author={Farmer, J Doyne and Sidorowich, John J},
  journal={Physical review letters},
  volume={59},
  number={8},
  pages={845},
  year={1987},
  publisher={APS}
}

@article{giannakis2019data,
  title={Data-driven spectral decomposition and forecasting of ergodic dynamical systems},
  author={Giannakis, Dimitrios},
  journal={Applied and Computational Harmonic Analysis},
  volume={47},
  number={2},
  pages={338--396},
  year={2019},
  publisher={Elsevier}
}

@incollection{chen2022discrete,
  title={Discrete Hessian complexes in three dimensions},
  author={Chen, Long and Huang, Xuehai},
  booktitle={The virtual element method and its applications},
  pages={93--135},
  year={2022},
  publisher={Springer}
}

@article{berwick2021discrete,
  title={Discrete vector bundles with connection and the Bianchi identity},
  author={Berwick-Evans, Daniel and Hirani, Anil N and Schubel, Mark D},
  journal={arXiv preprint arXiv:2104.10277},
  year={2021}
}

@article{dodziuk1976finite,
  title={Finite-difference approach to the Hodge theory of harmonic forms},
  author={Dodziuk, Jozef},
  journal={American Journal of Mathematics},
  volume={98},
  number={1},
  pages={79--104},
  year={1976},
  publisher={JSTOR}
}

@article{roweis2000nonlinear,
  title={Nonlinear dimensionality reduction by locally linear embedding},
  author={Roweis, Sam T and Saul, Lawrence K},
  journal={science},
  volume={290},
  number={5500},
  pages={2323--2326},
  year={2000},
  publisher={American Association for the Advancement of Science}
}

@article{christiansen2018generalized,
  title={Generalized finite element systems for smooth differential forms and Stokes’ problem},
  author={Christiansen, Snorre H and Hu, Kaibo},
  journal={Numerische Mathematik},
  volume={140},
  number={2},
  pages={327--371},
  year={2018},
  publisher={Springer}
}

@inproceedings{pedersen1995decorating,
  title={Decorating implicit surfaces},
  author={Pedersen, Hans K{\o}hling},
  booktitle={Proceedings of the 22nd annual conference on Computer graphics and interactive techniques},
  pages={291--300},
  year={1995}
}

@inproceedings{turk2001texture,
  title={Texture synthesis on surfaces},
  author={Turk, Greg},
  booktitle={Proceedings of the 28th annual conference on Computer graphics and interactive techniques},
  pages={347--354},
  year={2001}
}

@inproceedings{wei2001texture,
  title={Texture synthesis over arbitrary manifold surfaces},
  author={Wei, Li-Yi and Levoy, Marc},
  booktitle={Proceedings of the 28th annual conference on Computer graphics and interactive techniques},
  pages={355--360},
  year={2001}
}

@article{zhang2006vector,
  title={Vector field design on surfaces},
  author={Zhang, Eugene and Mischaikow, Konstantin and Turk, Greg},
  journal={ACM Transactions on Graphics (ToG)},
  volume={25},
  number={4},
  pages={1294--1326},
  year={2006},
  publisher={ACM New York, NY, USA}
}

@inproceedings{grimm1995modeling,
  title={Modeling surfaces of arbitrary topology using manifolds},
  author={Grimm, Cindy M and Hughes, John F},
  booktitle={Proceedings of the 22nd annual conference on Computer graphics and interactive techniques},
  pages={359--368},
  year={1995}
}

@inproceedings{crane2010trivial,
  title={Trivial connections on discrete surfaces},
  author={Crane, Keenan and Desbrun, Mathieu and Schr{\"o}der, Peter},
  booktitle={Computer Graphics Forum},
  volume={29},
  number={5},
  pages={1525--1533},
  year={2010},
  organization={Wiley Online Library}
}

@inproceedings{polthier2000variational,
  title={Variational approach to vector field decomposition},
  author={Polthier, Konrad and Preu{\ss}, Eike},
  booktitle={Data Visualization 2000: Proceedings of the Joint EUROGRAPHICS and IEEE TCVG Symposium on Visualization in Amsterdam, The Netherlands, May 29--30, 2000},
  pages={147--155},
  year={2000},
  organization={Springer}
}

@article{fisher2007design,
  title={Design of tangent vector fields},
  author={Fisher, Matthew and Schr{\"o}der, Peter and Desbrun, Mathieu and Hoppe, Hugues},
  journal={ACM transactions on graphics (TOG)},
  volume={26},
  number={3},
  pages={56--es},
  year={2007},
  publisher={ACM New York, NY, USA}
}

@article{knoppel2013globally,
  title={Globally optimal direction fields},
  author={Kn{\"o}ppel, Felix and Crane, Keenan and Pinkall, Ulrich and Schr{\"o}der, Peter},
  journal={ACM Transactions on Graphics (ToG)},
  volume={32},
  number={4},
  pages={1--10},
  year={2013},
  publisher={ACM New York, NY, USA}
}

@inproceedings{azencot2013operator,
  title={An operator approach to tangent vector field processing},
  author={Azencot, Omri and Ben-Chen, Mirela and Chazal, Fr{\'e}d{\'e}ric and Ovsjanikov, Maks},
  booktitle={Computer Graphics Forum},
  volume={32},
  number={5},
  pages={73--82},
  year={2013},
  organization={Wiley Online Library}
}

@inproceedings{de2014discrete,
  title={Discrete 2-tensor fields on triangulations},
  author={de Goes, Fernando and Liu, Beibei and Budninskiy, Max and Tong, Yiying and Desbrun, Mathieu},
  booktitle={Computer Graphics Forum},
  volume={33},
  number={5},
  pages={13--24},
  year={2014},
  organization={Wiley Online Library}
}

@article{belkin2003laplacian,
  title={Laplacian eigenmaps for dimensionality reduction and data representation},
  author={Belkin, Mikhail and Niyogi, Partha},
  journal={Neural computation},
  volume={15},
  number={6},
  pages={1373--1396},
  year={2003},
  publisher={MIT Press}
}

@article{von2008consistency,
  title={Consistency of spectral clustering},
  author={Von Luxburg, Ulrike and Belkin, Mikhail and Bousquet, Olivier},
  journal={The Annals of Statistics},
  pages={555--586},
  year={2008},
  publisher={JSTOR}
}

@inproceedings{hein2005graphs,
  title={From graphs to manifolds--weak and strong pointwise consistency of graph Laplacians},
  author={Hein, Matthias and Audibert, Jean-Yves and Von Luxburg, Ulrike},
  booktitle={International Conference on Computational Learning Theory},
  pages={470--485},
  year={2005},
  organization={Springer}
}

@article{trillos2018error,
  title={Error estimates for spectral convergence of the graph Laplacian on random geometric graphs towards the Laplace--Beltrami operator},
  author={Trillos, Nicolas Garcia and Gerlach, Moritz and Hein, Matthias and Slepcev, Dejan},
  journal={arXiv preprint arXiv:1801.10108},
  year={2018}
}

@article{belkin2006convergence,
  title={Convergence of Laplacian eigenmaps},
  author={Belkin, Mikhail and Niyogi, Partha},
  journal={Advances in neural information processing systems},
  volume={19},
  year={2006}
}

@article{trillos2018variational,
  title={A variational approach to the consistency of spectral clustering},
  author={Trillos, Nicolas Garcia and Slep{\v{c}}ev, Dejan},
  journal={Applied and Computational Harmonic Analysis},
  volume={45},
  number={2},
  pages={239--281},
  year={2018},
  publisher={Elsevier}
}

@article{shi2015convergence,
  title={Convergence of Laplacian spectra from random samples},
  author={Shi, Zuoqiang},
  journal={arXiv preprint arXiv:1507.00151},
  year={2015}
}

@article{dunson2021spectral,
  title={Spectral convergence of graph Laplacian and heat kernel reconstruction in $L^\infty$ from random samples},
  author={Dunson, David B and Wu, Hau-Tieng and Wu, Nan},
  journal={Applied and Computational Harmonic Analysis},
  volume={55},
  pages={282--336},
  year={2021},
  publisher={Elsevier}
}

@article{calder2022improved,
  title={Improved spectral convergence rates for graph Laplacians on $\varepsilon$-graphs and k-NN graphs},
  author={Calder, Jeff and Trillos, Nicolas Garcia},
  journal={Applied and Computational Harmonic Analysis},
  volume={60},
  pages={123--175},
  year={2022},
  publisher={Elsevier}
}

@article{garcia2020error,
  title={Error estimates for spectral convergence of the graph Laplacian on random geometric graphs toward the Laplace--Beltrami operator},
  author={Garc{\'\i}a Trillos, Nicol{\'a}s and Gerlach, Moritz and Hein, Matthias and Slep{\v{c}}ev, Dejan},
  journal={Foundations of Computational Mathematics},
  volume={20},
  number={4},
  pages={827--887},
  year={2020},
  publisher={Springer}
}

@article{moon2019visualizing,
  title={Visualizing structure and transitions in high-dimensional biological data},
  author={Moon, Kevin R and Van Dijk, David and Wang, Zheng and Gigante, Scott and Burkhardt, Daniel B and Chen, William S and Yim, Kristina and Elzen, Antonia van den and Hirn, Matthew J and Coifman, Ronald R and others},
  journal={Nature biotechnology},
  volume={37},
  number={12},
  pages={1482--1492},
  year={2019},
  publisher={Nature Publishing Group US New York}
}

@article{coifman2005geometric,
  title={Geometric diffusions as a tool for harmonic analysis and structure definition of data: Diffusion maps},
  journal={Proceedings of the national academy of sciences},
  volume={102},
  number={21},
  pages={7426--7431},
  year={2005},
  publisher={National Academy of Sciences}
}

@article{lin2010ricci,
  title={Ricci curvature and eigenvalue estimate on locally finite graphs},
  author={Lin, Yong and Yau, Shing-Tung},
  journal={Mathematical research letters},
  volume={17},
  number={2},
  pages={343--356},
  year={2010},
  publisher={International Press of Boston}
}

@article{lin2011ricci,
  title={Ricci curvature of graphs},
  author={Lin, Yong and Lu, Linyuan and Yau, Shing-Tung},
  journal={Tohoku Mathematical Journal, Second Series},
  volume={63},
  number={4},
  pages={605--627},
  year={2011},
  publisher={Mathematical Institute, Tohoku University}
}

@article{jungel2017discrete,
  title={Discrete Beckner inequalities via the Bochner--Bakry--Emery approach for Markov chains},
  author={J{\"u}ngel, Ansgar and Yue, Wen},
  year={2017}
}

@article{fathi2016entropic,
  title={Entropic Ricci curvature bounds for discrete interacting systems},
  author={Fathi, Max and Maas, Jan},
  year={2016}
}

@article{liu2018bakry,
  title={Bakry--{\'E}mery curvature and diameter bounds on graphs},
  author={Liu, Shiping and M{\"u}nch, Florentin and Peyerimhoff, Norbert},
  journal={Calculus of Variations and Partial Differential Equations},
  volume={57},
  number={2},
  pages={67},
  year={2018},
  publisher={Springer}
}

@article{cushing2020bakry,
  title={Bakry--{\'E}mery curvature functions on graphs},
  author={Cushing, David and Liu, Shiping and Peyerimhoff, Norbert},
  journal={Canadian Journal of Mathematics},
  volume={72},
  number={1},
  pages={89--143},
  year={2020},
  publisher={Canadian Mathematical Society}
}

@article{munch2018li,
  title={Li--Yau inequality on finite graphs via non-linear curvature dimension conditions},
  author={M{\"u}nch, Florentin},
  journal={Journal de Math{\'e}matiques Pures et Appliqu{\'e}es},
  volume={120},
  pages={130--164},
  year={2018},
  publisher={Elsevier}
}

@article{das2025learning,
  title={Learning Dynamical Systems with the Spectral Exterior Calculus},
  author={Das, Suddhasattwa and Giannakis, Dimitrios and Gu, Yanbing and Slawinska, Joanna},
  journal={arXiv preprint arXiv:2505.06061},
  year={2025}
}

@article{berry2015nonparametric,
  title={Nonparametric forecasting of low-dimensional dynamical systems},
  author={Berry, Tyrus and Giannakis, Dimitrios and Harlim, John},
  journal={Physical Review E},
  volume={91},
  number={3},
  pages={032915},
  year={2015},
  publisher={APS}
}

@article{berry2016forecasting,
  title={Forecasting turbulent modes with nonparametric diffusion models: Learning from noisy data},
  author={Berry, Tyrus and Harlim, John},
  journal={Physica D: Nonlinear Phenomena},
  volume={320},
  pages={57--76},
  year={2016},
  publisher={Elsevier}
}

@article{yarmola2010persistence,
  title={Persistence and computation of the cup product},
  author={Yarmola, Andrew},
  journal={Undergraduate honors thesis. Stanford University},
  year={2010}
}

@article{contessoto2021persistent,
  title={Persistent cup-length},
  author={Contessoto, Marco and M{\'e}moli, Facundo and Stefanou, Anastasios and Zhou, Ling},
  journal={arXiv preprint arXiv:2107.01553},
  year={2021}
}

@article{memoli2024persistent,
  title={Persistent cup product structures and related invariants},
  author={M{\'e}moli, Facundo and Stefanou, Anastasios and Zhou, Ling},
  journal={Journal of Applied and Computational Topology},
  volume={8},
  number={1},
  pages={93--148},
  year={2024},
  publisher={Springer}
}

@inproceedings{de2009persistent,
  title={Persistent cohomology and circular coordinates},
  author={De Silva, Vin and Vejdemo-Johansson, Mikael},
  booktitle={Proceedings of the twenty-fifth annual symposium on Computational geometry},
  pages={227--236},
  year={2009}
}

@inproceedings{perea2020sparse,
  title={Sparse circular coordinates via principal $\mathbb{Z}$-bundles},
  author={Perea, Jose A},
  booktitle={Topological Data Analysis: The Abel Symposium 2018},
  pages={435--458},
  year={2020},
  organization={Springer}
}

@phdthesis{yim2021local,
  title={Local inference of Morse indices using finite point samples},
  author={Yim, Ka Man},
  year={2021},
  school={University of Oxford}
}

@article{smale2009geometry,
  title={Geometry on probability spaces},
  author={Smale, Steve and Zhou, Ding-Xuan},
  journal={Constructive Approximation},
  volume={30},
  number={3},
  pages={311--323},
  year={2009},
  publisher={Springer}
}

@article{COIFMAN20065,
title = {Diffusion maps},
journal = {Applied and Computational Harmonic Analysis},
volume = {21},
number = {1},
pages = {5-30},
year = {2006},
note = {Special Issue: Diffusion Maps and Wavelets},
issn = {1063-5203},
doi = {https://doi.org/10.1016/j.acha.2006.04.006},
url = {https://www.sciencedirect.com/science/article/pii/S1063520306000546},
author = {Ronald R. Coifman and Stéphane Lafon}
}

@article{gine2006empirical,
  title={Empirical graph Laplacian approximation of Laplace-Beltrami operators: large sample results},
  author={Gin{\'e}, Evarist and Koltchinskii, Vladimir},
  journal={Lecture Notes-Monograph Series},
  pages={238--259},
  year={2006},
  publisher={JSTOR}
}

@article{belkin2008towards,
  title={Towards a theoretical foundation for Laplacian-based manifold methods},
  author={Belkin, Mikhail and Niyogi, Partha},
  journal={Journal of Computer and System Sciences},
  volume={74},
  number={8},
  pages={1289--1308},
  year={2008},
  publisher={Elsevier}
}

@article{berry2020spectral,
  title={Spectral exterior calculus},
  author={Berry, Tyrus and Giannakis, Dimitrios},
  journal={Communications on Pure and Applied Mathematics},
  volume={73},
  number={4},
  pages={689--770},
  year={2020},
  publisher={Wiley Online Library}
}

@article{robins1999towards,
  title={Towards computing homology from finite approximations},
  author = {Vanessa Robins},
  journal = {Topology proceedings},
  volume = {24},
  number = {1},
  pages = {503-532},
  year = {1999},
}

@article{edelsbrunner2002topological,
  title={Topological persistence and simplification},
  author={Edelsbrunner and Letscher and Zomorodian},
  journal={Discrete \& computational geometry},
  volume={28},
  pages={511--533},
  year={2002},
  publisher={Springer}
}

@article{bauer2021ripser,
  title={Ripser: efficient computation of Vietoris--Rips persistence barcodes},
  author={Bauer, Ulrich},
  journal={Journal of Applied and Computational Topology},
  volume={5},
  number={3},
  pages={391--423},
  year={2021},
  publisher={Springer}
}

@misc{bakry1985seminaire,
  title={S{\'e}minaire de probabilit{\'e}s XIX 1983/84},
  author={Bakry, Dominique and {\'E}mery, M},
  journal={Diffusions hypercontractives},
  volume={1123},
  pages={177--206},
  year={1985},
  publisher={Springer Berlin}
}

@article{kambhatla1997dimension,
  title={Dimension reduction by local principal component analysis},
  author={Kambhatla, Nandakishore and Leen, Todd K},
  journal={Neural computation},
  volume={9},
  number={7},
  pages={1493--1516},
  year={1997},
  publisher={MIT Press One Rogers Street, Cambridge, MA 02142-1209, USA journals-info~…}
}

@article{donoho2003hessian,
  title={Hessian eigenmaps: Locally linear embedding techniques for high-dimensional data},
  author={Donoho, David L and Grimes, Carrie},
  journal={Proceedings of the National Academy of Sciences},
  volume={100},
  number={10},
  pages={5591--5596},
  year={2003},
  publisher={National Acad Sciences}
}

@article{singer2012vector,
  title={Vector diffusion maps and the connection Laplacian},
  author={Singer, Amit and Wu, H-T},
  journal={Communications on pure and applied mathematics},
  volume={65},
  number={8},
  pages={1067--1144},
  year={2012},
  publisher={Wiley Online Library}
}

@article{desbrun2005discrete,
  title={Discrete exterior calculus},
  author={Desbrun, Mathieu and Hirani, Anil N and Leok, Melvin and Marsden, Jerrold E},
  journal={arXiv preprint math/0508341},
  year={2005}
}

@article{arnold2006finite,
  title={Finite element exterior calculus, homological techniques, and applications},
  author={Arnold, Douglas N and Falk, Richard S and Winther, Ragnar},
  journal={Acta numerica},
  volume={15},
  pages={1--155},
  year={2006},
  publisher={Cambridge University Press}
}

@article{arnold2010finite,
  title={Finite element exterior calculus: from Hodge theory to numerical stability},
  author={Arnold, Douglas and Falk, Richard and Winther, Ragnar},
  journal={Bulletin of the American mathematical society},
  volume={47},
  number={2},
  pages={281--354},
  year={2010}
}

@article{bartholdi2012hodge,
  title={Hodge theory on metric spaces},
  author={Bartholdi, Laurent and Schick, Thomas and Smale, Nat and Smale, Steve},
  journal={Foundations of Computational Mathematics},
  volume={12},
  pages={1--48},
  year={2012},
  publisher={Springer}
}

@inproceedings{zomorodian2004computing,
  title={Computing persistent homology},
  author={Zomorodian, Afra and Carlsson, Gunnar},
  booktitle={Proceedings of the twentieth annual symposium on Computational geometry},
  pages={347--356},
  year={2004}
}

@article{berry2016variable,
  title={Variable bandwidth diffusion kernels},
  author={Berry, Tyrus and Harlim, John},
  journal={Applied and Computational Harmonic Analysis},
  volume={40},
  number={1},
  pages={68--96},
  year={2016},
  publisher={Elsevier}
}

@article{jones2024diffusion,
  title={Diffusion Geometry},
  author={Jones, Iolo},
  journal={arXiv preprint arXiv:2405.10858},
  year={2024}
}

@article{azencot2015discrete,
  title={Discrete derivatives of vector fields on surfaces--an operator approach},
  author={Azencot, Omri and Ovsjanikov, Maks and Chazal, Fr{\'e}d{\'e}ric and Ben-Chen, Mirela},
  journal={ACM Transactions on Graphics (TOG)},
  volume={34},
  number={3},
  pages={1--13},
  year={2015},
  publisher={ACM New York, NY, USA}
}

@article{sritharan2021computing,
  title={Computing the Riemannian curvature of image patch and single-cell RNA sequencing data manifolds using extrinsic differential geometry},
  author={Sritharan, Duluxan and Wang, Shu and Hormoz, Sahand},
  journal={Proceedings of the National Academy of Sciences},
  volume={118},
  number={29},
  pages={e2100473118},
  year={2021},
  publisher={National Acad Sciences}
}

@article{hickok2023intrinsic,
  title={An Intrinsic Approach to Scalar-Curvature Estimation for Point Clouds},
  author={Hickok, Abigail and Blumberg, Andrew J},
  journal={arXiv preprint arXiv:2308.02615},
  year={2023}
}

@article{ali2023survey,
  title={A survey of vectorization methods in topological data analysis},
  author={Ali, Dashti and Asaad, Aras and Jimenez, Maria-Jose and Nanda, Vidit and Paluzo-Hidalgo, Eduardo and Soriano-Trigueros, Manuel},
  journal={IEEE Transactions on Pattern Analysis and Machine Intelligence},
  volume={45},
  number={12},
  pages={14069--14080},
  year={2023},
  publisher={IEEE}
}

@article{bamberger2025carr,
  title={Carr$\backslash$'e du champ flow matching: better quality-generalisation tradeoff in generative models},
  author={Bamberger, Jacob and Jones, Iolo and Duncan, Dennis and Bronstein, Michael M and Vandergheynst, Pierre and Gosztolai, Adam},
  journal={arXiv preprint arXiv:2510.05930},
  year={2025}
}

\appendix

\refstepcounter{section}
\section*{\thesection\hspace{0.5em} Markov chain properties}
\label{sec:markov-appendix}

\subsection*{Carré du champ as infinitesimal covariance}

We now give a proof of Proposition \ref{prop: covarianceeuclideanformula} from Section \ref{sec: overview} about the carré du champ operator for the heat diffusion on $\R^d$.
The analogous result holds for arbitrary Markov processes, i.e. if $P_t : \A \to \A$ is a Markov semigroup with carré du champ operator $\Gamma$ and transition kernel $p_t(x,dy)$, then
\begin{align*}
\Gamma(f,h)(p)
&= \lim_{t \to 0} \frac{1}{2t} \left( P_t(fh)(p) - P_t(f)(p) P_t(h)(p) \right) \\
&= \lim_{t \to 0} \frac{1}{2t} \left( \E[f(X)h(X) : X \sim p_t(p,\cdot)] - \E[f(X) : X \sim p_t(p,\cdot)] \E[h(X) : X \sim p_t(p,\cdot)] \right) \\
&= \lim_{t \to 0} \frac{1}{2t} \textnormal{Cov}\!\left[f(X), h(X) : X \sim p_t(p,\cdot)\right].
\end{align*}
See \cite{bakry2014analysis} for a proof of this general case.

\covarianceeuclideanformula*

\begin{proof}
We can Taylor expand
\[
f(y) = f(x) + (y - x) \cdot \nabla f + \ord(\|y-x\|^2),
\qquad
h(y) = h(x) + (y - x) \cdot \nabla h + \ord(\|y-x\|^2),
\]
so that, up to higher order terms,
\[
(f(y) - f(x)) (h(y) - h(x))
= \big((y-x) \cdot \nabla f \big)\big((y-x) \cdot \nabla h \big).
\]
By standard properties of the trace,
\[
\big((y-x) \cdot \nabla f \big)\big((y-x) \cdot \nabla h \big)
= \textnormal{tr}\!\big(\nabla f (\nabla h)^\top (y-x)(y-x)^\top\big),
\]
so
\begin{align*}
\int p_t(p,y) (f(y) - f(p)) (h(y) - h(p)) \, dy
&= \E\!\left[(f(X) - f(p))(h(X) - h(p)) : X \sim \mathcal{N}(p,2t\textbf{I})\right] \\
&= \E\!\left[\textnormal{tr}\!\big(\nabla f (\nabla h)^\top (X-p)(X-p)^\top\big) : X \sim \mathcal{N}(p,2t\textbf{I})\right] \\
&= 2t \, \textnormal{tr}(\nabla f (\nabla h)^\top) \\
&= 2t \, \nabla f \cdot \nabla h.
\end{align*}
Dividing by $2t$ gives the result.  
Finally, since $\E[f(X)] - f(p) = \ord(t)$ and similarly for $h$, replacing the expectation by the covariance gives the same limit
\[
\frac{1}{2t}\textnormal{Cov}\!\left[f(X), h(X) : X \sim \mathcal{N}(p,2t\textbf{I})\right]
\;\to\;
\nabla f \cdot \nabla h.
\]
as $t \to 0$.
\end{proof}

\subsection*{Symmetric kernels are self-adjoint Markov chains}

We now show that the Markov chains we use are self-adjoint if and only if they are the row-normalisations of symmetric kernel matrices.

\begin{prop}[Markov chain properties]
\label{prop: markov chain properties}
The following are equivalent:
\begin{enumerate}
    \item $\textbf{P}$ is constructed from a symmetric kernel matrix $\textbf{K}$ by row-normalising,
    \item $\textbf{P}$ is \boldblue{reversible} with respect to $\boldsymbol{\mu}$, meaning it satisfies the \q{detailed balance} condition
    \[
    \boldsymbol{\mu}_i \textbf{P}_{ij} = \boldsymbol{\mu}_j \textbf{P}_{ji}
    \]
    for all $i,j$,
    \item the linear operator $\textbf{P}: L^2(\textbf{A}, \boldsymbol{\mu}) \to L^2(\textbf{A}, \boldsymbol{\mu})$ is self-adjoint.
\end{enumerate}
If these conditions hold, then $\boldsymbol{\mu}$ is a stationary distribution for $\textbf{P}$.
\end{prop}

\begin{proof}
$(1) \implies (2)$: 
Assume that $\textbf{P}_{ij} = \textbf{K}_{ij}/D_i$ where $\textbf{K}$ is symmetric and $D_i = \sum_{j=1}^n \textbf{K}_{ij}$ are its row sums, so $\boldsymbol{\mu}_i = D_i / \sum_k D_k$.
Then
\[
\boldsymbol{\mu}_i \textbf{P}_{ij}
= \frac{\textbf{K}_{ij}}{\sum_k D_k}
= \frac{\textbf{K}_{ji}}{\sum_k D_k}
= \boldsymbol{\mu}_j \textbf{P}_{ji}
\]
so $\textbf{P}$ is reversible.

$(2) \implies (1)$: If $\textbf{P}$ is reversible, we can define a symmetric kernel $\textbf{K}_{ij} = \boldsymbol{\mu}_i \textbf{P}_{ij}$.
Let 
\[
D_i 
= \sum_j \textbf{K}_{ij} 
= \sum_j \boldsymbol{\mu}_i \textbf{P}_{ij} 
= \boldsymbol{\mu}_i.
\]
Then the Markov chain induced by $\textbf{K}$ is 
$
\textbf{K}_{ij}/D_i 
= (\boldsymbol{\mu}_i\textbf{P}_{ij})/\boldsymbol{\mu}_i 
= \textbf{P}_{ij}
$,
so $P$ is induced by a symmetric kernel.

$(2) \implies (3)$: Assume $\textbf{P}$ is reversible, so $\boldsymbol{\mu}_i \textbf{P}_{ij} = \boldsymbol{\mu}_j \textbf{P}_{ji}$.
Then
\[
\langle \textbf{P}\mathbf{f}, \mathbf{g} \rangle_{\boldsymbol{\mu}} 
= \sum_{i} (\textbf{P}\mathbf{f})_i \mathbf{g}_i \boldsymbol{\mu}_i
= \sum_{i,j} \textbf{P}_{ij} \mathbf{f}_j \mathbf{g}_i \boldsymbol{\mu}_i
= \sum_{i,j} \mathbf{f}_j \mathbf{g}_i ( \boldsymbol{\mu}_i \textbf{P}_{ij} )
\]
is symmetric in $\mathbf{f}$ and $\mathbf{g}$, so $\textbf{P}$ is self-adjoint as a linear operator $L^2(\textbf{A}, \boldsymbol{\mu}) \to L^2(\textbf{A}, \boldsymbol{\mu})$.

$(3) \implies (2)$: 
Notice that
\[
\langle \textbf{P}\mathbf{e}_k, \mathbf{e}_l \rangle_{\boldsymbol{\mu}} 
= \sum_{i,j} \textbf{P}_{ij} (\mathbf{e}_k)_j (\mathbf{e}_l)_i \boldsymbol{\mu}_i 
= \sum_{i,j} \textbf{P}_{ij} \delta_{jk} \delta_{il}\boldsymbol{\mu}_i 
= P_{kl} \boldsymbol{\mu}_l.
\]
If $\textbf{P}$ is self-adjoint, then
\(
\textbf{P}_{kl} \boldsymbol{\mu}_l
= \langle \textbf{P}\mathbf{e}_k, \mathbf{e}_l \rangle_{\boldsymbol{\mu}} 
= \langle \textbf{P}\mathbf{e}_l, \mathbf{e}_k \rangle_{\boldsymbol{\mu}} 
= \textbf{P}_{lk} \boldsymbol{\mu}_k
\)
for all $k,l$, so $\textbf{P}$ is reversible.

Finally, we see that reversibility implies that $\boldsymbol{\mu}$ is a stationary distribution, because
\[
(\boldsymbol{\mu}^T \textbf{P})_i
= \sum_j \boldsymbol{\mu}_j \textbf{P}_{ji}
= \sum_j \boldsymbol{\mu}_i \textbf{P}_{ij}
= \boldsymbol{\mu}_i
\]
for all $i$.
\end{proof}


This validates why the symmetric kernel approach is reasonably natural.
We could use a general Markov chain to compute diffusion geometry, and use its stationary distribution as the measure for an $L^2$ space of functions, but the heat diffusion operator $P$ would then not be self-adjoint.
The self-adjointness is theoretically expected, and also leads to a computational speedup when we try to find its eigenvectors.

\refstepcounter{section}
\section*{\thesection\hspace{0.5em} Operator forms of tensors}
\label{sec:operator-forms-appendix}

\subsection*{Operator form of the Hessian}

We can interpret the Hessian as a bilinear operator $\mathfrak{X}(M) \times \mathfrak{X}(M) \rightarrow \A$ via
$$
H(f)(X,Y) := H(f)(X\otimes Y) = g(H(f), X^\flat \otimes Y^\flat).
$$
In other words, the \q{operator form} of the Hessian is a map from $\A$ to the bilinear operators $\mathfrak{X}(M) \times \mathfrak{X}(M) \rightarrow \A$.
We can discretise this as an $n \times n_1d \times n_1d \times n_0$ 4-tensor $\textbf{H}^{\textnormal{op}}$ that maps functions in $\R^{n_0}$ to bilinear maps $\R^{n_1d} \times \R^{n_1d} \to \R^n$.
We first evaluate the action of $H(\p{i'})$ on the pair of vector fields $\p{i_1}\nabla x_{j_1}$ and $\p{i_2}\nabla x_{j_2}$ at a point $p$ the 6-tensor
\begin{equation*}
\label{eq: op Hessian}
\begin{split}
\textbf{H}^{\textnormal{op}}_{pi_1j_1i_2j_2i'}
&= H(\p{i'})(\p{i_1}\nabla x_{j_1}, \p{i_2}\nabla x_{j_2})(p) \\
&= \p{i_1}(p)\p{i_2}(p)H(\p{i'})(\nabla x_{j_1}, \nabla x_{j_2})(p) \\
&= \frac{1}{2} \p{i_1}(p)\p{i_2}(p)
\big[ \Gamma(x_{j_1}, \Gamma(x_{j_2}, \p{i'}))(p)
+ \Gamma(x_{j_2}, \Gamma(x_{j_1}, \p{i'}))(p)
- \Gamma(\p{i'}, \Gamma(x_{j_1}, x_{j_2}))(p)
\big] \\
&= \frac{1}{2}\textbf{U}_{pi_1} \textbf{U}_{pi_2}
\big[ 
\boldsymbol{\Gamma}^\textnormal{coord/mix}_{pj_1j_2i'}
+ \boldsymbol{\Gamma}^\textnormal{coord/mix}_{pj_2j_1i'}
- \boldsymbol{\Gamma}^\textnormal{mix/coord}_{pi'j_1j_2}
\big]
\end{split}
\end{equation*}
which has dimension $n \times n_1 \times d \times n_1 \times d \times n_0$.
We reshape $\textbf{H}^{\textnormal{op}}$ into an $n \times n_1d \times n_1d \times n_0$ 4-tensor, so, if $\textbf{f} \in \R^{n_0}$ is a function and $\textbf{X},\textbf{Y} \in \R^{n_1d}$ are vector fields, then
$$
H(f)(X,Y) = \Big( \sum_{i=1}^{n_0}\sum_{I,J = 1}^{n_1d} \textbf{H}^{\textnormal{op}}_{pIJi} \textbf{X}_I \textbf{Y}_J \textbf{f}_i \Big)_{p=1,...,n} \in \R^n
$$
computes the Hessian of $\textbf{f}$ in the directions $\textbf{X}$ and $\textbf{Y}$.

\paragraph{Hessian matrix.}

If we only want to compute the function $H(f)(\nabla x_i, \nabla x_j) \in \R^n$ for some $i,j$, then we can use the fact that
$$
\textbf{H}(\boldsymbol{\phi}_{i'})(\nabla \textbf{x}_i, \nabla \textbf{x}_j) 
= \big(
\boldsymbol{\Gamma}^\textnormal{coord/mix}_{pj_1j_2i'}
+ \boldsymbol{\Gamma}^\textnormal{coord/mix}_{pj_2j_1i'}
- \boldsymbol{\Gamma}^\textnormal{mix/coord}_{pi'j_1j_2}
\big)_{p=1,...,n} \in \R^n.
$$
At a fixed point $p$, $\boldsymbol{\Gamma}^\textnormal{coord/mix}_{pj_1j_2i'}
+ \boldsymbol{\Gamma}^\textnormal{coord/mix}_{pj_2j_1i'}
- \boldsymbol{\Gamma}^\textnormal{mix/coord}_{pi'j_1j_2}$ has dimension $d \times d \times n_0$, which, viewed as a linear map $\R^{n_0} \to \R^{d \times d}$, maps functions in $\R^{n_0}$ to their $d \times d$ Hessian matrix at $p$.

\subsection*{$(1,1)$-tensor form of the Levi-Civita connection.}

We often interpret the Levi-Civita connection as linear operator $\nabla_X : \mathfrak{X}(M) \to \mathfrak{X}(M)$, where $X \in \mathfrak{X}(M)$ is fixed, via
$$
\nabla_X Y := \nabla(X)(Y, \cdot)^\sharp \in \mathfrak{X}(M).
$$
This formulation of the Levi-Civita connection views it as a $(1,1)$-tensor that maps a vector field $X$ to a linear map on vector fields $\nabla_X : \mathfrak{X}(M) \to \mathfrak{X}(M)$.
We can discretise this as an $n_1d \times n_1d \times n_1d$ 3-tensor $\textbf{H}^{(1,1)}$ that maps vector fields $\textbf{X}$ in $\R^{n_1d}$ to linear maps $\R^{n_1d} \to \R^{n_1d}$ (i.e. $n_1d \times n_1d$ matrices).
We can evaluate $\boldsymbol{\nabla}^{\textnormal{weak,} (1,1)}$ with a minor adjustment from the above in the 6-tensor
\begin{align*}
\boldsymbol{\nabla}^{\textnormal{weak,} (1,1)}_{i_1j_1i_2j_2i'j'}
&= \inp{\p{i_1}\nabla{x_{j_1}}}{\nabla_{\p{i'}\nabla x_{j'}}(\p{i_2}\nabla{x_{j_2}})} \\
&= \int \nabla(\p{i'}\nabla x_{j'})(\p{i_2}\nabla x_{j_2}, \p{i_1}\nabla x_{j_1}) d\mu \\
&= \langle \p{i_1}\p{i_2} dx_{j_2}\otimes dx_{j_1}, \nabla(\p{i'}\nabla x_{j'})\rangle \\
&= \sum_{p=1}^n \textbf{U}_{pi_1} \textbf{U}_{pi_2} \boldsymbol\Gamma_p(\textbf{x}_{j_2}, \boldsymbol\phi_{i'}) \boldsymbol\Gamma_p(\textbf{x}_{j_1}, \textbf{x}_{j'}) \boldsymbol{\mu}_p \\
& \qquad
+ \frac{1}{2}\sum_{p=1}^n
 \textbf{U}_{pi_1} \textbf{U}_{pi_2} \textbf{U}_{pi'}
\big[ 
\boldsymbol{\Gamma}^\textnormal{coord/coord}_{pj_1j_2j'}
+ \boldsymbol{\Gamma}^\textnormal{coord/coord}_{pj_2j_1j'}
- \boldsymbol{\Gamma}^\textnormal{coord/coord}_{pj'j_1j_2}
\big] 
\boldsymbol{\mu}_p
\end{align*}
which has dimension $n_1 \times d \times n_1 \times d \times n_1 \times d$.
We reshape $\boldsymbol{\nabla}^{\textnormal{weak,} (1,1)}$ into an $n_1d \times n_1d \times n_1d$ 3-tensor, and obtain the strong form of the connection as $\boldsymbol{\nabla}^{(1,1)} = \big(\textbf{G}^1)\inv \boldsymbol{\nabla}^{\textnormal{weak,} (1,1)}$.
This is also an $n_1d \times n_1d \times n_1d$ 3-tensor, which we view as a map from vector fields in $\R^{n_1d}$ to the linear operators on vector fields $\R^{n_1d \times n_1d}$.
So, if $\textbf{X} \in \R^{n_1d}$ is a vector field, then $\boldsymbol{\nabla}^{(1,1)}\textbf{X} : \R^{n_1d} \to \R^{n_1d}$.
Explicitly, if $\textbf{X}, \textbf{Y} \in \R^{n_1d}$ then
$$
\nabla_X Y
= (\boldsymbol{\nabla}^{(1,1)}\textbf{X})(\textbf{Y})
= \Big( \sum_{I,J=1}^{n_1d} \boldsymbol{\nabla}^{(1,1)}_{KIJ}\textbf{Y}_I\textbf{X}_J \Big)_{K=1,...,n_1d} \in \R^{n_1d}.
$$

\end{document}